\documentclass[12pt, draft]{article}

\usepackage{pb-diagram,amsmath,amssymb}

\usepackage{amsthm}
\usepackage[T2A]{fontenc}
\usepackage[cp1251]{inputenc}
\usepackage[active]{srcltx}
\input prepictex
\newfont{\fiverm}{cmr5}
\input pictex
\input postpictex

\font\tenfrakt=eufm10   \font\tenscript=eusm10
\font\sevenfrakt=eufm7  \font\sevenscript=eusm7
\font\fivefrakt=eufm5   \font\fivescript=eusm5

\newfam\fraktfam
\textfont\fraktfam=\tenfrakt
\scriptfont\fraktfam=\sevenfrakt
\scriptscriptfont\fraktfam=\fivefrakt

\newfam\scriptfam
\textfont\scriptfam=\tenscript
\scriptfont\scriptfam=\sevenscript
\scriptscriptfont\scriptfam=\fivescript

\font\tenbbb=msbm10   \font\tenscript=msbm10
\font\sevenbbb=msbm7  \font\sevenscript=msbm7
\font\fivebbb=msbm5   \font\fivescript=msbm5

\newfam\bbbfam
\textfont\bbbfam=\tenbbb
\scriptfont\bbbfam=\sevenbbb
\scriptscriptfont\bbbfam=\fivebbb

\def\tr{{\rm tr}}

\newtheorem{theorem}{Theorem}
\newtheorem{definition}[theorem]{Definition}
\newtheorem{lemma}[theorem]{Lemma}
\newtheorem{proposition}[theorem]{Proposition}
\newtheorem{corollary}[theorem]{Corollary}

\def\myarrow{\ \hbox to 2em{\leaders
\hbox to 0.5ex{\hss\raise 0.55ex\hbox to 0.3ex{\hrulefill}\hss}
\hfill\,\llap{$>$}}\ }

\title{ Galois groups of co-abelian ball quotient covers    }
\author{ Azniv Kasparian  }
\date{      }

\begin{document}
\maketitle

\thispagestyle{empty}

\begin{abstract}

If $X'= \left( {\mathbb B} / \Gamma \right)'$ is a torsion free toroidal compactification of a discrete ball  quotient $X_o={\mathbb B} / \Gamma$ and
$\xi : (X', T = X'\setminus X_o) \rightarrow (X, D = \xi (T))$ is the blow-down of the $(-1)$-curves to the corresponding minimal model, then  $G'= Aut (X',T)$ coincides with the finite group $G=Aut(X,D)$. In particular, for an elliptic curve $E$  with endomorphism ring $R = End(E)$ and a split abelian surface $X = A = E \times E$, $G$ is a finite subgroup of $Aut(A) = \mathcal{T}_A \leftthreetimes GL(2,R)$, where $(\mathcal{T}_A,+) \simeq (A,+)$ is the translation group of $A$ and $GL(2,R) = \{ g \in R_{2 \times 2} \ \ \vert  \ \ \det(g) \in R^* \}$.

The present work classifies the finite subgroups  $H$ of $Aut (A = E \times E)$ for an arbitrary elliptic curve $E$. By the means of the  geometric invariants theory, it characterizes the Kodaira-Enriques types of $A/H \simeq \left( {\mathbb B} / \Gamma \right)'/H$, in terms of the fixed point sets of $H$ on $A$. The   abelian and  the  K3  surfaces $A/H$ are elaborated in \cite{KN}.  The first section provides necessary and sufficient conditions for $A/H$ to be a hyper-elliptic, ruled with elliptic base, Enriques or a rational surface. In such a way, it depletes the Kodaira-Enriques  classification of the finite Galois quotients $A/H$ of a split abelian surface $A = E \times E$. The second section   derives a  complete list of   the  conjugacy classes of the   linear automorphisms $g \in GL(2,R)$ of $A$ of finite order, by the means of their eigenvalues. The third section classifies the finite subgroups $H$ of $GL(2,R)$. The last section provides explicit generators and relations for the finite subgroups $H$ of $Aut(A)$ with K3, hyper-elliptic, rules with elliptic base or Enriques quotients $A/H \simeq \left( {\mathbb B} / \Gamma \right)'/H$.
\end{abstract}


Let
\[
{\mathbb B} = \{ z = (z_1, z_2) \in {\mathbb C} ^2  \ \ \vert \ \  |z_1|^2 + |z_2|^2 < 1 \} \simeq SU_{2,1} / S(U_2 \times U_1).
\]
be the complex $2$-ball. In \cite{BSA} Holzapfel settled the problem of the characterization of the projective surfaces, which are birational to an eventually singular ball quotient ${\mathbb B} / \Gamma $ by a lattice $\Gamma$ of $SU_{2,1}$. Note that if $\gamma \in \Gamma$ is a torsion element with isolated fixed points on ${\mathbb B}$ then ${\mathbb B} / \Gamma$ has isolated cyclic quotient singularity, which ought to be resolved in order to obtain a smooth open surface. The aforementioned resolution creates smooth rational curves of self-intersection $\leq -2$, which alter the local differential geometry of ${\mathbb B} / \Gamma$, modeled by ${\mathbb B}$. That is why we split the problem to the description of the minimal models $X_o$ of the smooth toroidal compactifications $X'_o = \left(  {\mathbb B} / \Gamma _o \right)'$ of torsion free $\Gamma _o$ and to the characterization of the birational equivalence classes of $X_o / H$ for appropriate finite automorphism groups $H$.
This reduction is based on   the fact that any finitely generated lattice $\Gamma$ in the simple Lie group $SU_{2,1}$ has a torsion free normal subgroup $\Gamma _o$ of finite index $[ \Gamma : \Gamma _o]$. Therefore ${\mathbb B} / \Gamma = \left( {\mathbb B} / \Gamma _o \right) / \left( \Gamma / \Gamma _o \right)$ and the classification of ${\mathbb B} / \Gamma$ is attempted by the classification of ${\mathbb B} / \Gamma _o$ and the finite automorphism groups $H = \Gamma / \Gamma _o$ of ${\mathbb B} / \Gamma _o$.

 According to the next proposition,  for any torsion free ball lattice $\Gamma _o$ and any $\Gamma < SU_{2,1}$, containing $\Gamma _o$ as a normal subgroup of finite index, the quotient group $\Gamma / \Gamma_o$ acts on the toroidal compactifying divisor $T = \left( {\mathbb B} / \Gamma _o \right)'\setminus \left( {\mathbb B} / \Gamma _o \right)$ and provides a compactification $\overline{{\mathbb B} / \Gamma} = \left( {\mathbb B} / \Gamma _o \right)'/ \left( \Gamma / \Gamma _o \right)$ of ${\mathbb B} / \Gamma$ with at worst isolated cyclic quotient singularities.
Therefore $H = \Gamma / \Gamma _o$ is a subgroup of $Aut ( X'_o, T)$. The birational equivalence classes of $\overline{{\mathbb B} / \Gamma }$ are to be described by the numerical invariants of the minimal resolutions $Y$ of the singularities of $\overline{{\mathbb B} / \Gamma }$. These can be computed by the means of the geometric invariant theory, applied to $X_o$ and a finite subgroup  $H$ of the biholomorphism group $Aut (X_o)$.

\begin{proposition}      \label{LatticeExtensionOverT}
Let $\Gamma$ be a lattice of $SU_{2,1}$ and $\Gamma _o$ be a normal torsion free subgroup of $\Gamma$ with finite index $[ \Gamma : \Gamma _o]$. Then the group $G = \Gamma / \Gamma _o$ acts on the toroidal compactifying divisor $T = \left( {\mathbb B} / \Gamma _o \right) '\setminus \left( {\mathbb B} / \Gamma _o \right)$ and the quotient $\left( {\mathbb B} / \Gamma _o \right)'/ G = \left( {\mathbb B} / \Gamma \right) \cup ( T/G) = \overline{{\mathbb B} / \Gamma }$ is a compactification of ${\mathbb B} / \Gamma$ with at worst isolated cyclic quotient singularities.
\end{proposition}

\begin{proof}

Recall that $p \in \partial _{\Gamma} {\mathbb B}$ is a $\Gamma$-rational boundary point exactly when the intersection $\Gamma \cap Stab _{SU_{2,1}} (p)$ is a lattice of $Stab _{SU_{2,1}} (p)$. Since $[ \Gamma : \Gamma _o] < \infty$, the quotient
\[
Stab _{SU_{2,1}} (p) / [ \Gamma \cap Stab _{SU_{2,1}} (p) ] =
\]
\[
= \left \{ Stab _{SU_{2,1}} (p) / \left[ \Gamma _o \cap Stab _{SU_{2,1}} (p) \right] \right \} /
\left \{ [ \Gamma \cap Stab _{SU_{2,1}} (p) / [ \Gamma _o \cap Stab _{SU_{2,1}} (p) ] \right \}
\]
has finite invariant volume exactly when $Stab _{SU_{2,1}} (p) / [ \Gamma _o \cap Stab _{SU_{2,1}} (p) ]$ has finite invariant volume. Therefore the $\Gamma$-rational boundary points coincide with the $\Gamma _o$-rational boundary points, $\partial _{\Gamma} {\mathbb B} = \partial _{\Gamma _o} {\mathbb B}$. It suffices to establish that the $\Gamma$-action on ${\mathbb B}$ admits local extensions on neighborhoods of the liftings of $T_i$ to complex lines through $p_i \in \partial _{\Gamma _o} {\mathbb B}$ with $Orb _{\Gamma _o} (p_i) = \kappa _i$. According to \cite{Mok}, the cusp $\kappa _i$, associated with the smooth elliptic curve $T_i$ has a neighborhood $N( \kappa _i) = T_i \times \Delta ^* (0,\varepsilon) \subset \left( {\mathbb B} / \Gamma _o \right)$ for a sufficiently small punctured disc $ \Delta ^* (0, \varepsilon ) = \{ z \in {\mathbb C} \ \ \vert \ \  |z| < \varepsilon \}$. The biholomorphisms $\gamma : {\mathbb B} \rightarrow {\mathbb B}$ from  $\Gamma$ extend to $\gamma : {\mathbb B} \cup \partial _{\Gamma _o} {\mathbb B} \rightarrow  {\mathbb B} \cup \partial _{\Gamma _o} {\mathbb B}$, as far as $\partial _{\Gamma _o} {\mathbb B}$ consists of isolated points. If $p_i \in \partial _{\Gamma _o} {\mathbb B}$, $\gamma (p_i) = p_j \in \partial _{\Gamma _o} {\mathbb B}$ and $\kappa _j = Orb _{\Gamma _o} (p_j)$ then there is a biholomorphism
\[
\gamma : N ( \kappa _i) \cap \gamma ^{-1} N(\kappa _j) \longrightarrow \gamma N(\kappa _i) \cap N(\kappa _j).
\]
For any $q_i \in T_i$ let $\Delta _{T_i} (q_i, \eta )$ be a sufficiently small disc on $T_i$, centered at $q_i$, which is contained in a $\pi _1 (T_i)$-fundamental domain, centered at $q_i$. One can view $\Delta _{T_i} (q_i, \eta) = \Delta _{\widetilde{T_i}} (q_i, \eta)$ as a disc on the lifting $\widetilde{T_i}$ of $T_i$ to a complex line through $p_i$. Then $N( \kappa _i, q_i) := \Delta _{\widetilde{T_i}} (q_i, \eta ) \times \Delta ^* (0, \varepsilon)$ is a bounded neighborhood of $q_i \in T_i$ on ${\mathbb B} /\Gamma _o$ and the holomorphic map
\[
\gamma : N( \kappa _i, q_i) \cap \gamma ^{-1} N(\kappa _j, q_j) \rightarrow \gamma N( \kappa _i, q_i) \cap N( \kappa _j, q_j) \subseteq N( \kappa _j, q_j) = \Delta _{\widetilde{T_j}} (q_j, \eta) \times \Delta ^* ( 0, \varepsilon)
\]
is bounded. Thus, $\gamma : {\mathbb B} \rightarrow {\mathbb B}$ is locally bounded around $\widetilde{T} = \sum\limits _{p_i \in \partial _{\Gamma _o} {\mathbb B}} \widetilde{T_i} (p_i)$ and admits a holomorphic extension $\gamma : {\mathbb B} \cup \widetilde{T} \rightarrow {\mathbb B} \cup \widetilde{T}$. This  induces a biholomorphism $\gamma \Gamma _o : \left( {\mathbb B} / \Gamma _o \right)'\rightarrow \left( {\mathbb B} / \Gamma _o \right)'$.

\end{proof}

The next proposition establishes that an arbitrary torsion free toroidal com\-pac\-ti\-fi\-ca\-tion  $\left( {\mathbb B} / \Gamma _o \right)'$ has finitely many  Galois quotients $\left( {\mathbb B} / \Gamma _o \right)'/ H = \overline{{\mathbb B} / \Gamma _H}$ with $\Gamma _H / \Gamma _o = H$. For torsion free $\left( {\mathbb B} / \Gamma _o \right)'$ with an abelian minimal model $X_o=A$, the result  is proved in \cite{WeakHolzConject}.
 Note also that \cite{APA} constructs an infinite series $\left \{ \left( {\mathbb B} / \Gamma _n \right)'\right \} _{n=1} ^{\infty}$ of mutually non-birational torsion free toroidal compactifications with abelian minimal models, which are finite Galois covers of a fixed $\left( \overline{ {\mathbb B} / \Gamma _{H_1}}, T(1) / H \right) = \left( \left( {\mathbb B} / \Gamma _n \right)', T(n) \right) / H_n$, $H_n \leq Aut \left( \left( {\mathbb B} / \Gamma _n \right)', T(n) \right)$.

\begin{proposition}   \label{EqualFiniteAutGroups}
Let $X'= \left( {\mathbb B} / \Gamma \right)' = \left( {\mathbb B} / \Gamma \right) \cup T$  be a torsion free toroidal  com\-pac\-ti\-fi\-ca\-tion  and $\xi : (X',T) \rightarrow (X = \xi (X'), D = \xi (T))$ be the blow-down of the $(-1)$-curves to the minimal model $X$ of $X'$. Then $Aut (X',T)$ is a finite group, which coincides with $Aut (X,D)$.
\end{proposition}

\begin{proof}

Let us denote $G = Aut (X,D)$, $G'=Aut (X',T)$ and observe that $X'$ is the blow-up of $X$ at the singular locus  $D^{\rm sing}$ of $D$. Since $D=\sum\limits _{i=1} ^h D_i$ has smooth elliptic  irreducible components $D_i$, the singular locus $D^{\rm sing} = \sum\limits _{1 \leq i < j \leq h} D_i \cap D_j$ and its complement $X \setminus D^{\rm sing}$ are $G$-invariant. We claim that the $G$-action extends to the exceptional divisor $E = \xi ^{-1} (D^{\rm sing})$ of $\xi$, so that $X'= (X \setminus D^{\rm sing}) \cup E$ is $G$-invariant. Indeed, for any $g \in G$ and $p \in D^{\rm sing}$ with $q=g(p)$, let us choose local holomorphic coordinates $x=(x_1,x_2)$, respectively, $y=(y_1,y_2)$  on   sufficiently small neighborhoods $N(p)$, $N(q)$  of $p$ and $q$  on $X$ with $gN(p) \subseteq N(q)$.
 Then $g: N(p) \rightarrow N(q) \subset {\mathbb C}^2$ consists of two local holomorphic functions $g=(g_1, g_2)$ on $N(p)$. By the very definition of a blow-up, \[
 \xi ^{-1} N (p) = \{ (x_1,x_2) \times [x_1 : x_2] \ \ \vert \ \  (x_1, x_2) \in N(p) \} \ \ \vert \mbox{  and }
 \]
 \[
 \xi ^{-1} N(q) = \{ (g_1(x), g_2 (x) ) \times [ g_1(x) : g_2 (x)] \ \ \vert \ \  g(x) = (g_1(x), g_2(x)) \in N(q) \},
\]
so that
\[
g: \xi ^{-1} N(p) \rightarrow \xi ^{-1} N(q),
\]
\[
 (x_1,x_2) \times [x_1 : x_2]  \mapsto (g_1(x), g_2(x)) \times [ g_1(x) : g_2(x) ]
 \]
extends the action of $g \in G$ to $\xi ^{-1}(D^{\rm sing} )$ and $G \subset Aut (X')$. Towards the $G$-invariance of $T$, note that the birational maps  $\xi : T_i \rightarrow \xi (T_i) = D_i$  of the smooth irreducible components $T_i$ of $T$ are biregular. Thus, the $G$-invariance of $D = \sum\limits _{i=1} ^h D_i$ implies the $G$-invariance of $T=\sum\limits _{i=1} ^h T_i$ and $G \subseteq G'=Aut(X',T)$. For the opposite inclusion $G'=Aut(X',T) \subseteq G=Aut(X,D)$ observe  that an arbitrary $g'\in G'$ acts on the union $E$ of the $(-1)$-curves on $X'$ and permutes the finite set  $\xi (E) = D^{\rm sing}$. In such a way, $g'$ turns to be a biregular morphism of $X = (X' \setminus E) \cup D^{\rm sing}$. The restriction of $g'$ on $T_i$ has image $g'(T_i) = T_j$ for some $1 \leq j \leq h$ and induces a biholomorphism $g': D_i \rightarrow D_j$. As a result, $g'\in G'$ acts on $D$ and $g'\in G=Aut(X,D)$.

In order to justify that $G = {\rm Aut} (X,D)$ is a finite group, let us consider the natural representation
\[
\varphi : G \rightarrow {\rm Sym} (D_1, \ldots , D_h)
 \]
 in the permutation group of the irreducible compo\-nents $D_1, \ldots , D_h$ of $D$. As far as the image $\varphi (G)$ is a finite group, it suffices to prove that the kernel $\ker \varphi$ is finite. Fix $p \in D^{\rm sing}$ and two local irreducible branches $U_o$ and $V_o$ of $D$ through $p$. If $U_o \subset D_i$ and $V_o \subset D_j$ for $i \neq j$ then consider the natural representation
 \[
 \varphi _o : \ker \varphi \rightarrow {\rm Sym} (D_i \cap D_j).
  \]
 The group homomorphism $\varphi _o$ has a finite image, so that the problem reduces to the finiteness of $G_o := \ker \left( \varphi _o \vert _{\ker \varphi} \right)$. By its  very definition, $G_o  \leq {\rm Stab} _G (p)$.  Let us move the origin of $D_i$ at $p$ and realize $G _o$ as  a subgroup of the finite cyclic group ${\rm End} ^* (D_i)$. After an eventual shrinking, $U_o$ is contained in a coordinate chart of $X$. Then $U = \cap _{g_o \in G_o} [g_o(U_o)]$ is a $G_o$-invariant neighborhood of $p$ on $D_i$. Similarly, pass to a $G_o$-invariant neighborhood $V \subset V_o$ of $p$ on  $D_j$, intersecting transversally $U$. Through any point $v \in V$ there is a local complex line $U(v)$, parallel to $U$. The union $W = \cup _{v \in V} U(v)$ is a neighborhood of $p$ on $X$, biholomorphic to $U \times V$. In holomorphic coordinates $(u,v) \in W$, one gets $G_o \leq {\rm End}^*(U) \times {\rm End}^*(V)$. Note that ${\rm End} ^*(U) \subseteq  {\rm End} ^*(D_i)$ and ${\rm End} ^*(D_i)$ is a finite cyclic group of order $1,2,3,4$ or $6$, so that   $|G_o| < \infty$.

\end{proof}

\newpage

\section{ Kodaira-Enriques classification of the finite Galois quotients of a split abelian surface}

Let $A = E \times E$ be the Cartesian square of an elliptic curve $E$.  For an arbitrary finite automorphism group $H \leq {\rm Aut} (A)$, we characterize the Kodaira-Enriques classification type of $A/H$ in terms of the fixed point set $Fix_A(H)$ of $H$ on $A$.  Partial results for this problem are provided by \cite{KN}. Namely, any $A/H$ is a finite cyclic Galois quotient of a smooth abelian surface $A/K$ or a normal model $A/K$ of a K3 surface.
  The surface $A/K$ is abelian exactly when $K = \mathcal{T}(H)$ is a translation group. The note \cite{KN} specifies  that  a necessary and sufficient condition for $A/[ \mathcal{T}(H) \langle h \rangle ]$  to have  irregularity $q(Y) = h^{1,0}(Y) =1$ is the presence of an entire  elliptic curve  in the fixed point set $Fix _A(h)$ of $h$. This result is similar to   S. Tokunaga and M. Yoshida's  study \cite{TY} of the discrete subgroups $\Lambda \leq {\mathbb C}^n \leftthreetimes U(n)$ with compact quotient ${\mathbb C}^n / \Lambda$. Namely, \cite{TY} establishes that if the linear part $\mathcal{L}(\Lambda)$ of such $\Lambda$ does not fix pointwise a complex line on ${\mathbb C}^2$, then ${\mathbb C}^n / \Lambda$ has vanishing irregularity.  Further, \cite{KN} observes  that if some $h \in H$ fixes pointwise an entire elliptic curve on $A$, then  the Kodaira dimension    $\kappa ( A/H) = - \infty$ drops down. Tokunaga and Yoshida prove the same statement for discrete subgroups $\Lambda \leq {\mathbb C}^n \leftthreetimes U(n)$ with compact quotient ${\mathbb C}^n / \Lambda$. The note \cite{KN} proves also  that if  $A/K$  is a K3 double cover of $A/H$ then $A/H$ is birational to an Enriques surface if and only if $A/K \rightarrow A/H$ is unramified.

      The present note establishes  that  an arbitrary   cyclic cover $\zeta ^K_H: A / K \rightarrow A/H$ of degree $\geq 3$ by  a K3 surfaces $A/K$ with isolated cyclic quotient singularities is ramified over a  finite set of points and $A/H$ is a  rational surface.
       If  a K3 surface $A/K$ is a double cover   $\zeta ^K _H : A/K \rightarrow A/H$  of $A/H$ then $A/H$ is birational to an Enriques surface exactly when $\zeta ^K_H$ is unramified. The quotients $A/H$ with ramified K3 double covers $\zeta ^K_H : A/K \rightarrow A/H$ are rational surfaces. If $H = \mathcal{T}(H) \langle h \rangle$ and the fixed points of  $\mathcal{L}(h)$ on $A$ contain an elliptic curve then $A/H$ is hyper-elliptic (respectively, ruled with an elliptic base) if and only if $H$ has not a fixed point on $A$ (respectively, $H$ has a fixed point on $A$, whereas $H$ has a pointwise fixed elliptic curve on $A$). If $H = \mathcal{T}(H)  \langle h \rangle$ and $\mathcal{L}(h)$ has isolated fixed points on $A$ then $A/H$  is a rational surface.

 In order to construct the normal subgroup $K$ of $H$, let us recall that the auto\-mor\-phism group ${\rm Aut} (A) = \mathcal{T}_A \rightthreetimes {\rm Aut} _{\check{o}_A} (A)$ of $A$ is a semi-direct product of the translation group $\mathcal{T} _A \simeq (A, +)$ and the stabilizer ${\rm Aut} _{\check{o}_A} (A)$ of the origin $\check{o}_A \in A$. Each $g \in {\rm Aut} _{\check{o}_A} (A)$ is a linear  transformation
 \[
 g = \left( \begin{array}{cc}
 a_{11}   &   a_{12}  \\
 a_{21}  &    a_{22}
 \end{array}    \right) \in GL_2 ({\mathbb C}),
 \]
 leaving invariant the fundamental group $\pi _1(A) = \pi _1(E) \times \pi _1(E)$ of $A = E \times E$. Therefore $a_{ij} \pi _1(E) \subseteq \pi _1(E)$ for all $1 \leq i, j \leq 2$ and  $ a_{ij} \in R$  for the endomorphism ring $R$ of $E$. The same holds for the entries of the inverse matrix
\begin{equation}      \label{inverse}
g^{-1} = \frac{1}{a_{11} a_{22} - a_{12}a_{21}} \left(  \begin{array}{rr}
a_{22}  &  -a_{12}  \\
-a_{21}  &  a_{11}
\end{array}   \right) \in {\rm Aut} _{\check{o}_A} (A).
\end{equation}
 Now, $\det (g)\in R$ and  $\det (g^{-1}) = ( \det(g)) ^{-1} \in R$ imply that  $\det (g) \in R^*$ is a unit. Thus, ${\rm Aut} _{\check{o}_A} (A)$ is contained in
\[
Gl (2, R) := \{ g \in ( R) _{2 \times 2} \ \ \vert \ \  \det (g) \in R ^* \}.
\]
The opposite inclusion $Gl(2, R ) \subseteq {\rm Aut} _{\check{o}_A} (A)$ is clear from (\ref{inverse}) and ${\rm Aut} _{\check{o}_A} (A) = Gl(2,  R)$.

 The map $\mathcal{L} : {\rm Aut}(A) \rightarrow   Gl (2, R)$, associating to $g \in {\rm Aut} (A)$ its linear part $\mathcal{L}(g) \in Gl (2, R)$ is a group homomorphism with kernel $\ker (\mathcal{L})  = \mathcal{T}_A$. Denote by $\mathcal{O}_{-d}$ the integers ring of an imaginary quadratic number field ${\mathbb Q} (\sqrt{-d})$. The determinant $\det : Gl( 2,R) \rightarrow  R^*$ is a group homomorphism in the cyclic units group
 \[
 R^* = \langle \zeta _{-d} \rangle \simeq
 \begin{cases}
 {\mathbb C}_2 &  \text{ for $R \neq {\mathbb Z}[i], \mathcal{O}_{-3}$,} \\
 {\mathbb C}_4  &  \text{ for $R = {\mathbb Z}[i] = \mathcal{O}_{-1}$,} \\
 {\mathbb C}_6 & \text{ for $R = \mathcal{O}_{-3}$ }
\end{cases}
\]
of order $o(R )$. For an arbitrary subgroup $H$ of ${\rm Aut} (A)$, let us denote by $K = K_H$ the kernel of the group homomorphism $\det \mathcal{L} : H \rightarrow  R^*$.  The image $\det \mathcal{L} (H) \leq ( R^*, .)$ is a cyclic group of order $m$, dividing $o( R^*)$, i.e., $\det \mathcal{L} (H) = \langle \zeta _{-d} ^k \rangle$ for some natural divisor  $k = \frac{o ( R^* )}{m}$ of $o(R^*)$. For an arbitrary $h_0 \in H$ with $\det \mathcal{L} (h_0) = \zeta _{-d} ^k$ the first homomorphism theorem reads as
\[
\{ K_H, h_0 K_H, \ldots , h_0 ^{m-1} K_H \} = H / K_H \simeq \langle \zeta _{-d} ^k \rangle =
\{ 1, \zeta _{-d} ^k, \zeta _{-d} ^{2k}, \ldots , \zeta _{-d} ^{(m-1)k} \}.
\]
Therefore $H = K_H \langle h_0 \rangle$ is a  product of $K_H = \ker ( \det \mathcal{L} \vert _H )$ and the  cyclic subgroup $\langle h_0 \rangle$ of $H$.

Denote by $E_1(H)$ the set of $h \in H$, whose linear parts $\mathcal{L}(h) \in GL_2 (R)$ have eigenvalue $1$ of multiplicity $1$. In other words, $h \in E_1(H)$ exactly when $\mathcal{L}(h)$ fixes pointwise an elliptic curve on $A$ through the origin $\check{o}_A$. Put $E_0(H)$ for the set of $h \in H$, whose linear parts have no eigenvalue $1$. Observe that $h \in E_0(H)$ if and only if $\mathcal{L}(h) \in GL(2, R)$ has isolated fixed points on $A$.

An automorphism $h \in H \setminus \{ {\rm Id} \}$ is called a reflection if fixes pointwise an elliptic curve on $A$. We claim that $h \in H$ is a reflection if and only if $h \in E_1(H)$ and $h$ has a fixed point on $A$. Indeed, if $h$ fixes an elliptic curve $F$ on $A$, then one can move the origin $\check{o}_A$ of $A$ on $F$, in order to represent $h$ by a linear transformation $h = \mathcal{L}(h) \in GL ( 2,R ) \setminus \{ {\rm Id} \} = E_1( GL(2,R)) \cup E_0 ( GL (2,R ))$. Any $h  = \mathcal{L}(h) \in E_0 ( GL (2,R))$ has isolated fixed points on $A$, so that $h = \mathcal{L}(h) \in E_1 (H)$ and $Fix _A(h) \neq \emptyset$. Conversely, if $h \in E_1(H)$  and $Fix_A(h) \neq \emptyset$, then after moving the origin  of $A$ at $\check{o}_A \in Fix_A(h)$, one attains $h = \mathcal{L}(h)$. Thus, $h$ fixes pointwise an elliptic curve on $A$ or $h$ is a reflection.

Towards the complete classification of the Kodaira-Enriques type of $A/H$, we use the following results from \cite{KN}:

\begin{proposition}   \label{AbelianK3}
(i) {\rm (cf. Corollary 5 from \cite{KN})} The quotient $A/H$ of $A = E \times E$ by a finite automorphism group $H$ is an abelian surface if and only if $H = \ker ( \mathcal{L} \vert _H ) = \mathcal{T}(H)$ is a translation group.

(ii) {\rm (Lemma 7 from \cite{KN})} The quotient $A/H$ is birational to a K3 surface if and only if $H = \ker ( \det \mathcal{L} \vert _H )$ and $H \varsupsetneq \ker ( \mathcal{L} \vert _H ) = \mathcal{T}(H)$.
\end{proposition}

\begin{proposition}   \label{PRAndQ1}
(i){\rm (cf. Lemma 11 from \cite{KN})}  If a finite automorphism group $H \leq {\rm Aut}(A)$ contains a reflection then $A/H$ is of Kodaira dimension $\kappa ( A/H) = - \infty$.

(ii) {\rm (cf. Proposition 12 from \cite{KN})} A smooth model $Y$ of $A/H$ is of irregularity $q(Y) = h^{1,0}(Y) =1$ if and only if
 $H = \mathcal{T}(H)  \langle h \rangle$ is a product of its normal  translation subgroup $\mathcal{T}(H) = \ker ( \mathcal{L} \vert _H)$ and a cyclic group
  $\langle h \rangle$, generated by $h \in E_1(H)$.
\end{proposition}

From now on, we consider only  subgroups $H \leq {\rm Aut}(A,T)$ with $\det \mathcal{L}(H) \neq \{ 1 \}$ and distinguish between translation
 $K = \ker ( \det \mathcal{L} \vert _H ) = \ker ( \mathcal{L} \vert _H ) = \mathcal{T}(H)$ and non-translation
 $K = \ker ( \det \mathcal{L} \vert _H ) \supsetneq \ker ( \mathcal{L} \vert _H) = \mathcal{T}(H)$.
Any $h \not \in K = \ker ( \det \mathcal{L} \vert _H)$ belongs to $E_1(H)$ or to $E_0(H)$.

\begin{proposition}    \label{AbelianCoverL1}
Let $H = \mathcal{T}(H) \langle h \rangle$ be a product of its (normal) translation subgroup $\mathcal{T}(H) = \ker ( \mathcal{L} \vert _H)$ and a cyclic group $\langle h \rangle$, generated by $h \in E_1 (H)$. Then:

(i) the fixed point set $Fix_A(H) = \emptyset$ of $H$ on $A$ is empty if and only if $A/H$ is a smooth hyper-elliptic surface;

(ii) the fixed point set $Fix_A(H) \neq \emptyset$ is non-empty if and only if $A/H$ is a smooth ruled surface with an elliptic base. If so, then $Fix_A(H)$ is of codimension $1$ in $A$.
\end{proposition}

\begin{proof}

According to Proposition \ref{PRAndQ1} (ii), $H = \mathcal{T}(H) \langle h \rangle$ with $h \in E_1 (H)$ if and only if any smooth model $Y$ of $A/H$ has irregularity $q(Y) = h^{1,0} (Y) =1$. More precisely, $Y$ is a hyper-elliptic surface or a ruled surface with an elliptic base.

If $Fix _A(H) = \emptyset$ then $A \rightarrow A/H$ is an unramified cover and the Kodaira dimension $\kappa (A/H) = \kappa (A) = 0$. Therefore $A/H$ is hyper-elliptic.

Suppose that there is an $H$-fixed point $p \in Fix _A(H)$ and move the origin $\check{o}_A$ of $A$ at $p$. For any $h_1 \in Stab _H( \check{o} _A) \setminus \{ Id_A \}$ one has $\check{o}_A = h_1 ( \check{o}_A) = \tau ( h_1) \mathcal{L}(h_1) ( \check{o}_A) = \tau (h_1) ( \check{o}_A)$, so that $h_1$ has trivial translation part $\tau (h_1) =  \tau _{\check{o}_A}$. As a result, $h_1 = \mathcal{L} (h_1) \in E_1(H) \setminus \{ Id_A \}$ is a reflection and fixes pointwise an elliptic curve on $A$.  In particular, $Fix_A(H)$ is of complex  codimension $1$.  If
\[
\mathcal{L}(h) = \left( \begin{array}{cc}
1  &  0  \\
0  &  \lambda _2(h)
\end{array}  \right) \quad  \mbox{  with  } \ \  \lambda _2(h) \neq 1
\]
then
\[
h_1 = \left( \begin{array}{cc}
1  &  0  \\
0   &  \lambda _2 (h) ^i
\end{array}  \right) \quad \mbox{  with  }  \ \ i \in {\mathbb Z}, \ \ \lambda _2 (h) ^i \neq 1.
\]
By Proposition \ref{PRAndQ1} (i), the quotient $A / \langle h_1 \rangle $ by the cyclic  group  $\langle h_1 \rangle$, generated by the reflection $h_1 = \mathcal{L}(h_1) \in E_1(H)$ is of Kodaira dimension
 $\kappa ( A / \langle h_1 \rangle ) = - \infty$. Along the finite (not necessarily Galois) cover $A / \langle h_1 \rangle \rightarrow A/H$, one has $\kappa ( A / \langle h_1 \rangle ) \geq \kappa (A / H)$, whereas $\kappa ( A/H) = - \infty$ and $A/H$ is birational to a ruled surface with an elliptic base. Note that all $h \in H$ with $Fix _A (h) \neq \emptyset$ are reflections, so that the quotient $A/H$ is a smooth surface by a result of Chevalley \cite{Chevalley}.

That proves the proposition, as far as the assumption $Fix_A(H) \neq \emptyset$ for a hyper-elliptic $A/H$ leads to a contradiction, as well as the assumption $Fix _A(H) = \emptyset$ for a ruled $A/H$ with an elliptic base.

\end{proof}

\begin{proposition}    \label{AbelianCoverE0}
Let $H = \mathcal{T}(H) \langle h \rangle $ for some
\[
h \in E_0 (H) = \{ h \in H \ \ \vert \ \  \lambda _j \mathcal{L}(h) \neq 1, \ \ 1 \leq j \leq 2 \}
\]
  with $\det \mathcal{L}(h) \neq 1$. Then $A/H$ is a rational surface.
\end{proposition}

\begin{proof}

We claim that $A/H$ with $A = E \times E$ is simply connected. To this end, let us denote by $R$ the endomorphism ring of $E$ and  lift $H$ to a subgroup $\widetilde{H}$ of the affine-linear group  $Aff ({\mathbb C}^2, R) = ( {\mathbb C}^2, +) \leftthreetimes GL(2,R)$ , containing  $( \pi _1(A), +)$ as a normal subgroup with quotient $\widetilde{H} / \pi _1 (A) = H$. Then
\[
A/H = \left[ {\mathbb C}^2 / \pi _1(A) \right] / \left[ \widetilde{H} / \pi _1(A) \right] \simeq {\mathbb C}^2 / \widetilde{H}.
\]
The universal cover $\widetilde{A} = {\mathbb C}^2$ of $A$ is a path connected, simply connected locally compact metric space and $\widetilde{H}$ is a discontinuous group of homeomorphisms of ${\mathbb C}^2$. That allows to apply Armstrong's result \cite{Armstrong} and conclude that
\[
\pi _1(A/H) = \pi _1 \left( {\mathbb C}^2 / \widetilde{H} \right) \simeq \widetilde{H} / \widetilde{N},
\]
where $\widetilde{N}$ is the normal subgroup of $\widetilde{H}$, generated by $\widetilde{h} \in \widetilde{H}$ with $Fix _{{\mathbb C}^2} ( \widetilde{h} ) \neq \emptyset$. There remains to be shown the coincidence $\widetilde{H}$ = $\widetilde{N}$.
In the case under consideration, let us choose generators $\tau _{(P_i, Q_i)}$ of $\mathcal{T}(H)$, $1 \leq i \leq m$ and fix liftings $(p_i,q_i) \in {\mathbb C} ^2 = \widetilde{A}$ of $( p_i + \pi _1(E), q_i + \pi _1(E) ) = (P_i, Q_i)$. If $\pi _1(E) = \lambda _1 {\mathbb Z} + \lambda _2 {\mathbb Z}$ for some $\lambda _1, \lambda _2 \in {\mathbb C}^*$ with $\frac{\lambda _2}{\lambda _1} \in {\mathbb C} \setminus {\mathbb R}$, then $\pi _1(A) = \pi _1(E) \times \pi _1(E)$ is generated by
\[
\Lambda _{11} = ( \lambda _1, 0), \ \ \Lambda _{12} = ( \lambda _2, 0), \ \ \Lambda _{21} = (0, \lambda _1) \ \
 \mbox{   and   } \ \ \Lambda _{22} = (0, \lambda _2).
\]
Let $\widetilde{h} = \tau _{(u,v)} \mathcal{L}(h) \in \widetilde{H}$ be a lifting of $h = \tau _{(U,V)} \mathcal{L}(h) \in H$, i.e., $( u + \pi _1(E), v + \pi _1(E)) = (U,V)$. Then $\widetilde{H}$ is generated by its subset
\[
S = \left \{ \Lambda _{ij}, \ \  \tau _{(p_k,q_k)}, \ \ \widetilde{h} \ \ \vert \ \  1 \leq i, j \leq 2, \ \ 1 \leq k \leq m \right \}.
\]
Since $\mathcal{L}(h)$ has eigenvalues $\lambda _1 \mathcal{L}(h) \neq 1$, $\lambda _2 \mathcal{L}(h) \neq 1$, for any $(a,b) \in {\mathbb C}^2$ the au\-to\-mor\-phism $\tau _{(a,b)} \mathcal{L}(h) \in Aut ( {\mathbb C}^2 )$ has a fixed point on ${\mathbb C}^2$. One can replace the generators $\Lambda _{ij}$ and $\tau _{(p_k,q_k)}$ of $\widetilde{H}$ by $\Lambda _{ij} \widetilde{h}$, respectively, $\tau _{(p_k, q_k)} \widetilde{h}$, since
\[
\langle S \rangle  \supseteq \{ \Lambda _{ij} \widetilde{h}, \ \ \tau _{(p_k,q_k)} \widetilde{h},  \ \ \widetilde{h} \ \ \vert \ \  1 \leq i, j \leq 2, \ \ 1 \leq k \leq m \}
 \]
 and $\Lambda _{ij}, \tau _{(p_k,q_k)} \in \langle  \{ \Lambda _{ij} \widetilde{h}, \ \ \tau _{(p_k,q_k)} \widetilde{h}, \ \  \widetilde{h} \ \ \vert \ \
  1 \leq i, j \leq 2, \ \  1 \leq k \leq m \} \rangle $.
 Thus
 \[
 \widetilde{H} = \langle \Lambda _{ij} \widetilde{h}, \ \ \tau _{(p_k,q_k)} \widetilde{h}, \ \ \widetilde{h} \ \ \vert \ \
  1 \leq i, j \leq 2, \ \ 1 \leq k \leq m \rangle
 \]
 coincides with $\widetilde{N}$, because $\widetilde{H}$ is generated by elements with fixed points. As a result, $\pi _1 (A/H) = \{ 1 \}$.

 Note that the simply connected surfaces $A/H$ are either rational or K3. According to $\det \mathcal{L}(h) \neq 1$, the quotient $A/H$ is not birational to a K3 surface, so that $A/H$ is a rational surface with isolated cyclic quotient singularities.

 \end{proof}

 \begin{proposition}    \label{K3Cover}
 Let $H  < Aut (A)$ be a finite subgroup of the form $H = K \langle h \rangle$ with $\mathcal{L} (K) < SL(2, R)$ and $\det \mathcal{L}(H) = \langle \det \mathcal{L} (h) \rangle \neq \{ 1 \}$.

 (i) The complement $H \setminus K$ has fixed points on $A$, $Fix _A(H \setminus K) \neq \emptyset$ if and only if $A/H$ is a rational surface;

 (ii) The complement $H \setminus K$ has no fixed points on $A$, $Fix _A(H \setminus K) = \emptyset$ if and only if $A/H$ is birational to an Enriques surface $Y$. If so, then the K3 universal cover $\widetilde{Y}$ of $Y$ is birational to $A/K$ and the index $[ H : K] =2$.
  \end{proposition}

 \begin{proof}

 First of all, the $H/K$-Galois cover $\zeta : A/K \rightarrow A/H$ is ramified  if and only if the complement $H \setminus K$ has a fixed point on $A$. More precisely, a point $Orb_K(p) \in A/K$, $p \in A$ is fixed by $hK \in H/K \setminus \{ K \}$ exactly when $h Orb _K (p) = Orb _K(p)$ or
 \begin{equation}   \label{FixedOrbit}
 \{ hk (p) \ \ \vert \ \  k \in K \} = \{ k(p) \ \ \vert \ \  k \in K \}.
 \end{equation}
 The condition (\ref{FixedOrbit}) implies the existence of $k_o \in K$ with $h(p) = k_o (p)$. Therefore $h _1 = k_o^{-1} h  \in Stab _H(p) \setminus K$ has a fixed point and
 \[
 h_1 K = (k_o^{-1} h) K = k_o ^{-1} (hK) = k_o ^{-1} K h = K h = h K,
  \]
  as far as $K$ is a normal subgroup of $H$. Conversely, if $h_1(p) = p$ for some $h_1 \in H \setminus K$ then $K_p = K h_1(p) = h_1 K(p)$ and the point $Orb _K (p) \in A/K$ is fixed by $h_1 K \in H /K$.

  Note that the presence of a covering $\zeta  : A/K \rightarrow A/H$ by a (singular) K3 model $A/K$ implies the vanishing $q(X) = h^{1,0}(X)$ of the irregularity of any smooth model $X$ of $A/H$, as far as $q(X) \leq q(Y) = 0$ for any smooth $H/K$-Galois cover $Y$ of $X$, birational to $A/K$.
  The smooth projective surfaces  $S$  with irregularity  $q(S)=0$ and Kodaira dimension $\kappa (S) \leq 0$ are the rational, K3 and Enriques $S$. Due to $\mathcal{L}(h) \neq 1$, the smooth model $X$ of  $A/H$  is not a K3 surface. Thus, $X$ is either an Enriques or a rational surface.

  If $Fix _A(H \setminus K) = \emptyset$ and $\zeta : A/K \rightarrow A/H$ in unramified, then $\kappa (X) = \kappa (Y) =0$ by \cite{Peters} and $X$ is an Enriques surface.

  Let us assume that $Fix _A (H \setminus K) \neq \emptyset$ and the minimal resolution $Y$ of the singularities of $A/H$ is an Enriques surface. Consider the minimal resolution $\rho _1 : Y \rightarrow A/K$ of the singularities of $A/K$ and the resolution $\nu _2:  X_2  \rightarrow A/H$ of $\zeta (A/H)^{\rm sing}$. Then there is a commutative diagram
  \begin{equation}   \label{FirstResolution}
  \begin{diagram}
  \node{A/K}   \arrow{s,r}{\zeta} \node{Y}   \arrow{w,t}{\rho _1}  \arrow{s,r}{\zeta _1}  \\
  \node{A/H}  \node{X_2}  \arrow{w,t}{\nu _2}
  \end{diagram}
  \end{equation}
  with $H/K$-Galois cover $\zeta _1$, ramified over the pull-back $\nu _2 ^{-1} B( \zeta)$ of the branch locus $B(\zeta) \subset A/H$ of $\zeta$. The minimal resolution $\mu _2 : X \rightarrow X_2$ of  the singularities $X_2^{\rm sing} = ( A/H) ^{\rm sing} \setminus \zeta ( A/K) ^{\rm sing}$ of $X_2$  and
  $\zeta _1 : Y \rightarrow X_2$ give rise to the fibered product commutative diagram
  \begin{equation}    \label{SecondResolution}
  \begin{diagram}
  \node{Y}  \arrow{s,r}{\zeta _1} \node{Z = Y \times _{X_2} X} \arrow{w,t}{{\rm  pr}_1}   \arrow{s,r}{\zeta _2}  \\
  \node{X_2}  \node{X}  \arrow{w,t}{\mu _2}
  \end{diagram},
  \end{equation}
  with ramified $H/K$-Galois cover $\zeta _2$ and birational ${\rm pr} _1$. Note that $Z$ is a smooth surface, since otherwise $\emptyset \neq {\rm pr} _1 ( Z^{\rm sing} ) \subseteq X^{\rm sing} = \emptyset$. Moreover, $Z$ is of type K3.
  Let us consider the universal double covering $U_X : \widetilde{X} \rightarrow X$ of $X$ by a K3 surface $\widetilde{X}$. Since $Z$ is simply connected and $U_X : \widetilde{X} \rightarrow X$ is unramified, the finite cover $\zeta _2 : Z \rightarrow X$ lifts to a morphism $\widetilde{\zeta} : Z \widetilde{X}$, closing the commutative diagram
  \begin{equation}    \label{K3UnivCover}
  \begin{diagram}
  \node{\mbox{}}  \node{\widetilde{X}}  \arrow{s,r}{U_X}  \\
  \node{Z}  \arrow{ne,l}{\widetilde{\zeta}}  \arrow{e,t}{\zeta _2} \node{X}
  \end{diagram}.
  \end{equation}
The finite ramified morphism $\zeta _2 = U_X \widetilde{\zeta}$ has finite ramified factor $\widetilde{\zeta}$, as far as the universal covering  $U_X : \widetilde{X} \rightarrow X$ is unramified. If $B(\widetilde{\zeta}) \subset Z$ is the branch locus of $\widetilde{\zeta}$ then the canonical divisor
\[
\mathcal{O}_Z = \mathcal{K}_Z = \widetilde{\zeta}^* \mathcal {K}_{\widetilde{X}} + B( \widetilde{\zeta}) =
\widetilde{\zeta} ^* \mathcal{O} _{\widetilde{X}} + B( \widetilde{\zeta}),
\]
which is an absurd. Therefore, $Fix _A (H \setminus K) \neq \emptyset$ implies that $A/H$ is a rational surface.

If $\zeta : A/K \rightarrow A/H$ is unramified and $A/H$ is an Enriques surface then $\zeta _1 : Y \rightarrow X_2$ from diagram (\ref{FirstResolution}) and $\zeta _2 : Z \rightarrow X$ from (\ref{SecondResolution}) are unramified.  As a result, $\widetilde{\zeta} : Z \rightarrow \widetilde{X}$ from diagram (\ref{K3UnivCover}) is a finite ramified cover of smooth simply connected surfaces, whereas $\deg ( \widetilde{\zeta}) =1$ and $Z$ coincides with the universal cover $\widetilde{X}$ of $X$. Thus, $\widetilde{X}$ is birational to $A/K$ and
\[
\deg ( \zeta) = \deg ( \zeta _1) = \deg ( \zeta _2) = \deg ( U_X) =2,
\]
so that $[H : K] = | H/K| = \deg (\zeta) =2$.

 \end{proof}

By the very construction, the surfaces $A/H$ and $\overline{ {\mathbb B} / \Gamma _H} = \left( {\mathbb B} / \Gamma \right)'/ H $ are si\-mul\-ta\-neous\-ly singular. The classical work \cite{Chevalley} of Chevalley  establishes that $A/H$ is singular if and only if there is $h \in H$, whose linear part $\mathcal{L}(h) \in GL(2,R)$ has eigenvalues $\{ \lambda _1 \mathcal{L}(h), \lambda _2 \mathcal{L} (h) \} \not \ni 1$. Thus, $A/H$ and $\overline{{\mathbb B} / \Gamma _H}$ are smooth exactly when birational  to a  hyper-elliptic or a  ruled surface with an  elliptic base.

 Let $T_i$ be an irreducible compo\-nent of  $T = \left( {\mathbb B} / \Gamma \right)'\setminus \left( {\mathbb B} / \Gamma \right)$  of ${\mathbb B} / \Gamma$. Then the irreducible component $Orb _H(T_i) / H$ of $T/H = \left( \overline{{\mathbb B} / \Gamma _H} \right)  \setminus \left( {\mathbb B} / \Gamma _H \right)$ is elliptic (respectively, rational) if and only if $Fix _A(H) \cap D_i = \emptyset$ (respectively, $Fix _A(H) \cap D_i \neq \emptyset$) for the image $D_i = \xi (T_i)$ of $T_i$ under the blow-down $\xi : \left( {\mathbb B} / \Gamma \right)'\rightarrow A$ of the $(-1)$-curves.


\newpage
\section{Linear automorphisms of finite order }

Throughout this section, let $R$ be the endomorphism ring of an elliptic curve $E$.  It is well known that $R = {
\mathbb Z} + f \mathcal{O}_{-d}$ for a natural number $f \in {\mathbb N}$, called the conductor of $E$ and integers ring $\mathcal{O}_{-d}$ of an imaginary quadratic number field ${\mathbb Q} ( \sqrt{-d})$. More precisely, $\mathcal{O}_{-d} = {\mathbb Z} + \omega _{-d} {\mathbb Z}$ with
\[
\omega _{-d} = \begin{cases}
\sqrt{-d}  & \text{ for $-d \not \equiv 1 ({\rm mod}  \ \ 4)$,} \\
\frac{1 + \sqrt{-d}}{2} & \text{for $-d \equiv 1 ({\rm mod} \ \ 4)$.}
\end{cases}
\]
and $R = {\mathbb Z} + f \omega _{-d} {\mathbb Z}$ for $R \neq  {\mathbb Z}$. In particular, $R$ is a subring of ${\mathbb Q} ( \sqrt{-d})$. We write $R \subset {\mathbb Q} (\sqrt{-d})$ both, for the case of $R = {\mathbb Z} + f \omega _{-d} {\mathbb Z}$ or $R = {\mathbb Z}$, without specifying the presence of a complex multiplication on $E$. (For $R = {\mathbb Z}$ one hat $R \subset {\mathbb Q} (\sqrt{-d})$ for $\forall d \in {\mathbb N}$.)

The automorphism group of the abelian surface $A = E \times E$ is a semi-direct product
\[
{\rm Aut} (A) = (A, +) \rtimes GL(2,R)
\]
of its translation subgroup $(A,+)$ and the isotropy group
\[
{\rm Aut} _{\check{o}_A} (A) = GL(2,R) = \{ g \in R_{2 \times 2} \ \ \vert \ \  \det (g) \in R^* \}
 \]
 of the origin $\check{o}_A \in A$.

 \begin{lemma}    \label{UnitsGroupR}
 Let $R$ be the endomorphism ring of an elliptic curve $E$. If $R$ is different from $\mathcal{O}_{-1} = {\mathbb Z} [i]$ and $\mathcal{O}_{-3}$ then
 \[
 R^* = \langle -1 \rangle = \{ \pm 1 \} = {\mathbb C} _2
 \]
 is the cyclic group of the square roots of the unity.

 If $R = {\mathbb Z}[i]$ is the ring of the Gaussian integers then
 \[
 R^* = \langle i \rangle = \{ \pm 1, \pm i \} = {\mathbb C} _4
 \]
 is the cyclic group of the roots of unity of order $4$.

 The units group of Eisensten integers $R = \mathcal{O}_{-3}$ is the cyclic group
 \[
 R^* =
 \langle e^{ \frac{ 2 \pi i}{6}} \rangle = \{ \pm 1, \ \ e^{ \pm \frac{ 2 \pi i}{3}}, \ \ e^{ \pm \frac{\pi i}{3}} \} = {\mathbb C}_6
 \]
 of the sixth roots of unity.
  \end{lemma}

\begin{proof}

Recall that the units group $\mathcal{O}_{-d}^*$ of the integers ring $\mathcal{O}_{-d}$ of an imaginary quadratic number field ${\mathbb Q}  (\sqrt{-d})$ is
\[
\mathcal{O}_{-d} ^* = \langle -1 \rangle \simeq {\mathbb C}_2 \ \ \mbox{  for   } \ \ d \neq 1,3 \ \ \mbox{ and}
\]
\[
\mathcal{O}_{-1} ^* =  {\mathbb Z} [i] ^* = \langle i \rangle = {\mathbb C}_4,
\]
\[
\mathcal{O}_{-3} ^* = \langle e^{ \frac{2\pi i}{6}}  \rangle = {\mathbb C}_6.
\]
The units group $R^*$ of the subring $R = {\mathbb Z} + f \mathcal{O}_{-d}$ of $\mathcal{O}_{-d}$ is a subgroup of $\mathcal{O}_{-d} ^*$, so that $R^* = \langle -1 \rangle \simeq {\mathbb C}_2$ for $R = {\mathbb Z}$ or $R = {\mathbb Z} + f \mathcal{O}_{-d}$ with $d \in {\mathbb N} \setminus \{ 1, 3 \}$, $f \in {\mathbb N}$. In the case of $R = {\mathbb Z} + f \mathcal{O}_{-1}$, the assumption $i \in R^*$ implies $R = \mathcal{O}_{-1}$ and happens only for the conductor $f=1$. Similarly, the existence of $e^{\frac{2 \pi i}{3}} \in R^* \setminus  \{ \pm 1 \}$ for $R = {\mathbb Z} + f \mathcal{O}_{-3}$ forces
\[
e^{\frac{2 \pi i}{3}} = - \frac{1}{2} + \frac{\sqrt{3}i}{2} = - 1 + \frac{1 + \sqrt{-3}}{2} = -1 + \omega _{-3} \in R^*,
\]
whereas $\omega _{-3} \in R$ and $R = \mathcal{O}_{-3}$.

\end{proof}

Towards the description of $g \in GL(2,R)$ of finite order, let us recall that the polynomials
\[
f(x) = x^n + a_1 x^{n-1} + \ldots + a_{n-1} x + a_n \in {\mathbb Z}[x]
\]
with leading coefficient $1$ are called monic.

\begin{definition}   \label{IntegralElement}
If $A$ is a subring with unity of a ring $B$ then $b \in B$ is integral over $A$ if annihilates  a monic polynomial
\[
f(x) = x^n + a_1 x^{n-1} + \ldots + a_{n-1} x + a_n \in A[x]
\]
with coefficients from $A$.
\end{definition}

It is well known (cf. \cite{AM}) that $b \in B$ is integral over $A$ if and only if the polynomial ring $A[b]$ is a finitely generated $A$-module.

\begin{definition}   \label{AlgebraicInteger}
The complex numbers $c \in {\mathbb C}$, which are integral over ${\mathbb Z}$ are called algebraic integers.
\end{definition}

Any algebraic integer $c$ is algebraic over ${\mathbb Q}$. If $g(x) \in {\mathbb Q} [x] \setminus {\mathbb Q}$ is a polynomial of minimal degree $k$ with a root $c$ then $g(x)$ divides any $h(x) \in {\mathbb Q} [x] \setminus {\mathbb Q}$ with $h(c) =0$. An arbitrary  $g'(x) \in {\mathbb Q} [x]$ of degree $k$ with a root $c$ is of the form $g'(x) = q g(x)$ for some ${\mathbb Q}^*$. The polynomials $q g(x)$ with arbitrary $q \in {\mathbb  Q}^*$ are referred to as minimal polynomials of $c$ over ${\mathbb Q}$. If $c$ is algebraic over ${\mathbb Q}$ then the ring of the  polynomials ${\mathbb Q}[c]$ of $c$ with rational coefficients coincides with the field ${\mathbb Q}(c)$ of the rational functions of $c$, ${\mathbb Q}[c] = {\mathbb Q}(c)$ and the degree $[ {\mathbb Q} (c) : {\mathbb Q}]$ equals the degree of a minimal polynomial of $c$ over ${\mathbb Q}$.

\begin{definition}    \label{DegreeOverQ}
If $c \in {\mathbb C}$ is algebraic over ${\mathbb Q}$, then $[ {\mathbb Q} (c) : {\mathbb Q}] = \dim _{\mathbb Q} {\mathbb Q}(c)$ is called the degree of $c$ over ${\mathbb Q}$.
\end{definition}

Let $c$ be an algebraic integer and $f(x) \in {\mathbb Z} [x] \setminus {\mathbb Z}$ be a monic polynomial of minimal degree with a root $c$. Then any $h(x) \in {\mathbb Z}[x]$ with $h(c)=0$ is divisible by $f(x)$. Thus, $f(x)$ is unique and referred to as the minimal integral relation of $c$. If $f(x) \in {\mathbb Z}[x] \setminus {\mathbb Z}$ is the minimal integral relation of $c \in {\mathbb C}$ and $g(x) \in {\mathbb Q} [x] \setminus {\mathbb Q}$ is a minimal polynomial of $c$ over ${\mathbb Q}$, then $g(x) = q f(x)$ for the leading coefficient $q = LC(g) \in {\mathbb Q}^*$ of $g(x)$. More precisely, $g(x)$ divides $f(x)$ and $f(x)$ is indecomposable over ${\mathbb Q}$, as far as it is indecomposable over ${\mathbb Z}$. In such a way, one obtains the following

\begin{lemma}   \label{DegreeOfMinIntegralRelation}
If $c \in {\mathbb C}$ is an algebraic integer, then the degree $\deg _{\mathbb Q} (c) = [ {\mathbb Q} (c) : {\mathbb Q}]$ of $c$ over ${\mathbb Q}$ equals the degree of the minimal integral relation
\[
f(x) = x^n + a_1 x^{n-1} + \ldots + a_{n-1} x + a_n \in {\mathbb Z}[x] \ \ \mbox{ of  } \ \ c.
\]
\end{lemma}

\begin{lemma}   \label{EigenvaluesArbitraryGL2R}
Let  $E$ be an elliptic curve, $R = End(E)$ and  $g \in GL(2, R)$.  Then any eigenvalue  $\lambda _1$  of $g$ is an algebraic integer of degree
 $1, 2$ or $4$ over ${\mathbb Q}$.
\end{lemma}

\begin{proof}

It suffices to observe that if $A  \subset B$ are subrings with unity of a ring $C$, $A$ is a Noetherian ring, $B$ is a  finitely generated $A$-module and $c \in C$ is integral over $B$, then $c$ is integral over $A$. Indeed, let $f \in {\mathbb N}$  be the conductor of $E$ and
\begin{equation}   \label{StandardGenerator}
\omega _{-d} = \begin{cases}
\sqrt{-d}  &  \text{ for $-d \not \equiv 1 ({\rm mod} 4)$,}   \\
\frac{1 + \sqrt{-d}}{2}  &  \text{ for $-d \equiv 1 ({\rm mod} 4)$.}
\end{cases}
\end{equation}
Then the integers ring ${\mathbb Z}$ is Noetherian and  the endomorphism ring
\[
R = {\mathbb Z} + f \mathcal{O}_{-d} = {\mathbb Z} + f \omega _{-d} {\mathbb Z}
\]
of $E$ is a free ${\mathbb Z}$-module of rank $2$. The eigenvalue $\lambda _1 \in {\mathbb C}$ of $g \in GL(2, R)$ is a root of the characteristic polynomial
\[
\mathcal{X}_g ( \lambda) = \lambda ^2 - \tr(g) \lambda + \det(g) \in R[ \lambda]
\]
of $g$, so that $\lambda _1$ is integral over $R$. According to the claim, $\lambda _1$ is integral over ${\mathbb Z}$ or $\lambda _1 \in {\mathbb C}$ is an algebraic integer. On one hand, the degree of $\lambda _1 $ over ${\mathbb Q} (\sqrt{-d})$ is
\[
\deg _{{\mathbb Q} ( \sqrt{-d})} ( \lambda _1) = [ {\mathbb Q} ( \sqrt{-d}, \lambda _1) : {\mathbb Q} ( \sqrt{-d}) ] = 1 \ \ \mbox{  or  } \ \ 2,
\]
so that
\[
[ {\mathbb Q} ( \sqrt{-d}, \lambda _1 ) : {\mathbb Q}] =
[ {\mathbb Q} ( \sqrt{-d}, \lambda _1 ) : {\mathbb Q} ( \sqrt{-d}) ] [ {\mathbb Q} ( \sqrt{-d}) : {\mathbb Q} ] = 2 \ \ \mbox{ or  } \ \  4.
\]
On the other hand, the inclusions
\[
{\mathbb Q} \subseteq {\mathbb Q} ( \lambda _1) \subseteq {\mathbb Q} ( \sqrt{-d}, \lambda _1)
\]
of subfields imply that
\[
[ {\mathbb Q} ( \lambda _1) : {\mathbb Q}] =
\frac{ [ {\mathbb Q} ( \sqrt{-d}, \lambda _1) : {\mathbb Q}]}{[ {\mathbb Q} ( \sqrt{-d}, \lambda _1) : {\mathbb Q} ( \lambda _1) ]}.
\]
Therefore, the degree $\deg _{\mathbb Q} ( \lambda _1) = [ {\mathbb Q}  ( \lambda _1 ) : {\mathbb Q} ]$ of $\lambda _1$ over ${\mathbb Q}$ is a divisor of the degree  $[ {\mathbb Q} ( \sqrt{-d}, \lambda _1 ) : {\mathbb Q} ]$ or $\deg _{\mathbb Q} ( \lambda _1) \in \{ 1, 2, 4 \}$.

In order to justify the claim, recall that $c \in C$ is integral over $B$ if and only if the polynomial ring $B [c] = B + Bc + \ldots + B c^{n-1}$ is a finitely generated $B$-module. If $B = A \beta _1 + \ldots + A \beta _s$  is a finitely generated $A$-module, then
\[
B [c] = \sum\limits _{i=1} ^s \sum\limits _{j=0} ^{n-1} A \beta _i c^j
\]
is a finitely generated $A$-module. Since $A$ is a Noetherian ring, the $A$-submodule $A[c]$ of $B[c]$ is a finitely generated $A$-module.

\end{proof}

Note that if $h = \tau _{(U,V)} \mathcal{L}(h) \in H \leq Aut(A)$ is an automorphism of $A = E \times E$  of finite order $r$ then
\[
h^r = \tau _{ \sum\limits _{s=0} ^{r-1} \mathcal{L}(h) ^s \left(  \begin{smallmatrix}
U  \\
V
\end{smallmatrix}  \right)} \mathcal{L}(h) ^r = Id
\]
implies that  $\sum\limits _{s=0} ^{r-1} \mathcal{L}(h) ^s \left(  \begin{smallmatrix}
U  \\
V
\end{smallmatrix}  \right) = \check{o}_A$ and $\mathcal{L}(h) ^r = I_2$.  In other words, the automorphisms $h \in Aut(A)$ of finite order  have linear parts $\mathcal{L}(h) \in GL(2,R)$ of finite order.

From now on, we concentrate on $g \in GL(2,R)$ of finite order.

\begin{proposition}   \label{EigenvaluesFiniteOrder}
If $R$ is the endomorphism ring of an elliptic curve $E$ and $g \in GL(2, R)$ is of finite order $r$, then $g$ is diagonalizable and  the eigenvalues $\lambda _j$ of $g$ are primitive roots of unity of degree $r_j = 1,2,3,4, 6,8$ or $12$.
\end{proposition}

\begin{proof}

Let us assume that $g \in GL(2,R)$ of finite order $r$ is not diagonalizable. Then there exists $S \in GL(2, {\mathbb C})$, reducing $g$ to its Jordan normal form
\[
J = S^{-1} g S = \left( \begin{array}{cc}
\lambda _1  &  1  \\
0  &  \lambda _1
\end{array}  \right).
\]
By an induction on $n$, one verifies that
\[
J^n = \left( \begin{array}{cc}
\lambda _1 ^n  &  (n-1) \lambda _1 ^{n-1}  \\
0  & \lambda _1 ^n
\end{array}  \right) \ \ \mbox{  for   } \ \ \forall n \in {\mathbb N}.
\]
In particular,
\[
I_2 = S^{-1} I_2 S = S^{-1} g^r S = (S^{-1} g S) ^r = J^r =
\left( \begin{array}{cc}
\lambda _1 ^r  &  (r-1) \lambda _1 ^{r-1}  \\
0  & \lambda _1 ^r
\end{array}  \right)
\]
is an absurd, justifying the diagonalizability of $g$.

If
\[
D = S^{-1} g S = \left( \begin{array}{cc}
\lambda _1  &  0  \\
0  &  \lambda _2
\end{array}  \right)
\]
is a diagonal form of $g$ then
\[
I_2 = S^{-1} I_2 S = S^{-1} g^r S = (S^{-1}gS)^r = \left( \begin{array}{cc}
\lambda _1 ^r  &  0  \\
0  &  \lambda _2 ^r
\end{array} \right)
\]
reveals that $\lambda _1$ and $\lambda _2$ are $r$-th roots of unity.

Thus, $\lambda _j$ are of finite order $r_j$, dividing $r$ and the least common multiple $m = LCM( r_1,r_2) \in {\mathbb N}$ divides $r$. Conversely,
\[
I_2 = \left( \begin{array}{cc}
\lambda _1 ^m  & 0  \\
0  & \lambda _2 ^m
\end{array}  \right) = (S^{-1} g S) ^m = S^{-1} g^m S
\]
implies that $g^m = S I_2 S^{-1} = I_2$, so that $r \in {\mathbb N}$ divides $m \in {\mathbb N}$ and $r=m$.

Let $\lambda _j \in {\mathbb C}^*$ be a primitive $r_j$-th root of unity. Then the cyclotomic polynomials $\Phi _{r_j}(x) \in {\mathbb Z}[x]$ are the minimal integral relations of $\lambda _j$. More precisely, the minimal integral relations $f_j(x) \in {\mathbb Z}[x] \setminus {\mathbb Z}$ of $\lambda _j$ are monic polynomials of degree $\deg _{\mathbb Q} ( \lambda _j)$. On the other hand, $\Phi _{r_j} (x) \in {\mathbb Z}[x] \setminus {\mathbb Z}$ are irreducible over ${\mathbb Z}$ and ${\mathbb Q}$. Therefore $\Psi _{r_j} (x)$ are minimal polynomials of $\lambda _j$ over ${\mathbb Q}$ and $\Psi _{r_j} (x) = q f_j(x)$ for some $q \in {\mathbb Q}^*$. As far as $\Phi _{r_j} (x)$ and $f_j(x)$ are monic, there follows $q=1$ and $\Phi _{r_j} (x) \equiv f_j(x) \in {\mathbb Z}[x]$.

 Recall Euler's function
\[
\varphi : {\mathbb N} \longrightarrow {\mathbb N},
\]
associating to each $n \in {\mathbb N}$ the number of the residues $0 \leq r \leq n-1$ modulo $n$, which are relatively prime to $n$. The degree of $\Phi _{r_j}(x)$ is $\varphi (r_j)$. If $r_j = p_1 ^{a_1} \ldots p_m ^{a_m}$ is the unique factorization of $r_j \in {\mathbb N}$ into a product of different  prime numbers $p_s$, then
\[
\varphi \left( p_1 ^{a_1} \ldots p_m ^{a_m} \right) = \varphi \left( p_1 ^{a_1} \right) \ldots \varphi \left( p_m ^{a_m} \right) =
p_1 ^{a_1-1} (p_1-1) \ldots p_m ^{a_m-1} (p_m-1).
\]
According to Lemma \ref{EigenvaluesArbitraryGL2R}, the algebraic integers $\lambda _j$ are of degree
\[
\deg _{\mathbb Q} ( \lambda _j) = \deg \Phi _{r_j} (x) = \varphi ( r_j) = 1, 2, \mbox{ or } 4.
\]
If $r_j$ has a prime divisor $p \geq 7$ then $\varphi (r_j)$ has a factor $p-1 \geq 6$, so that $\varphi ( r_j) > 4$. Therefore $r_j = 2^a 3^b 5^c$ for some non-negative integers $a,b,c$. If $c \geq 1$ then
\[
\varphi (r_j) = \varphi ( 2^a 3^b ) \varphi ( 5^c) = \varphi ( 2^a 3^b) 5^{c-1} .4 \in \{ 1,2,4 \}
\]
exactly when $\varphi ( r_j) = 4$, $c=1$ and $\varphi ( 2^a 3^b) =1$. For $b \geq 1$ one has
\[
\varphi ( 2^a 3^b) = \varphi ( 2^a) 3^{b-1} .2 > 1,
\]
so that $\varphi ( 2^a 3^b) =1$ requires $b=0$ and $\varphi ( 2^a) = 1$. As a result, $a=0$ or $1$ and $r_j =5$ or $10$, if $5$ divides $r_j$. From now on, let us assume that $r_j = 2^a 3^b$ with $a, b \in {\mathbb N} \cup \{ 0 \}$. If $b \geq 2$ then $\varphi ( r_j) = \varphi ( 2^a).3^{b-1}.2$ with $b-1 \geq 1$ is divisible by $3$ and cannot equal $1,2$ or $4$. Therefore $r_j = 2^a .3$ or $r_j = 2^a$ with $a \geq 0$. Straightforwardly,
\[
\varphi ( 2^a .3) = 2 \varphi ( 2^a) \in \{ 1,2,4 \}
\]
exactly when $\varphi ( 2^a) =1$ or $\varphi (2^a ) =2$. These amount to $a \in \{ 0,1,2 \}$ and reveal that $3,6,12$ are possible values for $r_j$. Finally, $\varphi (r_j) = \varphi (2^a) \in \{ 1,2,4\}$ for $r_j = 1,2,4$ or $8$. Thus, $\varphi (r_j) \in \{ 1,2,4\}$ if and only if
\[
r _j \in \{ 1,2,3,4,5,6,8,10,12 \}.
\]

In order to exclude $r_j=5$ and $r_j=10$ with $\varphi (5) = \varphi (10) =4$, recall that $\lambda _j$ is of degree $\deg _{{\mathbb Q}( \sqrt{-d})} ( \lambda _j) = [ {\mathbb Q} ( \sqrt{-d}, \lambda _j) : {\mathbb Q} ( \sqrt{-d}) ] \leq 2$ over ${\mathbb Q} ( \sqrt{-d})$, so that
\[
[ {\mathbb Q} ( \sqrt{-d}, \lambda _j) : {\mathbb Q}] =
 [ {\mathbb Q} ( \sqrt{-d}, \lambda _j) : {\mathbb Q} ( \sqrt{-d})][ {\mathbb Q} ( \sqrt{-d}) : {\mathbb Q} ] \leq 4.
\]
On the other hand,
\[
{\mathbb Q} \subset {\mathbb Q} ( \lambda _j) \subseteq {\mathbb Q} ( \sqrt{-d}, \lambda _j)
\]
implies that
\[
[ {\mathbb Q} ( \sqrt{-d}, \lambda _j) : {\mathbb Q}] =
[ {\mathbb Q} ( \sqrt{-d}, \lambda _j) : {\mathbb Q} ( \lambda _j)][{\mathbb Q} ( \lambda _j) : {\mathbb Q}] =
4 [ {\mathbb Q} ( \sqrt{-d}, \lambda _j) : {\mathbb Q} ( \lambda _j)] \geq 4,
\]
whereas $[ {\mathbb Q} ( \sqrt{-d}, \lambda _j) : {\mathbb Q} ] = [ {\mathbb Q} ( \lambda _j) : {\mathbb Q}] = 4$ and $[ {\mathbb Q} ( \sqrt{-d}, \lambda _j) : {\mathbb Q} ( \lambda _j)] =1$. Therefore ${\mathbb Q} ( \sqrt{-d}, \lambda _j) = {\mathbb Q} ( \lambda _j)$, so that $\sqrt{-d} \in {\mathbb Q} ( \lambda _j)$ and ${\mathbb Q} ( \sqrt{-d}) \subset {\mathbb Q} ( \lambda _j)$ with
\[
[ {\mathbb Q} ( \lambda _j) : {\mathbb Q} ( \sqrt{-d})] =
\frac{[ {\mathbb Q} ( \lambda _j) : {\mathbb Q}]}{[ {\mathbb Q} ( \sqrt{-d}) : {\mathbb Q}]} = \frac{4}{2} =2.
\]
As far as ${\mathbb Q} ( \sqrt{-d})$ and ${\mathbb Q} ( \lambda _j)$ are finite Galois extensions of ${\mathbb Q}$ (i.e., normal and separable), the subfield ${\mathbb Q} ( \sqrt{-d})$ of ${\mathbb Q} ( \lambda _1)$ of index $[ {\mathbb Q} ( \lambda _1) : {\mathbb Q} ( \sqrt{-d}) ] =2$ is the fixed point set of a subgroup $H$ of the Galois group $Gal ( {\mathbb Q} ( \lambda _j) / {\mathbb Q})$ with  $|H| =2$. The minimal polynomial of $\lambda _j$ over ${\mathbb Q}$ is the cyclotomic polynomial $\Phi _{r_j} (x) \in {\mathbb Z}[x]$ of degree $\deg  ( \Phi _{r_j}) = \varphi (r_j) =4$ for $r_j \in \{ 5, 10 \}$ and the Galois group
\[
Gal ( {\mathbb Q} ( \lambda _j) / {\mathbb Q}) \simeq {\mathbb Z}_{r_j}^*
\]
coincides with the multiplicative group ${\mathbb Z}_{r_j}^*$ of the congruence ring ${\mathbb Z}_{r_j}$ modulo $r_j$. More precisely, the roots of $\Phi _{r_j}(x)$ are $\{ \lambda _j ^s \ \ \vert \ \  s \in {\mathbb Z}_{r_j} ^* \}$ and for any $s \in {\mathbb Z}_{r_j}^*$  the correspondence $\lambda _j \mapsto \lambda _j ^s$ extends to an automorphism of ${\mathbb Q} ( \lambda _j)$, fixing ${\mathbb Q}$. The groups
\[
{\mathbb Z} _5 ^* = \{ \pm 1 ({\rm mod} 5), \ \  \pm 3 ({\rm mod} 5) \} = \langle 3 ({\rm mod} 5) \rangle =  \langle -3 ({\rm mod} 5) \rangle \simeq {\mathbb C}_4
\]
and
\[
{\mathbb Z}_{10}^* = \{  \{ \pm 1 ({\rm mod} 10), \ \ \pm 3 ({\rm mod} 10) \} = \langle 3 ({\rm mod} 10) \rangle =  \langle -3 ({\rm mod} 10) \rangle \simeq
 {\mathbb C}_4
\]
are cyclic and contain unique subgroups $H_5 = \langle -1 ({\rm mod} 5) \rangle$, respectively, $H_{10} = \langle -1 ({\rm mod} 10) \rangle$ or order $2$. Denote by $h$ the generator of $H_5$ or $H_{10}$ with $h( \lambda _j) = \lambda _j^{-1}$, $h \vert {\mathbb Q} = Id_{\mathbb Q}$. In both cases, the degree
\[
\deg _{{\mathbb Q}( \sqrt{-d})} ( \lambda _j) = [ {\mathbb Q} ( \lambda _j , \sqrt{-d}) : {\mathbb Q} ( \sqrt{-d})] =
[ {\mathbb Q}(\lambda _j) : {\mathbb Q}( \sqrt{-d})] =2,
\]
so that the characteristic polynomial
\[
\mathcal{X} _g(\lambda ) = \lambda ^2 - \tr(g) \lambda + \det(g) \in R[ \lambda ] \subset {\mathbb Q}( \sqrt{-d}) [\lambda]
\]
of $g$ is irreducible over ${\mathbb Q}( \sqrt{-d})$. In fact, $\mathcal{X} _g( \lambda )$ is a minimal polynomial of $\lambda _j$ over ${\mathbb Q} ( \sqrt{-d})$ and divides the cyclotomic polynomial $\Phi _{r_j} ( \lambda ) \in {\mathbb Z} [ \lambda] \subset {\mathbb Q} ( \sqrt{-d} ) [ \lambda]$ with $\Phi _{r_j} ( \lambda _j) =0$. In particular, the other eigenvalue $\lambda _{3-j}$ of $g$ is a root of $\Phi _{r_j}( \lambda)$ or a primitive $r_j$-th root of unity. That allows to express $\lambda _{3-j} = \lambda _j^t$ by some $t \in {\mathbb Z} _{r_j} ^*$. According to
\[
\lambda _j ^{t+1} = \lambda _j \lambda _j ^t = \lambda _j \lambda _{3-j}= \det(g) \in R^* \subset {\mathbb Q}( \sqrt{-d}) = {\mathbb  Q} ( \lambda _j) ^{ \langle h \rangle},
\]
one has
\[
\lambda _j ^{t+1} = h( \lambda _j ^{t+1}) = \lambda _j ^{-t-1} \ \ \mbox{  or  } \ \  \lambda _j ^{2(t+1)} =1.
\]
If $\lambda _j$ is a primitive fifth root of unity then $\lambda _j ^{2(t+1)} =1$ requires that $2(t+1)$ to be divisible by $5$. Since $GCD(2,5)=1$, $5$ is to divide $t+1$ or $t \equiv -1 ({\rm mod} 5)$. Similarly, if $\lambda _j$ is a primitive tenth root of unity then $10$ divides $2 (t+1)$, i.e., $2(t+1) = 10z$ for some $z \in {\mathbb Z}$. As a result, $5$ divides $t+1$ and $ t \equiv -1 ({\rm mod} 10)$. Thus, for any $r_1 \in \{ 5, 10 \}$ there follows $\lambda _{3-j} = \lambda _j ^t = \lambda _j ^{-1}$. Expressing
 $\lambda _j = e^{ \frac{2 \pi is}{r_j}}$ for some natural number $1 \leq s \leq r_j-1$, relatively prime to $r_j$, one observes  that
 \[
 \tr(g) = \lambda _j + \lambda _{3-j} = \lambda _j + \lambda _j^{-1} = e^{\frac{2 \pi is}{r_j}} + e^{ - \frac{2 \pi is}{r_j}} = 2 \cos \left( \frac{2 \pi s}{r_j} \right) \in R \cap {\mathbb R}.
 \]
We claim that $R \cap {\mathbb R} = {\mathbb Z}$. The inclusion ${\mathbb Z} \subseteq R \cap {\mathbb R}$ is clear. Conversely, let
\[
r \in {\mathbb R} \cap R = {\mathbb R} \cap ( {\mathbb Z} + f \omega _{-d} {\mathbb Z})
\]
for the conductor $f \in {\mathbb N}$ of $E$ and $\omega _{-d}$ from  (\ref{StandardGenerator}). In the case of $ -d \not \equiv 1 ({\rm mod} 4)$ there exist $a,b \in {\mathbb Z}$ with $r = a + f \sqrt{-d} b$. The complex number $a -r + f \sqrt{-d} b =0$ vanishes exactly when its real part $a-r=0$ and its imaginary part $f \sqrt{d} b=0$ are zero. Therefore $b=0$ and $r=a \in {\mathbb Z}$, i.e., ${\mathbb R} \cap R  \subseteq {\mathbb Z}$ for $-d \not \equiv 1 ({\rm mod} 4)$.

If $-d \equiv 1({\rm mod} 4)$ then
\[
r = a + f b \frac{( 1 + \sqrt{-d})}{2} \ \ \mbox{for some }  \ \ a, b \in {\mathbb Z}
\]
yields
\[
\left|  \begin{array}{c}
r = a + \frac{fb}{2}  \\
\frac{f \sqrt{d}}{2} b =0
\end{array}  \right.
\]
by comparison of the real and imaginary parts. As a result, again $b=0$ and $r=a \in {\mathbb Z}$, i.e., ${\mathbb R} \cap R \subseteq {\mathbb Z}$ for $-d \equiv 1 ({\rm mod} 4)$. That justifies ${\mathbb R} \cap R =  {\mathbb Z}$ and implies that $\tr(g) = 2 \cos \left( \frac{2 \pi s}{r_j} \right) \in {\mathbb Z}$. Bearing in mind the $\cos \left( \frac{2 \pi s}{r_j} \right) \in [ -1,1]$, one concludes
\begin{equation}  \label{ListCosValues}
\tr(g) = 2 \cos \left( \frac{2 \pi s}{r_j} \right) \in [ -2, 2] \cap {\mathbb Z} = \{ 0 , \pm1, \pm 2 \} \ \ \mbox{ or  }
\end{equation}
\[
\cos \left( \frac{2 \pi s}{r_j} \right) \in \left \{ 0, \pm \frac{1}{2}, \pm 1 \right \}.
\]
For a natural  number $1 \leq s \leq r_j-1$, one has $\frac{2 \pi s}{r_1} \in [ 0, 2 \pi )$. The solutions of $\cos (x) =0$ in $[0, 2 \pi )$ are $\frac{\pi }{2}$ and $\frac{3 \pi }{2}$, while $\cos (x) = \pm 1$ holds for $x \in \{ 0, \pi \}$. Finally, $\cos (x) = \pm \frac{1}{2}$ is satisfied by $x \in \left \{ \frac{\pi}{3}, \frac{2 \pi }{3}, \frac{4 \pi }{3}, \frac{5 \pi}{3}  \right \}$,  so that (\ref{ListCosValues}) implies
\begin{equation}  \label{ValuesList}
\frac{2 \pi s}{r_j} \in \left \{ 0, \ \ \frac{\pi}{2}, \pi, \ \  \frac{3 \pi}{2}, \ \  \frac{\pi}{3}, \ \ \frac{2 \pi }{3}, \ \  \frac{4 \pi }{3}, \frac{5 \pi }{3} \right \}.
\end{equation}
For $r_j =5$ or $10$ this is an absurd, so that
\[
r_j \in \{ 1,2,3,4,6,8,12 \}.
\]

\end{proof}

Now we are ready to describe the elements of $GL(2,R)$ of finite order, by specifying their eigenvalues $\lambda _1, \lambda _2$. The roots $\lambda _1, \lambda _2$ of the characteristic polynomial
\[
\mathcal{X} _g ( \lambda ) = \lambda ^2 - \tr(g) \lambda + \det(g) \in R[\lambda ]
\]
of $g$ are in a bijective correspondence with the trace $\tr(g) = \lambda _1 + \lambda _2 \in R $ and the determinant $\det(g) = \lambda _1 \lambda _2 \in R^*$ of $g$. Making use of Lemma \ref{UnitsGroupR}, we subdivide the problem to the description of finite order $g \in GL(2,R)$ with a fixed determinant $\det(g) \in R^*$. The traces of such $g$ take finitely many values and allow to list explicitly the eigenvalues of all $g \in GL(2,R)$ of finite order. The classification of the unordered pairs of eigenvalues $\lambda _1, \lambda _2$ of $g \in GL(2,R)$ of finite order is a more specific result than Proposition \ref{EigenvaluesFiniteOrder}. Note that the next classification of $\lambda _1, \lambda _2$ is derived independently of Proposition \ref{EigenvaluesFiniteOrder}.

Let us start with the case of $\det(g)=1$. The next proposition puts in a bijective correspondence the traces $\tr(g)$ of $g \in SL(2,R)$ with the orders $r$ of $g$.

\begin{proposition}   \label{ListK}
If $g \in SL(2,R)$ is of finite order $r$ then the trace
\begin{equation}    \label{TraceValues}
{\rm tr} (g) \in \{ \pm 2, \ \  \pm 1, \ \  0 \}.
\end{equation}
The eigenvalues $\lambda _1, \lambda _2$ of $g$ are of order
\begin{equation}    \label{OrdersList}
r_1 = r_2 = r \in \{ 1,  2,   3,   4,   6 \}.
\end{equation}
More precisely,

(i) ${\rm tr} (g) =2$ or $\lambda _1 = \lambda _2 =1$, $g = I_2$ if and only if $g$ is of order $1$;

(ii) ${\rm tr} (g) = -2$ or $\lambda _1 = \lambda _2 = -1$, $g = - I_2$ if and only if $g$ is of order $2$;

(iii) ${\rm tr} (g) = 1$ or $\lambda _1 = e^{\frac{\pi i}{3}},$ $\lambda _2 = e ^{ - \frac{\pi i}{3}}$ if and only if $g$ is of order $6$;

(iv) ${\rm tr} (g) = - 1$ or $\lambda _1 = e^{\frac{2 \pi i}{3}}$, $\lambda _2 = e^{ - \frac{2 \pi i}{3}}$ if and only if $g$ is of order $3$;

(v) ${\rm tr} (g) = 0$ or $\lambda _1 = i$, $\lambda _2 = -i$ if and only if $g$ is of order $4$.
\end{proposition}

\begin{proof}

If $g \in SL(2,R)$ is of order $r$ then the eigenvalues $\lambda _j$ of $g$ are of finite order $r_j$, dividing $r = LCM(r_1, r_2)$. According to
\[
1 = \det(g) = \lambda _1 \lambda _2,
\]
one has $\lambda _1 = e^{\frac{ 2 \pi i s}{r_1}}$, $\lambda _2 = e^{ - \frac{2 \pi is}{r_1}}$ for some natural number $1 \leq s \leq r_1-1$, relatively prime to $r_1$. Thus, $\lambda _2$ is a primitive $r_1$-th root and $r_1 =r_2 = LCM(r_1,r_2) =r$. As in the proof of Proposition \ref{EigenvaluesFiniteOrder},
\[
\tr(g) = \lambda _1 + \lambda _2 = e^{ \frac{ 2 \pi is}{r_1}} + e^{ - \frac{2 \pi is}{r_1}} =
2 \cos \left( \frac{2 \pi s}{r_1} \right) \in {\mathbb R} \cap R = {\mathbb Z}
\]
and $\cos \left( \frac{2 \pi s}{r_1} \right) \in [ -1,1]$ specify (\ref{TraceValues}). Consequently,
\[
\cos \left( \frac{2 \pi s}{r_1} \right) \in \left \{ 0, \ \ \pm  \frac{1}{2}, \ \  \pm1 \right \} \ \ \mbox{  and  }
\]
\[
\frac{2 \pi s}{r_1} \in \left \{ 0, \ \ \frac{\pi}{2}, \ \  \pi, \ \ \frac{3\pi}{2}, \ \  \frac{\pi}{3}, \ \ \frac{2 \pi}{3}, \ \ \frac{4 \pi}{3}, \ \
\frac{5 \pi}{3} \right \},
\]
as in (\ref{ValuesList}). Straightforwardly, $\lambda _1 = e^0 =1$ is of order $1$, $\lambda _1 = e^{ \pi i} = -1$ is of order $2$, $\lambda _1 \in \left \{ e^{\frac{\pi i}{2}}, e^{\frac{3 \pi i}{2}} \right \}$ are of order $4$, $\lambda _1 \in \left \{ e^{\frac{2 \pi i}{3}}, e^{\frac{4 \pi i}{3}} \right \}$ are of order $3$ and $\lambda _1 \in \left \{ e^{\frac{\pi i}{3}}, e^{\frac{5 \pi i}{3}} \right \}$ are of order $6$. That justifies (\ref{OrdersList}).

If $g$ is of order $r=1$ then $\lambda _1  \in {\mathbb C}^*$ is of order $r_1=1$, so that $\lambda _1 =1$. Consequently, $\lambda _2=1$ and $g=I_2$, as far as $I_2$ is the only conjugate of the scalar matrix $I_2$. The trace $\tr(g) = \tr(I_2) =2$. Conversely, if  $\lambda _1 = \lambda _2 =1$, then $g=I_2$ is of order $1$.

An automorphism $g \in SL(2,R)$ of order $r=2$ has eigenvalues $\lambda _1, \lambda _2 \in {\mathbb C}^*$ of order $2$, or $\lambda _1 = \lambda _2 = -1$. Consequently, $g = - I_2$ and $\tr(g) = -2$. Conversely, for  $\lambda _1 = \lambda _2 -1$ the matrix $g = - I_2$ is of order $2$.

Let us suppose that $g \in SL(2,R)$ is of order $3$. Then the eigenvalues $\lambda _1, \lambda _2$ of $g$ are of order $3$ or $\lambda _1 = e^{ \frac{2 \pi i}{3}}$, $\lambda _2 = e^{ - \frac{ 2 \pi i}{3}}$, up to a transposition. The trace $\tr(g) = \lambda _1 + \lambda _2 = -1$. Conversely, if  $\lambda _1 = e^{ \frac{2 \pi i}{3}}$, $\lambda _2 = e^{ - \frac{2 \pi i}{3}}$ then $r=r_1=r_2 =3$.

For $g \in SL(2,R)$ of order $4$ one has $\lambda _1, \lambda _2 \in {\mathbb C}^*$ of order $4$ or $\lambda _1 =i$, $\lambda _2 = -i$, up to a transposition.
The trace $\tr(g) = \lambda _1 + \lambda _2 = 0$. Conversely, for $\lambda _1=i$, $\lambda _2 = -i$ there follows $r=r_1=r_2=4$.

Suppose that $g \in SL(2,R)$ is of order $6$. Then $\lambda _1, \lambda _2 \in {\mathbb C}^*$ are of order $6$ or $\lambda _1 = e^{\frac{\pi i}{3}}$, $\lambda _2 = e^{ - \frac{ \pi i}{3}}$, up to a transposition. The trace $\tr(g) = \lambda _1 + \lambda _2 =1$. Conversely, the assumption $\lambda _1 = e^{\frac{\pi i}{3}}$, $\lambda _2 = e^{ - \frac{ \pi i}{3}}$ implies $r=r_1=r_2=6$.

Note that
\[
g_1 = \left( \begin{array}{rr}
1  &   1  \\
-3  &  -2
\end{array} \right), \ \
g_2 = \left( \begin{array}{rr}
1  &  -2  \\
1  & -1
\end{array}  \right), \ \
g_3 = \left( \begin{array}{rr}
2  &  1   \\
-3  & -1
\end{array}  \right) \in SL(2,{\mathbb Z}) \subseteq SL(2, R)
\]
with $\tr(g_1) = -1$, $\tr(g_2) =0$, $\tr(g_3) =1$ realize all the possibilities,  listed in the statement of the proposition.

\end{proof}

If $E$ is an elliptic curve with complex multiplication by an imaginary quadratic number field ${\mathbb Q} ( \sqrt{-d})$ and conductor $f \in {\mathbb N}$ then we denote the endomorphism ring of $E$ by
\[
R_{-d,f} = {\mathbb Z} + f \mathcal{O}_{-d} = {\mathbb Z} + f \omega _{-d} {\mathbb Z},
\]
where $\omega _{-d}$ is the non-trivial generator of $\mathcal{O}_{-d}$ as a ${\mathbb Z}$-module, given in (\ref{StandardGenerator}). If $E$ has no complex multiplication, we put
\[
R_{0,1} := {\mathbb Z}.
\]

\begin{proposition}   \label{ListS2}
Let $g \in GL(2, R_{-d,f})$ be a linear automorphism of $A = E \times E$  of order $r$, with $\det(g) = -1$ and eigenvalues $\lambda _1(g), \lambda _2(g) \in {\mathbb C}^*$.

(i) The automorphism  $g$ is of order $2$ if and only if  its trace is $\tr(g)=0$ or, equivalently, $\lambda _1 (g) = -1$, $\lambda _2 (g) =1$.

(ii) If $R_{-d,f} \neq {\mathbb Z}[i], \mathcal{O}_{-2}, \mathcal{O}_{-3}, R_{-3,2}$ then any $g \in GL(2, R_{-d,f}) \setminus SL(2,R)$ is of order $2$.

(iii) If $g \in GL(2, \mathcal{O}_{-2})$  is of order $r >2$  and  $\det(g) = -1$ then $r=8$ and the trace  $\tr(g) \in \{ \pm \sqrt{-2} \}$.

More precisely,

(a)  $\tr(g) = \sqrt{-2}$ if and only if $\lambda _1 (g) = e^{ \frac{\pi i}{4}}$, $\lambda _2 (g) = e^{ \frac{3 \pi i}{4}}$;

(b) $\tr(g) = - \sqrt{-2}$ if and only if $\lambda _1 (g) = e^{ \frac{5 \pi i}{4}}$, $\lambda _2 (g) = e^{ - \frac{\pi i}{4}}$.

(iv)  If $g \in GL(2, {\mathbb Z}[i])$ is of order $r >2$ and  $\det(g) = -1$, then $r \in \{4, 12 \}$ and the trace $\tr(g) \in \{ \pm i, \pm 2i \}$.

More precisely,

(a) $\tr(g) =2i$ exactly when $g = iI_2$;

(b) $\tr(g) = -2i$ exactly when $g = - i I_2$;

(c) $\tr(g) =i$ exactly when $\lambda _1(g) = e^{\frac{\pi i}{6}}$, $\lambda _2(g) = e^{ \frac{5 \pi i}{6}}$;

(d) $\tr(g) = -i$ exactly when $\lambda _1 (g) = e^{ \frac{ 7 \pi i}{6}}$, $\lambda _2 (g) = e^{ - \frac{\pi i}{6}}$.

(v) If $g \in GL(2, R_{-3,f})$ with  $R_{-3,f}  \in \{ R_{-3,1} = \mathcal{O}_{-3}, R_{-3,2} = {\mathbb Z} + \sqrt{-3} {\mathbb Z} \}$
 is of order $r >2$ and  $\det(g) = -1$ then $r=6$ and the trace $\tr(g) \in \{ \pm \sqrt{-3} \}$.

 More precisely,

(a) $\tr(g) = \sqrt{-3}$ if and only if $\lambda _1 (g) = e^{\frac{\pi i}{3}}$, $\lambda _2 (g) = e^{ \frac{2 \pi i}{3}}$;

(b) $\tr(g) = - \sqrt{-3}$ if and only if $\lambda _1 (g) = e^{ - \frac{2 \pi i}{3}}$, $\lambda _2 (g) = e^{ - \frac{ \pi i}{3}}$.
\end{proposition}

\begin{proof}

 The eigenvalues $\lambda _1(g), \lambda _2(g) \in {\mathbb C}^*$ of $g \in GL(2,R_{-d,f})$ with $\det(g) =-1$ are subject to  $\lambda _2 (g) = - \lambda _1 (g) ^{-1}$. More precisely, if  $\lambda _1 (g) = e^{ \frac{ 2 \pi si}{r_1}}$ is a primitive $r_1$-th root of unity then $\lambda _2 (g) = - e^{ - \frac{ 2 \pi si}{r_1}}$. The trace
\begin{equation}  \label{TraceForDet-1}
\tr(g) = \lambda _1 (g) + \lambda _2 (g) = e^{ \frac{ 2 \pi si}{r_1}} - e^{ - \frac{ 2 \pi si}{r_1}} = 2 i \sin \left( \frac{ 2 \pi s}{r_1} \right) \in R_{-d,f} \cap i{\mathbb R}.
\end{equation}
We claim that
\[
R_{-d,f} \cap i {\mathbb R} = \begin{cases}
f \sqrt{-d} {\mathbb Z} &  \text{ for $-d \not \equiv 1 ({\rm mod} 4)$ or $-d \equiv 1({\rm mod}4)$, $f \equiv 1({\rm mod}2)$,}   \\
\frac{f}{2} \sqrt{-d} {\mathbb Z}  &  \text{ for $-d \equiv 1({\rm mod} 4)$, $f \equiv 0 ({\rm mod} 2)$.}
\end{cases}
\]
Indeed, if $-d \not \equiv 1( {\rm mod} 4)$ then $\mathcal{O}_{-d} = {\mathbb Z} + \sqrt{-d} {\mathbb Z}$ and $R_{-d,f} = {\mathbb Z} + f \sqrt{-d} {\mathbb Z}$ contains $f \sqrt{-d}$, i.e., $f \sqrt{-d} {\mathbb Z} \subseteq R_{-d,f} \cap i {\mathbb R}$. Any $ir = a + bf \sqrt{-d} \in i {\mathbb R} \cap R_{-d,f}$ with $r \in {\mathbb R}$, $a,b \in {\mathbb Z}$ has imaginary part $r = bf \sqrt{d}$, so that $i {\mathbb R} \cap R_{-d,f} \subseteq f \sqrt{-d} {\mathbb Z}$ and $i {\mathbb R} \cap R_{-d,f} = f \sqrt{-d} {\mathbb Z}$.

Suppose that $-d \equiv 1 ({\rm mod} 4)$ and the conductor $f = 2k+1 \in {\mathbb N}$ is odd. Then $R_{-d,2k+1} = {\mathbb Z} + f \frac{(1 + \sqrt{-d})}{2} {\mathbb Z}$ contains $f \sqrt{-d} = - f + (2f) \frac{(1 + \sqrt{-d})}{2}$, so that $f \sqrt{-d} {\mathbb Z} \subseteq R_{-d,2k+1} \cap i {\mathbb R}$. Any $ir = a + \frac{bf}{2} (1 + \sqrt{-d})$ with $r \in {\mathbb R}$, $a,b \in {\mathbb Z}$ has real part $a + \frac{bf}{2} =0$ and imaginary part $r = \frac{bf}{2} \sqrt{d}$. Note that $\frac{bf}{2} = \frac{b(2k+1)}{2} = -a \in {\mathbb Z}$ is an integer only for an even $b = 2b_1$, $b_1  \in {\mathbb Z}$, so that $r = b_1 f \sqrt{d}$ and $i {\mathbb R} \cap R_{-d,2k+1} \subseteq f \sqrt{-d} {\mathbb Z}$. That justifies $i {\mathbb R} \cap R_{-d,2k+1} = f \sqrt{-d} {\mathbb Z}$ for $ -d \equiv 1 ({\rm mod} 4)$, $f \equiv 1({\rm mod}2)$.

Finally, for $ -d \equiv 1 ({\rm mod} 4)$ and an even conductor $f=2k \in {\mathbb N}$ the endomorphism ring $R_{-d,2k} = {\mathbb Z} + k (1 + \sqrt{-d}) {\mathbb Z}$ contains $k \sqrt{-d}$, so that $k \sqrt{-d} {\mathbb Z} \subseteq i {\mathbb R} \cap R_{-d,2k}$. Note that $ir = a + bk ( 1 + \sqrt{-d})$ with $r \in {\mathbb R}$, $a,b \in {\mathbb Z}$ has real part $a + bk =0$ and imaginary part $r = bk \sqrt{d}$, so that $i {\mathbb R} \cap R_{-d,2k} \subseteq k \sqrt{-d} {\mathbb Z}$ and $i {\mathbb R} \cap R_{-d,2k} = k \sqrt{-d} {\mathbb Z}$.

Now, (\ref{TraceForDet-1}) implies that
\[
2 \sin \left( \frac{2 \pi s}{r_1} \right) \in [ -2,2] \cap i ( R_{-d,f} \cap i {\mathbb R}) =
\]
\[
= \begin{cases}
[-2,2] \cap f \sqrt{d} {\mathbb Z} & \text{ for $-d \not \equiv 1 ({\rm mod} 4)$ or $-d \equiv 1 ({\rm mod} 4)$, $f \equiv 1 ({\rm mod} 2)$,}  \\
[-2,2] \cap \frac{f}{2} \sqrt{d} {\mathbb Z} & \text{ for $-d \equiv 1({\rm mod} 4)$, $f \equiv 0 ({\rm mod} 2)$.}
\end{cases}
\]
If $d \geq 5$ then $\sqrt{d} \geq \sqrt{5} > 2$ and $[-2,2] \cap f \sqrt{d} {\mathbb Z} = \{ 0 \}$ for $\forall f \in {\mathbb N}$ and $[-2,2] \cap \frac{f}{2} \sqrt{d} {\mathbb Z} = \{ 0 \}$ for $\forall f \in 2 {\mathbb N}$. Note that $\sin \left( \frac{2 \pi s}{r_1} \right) =0$ for some natural number $1 \leq s \leq r_1-1$ with $GCD(s,r_1) =1$ has unique solution $\frac{2 \pi s}{r_1} = \pi$, since $\frac{2 \pi s}{r_1} \in (0, 2 \pi )$. That implies $2s = r_1$, whereas $s$ divides $r_1$ and $s = GCD(s,r_1) =1$, $r_1=2$. Thus, $\lambda _1 = e^{ \frac{2 \pi i}{2}} = e^{\pi i} = -1$, $\lambda _2 = -(-1) =1$ and $g $ is conjugate to
\[
D_2 = \left( \begin{array}{rr}
-1  &  0   \\
0  &  1
\end{array} \right).
\]
In particular, $g$ is of order $2$. Note that the case of $g \in GL(2,R)$ with $\lambda _1 =-1$, $\lambda _2 =1$ is realized by the diagonal matrix $D_2 \in GL(2, {\mathbb Z}) \leq GL(2,R_{-d,f})$.

If $d=1$ and $f \geq 3$ then $2 \sin \left( \frac{2 \pi s}{r_1} \right) \in [-2,2] \cap f {\mathbb Z} = \{ 0 \}$ and $D_2$ is the only diagonal form  for $g$. For $d=2$ and $f \geq 2$ the intersection $[-2,2] \cap f \sqrt{2} {\mathbb Z} = \{ 0 \}$, so that any $g \in GL(2, R_{-2,f})$  with $f \geq 2$ and $\det(g)=-1$ is conjugate to $D_2$. If $d=3$ and $f = 2k+1 \geq 3$ then $[-2,2] \cap f \sqrt{3} {\mathbb Z} = \{ 0 \}$. Similarly, for $d=3$ and $f =2k \geq 4$ one has $[-2,2] \cap k \sqrt{3} {\mathbb Z} = \{ 0 \}$. In such a way, the existence of $g \in GL(2, R_{-d,f})$ with $\det(g) = -1$, $\tr(g) \neq 0$ requires $R_{-d,f}$ to be among
\[
R_{-1,1} = \mathcal{O}_{-1} = {\mathbb Z}[i],  \ \ R_{-1,2} = {\mathbb Z} + 2i {\mathbb Z},  \ \
R_{-2,1} = \mathcal{O}_{-2} = {\mathbb Z} + \sqrt{-2} {\mathbb Z},
\]
\[
R_{-3,1} = \mathcal{O}_{-3} = {\mathbb Z} + \frac{1 + \sqrt{-3}}{2} {\mathbb Z} \ \ \mbox{ or  } \ \
R_{-3,2} = {\mathbb Z} + 2 \left( \frac{1 + \sqrt{-3}}{2} \right) {\mathbb Z} = {\mathbb Z} + \sqrt{-3} {\mathbb Z}.
\]

The next considerations exploit the following simple observation: If $a,b$ are relatively prime natural numbers and $s,r_1$ are relatively prime natural numbers then $as = br_1$ if and only if $s=b$ and $r_1=a$. Namely, $b$ divides $as$ and $GCD(a,b)=1$ requires $b$ to divide $s$. Thus, $s = b s_1$ for some $s_1 \in {\mathbb N}$ and $as_1 = r_1$. Now $s_1$ is a natural common divisor of the relatively prime $s,r_1$, so that $s_1=1$, $s=b$ and $r_1=a$.

For $d=1$ and $f=2$ one has $2 \sin \left( \frac{2 \pi s}{r_1} \right) \in [ -2,2] \cap f {\mathbb Z} = \{ 0 , \pm 2 \}$. Let $\tr(g) =2i$ or
 $\sin \left( \frac{2 \pi s}{r_1} \right) =1$ for $r_1 \in {\mathbb N}$ and some natural number $1 \leq s \leq r_1-1$, $GCD(s,r_1)=1$. Then
$\frac{2 \pi s}{r_1} = \frac{\pi}{2}$ or $4s =r_1$. As a result, $s=1$, $r_1=4$ and $\lambda _1 = e^{ \frac{\pi i}{2}} =i$, $\lambda _2 = - e^{ - \frac{ \pi i}{2}} = i$. Now $g = iI_2$ as the unique matrix, conjugate to the scalar matrix $iI_2$. However, $i I_2 \not \in GL(2,R_{-1,2}) = GL(2, {\mathbb Z} + 2i {\mathbb Z})$, so that $g=iI_2$ is not a solution of the problem. For $\tr(g) = -2i$ one has $\sin \left( \frac{ 2 \pi s}{r_1} \right) = -1$, whereas $\frac{2 \pi s}{r_1} = \frac{ 3 \pi}{2}$ and $4s=3r_1$. Thus, $s=3$, $r_1 =4$ and $\lambda _1 = e^{ \frac{ 3 \pi i}{3}} = -i$, $\lambda _2 = - e^{ - \frac{3 \pi i}{3}} = -i$. That determines a unique $g = -i I_2$. But $-iI_2 \not \in GL(2,R_{-1,2}) = GL(2, {\mathbb Z} + 2i {\mathbb Z})$, so that $\lambda _1=1$, $\lambda _2 =-1$ are the only possible eigenvalues for $g \in GL(2, R_{-1,2})$ of finite order with $\det(g) =-1$.

In the case of $d=1$ and $f=1$, note that $ 2 \sin \left( \frac{2 \pi s}{r_1} \right) \in [-2,2] \cap {\mathbb Z} = \{ 0, \pm 1, \pm 2\}$. Besides $g \in GL(2, {\mathbb Z}[i])$ with $\det(g) = -1$, $\tr(g)=0$, one has $g = iI_2 \in GL(2, {\mathbb Z}[i])$ and $g = -iI_2 \in GL(2, {\mathbb Z}[i])$. The case of $\tr(g) = i$ corresponds to $\sin \left( \frac{ 2 \pi s}{r_1} \right) = \frac{1}{2}$ and holds for $\frac{2 \pi s}{r_1} = \frac{\pi}{6}$ or $\frac{2 \pi s}{r_1} = \frac{5 \pi}{6}$. Note that $12s = r_1$ implies $s=1$, $r_1=12$ and $\lambda _1  = e^{ \frac{\pi i}{6}} = \frac{\sqrt{3}}{2} + \frac{1}{2}i$, $\lambda _2 = - e^{ - \frac{ \pi i}{6}} = - \frac{\sqrt{3}}{2} + \frac{1}{2}i = e^{ \frac{5 \pi i}{6}}$. Thus, $g$ is of order $r = LCM(12,12) =12$. This possibility is realized, for instance, by
\[
g(i) = \left(  \begin{array}{rr}
1  &   1  \\
i  & ( -1+i)
\end{array}  \right) \in GL(2, {\mathbb Z}[i]) \ \ \mbox{  with  } \ \  \det(g(i)) = -1, \ \ \tr(g (i)) =i.
\]

If $12s=5r_1$ then $s=5$, $r_1=12$ and $\lambda _1 = e^{\frac{5 \pi i}{6}} = - \frac{\sqrt{3}}{2} + \frac{1}{2}i$, $\lambda _2 = - e^{ - \frac{ 5 \pi i}{6}} = \frac{\sqrt{3}}{2} + \frac{1}{2}i = e^{ \frac{\pi i}{6}}$, which was already obtained. Note that $\tr(g) = -i$ amounts to $\sin \left( \frac{2 \pi s}{r_1} \right) = - \frac{1}{2}$ and holds for $\frac{ 2 \pi s}{r_1} = \frac{ 7 \pi}{6}$ or $\frac{ 2 \pi s}{r_1} = \frac{11 \pi}{6}$. If $12s =7r_1$ then $s=7$, $r_1 =12$ and $\lambda _1 = e^{ \frac{ 7 \pi i}{6}} = - \frac{\sqrt{3}}{2} - \frac{1}{2}i$, $\lambda _2 = - e^{ - \frac{ 7 \pi i}{6}} = \frac{\sqrt{3}}{2} - \frac{1}{2} i = e^{ - \frac{\pi i}{6}}$ and $g$ is of order $r = LCM (12,12) =12$. Note that
\[
g(-i) = \left( \begin{array}{rr}
1  &  1  \\
-i  &  (-1-i)
\end{array} \right) \in GL(2, {\mathbb Z} [i] ) \ \ \mbox{  with  } \ \  \det ( g(-i)) = -1, \ \ \tr(g(-i)) = -i
\]
realizes the aforementioned possibility.

 In the case of $12s = 11r_1$ one has $s=11$, $r_1 =12$ and $\lambda _1 = e^{ \frac{ 11 \pi i}{6}} = \frac{\sqrt{3}}{2} - \frac{1}{2} i$, $\lambda _2 = - e^{ \frac{\pi i}{6}} = - \frac{\sqrt{3}}{2} - \frac{1}{2}i$, which is already listed as a solution. That concludes the considerations for $g \in GL(2, {\mathbb Z}[i])$ with $\det(g) = -1$.

If $d=2$ and $f=1$ then $2 \sin \left( \frac{ 2 \pi s}{r_1} \right) \in [ -2,2] \cap \sqrt{2} {\mathbb Z} = \{ 0 , \pm \sqrt{2} \}$. Note that $\sin \left( \frac{ 2 \pi s}{r_1} \right) = \frac{\sqrt{2}}{2}$ holds for $\frac{ 2 \pi s}{r_1} = \frac{ \pi}{4}$ or $\frac{ 2 \pi s}{r_1} = \frac{3 \pi}{4}$. The equality $r_1 = 8s$ implies $s=1$ and $r_1=8$. As a result, $\lambda _1 = e^{\frac{ \pi i}{4}} = \frac{\sqrt{2}}{2} + \frac{\sqrt{2}}{2} i$, $\lambda _2 = - e^{ - \frac{ \pi i}{4}} = - \frac{\sqrt{2}}{2} + \frac{\sqrt{2}}{2} i = e^{ \frac{ 3 \pi i}{4}}$. Observe that
\[
g( \sqrt{-2}) = \left( \begin{array}{rr}
1  &  1  \\
\sqrt{-2}  & ( -1 + \sqrt{-2})
\end{array}  \right) \in GL(2, \mathcal{O}_{-2}), \mathcal{O}_{-2} = {\mathbb Z} + \sqrt{-2} {\mathbb Z}
\]
with $\det ( g( \sqrt{-2})) = -1$, $\tr(g ( \sqrt{-2})) = \sqrt{-2}$ realizes the aforementioned possibility.
If $8s=3r_1$ then $s=3$, $r_1 =8$ and $\lambda _1 = e^{ \frac{ 3 \pi i}{4}} = - \frac{\sqrt{2}}{2} + \frac{\sqrt{2}}{2} i$,
$\lambda _2 = - e^{ - \frac{3 \pi i}{4}} = \frac{\sqrt{2}}{2} + \frac{\sqrt{2}}{2} i = e^{ \frac{\pi i}{4}}$. These eigenvalues have been already mentioned.

For $\sin \left( \frac{ 2 \pi s}{r_1} \right) = - \frac{\sqrt{2}}{2}$ there follows  $\frac{ 2 \pi s}{r_1} = \frac{ 5 \pi }{4}$ or
$\frac{ 2 \pi s}{r_1} = \frac{7 \pi}{4}$.
If $8s = 5r_1$ then $s=5$, $r_1 =8$ and
$\lambda _1 = e^{ \frac{ 5 \pi i}{4}} = - \frac{\sqrt{2}}{2} - \frac{\sqrt{2}}{2} i$,
$\lambda _2 = - e^{ - \frac{ 5 \pi i}{4}} = \frac{\sqrt{2}}{2} - \frac{\sqrt{2}}{2} i = e^{ - \frac{\pi i}{4}}$.
The corresponding automorphism $g$ is of order $r = LCM(8,8) =8$. Note that
\[
g ( - \sqrt{-2}) = \left( \begin{array}{rr}
1 &    1   \\
- \sqrt{-2}   & (-1 -\sqrt{-2})
\end{array}  \right) \in GL(2, \mathcal{O}_{-2})
\]
with $\det ( g( - \sqrt{-2})) = -1$, $\tr(g ( - \sqrt{-2})) = - \sqrt{-2}$.
realizes this possibility.
In the case of $8s = 7r_1$, one has $s=7$, $r_1=8$. The eigenvalues
$\lambda _1 = e^{ \frac{ 7 \pi i}{4}} = \frac{\sqrt{2}}{2} - \frac{\sqrt{2}}{2} i$,
$\lambda _2 = - e^{ - \frac{7 \pi i}{4}} = - \frac{\sqrt{2}}{2} - \frac{\sqrt{2}}{2} i$ were already obtained. That concludes the considerations for $d=2$.

If $d=3$ and $f=1$, note that $2 \sin \left( \frac{2 \pi s}{r_1} \right) \in [-2,2] \cap \sqrt{3} {\mathbb Z} = \{ 0, \pm \sqrt{3} \}$. Similarly, for $d=3$ and $f=2$ one has $2 \sin \left( \frac{ 2 \pi s}{r_1} \right) \in [-2,2] \cap \sqrt{3} {\mathbb Z} = \{ 0, \pm \sqrt{3} \}$. If $\sin \left( \frac{ 2 \pi  s}{r_1} \right) = \frac{\sqrt{3}}{2}$ then $\frac{2 \pi s}{r_1} = \frac{\pi}{3}$ or $\frac{2 \pi s}{r_1} = \frac{2 \pi}{3}$. In the case of $6s=r_1$ there follows $s=1$, $r_1=6$. The eigenvalues
$\lambda _1 = e^{ \frac{ \pi i}{3}} = \frac{1}{2} + \frac{\sqrt{3}}{2} i$,
$\lambda _2 = - e^{ - \frac{\pi i}{3}} = - \frac{1}{2} + \frac{\sqrt{3}}{2} i = e^{ \frac{2 \pi i}{3}}$ and $g$ is of order $r = LCM(6,3) =6$. The automorphism
\[
g( \sqrt{-3}) = \left( \begin{array}{rr}
1  &  1  \\
\sqrt{-3}   &  (-1 + \sqrt{-3})
\end{array}  \right) \in  GL(2, R_{-3,2} ) \leq GL(2, \mathcal{O}_{-3})
\]
with $\det(g( \sqrt{-3})) = -1$, $\tr(g ( \sqrt{-3})) = \sqrt{-3}$ realizes the aforementioned possibility.
 If $3s=r_1$ then $s=1$, $r_1=3$ and  $\lambda _1 = e^{ \frac{ 2 \pi i}{3}} = - \frac{1}{2} + \frac{\sqrt{3}}{2} i$,
$\lambda _2 = - e^{ - \frac{ 2 \pi i}{3}} = \frac{1}{2} + \frac{\sqrt{3}}{2} i = e^{\frac{\pi i}{3}}$,
which was already  obtained.

If $\sin \left( \frac{ 2 \pi s}{r_1} \right) = - \frac{\sqrt{3}}{2}$ then $\frac{ 2 \pi s}{r_1} = \frac{4 \pi}{3}$ or
$\frac{2 \pi s}{r_1} = \frac{5 \pi}{3}$. In the case of $3s=2r_1$ note that $s=2$, $r_1=3$ and
$\lambda _1 = e^{ \frac{ 4 \pi i}{3}} = - \frac{1}{2} - \frac{\sqrt{3}}{2} i$,
$\lambda _2 = - e^{ -  \frac{ 4 \pi i}{3}} = \frac{1}{2} - \frac{\sqrt{3}}{2} i = e^{ - \frac{\pi i}{3}}$.
 The automorphisms $g$ with such eigenvalues are of order $r = LCM(3,6) = 6$. In particular,
\[
g ( - \sqrt{-3}) = \left( \begin{array}{rr}
1  &  1  \\
- \sqrt{-3}  & (-1-\sqrt{-3})
\end{array}  \right) \in  GL(2, R_{-3,2}) \leq GL(2, \mathcal{O}_{-3})
\]
with $\det (g ( - \sqrt{-3})) =-1$, $\tr(g( - \sqrt{-3})) = - \sqrt{-3}$ realizes the  aforementioned possibility.

If $6s=5r_1$ then $s=5$, $r_1=6$ and
$\lambda _1 = e^{\frac{ 5 \pi i}{3}} = e^{ - \frac{\pi i}{3} } = \frac{1}{2} - \frac{\sqrt{3}}{2} i$,
$\lambda _2 = - e^{ \frac{\pi i}{3}} = - \frac{1}{2} - \frac{\sqrt{3}}{2} i = e^{\frac{4 \pi i}{3}}$.
These eigenvalues are already obtained. That concludes the considerations for $d=3$ and the description of all $g \in GL(2, R_{-d,f})$ with $\det(g) =-1$.

\end{proof}

\begin{proposition}    \label{ListS4+}
If $g \in GL(2, {\mathbb Z}[i])$ is of finite order $r$ and $\det(g) =i$ then
\[
\tr(g) \in \{ 0, \pm (1+i) \}, \quad r \in \{ 4, 8 \}.
\]
 More precisely,

(i) $\tr(g) =0$ or $\lambda _1 = e^{\frac{3 \pi i}{4}}$, $\lambda _2 = e^{ - \frac{ \pi i}{4}}$ if and only if $g$ is of order $8$;

(ii) if $\tr(g) = 1+i$  or $\lambda _1 =i$, $\lambda _2 =1$ then $g$ is of order $4$;

(iii) if $\tr(g) = -1-i$ or $\lambda _1 = -i$, $\lambda _2 = -1$ then $g$ is of order $4$.
\end{proposition}

\begin{proof}

If $\lambda _1 = e^{ \frac{ 2 \pi si}{r_1}}$ for the order $r_1 \in {\mathbb N}$ of $\lambda _1 \in {\mathbb C}^*$ and some natural number $1 \leq s < r_1$, $GCD(s,r_1) =1$, then $\lambda _2 = \det (g) \lambda _1 ^{-1} = i e^{ - \frac{ 2 \pi si}{r_1}}$. Therefore, the trace
\[
\tr(g) = \lambda _1 + \lambda _2 = \left[ \cos \left( \frac{2 \pi s}{r_1} \right)  + \sin \left( \frac{2 \pi s}{r_1} \right) \right] (1+i) =
\]
\[
= \sqrt{2} \sin \left( \frac{2 \pi s}{r_1} + \frac{\pi}{4} \right) (1+i) \in {\mathbb Z}[i] = {\mathbb Z} + i {\mathbb Z}
\]
if and only if the real part
\[
\sqrt{2} \sin \left( \frac{2 \pi s}{r_1} + \frac{\pi}{4} \right) \in {\mathbb Z} \cap [ - \sqrt{2}, \sqrt{2} ] = \{ 0, \pm 1 \}.
\]
As a result, $\tr(g) \in \{ 0, \pm (1+i) \}$. If $\tr(g) =0$ or, equivalently, $\sin \left( \frac{2 \pi s}{r_1} + \frac{\pi}{4} \right) = 0$ for $\frac{2 \pi s}{r_1} + \frac{\pi}{4} \in \left( \frac{\pi}{4}, \frac{9 \pi}{4} \right)$ then $\frac{2 \pi s}{r_1} + \frac{\pi}{4} = \pi$ or $\frac{ 2 \pi s}{r_1} + \frac{\pi}{4} = 2 \pi$.  For $\frac{ 2s}{r_1} = \frac{3}{4}$ there follows $8s=3r_1$ and $s=3$, $r_1=8$. As a result,
$\lambda _1 = e^{ \frac{ 3 \pi i}{4}} = - \frac{\sqrt{2}}{2} + \frac{\sqrt{2}}{2} i$,
$\lambda _2 = i e^{ - \frac{ 3 \pi i}{4}} = \frac{\sqrt{2}}{2} - \frac{\sqrt{2}}{2} i = e^{ - \frac{\pi i}{4}}$ and $g$ is of order $r = LCM(8,8) =8$.
 For instance,
\[
g_i (0) = \left( \begin{array}{rr}
i  &  i  \\
(-1-i)  &  -i
\end{array}  \right) \in GL(2, {\mathbb Z}[i])
\]
with $\det( g_i(0)) = i$, $\tr(g _i(0)) = 0$ attains this possibility.

If $\frac{2s}{r_1} = \frac{7}{4}$ then  $8s=7r_1$ and $s=7$, $r_1=8$. The eigenvalues
$\lambda _1 = e^{ \frac{7 \pi i}{4}} = e^{ - \frac{ \pi i}{4}} = \frac{\sqrt{2}}{2} - \frac{\sqrt{2}}{2} i$,
$\lambda _2 = i e^{ \frac{\pi i}{4}} = - \frac{\sqrt{2}}{2} + \frac{\sqrt{2}}{2} i = e^{ \frac{3 \pi i}{4}}$
are already obtained.

In the case of $\tr(g) = 1+i$, one has $\sin \left( \frac{2 \pi s}{r_1} + \frac{\pi}{4} \right) = \frac{\sqrt{2}}{2}$, which is equivalent to
$\frac{ 2 \pi s}{r_1} + \frac{\pi}{4} = \frac{3 \pi}{4}$ for $\frac{2 \pi s}{r_1} + \frac{\pi}{4} \in \left( \frac{\pi}{4}, \frac{9 \pi}{4} \right)$. Now,  $\frac{2s}{r_1} = \frac{1}{2}$, whereas $4s=r_1$ and $s=1$, $r_1=4$. The eigenvalues are
$\lambda _1 = e^{ \frac{\pi i}{2}} = i$,
$\lambda _2 = i e^{ - \frac{\pi i}{2}} = 1$ and $g$ is of order $r = LCM(4,1)=4$. Note that
\[
g_i (1+i) = \left( \begin{array}{rr}
i  &   0  \\
0  &   1
\end{array}  \right) \in GL(2, {\mathbb Z}[i])
\]
with $\det( g_i(1+i)) =i$, $\tr( g_i(1+i)) = 1+i$ realizes this case.

Finally, for $\tr(g) = -1-i$ there follows $\sin \left( \frac{2 \pi s}{r_1} + \frac{\pi}{4} \right) = - \frac{\sqrt{2}}{2}$. Consequently, $\frac{2 \pi s}{r_1} + \frac{\pi}{4} = \frac{5 \pi}{4}$ or $\frac{2 \pi s}{r_1} + \frac{\pi}{4} = \frac{7 \pi}{4}$ for $\frac{2 \pi s}{r_1} + \frac{\pi}{4} \in \left( \frac{\pi}{4}, \frac{9 \pi}{4} \right)$. In the case of $\frac{2s}{r_1} =1$ one has $s=1$, $r_1=2$. The eigenvalues of $g$ are $\lambda _1 = e^{ \pi i} = -1$,
$\lambda _2 = i e^{ - \pi i} = -i$, so that $g$ is of order $r = LCM (2,4) =4$. This possibility is realized by
\[
g_i (-1-i) = \left( \begin{array}{cc}
-i  &  0  \\
0  &  -1
\end{array}  \right) \in GL(2,{\mathbb Z}[i])
\]
with $\det( g_i (-1-i)) = i$, $\tr(g _i (-1-i)) = -1-i$.

If $\frac{2s}{r_1} = \frac{3}{2}$ then $4s=3r_1$ and $s=3$, $r_1=4$. The eigenvalues
$ \lambda _1 = e^{ \frac{ 3 \pi i}{2}} = -i$,
$\lambda _2 = i e^{ - \frac{ 3 \pi i}{2}} = -1$
are already obtained. That concludes the description of the eigenvalues of all $g \in GL(2,{\mathbb Z}[i])$ of finite order with $\det(g) =i$.

\end{proof}

\begin{proposition}  \label{ListS4-}
If $g \in GL(2, {\mathbb Z}[i])$ is of finite order $r$ and $\det(g) = -i$ then
\[
\tr(g) \in \{ 0, \pm (1-i) \}, \quad r \in \{ 4, 8 \}.
\]
More precisely,

(i) $\tr(g) =0$ or $\lambda _1 = e^{\frac{\pi i}{4}}$,
$\lambda _2 = e^{ \frac{5 \pi i}{4}}$
if and only if $g$ is of order $8$;

  (ii) if $\tr(g) = 1-i$ or
$\lambda _1 = -i$, $\lambda _2 =1$  then $g$ is of order $4$;

(iii) if $\tr(g) = -1+i$ or $\lambda _1 = i$, $\lambda _2 =-1$ then $g$ is of order $4$.

\end{proposition}

\begin{proof}

If one of the eigenvalues of $g$ is $\lambda _1 = e^{ \frac{ 2 \pi s i}{r_1} }$ then the other one is $\lambda _2 = -i e^{ - \frac{ 2 \pi si}{r_1}}$. Thus, the trace
\[
\tr(g) = \lambda + \lambda _2 = \left[ \cos \left( \frac{ 2 \pi s}{r_1} \right)  - \sin \left( \frac{ 2 \pi s}{r_1} \right) \right] (1-i) =
\sqrt{2} \cos \left( \frac{ 2 \pi s}{r_1} + \frac{\pi}{4} \right) (1-i)
\]
belongs to ${\mathbb Z}[i] = {\mathbb Z} + {\mathbb Z}i$ if and only if $\sqrt{2} \cos \left( \frac{ 2 \pi s}{r_1} + \frac{\pi}{4} \right) \in {\mathbb Z}$.
As a result,
\[
\sqrt{2} \cos \left( \frac{ 2 \pi s}{r_1} + \frac{\pi}{4} \right) \in {\mathbb Z} \cap [ - \sqrt{2}, \sqrt{2}] = \{ 0, \pm 1 \}
\]
or $\tr(g) \in \{ 0, \pm (1-i) \}$. Note that $\tr(g) =0$ reduces to $\cos \left( \frac{ 2 \pi s}{r_1} + \frac{\pi}{4} \right) =0$ with solutions $\frac{ 2 \pi s}{r_1} + \frac{\pi}{4} = \frac{\pi}{2}$ or $\frac{ 2 \pi s}{r_1} + \frac{\pi}{4} = \frac{3 \pi}{2}$.
If $\frac{2s}{r_1} = \frac{1}{4}$ then $8s=r_1$ and $s=1$, $r_1=8$. The eigenvalues of $g$ are
$\lambda _1 = e^{ \frac{ \pi i}{4}} = \frac{\sqrt{2}}{2} + \frac{\sqrt{2}}{2} i$,
$\lambda _2 = -i e^{ - \frac{ \pi i}{4}} = - \frac{\sqrt{2}}{2} - \frac{\sqrt{2}}{2} i$ and $g$ is of order $r = LCM(8,8)=8$. Note that
\[
g_{-i} (0) = \left( \begin{array}{rr}
-i  &  -i  \\
(-1+i)  & i
\end{array}  \right) \in GL(2, {\mathbb Z}[i])
\]
with $\det ( g_{-i} (0)) = -i$, $\tr(g _{-i} (0)) = 0$ realizes the aforementioned possibility. In the case of $\frac{2 \pi s}{r_1} = \frac{5}{4}$ there holds $8s=5r_1$, whereas $s=5$, $r_1 =8$ and $\lambda _1 = e^{\frac{ 5 \pi i}{4}} = - \frac{\sqrt{2}}{2} - \frac{\sqrt{2}}{2} i$, $\lambda _2 = -i e^{ - \frac{ 5 \pi i}{4}} = \frac{\sqrt{2}}{2} + \frac{\sqrt{2}}{2} i = e^{\frac{ \pi i}{4}}$. This case has been already discussed.

For $\tr(g) = 1-i$ one has $\cos \left( \frac{ 2 \pi s}{r_1} + \frac{\pi}{4} \right) = \frac{\sqrt{2}}{2}$, which reduces to $\frac{ 2 \pi s}{r_1} + \frac{\pi}{4} = \frac{ 7 \pi}{4}$ for $\frac{ 2 \pi s}{r_1} + \frac{\pi}{4} \in \left( \frac{\pi}{4}, \frac{9 \pi}{4} \right)$. Now $\frac{2s}{r_1} = \frac{3}{2}$ reads as $4s=3r_1$ and determines $s=3$, $r_1=4$.  The eigenvalues of $g$ are
$\lambda _1 = e^{ \frac{ 3 \pi i}{2}} = -i$,
$ \lambda _2 = -i e^{ - \frac{ 3 \pi i}{2}} =1$
and $g$ is of order $r = LCM(4,1)=4$. This possibility is realized by
\[
g_{-i} (1-i) = \left( \begin{array}{rr}
-i  &  0  \\
0  &  1
\end{array}  \right) \in GL (2, {\mathbb Z}[i])
\]
with $\det( g_{-i} (1-i)) = -i$, $\tr(g_{-i} (1-i)) = 1-i$.

Finally, $\tr(g) = -1+i$ is equivalent to $\cos \left( \frac{ 2 \pi s}{r_1} + \frac{\pi}{4} \right) = - \frac{\sqrt{2}}{2}$ and holds for $\frac{ 2 \pi s}{r_1} + \frac{\pi}{4} = \frac{ 3 \pi}{4}$ or $\frac{ 2 \pi s}{r_1} + \frac{\pi}{4} = \frac{ 5 \pi}{4}$. In the case of $\frac{2s}{r_1} = \frac{1}{2}$, one has $4s=r_1$ and $s=1$, $r_1=4$. The eigenvalues of $g$ are
$\lambda _1 = e^{ \frac{ \pi i}{2}} = i$,
$\lambda _2 = -i e^{ - \frac{ \pi i}{2}} = -1$ and $g$ is of order $r = LCM(4,2)=4$. The automorphism
\[
g_{-i} (-1+i) = \left( \begin{array}{rr}
i  &  0  \\
0  &  -1
\end{array}  \right) \in GL(2, {\mathbb Z}[i])
\]
realizes the case under discussion. For $\frac{2s}{r_1} = 1$ there follow $s=1$, $r_1 =2$ and $\lambda _1 = e^{\pi i} = -1$, $\lambda _2 = - i e^{ - \pi i} = i$, which was already discussed.  That concludes the description of the automorphisms $g \in GL(2, {\mathbb Z} [i])$ with $\det(g) = -i$.

\end{proof}

\begin{proposition}    \label{ListS6+}
If $g \in GL(2, \mathcal{O}_{-3})$ is of finite order $r$ and $\det(g) = e^{ \frac{\pi i}{3}}$ then
\[
r=6 \ \ \mbox{  and  } \ \ \tr(g) \in \left
\{ 0, \pm \left( \frac{3}{2} + \frac{\sqrt{-3}}{2}  \right) \right \}.
\]
More precisely,

(i) $\tr(g) =0$ exactly when $\lambda _1 = e^{ \frac{ 2 \pi i}{3}}$,
$\lambda _2 = e^{ - \frac{ \pi i}{3}}$;

(ii) $\tr(g) = \frac{3}{2} + \frac{\sqrt{-3}}{2}$ exactly when
$\lambda _1 = e^{ \frac{ \pi i}{3}}$, $\lambda _2 =1$;

(iii) $\tr(g) = - \frac{3}{2} - \frac{\sqrt{-3}}{2}$ exactly when
$\lambda _1 = e^{ - \frac{ 2 \pi i}{3}}$, $\lambda _2 =-1$.
\end{proposition}

\begin{proof}

If $\lambda _1 = e^{ \frac{ 2 \pi si}{r_1}}$  then $\lambda _2 = e^{ \frac{\pi i}{3}} e^{ -  \frac{ 2 \pi si}{r_1}}$ and the trace
\[
\tr(g) = \lambda _1 + \lambda _2 = ( \sqrt{3} +i) \sin \left(  \frac{ 2 \pi s}{r_1} + \frac{\pi}{3} \right)
\]
belongs to $\mathcal{O}_{-3} = {\mathbb Z} + \frac{1 + \sqrt{-3}}{2} {\mathbb Z}$ if and only if $\sin \left(  \frac{ 2 \pi s}{r_1} + \frac{\pi}{3} \right) \in \frac{\sqrt{3}}{2} {\mathbb Z}$. Combining with $\sin \left(  \frac{ 2 \pi si}{r_1} + \frac{\pi}{3} \right) \in [-1,1]$, one gets $\sin \left(  \frac{ 2 \pi s}{r_1} + \frac{\pi}{3} \right) \in \frac{\sqrt{3}}{2} {\mathbb Z} \cap [-1,1] = \left \{ 0 , \pm \frac{\sqrt{3}}{2} \right \}$  and, respectively, $\tr(g) \in \left\{ 0, \pm \left( \frac{3}{2} + \frac{\sqrt{-3}}{2} \right) \right \}$.

If $\sin \left(  \frac{ 2 \pi s}{r_1} + \frac{\pi}{3} \right) =0$ then $\frac{ 2 \pi s}{r_1} + \frac{\pi}{3} = \pi$ or $\frac{ 2 \pi s}{r_1} + \frac{\pi}{3} = 2 \pi$. For $\frac{2s}{r_1} = \frac{2}{3}$ there follows $s=1$, $r_1=3$ and $\lambda _1 = e^{ \frac{ 2 \pi i}{3}} = - \frac{1}{2} + \frac{\sqrt{-3}}{2}$,
$\lambda _2 = e^{\frac{ \pi i}{3}} e^{ - \frac{ 2 \pi i}{3}} = e^{ - \frac{\pi i}{3}} = \frac{1}{2} - \frac{\sqrt{-3}}{2} $.
The automorphisms $g \in GL(2, \mathcal{O}_{-3})$ with such eigenvalues are of order $r = LCM(3,6)=6$. For instance,
\[
\left( \begin{array}{rr}
e^{ \frac{ 2 \pi i}{3}}  &  0  \\
0  &  e^{ - \frac{ \pi i}{3}}
\end{array}  \right) \in GL(2, \mathcal{O}_{-3})
\]
attains the aforementioned possibility.

In the case of $\frac{2s}{r_1} = \frac{5}{3}$ one has $s=5$, $r_1=6$ and
$\lambda _1 = e^{ - \frac{ \pi i}{3}}$,
$\lambda _2 = e^{ \frac{\pi i}{3}}  e^{ \frac{\pi i}{3}}  = e^{ \frac{ 2 \pi i}{3}}$,
which was already obtained.

Note that $\sin \left( \frac{ 2 \pi s}{r_1}  + \frac{\pi}{3} \right) = \frac{\sqrt{3}}{2}$ for $\frac{ 2 \pi s}{r_1} + \frac{\pi}{3} \in \left( \frac{\pi}{3}, \frac{ 7 \pi}{3} \right)$ implies $\frac{ 2 \pi s}{r_1}  + \frac{\pi}{3} = \frac{ 2 \pi}{3}$, whereas $6s=r_1$ and $s=1$, $r_1=6$. The corresponding eigenvalues are
$\lambda _1 = e^{\frac{ \pi i}{3}} = \frac{1}{2} + \frac{\sqrt{3}}{2} i$,
$\lambda _2 = e^{\frac{\pi i}{3}} e^{ - \frac{ \pi i}{3}} =1$ and $g$ is of order $r=LCM(6,1)=6$. Note that
\[
\left( \begin{array}{rr}
e^{ \frac{ \pi i}{3}}  &   0  \\
0  &   1
\end{array}  \right) \in GL(2, \mathcal{O}_{-3})
\]realizes this possibility.

The equality $\sin \left( \frac{ 2 \pi s}{r_1}  + \frac{\pi}{3} \right) = - \frac{\sqrt{3}}{2}$ holds for $\frac{ 2 \pi s}{r_1}  + \frac{\pi}{3} = \frac{ 4 \pi }{3}$ or $\frac{ 2 \pi s}{r_1}  + \frac{\pi}{3} = \frac{ 5 \pi}{3}$. If $2s=r_1$ then $s=1$, $r_1=2$ and
$\lambda _1 = e^{ \pi i} = -1$,
$\lambda _2 = e^{ \frac{ \pi i}{3}} e^{ - \pi i} = e^{ - \frac{ 2 \pi i}{3}} = - \frac{1}{2} - \frac{\sqrt{3}}{2} i$. The automorphism $g$ is of order $r = LCM(2,3)=6$. Note that
\[
\left( \begin{array}{rr}
e^{ - \frac{ 2 \pi i}{3}} &  0  \\
0  & -1
\end{array} \right) \in GL(2, \mathcal{O}_{-3})
\]
attains this possibility and concludes the proof of the proposition.

\end{proof}

\begin{proposition}  \label{ListS6-}
If $g \in GL(2, \mathcal{O}_{-3})$ is of finite order $r$ and $\det(g) = e^{ - \frac{\pi i}{3}}$ then
\[
r=6 \ \ \mbox{ and    } \ \  \tr(g) \in \left \{ 0, \pm \left( \frac{3}{2} - \frac{\sqrt{-3}}{2} \right) \right \}.
\]
More precisely,

(i) $\tr(g) =0$ exactly when $\lambda _1 = e^{ \frac{ \pi i}{3}}$, $\lambda _2 = e^{ - \frac{ 2 \pi i}{3}}$;

(ii) $\tr(g) = \frac{3}{2} - \frac{\sqrt{3}}{2} i$ exactly when $\lambda _1 = e^{  - \frac{ \pi i}{3}}$, $\lambda _2 =1$;

(iii) $\tr(g) = - \frac{3}{2} + \frac{\sqrt{3}}{2} i$ exactly when $\lambda _1 = к ^{ \frac{ 2 \pi i}{3}}$, $\lambda _2 =-1$.
\end{proposition}

\begin{proof}

If $\lambda _1 = e^{ \frac{ 2 \pi si}{r_1}}$ then $\lambda _2 = e^{ - \frac{ \pi i}{3}} e^{ - \frac{ 2 \pi si}{r_1} }$ and the trace
\[
\tr(g) = \lambda _1 + \lambda _2 = ( - \sqrt{3} + i) \sin \left( \frac{ 2 \pi s}{r_1} - \frac{\pi}{3} \right)
\]
belongs to $\mathcal{O}_{-3} = {\mathbb Z} + \frac{1 + \sqrt{3}i}{2} {\mathbb Z}$ if and only if
$\sin \left( \frac{ 2 \pi s}{r_1} - \frac{\pi}{3} \right)  \in \frac{\sqrt{3}}{2} {\mathbb Z}$.
As a result, $\sin \left( \frac{ 2 \pi s}{r_1} = \frac{\pi}{3} \right) \in \frac{\sqrt{3}}{2} {\mathbb Z} \cap [ -1,1] =
\left \{ 0, \pm \frac{\sqrt{3}}{2} \right \}$ and $\tr(g) \in \left \{ 0, \pm \left( \frac{3}{2} - \frac{\sqrt{3}}{2}i \right) \right \}$.

The equation $\sin \left( \frac{ 2 \pi s}{r_1} - \frac{\pi}{3} \right) =0$ for
$\frac{ 2 \pi s}{r_1} - \frac{\pi}{3} \in \left( - \frac{\pi}{3}, \frac{5 \pi}{3} \right)$ has solutions
$\frac{ 2 \pi s}{r_1} - \frac{\pi}{3} =0$ and $\frac{ 2 \pi s}{r_1} - \frac{\pi}{3} = \pi$.

If $6s=r_1$ then $s=1$, $r_1=6$ and
$\lambda _1 = e^{ \frac{ \pi i}{3}} = \frac{1}{2} + \frac{\sqrt{3}}{2} i$,
$\lambda _2 = e^{ - \frac{ \pi i}{3}} e^{ - \frac{\pi i}{3}} = -\frac{1}{2} - \frac{\sqrt{3}}{2} i$. The automorphisms $g \in GL(2, \mathcal{O}_{-3})$
with such eigenvalues are of order $r = LCM (6,3)=6$. For instance,
\[
\left( \begin{array}{rr}
e^{ \frac{ \pi i}{3}}  &  0  \\
0  &  e^{ - \frac{ 2 \pi i}{3}}
\end{array}  \right) \in GL(2, \mathcal{O}_{-3})
\]
attains this case.

If $\sin \left( \frac{ 2 \pi s}{r_1} - \frac{\pi}{3} \right) = \frac{\sqrt{3}}{2}$ then $ \frac{ 2 \pi s}{r_1} - \frac{\pi}{3} = \frac{\pi}{3}$ or
$\frac{ 2 \pi s}{r_1} - \frac{\pi}{3} = \frac{ 2 \pi}{3}$. For $3s=r_1$ one has $s=1$, $r_1=3$ and
$\lambda _1 = e^{ \frac{ 2 \pi i}{3}} = - \frac{1}{2} + \frac{\sqrt{3}}{2} i$,
$\lambda _2 = e^{ - \frac{ \pi i}{3}} e^{ - \frac{ 2 \pi i}{3}} = e^{ - \pi i} =-1$, attained by
\[
\left( \begin{array}{rr}
e^{ \frac{ 2 \pi i}{3}}  &  0  \\
0  &  -1
\end{array}  \right) \in GL(2, \mathcal{O}_{-3}).
\]
All  $g \in GL(2, \mathcal{O}_{-3})$ with such eigenvalues are of order $r = LCM (3,2)=6$.

 In the case of $2s=r_1$ there follows $s=1$, $r_1 =2$ and
 $\lambda _1 = e^{ \pi i} = -1$,
 $\lambda _2 = e^{ - \frac{ \pi i}{3}} e^{ - \frac{ \pi i}{3}} = e^{ - \frac{ 2 \pi i}{3}}$,
 which is already discussed.

 The equation $\sin \left( \frac{ 2 \pi s}{r_1} - \frac{\pi}{3} \right) = - \frac{\sqrt{3}}{2}$ for $\frac{ 2 \pi s}{r_1} - \frac{\pi}{3} \in \left( - \frac{\pi}{3}, \frac{5 \pi}{3} \right)$ has solution $\frac{ 2 \pi s}{r_1} - \frac{\pi}{3} = \frac{ 5 \pi}{3}$. Therefore $6s=5r_1$ and $s=5$, $r_1=6$, As a result,
 $\lambda _1 = e^{ \frac{ 5 \pi i}{3}} = \frac{1}{2} - \frac{\sqrt{3}}{2} i$,
 $\lambda _2 = e^{ - \frac{\pi i}{3}} e^{ \frac{ \pi i}{3}} =1$
 and $g$ is of order $r = LCM( 6,1)=6$. Note that
 \[
 \left( \begin{array}{rr}
 e^{ - \frac{ \pi i}{3}}  &  0  \\
 0  &  1
 \end{array}  \right) \in GL(2, \mathcal{O}_{-3})
 \]
 attains this possibility and concludes the proof of the proposition.

\end{proof}

\begin{proposition}  \label{ListS3+}
If $g \in GL(2, \mathcal{O}_{-3})$ is of finite order $r$ and $\det(g) = e^{ \frac{ 2 \pi i}{3}}$ then
\[
\tr(g) \in \left \{ 0, \pm \frac{( 1 + \sqrt{-3})}{2}, \pm (1 + \sqrt{-3}) \right \}, \quad r \in \{ 3, 6, 12 \}.
\]
More precisely,

(i) $\tr(g) =0$ or $\lambda _1 = e^{ \frac{ 5 \pi i}{6}}$,
$\lambda _2 = e^{ - \frac{\pi i}{6}}$ if and only if $g$ is of order $12$;

(ii) if $\tr(g) = \frac{1 + \sqrt{3}i}{2}$ or
$\lambda _1 = e^{ \frac{2 \pi i}{3}}$, $\lambda _2 =1$ then $g$ is of order $3$;

(iii) if $\tr(g) = -1 - \sqrt{3}i$ or $g = e^{ - \frac{ 2 \pi i}{3}} I_2$ then $g$ is of order $3$;

(iv) if $\tr(g) = \frac{-1 - \sqrt{3}i}{2}$ or
$\lambda _1 = e^{ - \frac{ \pi i}{3}}$, $\lambda _2 =-1$ then $g$ is of order $6$;

(v) if $\tr(g) = 1 + \sqrt{3}i$ or $g = e^{\frac{ \pi i}{3}} I_2$ then $g$ is of order $6$.
\end{proposition}

\begin{proof}

If $\lambda _1 = e^{ \frac{ 2 \pi si}{r_1} }$ then $\lambda _2 = e^{ \frac{ 2 \pi i}{3}} e^{ - \frac{ 2 \pi si}{r_1}}$ and the trace
\[
\tr(g) = \lambda _1 + \lambda _2 = (1 + \sqrt{3}i) \sin \left( \frac{2 \pi s}{r_1} + \frac{\pi}{6} \right)
\]
belongs to $\mathcal{O}_{-3} = {\mathbb Z} + \frac{1 + \sqrt{3}i}{2} {\mathbb Z}$ if and only if $2 \sin \left( \frac{2 \pi s}{r_1} + \frac{\pi}{6} \right) \in {\mathbb Z}$. Combining with $\sin \left( \frac{2 \pi s}{r_1} + \frac{\pi}{6} \right) \in [-1,1]$, one obtains
$
2 \sin \left( \frac{2 \pi s}{r_1} + \frac{\pi}{6} \right) \in {\mathbb Z} \cap [-2,2] = \left \{ 0, \pm1, \pm2 \right \}
$
and, respectively,
\[
\tr(g) \in \left \{  0, \pm \frac{(1 + \sqrt{3}i)}{2}, \pm (1 + \sqrt{3}i) \right \}.
\]

If $\sin \left( \frac{2 \pi s}{r_1} + \frac{\pi}{6} \right) =0$ for $\frac{2 \pi s}{r_1} + \frac{\pi}{6} \in \left( \frac{\pi}{6}, \frac{13\pi}{6} \right)$ then
$\frac{2 \pi s}{r_1} + \frac{\pi}{6} = \pi$ or $\frac{2 \pi s}{r_1} + \frac{\pi}{6} = 2 \pi$.

For $12s=5r_1$ one has $s=5$, $r_1=12$ and
$\lambda _1 = e^{ \frac{ 5 \pi i}{6}} = - \frac{\sqrt{3}}{2} + \frac{1}{2}i$,
$\lambda _2 = e^{ \frac{ 2 \pi i}{3}} e^{ - \frac{ 5 \pi i}{6}} = e^{ - \frac{\pi i}{6}} = \frac{\sqrt{3}}{2} - \frac{1}{2}i$. Therefore $g$ is of order $r = LCM( 12,12)=12$. Note that
\[
\left( \begin{array}{rr}
e^{ \frac{ 5 \pi i}{6}}  &  0  \\
0  &  e^{ - \frac{ \pi i}{6}}
\end{array}  \right) \in GL(2, \mathcal{O}_{-3})
\]
attains this possibility.

In the case of $12s=11r_1$ there follows $s=11$, $r_1=12$. As a result,
$\lambda _1 = e^{ \frac{11 \pi i}{6}} = \frac{\sqrt{3}}{2} - \frac{1}{2} i$,
$\lambda _2 =  e^{ \frac{ 2 \pi i}{3}} e^{ \frac{ \pi i}{6}} = e^{ \frac{ 5 \pi i}{6}} = - \frac{\sqrt{3}}{2} + \frac{1}{2}i$, which was already obtained.

If $\sin \left( \frac{2 \pi s}{r_1} + \frac{\pi}{6} \right) = \frac{1}{2}$ for $\frac{2 \pi s}{r_1} + \frac{\pi}{6} \in \left(  \frac{\pi}{6}, \frac{13 \pi}{6} \right)$ then $\frac{2 \pi s}{r_1} + \frac{\pi}{6} = \frac{ 5 \pi}{6}$ and $3s=r_1$. Therefore $s=1$, $r_1=3$ and $\lambda _1 = e^{ \frac{ 2 \pi i}{3}} = - \frac{1}{2} + \frac{\sqrt{3}}{2} i$, $\lambda _2 = e^{ \frac{ 2 \pi i}{3}} e^{ - \frac{ 2 \pi i}{3}} =1$. The order of $g$ is $r = LCM( 3,1)=3$. This possibility is attained by
\[
\left( \begin{array}{rr}
e^{ \frac{ 2 \pi i}{3}}  &  0  \\
0  & 1
\end{array}  \right) \in GL(2, \mathcal{O}_{-3}).
\]

The equation $\sin \left( \frac{2 \pi s}{r_1} + \frac{\pi}{6} \right) = - \frac{1}{2}$ has solutions $\frac{2 \pi s}{r_1} + \frac{\pi}{6} = \frac{ 7 \pi}{6}$
and $\frac{2 \pi s}{r_1} + \frac{\pi}{6} = \frac{11 \pi}{6}$.

If $2s=r_1$ then $s=1$, $r_1=2$, $\lambda _1 = e^{ \pi i} = -1$,
$\lambda _2 = e^{ \frac{ 2 \pi i}{3}} e^{ - \pi i} = e^{ - \frac{ \pi i}{3}} = \frac{1}{2} - \frac{\sqrt{3}}{2} i$ and $g$ is of order $r = LCM( 2,6) =6$. For instance,
\[
\left( \begin{array}{rr}
e^{ - \frac{ \pi i}{3}}  &  0  \\
0  &  -1
\end{array}  \right) \in GL(2, \mathcal{O}_{-3})
\]
attains these eigenvalues.

For $6s=5r_1$ one has $s=5$, $r_1=6$
$\lambda _1 = e^{ \frac{ 5 \pi i}{3}} = e^{ - \frac{ \pi i}{3}} = \frac{1}{2} - \frac{\sqrt{3}}{2} i$,
$\lambda _2 = e^{\frac{ 2 \pi i}{3}} e^{ \frac{ \pi i}{3}} = e^{ \pi i} = -1$,
which is already obtained.

Note that $\sin \left( \frac{2 \pi s}{r_1} + \frac{\pi}{6} \right) =1$ is equivalent to $\frac{2 \pi s}{r_1} + \frac{\pi}{6} = \frac{\pi}{2}$, whereas $6s=r_1$ and $s=1$, $r_1=6$. The eigenvalues
$\lambda _1 = e^{ \frac{ \pi i}{3}} = \frac{1}{2} + \frac{\sqrt{3}}{2} i$,
$\lambda _2 = e^{ \frac{ 2 \pi i}{3}} e^{ - \frac{ \pi i}{3}} = e^{ \frac{ \pi i}{3}} = \frac{1}{2} + \frac{sqrt{3}}2 i$ are equal, so that
$g = e^{ \frac{ \pi i}{3}} I_2$ and $r = LCM(6,6)=6$.

If $\sin \left( \frac{2 \pi s}{r_1} + \frac{\pi}{6} \right) = -1$ then $\frac{2 \pi s}{r_1} + \frac{\pi}{6} = \frac{ 3 \pi}{2}$ and $3s = 2r_1$, $s=2$, $r_1 =3$. Then $\lambda _1 = e^{ \frac{ 4 \pi i}{3}} = e^{ - \frac{ 2 \pi i}{3}} = - \frac{1}{2} - \frac{\sqrt{3}}{2} i$,
$\lambda _2 = e^{ \frac{ 2 \pi i}{3}} e^{ \frac{ 2 \pi i}{3}} = e^{ - \frac{ 2 \pi i}{3}}$ determine uniquely $g = e^{ - \frac{ 2 \pi i}{3}} I_2$ of order $r = LCM( 3,3)=3$. That concludes the description of $g \in GL(2, \mathcal{O}_{-3})$ of finite order and $\det(g) = e^{ \frac{ 2 \pi i}{3}}$.

\end{proof}

\begin{proposition}    \label{ListS3-}
If $g \in GL(2, \mathcal{O}_{-3})$ is of finite order $r$ and $\det(g) = e^{ - \frac{ 2 \pi i}{3}}$ then
\[
\tr(g) \in \left \{ 0, \pm \frac{(1 - \sqrt{-3})}{2}, \ \ \pm (1 - \sqrt{-3}) \right \}, \quad r \in \{ 3, 6, 12 \}.
\]
More precisely,

(i) $\tr(g) =0$ or $\lambda _1 = e^{ \frac{ \pi i}{6}}$, $\lambda _2 = e^{ - 5 \frac{\pi i}{6}}$ if and only if $g$ is of order $12$;

(ii) if $\tr(g) = \frac{1 - \sqrt{3}i}{2}$ or $\lambda _1 = e^{   \frac{ 4 \pi i}{3}}$, $\lambda _2=1$ then $g$ is of order $3$;

(iii) if $\tr(g) = -1 + \sqrt{3} i$ or $g = e^{ \frac{ 2 \pi i}{3}} I_2$ then $g$ is of order $3$;

(iv) if $\tr(g) = \frac{-1 + \sqrt{3}i}{2}$ or $\lambda _1 = e^{ \frac{ \pi i}{3}}$, $\lambda _2 =-1$ then $g$ is of order $6$;

(v) if $\tr(g) = 1 - \sqrt{3} i$ or $g = e^{ - \frac{ \pi i}{3}} I_2$ then $g$ is of order $6$.
\end{proposition}

\begin{proof}

If $\lambda _1 = e^{ \frac{ 2 \pi si}{r_1} }$ then $\lambda _2 = e^{ - \frac{ 2 \pi i}{3}} e^{ - \frac{ 2 \pi si}{r_1}}$ and the trace
\[
\tr(g) = \lambda _1 + \lambda _2 = ( -1 + \sqrt{3}i) \sin \left( \frac{ 2 \pi s}{r_1}  - \frac{\pi}{6} \right)
\]
belongs to $\mathcal{O}_{-3} = {\mathbb Z} + \frac{1 + \sqrt{3}i}{2} {\mathbb Z}$ if and only if $2 \sin \left( \frac{ 2 \pi s}{r_1}  - \frac{\pi}{6} \right) \in {\mathbb Z}$. Combining with $2 \sin \left( \frac{ 2 \pi s}{r_1}  - \frac{\pi}{6} \right) \in [-2,2]$, one concludes that
$ \sin \left( \frac{ 2 \pi s}{r_1}  - \frac{\pi}{6} \right) \in \left \{ 0, \pm \frac{1}{2}, \pm 1 \right \}$ and $\tr(g) \in \left \{ 0, \pm \frac{(1 - \sqrt{3}i)}{2}, \pm ( 1 - \sqrt{3}i) \right \}$.

If $\sin \left( \frac{ 2 \pi s}{r_1}  - \frac{\pi}{6} \right) =0$ with $\frac{ 2 \pi s}{r_1}  - \frac{\pi}{6}  \in \left( - \frac{\pi}{6}, \frac{11 \pi}{6} \right)$ then $\frac{ 2 \pi s}{r_1}  - \frac{\pi}{6} =0$ or $\frac{ 2 \pi s}{r_1}  - \frac{\pi}{6} = \pi$.

For $12s=r_1$ one has $s=1$, $r_1=12$,
$\lambda _1 = e^{ \frac{ \pi i}{6}} = \frac{\sqrt{3}}{2} + \frac{1}{2}i$,
$\lambda _2 = e^{ - \frac{ 2 \pi i}{3}} e^{ - \frac{ \pi i}{6}} = e^{ - \frac{ 5 \pi i}{6}} = - \frac{\sqrt{3}}{2} - \frac{1}{2} i$, so that $g$ is of order $r = LCM(12,12)=12$. For instance,
\[
\left( \begin{array}{rr}
e^{ \frac{ \pi i}{6}}  &  0  \\
0  &  e^{ - \frac{ 5 \pi i}{6}}
\end{array}  \right) \in GL(2, \mathcal{O}_{-3})
\]
attains this case.

For $12s = 7r_1$ there follows $s=7$, $r_1=12$,
$\lambda _1 = e^{ \frac{ 7 \pi i}{6}} = e^{ - \frac{ 5 \pi i}{6}}$,
$\lambda _2 = e^{ - \frac{ 2 \pi i}{3}} e^{ \frac{ 5 \pi i}{6}} = e^{ \frac{ \pi i}{6}}$, which is already discussed.

In the case of $\sin \left( \frac{ 2 \pi s}{r_1}  - \frac{\pi}{6} \right) = \frac{1}{2}$ note that $\frac{ 2 \pi s}{r_1}  - \frac{\pi}{6} = \frac{\pi}{6}$ or
$\frac{ 2 \pi s}{r_1}  - \frac{\pi}{6} = \frac{ 5 \pi}{6}$.

If $6s=r_1$ then $s=1$, $r_1=6$,
$\lambda _1 = e^{ \frac{ \pi i}{3}} = \frac{1}{2} + \frac{\sqrt{3}}{2} i$,
$\lambda _2 = e^{ - \frac{ 2 \pi i}{3}} e^{ - \frac{ \pi i}{3}} = e^{ - \pi i} = -1$ and $g$ is of order $r = LCM( 6,2) =6$. Note that
\[
\left( \begin{array}{rr}
e^{ \frac{ \pi i}{3}}  &  0  \\
0  &  -1
\end{array}  \right) \in GL(2, \mathcal{O}_{-3})
\]
attains this case.

For $2s=r_1$ there follows $s=1$, $r_1 =2$,
$\lambda _1 = e^{ \pi i} = -1$,
$\lambda _2 = e^{ - \frac{ 2 \pi i}{3}} e^{ - \frac{ \pi i}{3}} = e^{ - \frac{ 5 \pi i}{3}} = e^{ \frac{ \pi i}{3}}$, which is already obtained.

Note that $\sin \left( \frac{ 2 \pi s}{r_1}  - \frac{\pi}{6} \right) = - \frac{1}{2}$ for $\frac{ 2 \pi s}{r_1}  - \frac{\pi}{6} \in \left( - \frac{\pi}{6}, \frac{ 11 \pi}{6} \right)$ implies $\frac{ 2 \pi s}{r_1}  - \frac{\pi}{6} = \frac{ 7 \pi}{6}$, whereas $3s=2r_1$, $s=2$ and $r_1 =3$. Then
 $\lambda _1 = e^{ \frac{ 4 \pi i}{3}} = - \frac{1}{2} - \frac{\sqrt{3}}{2}i$,
$\lambda _2 = e^{ - \frac{ 2 \pi i}{3}} e^{ - \frac{ 4 \pi i}{3}} = e^{ - 2 \pi i} =1$ and $g$ is of order $r = LCM(3,1)=3$, attained by
\[
\left( \begin{array}{rr}
e^{ \frac{ 4 \pi i}{3}}  &  0  \\
0  &  1
\end{array}  \right) \in GL(2, \mathcal{O}_{-3}).
\]

If $\sin \left( \frac{ 2 \pi s}{r_1}  - \frac{\pi}{6} \right) =1$ then $\frac{ 2 \pi s}{r_1} - \frac{\pi}{6} = \frac{\pi}{2}$ or $3s=r_1$. As a result, $s=1$, $r_1=3$, $\lambda _1 = e^{ \frac{ 2 \pi i}{3}} = - \frac{1}{2} + \frac{\sqrt{3}}{2} i$,
$\lambda _2 = e^{ - \frac{2 \pi i}{3}} e^{ - \frac{ 2 \pi i}{3}} = e^{ \frac{ 2 \pi i}{3}}$, whereas $g = e^{ \frac{ 2 \pi i}{3}} I_2 \in GL(2, \mathcal{O}_{-3})$ is a scalar matrix of order $3$.

Finally, $\sin \left( \frac{ 2 \pi s}{r_1} - \frac{\pi}{6} \right) = -1$ holds for $\frac{ 2 \pi s}{r_1} - \frac{\pi}{6} = \frac{ 3 \pi}{2}$, i.e., $6s=5r_1$ and $s=5$, $r_1=6$. Now
$\lambda _1 = e^{ - \frac{ \pi i}{3}} = \frac{1}{2} - \frac{\sqrt{3}}{2} i$,
$\lambda _2 = e^{ - \frac{ 2 \pi i}{3}} e^{ \frac{ \pi i}{3}} = e^{ - \frac{ \pi i}{3}}$, so that $g = e^{ - \frac{ \pi i}{3}} I_2 \in GL(2, \mathcal{O}_{-3})$ is a scalar matrix of order $6$. That concludes the proof of the proposition.

\end{proof}


\section{Finite linear automorphism groups of $E \times E$}

The classification of the finite subgroups $K$ of $SL(2,R)$ for an endomorphism ring $R$ of an elliptic curve $E$ starts with a classification of the Sylow subgroups $H_{p^k}$ of $K$.

\begin{proposition}     \label{SylowSubgroupsK}
If $K$ is a finite subgroup of $SL(2,R)$ then $K$ is of order $|K| = 2^a 3^b$ for some integers $0 \leq a \leq 3$, $0 \leq b \leq 1$.

 If $K$ is of even order then the Sylow $2$-subgroup $H_{2^a}$ of $K$ is isomorphic to ${\mathbb C}_2$, ${\mathbb C}_4$ or the quaternion group
\[
{\mathbb Q} _8 = \langle g_1, g_2 \ \ \vert \ \  g_1^2 = g_2 ^2 = - I_2, \ \  g_2 g_1 = - g_1 g_2 \rangle
\]
of order $8$.

If the order of $K$ is divisible by $3$ then the Sylow $3$-subgroup $H_3$ of $K$ is isomorphic to the cyclic group ${\mathbb C}_3$ of the third roots of unity.
\end{proposition}

\begin{proof}

 According to the First Sylow Theorem, if $|K| = p_1^{m_1} \ldots p_s ^{m_s}$ for some rational primes $p_j \in {\mathbb N}$ and some $m_j \in {\mathbb N}$, then for any $1 \leq i \leq k$ there is a subgroup $H_{p_j^i} \leq K$ of order $| H_{p_j^i} | = p_j ^{i}$. In particular, any $H_{p_j} = \langle g_{p_j} \rangle \simeq {\mathbb C} _{p_j}$ of prime order $p_j$, dividing $|K|$ is cyclic and there is an element $g_{p_j} \in K$ of order $p_j
$. By Proposition \ref{ListK}, the order of an element $g \in SL(2,R)$ is $1, 2, 3, 4, 6$ or $\infty$. As a result, if $g \in SL(2,R)$ is of prime order $p$ then $p=2$ or $3$. In other words, $K$ is of order $|K| = 2^a 3^b$ for some non-negative integers $a,b$.

Suppose that $b \geq 1$ and consider the Sylow subgroup $H_{3^b} \leq K$ of order $3^b$. Then any $h \in H_{3^b} \setminus \{ I_2 \}$ is of order $3$ since there is no $g \in SL(2,R)$, whose order is divisible by $9$. We claim that $H_{3^b} = \langle h_1 \rangle \simeq {\mathbb C}_3$ is a cyclic group of order $3$. Otherwise, $b \geq 2$ and there exists $h_2 \in H_{3^b} \setminus \langle h_1 \rangle$. Note that $h_1^j h_2 \in H_{3^b}$ with $1 \leq j \leq 2$ are of order $3$, as far as $h_1 ^j h_2 = I_2$ implies $h_2 = h_1 ^{-j} \in \langle h_1 \rangle$, contrary to the choice of $h_2$. We are going to show that if $h_1, h_2, h_1h_2 \in SL(2,R)$ are of order $3$ then $h_1 ^2 h_2 = I_2$, so that there is no $h_2 \in H_{3^b} \setminus \langle h_1 \rangle
$ and $H_{3^b} = \langle h_1 \rangle \simeq {\mathbb C}_3$.  According to Proposition \ref{ListK}, $g \in SL(2,R)$ is of order $3$ if and only if ${\rm tr} (g) = -1$ and $g$ is conjugate to
\[
D_g = \left( \begin{array}{cc}
e^{\frac{2 \pi i}{3}}   &  0   \\
0   &  e^{- \frac{2 \pi i}{3}}
\end{array}  \right).
\]
Similarly, $g \in SL(2,R)$ coincides with the identity matrix $I_2$ exactly when ${\rm tr} (g) =2$. Thus, we have to check that if $h_1, h_2 \in SL(2,R)$ satisfy ${\rm tr} (h_1) = {\rm tr} (h_2) = {\rm tr} (h_1 h_2) = -1$ then ${\rm tr} ( h_1 ^2 h_2) = 2$. Let
\[
D_1 = S^{-1} h_1 S = \left( \begin{array}{cc}
e^{\frac{2 \pi i}{3}}  &  0  \\
0  &  e^{- \frac{2 \pi i}{3}}
\end{array}   \right)
\]
be a diagonal form of $h_1$ for some $S \in GL(2, {\mathbb C})$ and
\[
D_2 = S^{-1} h_2 S = \left(  \begin{array}{cc}
a  &  b  \\
c  &  d
\end{array}    \right) \in GL(2, {\mathbb C}).
\]
(More precisely, if $Q(R) = {\mathbb Q}$ or ${\mathbb Q}( \sqrt{-d})$ is the fraction field of $R$ then the eigenvectors of $h_1$ have entries from $Q(R) ( \sqrt{-3})$, so that $S, D_2 \in Q(R)( \sqrt{-3}) _{2 \times 2}$ have entries from $Q(R) ( \sqrt{-3}) = {\mathbb Q}( \sqrt{-3})$ or ${\mathbb Q} ( \sqrt{-d}, \sqrt{-3})$.) Since the determinant and the trace of a matrix are invariant under conjugation, the statement is equivalent to the fact that if $\det (D_2) =1$ and ${\rm tr} (D_2) = {\rm tr} (D_1 D_2) = -1$ then ${\rm tr} ( D_1^2 D_2) = 2$. Indeed, if $d = -a -1$ and ${\rm tr} (D_1 D_2) = e^{\frac{2 \pi i}{3}} a - e^{ - \frac{2 \pi i}{3}} (a+1) = -1$ then $a = e^{ \frac{2 \pi i}{3}}$, $d = e ^{ - \frac{ 2 \pi i}{3}}$, whereas ${\rm tr} ( D_1^2 D_2) =2$. That proves the non-existence of $h_2 \in H_{3^b} \setminus \langle h_1 \rangle$ and $H_{3^b} = H_3 = \langle h_1 \rangle \simeq {\mathbb C}_3$.

Suppose that $K$ is of even order and denote by $H_{2^a}$ the Sylow $2$-subgroup of $K < SL(2,R)$ of order $2^a \geq 2$. Then any $g \in H_{2^a} \setminus \{ I_2 \}$ is of order
\[
r \in \{ 2^i  \ \ \vert \ \  i \in {\mathbb N} \} \cap \{ 1, 2, 3, 4, 6 \} = \{ 2, 4 \}.
\]
 Recall from Proposition \ref{ListK} that there is a unique element $ - I_2$ of $SL(2,R)$ of order $2$ and $g \in SL(2,R)$ is of order $4$ if and only if the trace ${\rm tr} (g) =0$. For $a=1$ the Sylow subgroup $H_2  = \langle - I_2 \rangle \simeq {\mathbb C}_2$ is cyclic of order $2$. If $a=2$ then $H_4 = \langle g \rangle \simeq {\mathbb C}_4$ is cyclic of order $4$, since $SL(2,R)$ has a unique element $- I_2$ of order $2$. From now on, let us assume that $a \geq 3$ and fix an element $g_1 \in H_{2^a}$ of order $4$. Due to $g_1 ^2 = - I_2 \in \langle g_1 \rangle$, any $g_2 \in H_{2^a} \setminus \langle g_1 \rangle$ is of order $4$ and $g_2 ^2 = - I_2$. Moreover, $g_1 g_2 \in H_{2^a}$ is of order $4$, as far as $g_1 g_2 = \pm I_2$ requires $g_2 = \mp g_1 \in \langle g_1 \rangle$, contrary to the choice of $g_2$. We claim that if $g_1, g_2 \in SL(2,R)$ of order $4$ have product $g_1g_2$ of order $4$ then they generate a quaternion group
\[
\langle g_1, g_2 \ \ \vert \ \  g_1^2 = g_2 ^2 = - I_2, \ \ g_2g_1 = - g_1 g_2 \rangle \simeq {\mathbb Q}_8
\]
of order $8$. In other words, if $g_1, g_2 \in R_{2 \times 2}$ have $\det (g_1) = \det(g_2) =1$ and ${\rm tr} (g_1) = {\rm tr} (g_2) = {\rm tr}(g_1g_2) = 0$ then $g_2 g_1 = - g_1 g_2$. In particular, if $g_1, g_2 \in SL(2,R)$ of order $4$ have product $g_1g_2$ of order $4$ then $g_2 \not \in \langle g_1 \rangle = \{ \pm I_2, \ \ \pm g_1 \}$. To this end, let
\[
D_1 = S^{-1} g_1 S = \left( \begin{array}{cc}
i  &  0  \\
0  &  -i
\end{array}   \right)
\]
be the diagonal form of $g_1$ and
\[
D_2 = S^{-1} g_2 S = \left(  \begin{array}{cc}
a  &   b  \\
c  &  d
\end{array}   \right)
\]
for appropriate matrices $S$ and $D_2$ with entries from $Q(R)(\sqrt{-1}) = {\mathbb Q} (\sqrt{-1})$ or ${\mathbb Q}( \sqrt{d}, \sqrt{-1})$. The determinant and the trace are invariant under conjugation, so that suffices to show that if $\det(D_2) =1$ and ${\rm tr} (D_2) = {\rm tr} (D_1 D_2) = 0$ then $D_2 D_1 = - D_1 D_2$, whereas
\[
g_2 g_1 = (S D_2 S^{-1}) (S D_1 S^{-1}) = S (D_2 D_1) S^{-1} =
\]
\[
 = S ( - D_1 D_2) S^{-1} = - (S D_1 S^{-1}) (S D_2 S^{-1}) = - g_1 g_2.
\]
Indeed, ${\rm tr} (D_2) = a + d = 0$ and ${\rm tr} (D_1 D_2) = i (a - d) = 0$ require $a = d = 0$. Now, $\det(D_2) = - bc =1$ determines $c = - \frac{1}{b}$ for some $b \in {\mathbb Q} ( \sqrt{d}, \sqrt{-1})$ and
\[
D_2 D_1 = \left(  \begin{array}{cc}
0  &  - i b  \\
- \frac{i}{b}  &  0
\end{array}   \right) = - D_1 D_2.
\]
Thus, if $a=3$ then the Sylow $2$-subgroup of $K$ is isomorphic to the quaternion group ${\mathbb Q}_8$ of order $8$,
\[
H_8 = \langle g_1, g_2 \ \ \vert \ \  g_1^2 = g_2^2 = - I_2, \ \ g_2g_1 = - g_1 g_2 \rangle \simeq {\mathbb Q}_8.
\]
There remains to be rejected the case of $a \geq 4$. The assumption $a \geq 4$ implies the existence of $g_3 \in H_{2^a} \setminus \langle g_1, g_2 \rangle$. Any such $g_3$ is of order $4$, together with the products $g_1g_3 \in H_{2^a}$ for $1 \leq j \leq 2$, since $g_j g_3 = \pm I_2$ amounts to $g_3 = \pm g_j ^3 \in \langle g_j \rangle$ and contradicts the choice of $g_3$. Thus, the subgroups
\[
\langle g_1, g_3 \ \ \vert \ \  g_1^2 = g_3 ^2 = - I_2, \ \ g_3 g_1 = - g_1 g_3 \rangle \simeq
\]
\[
\langle g_2, g_3, \ \ \vert \ \  g_2^2 = g_3 ^2 = - I_2, \ \ g_3g_2 = - g_2 g_3 \rangle \simeq {\mathbb Q}_8
\]
are also isomorphic to ${\mathbb Q}_8$. In particular,
\[
D_3 = S^{-1} g_3 S = \left(  \begin{array}{cc}
0   &  b_3  \\
- \frac{1}{b_3}   &   0
\end{array}   \right)
\]
with $b_3 \in {\mathbb Q} ( \sqrt{d}, \sqrt{-1}) ^*$ is subject to
\[
D_3 D_2 = \left(  \begin{array}{cc}
- \frac{b_3}{b}   &  0  \\
0  &  - \frac{b}{b_3}
\end{array}  \right) =
\left( \begin{array}{cc}
 \frac{b}{b_3}  &  0  \\
0  &  \frac{b_3}{b}
\end{array}  \right) = - D_2 D_3,
\]
whereas $b_3^2  = - b^2$ or $b_3 = \pm i b$. As a result, $D_3 = D_1D_2$ and $g_3 = g_1g_2$, contrary to the choice of $g_3 \not \in \langle g_1, g_2 \rangle$. Therefore $a < 4$ and the Sylow $2$-subgroup of a finite group $K < SL(2,R)$ is $H_2 \simeq {\mathbb C}_2$, $H_4 \simeq {\mathbb C}_4$ or $H_8 \simeq {\mathbb Q}_8$.

\end{proof}

\begin{proposition}    \label{K}
Any finite subgroup $K$ of $SL(2,R)$ is isomorphic to one of the following:
\[
K_1 = \{ I_2 \},
\]
\[
K_2 = \langle - I_2 \rangle \simeq {\mathbb C}_2,
\]
\[
K_3 = \langle g_1 \rangle \simeq {\mathbb C} _4 \ \ \mbox{\rm  for some } \ \ g_1 \in SL(2,R) \ \ \mbox{\rm  with } \ \ {\rm tr} (g_1) = 0,
\]
\[
K_4 = \langle g_1, g_2 \ \ \vert \ \  g_1 ^2 = g_2 ^2 = - I_2,  \ \ g_2 g_1 g_2 = g_1 \rangle \simeq {\mathbb Q}_8,
\]
\[
K_5 = \langle g_3 \rangle \simeq {\mathbb C}_3 \ \ \mbox{ \rm for some } \ \ g_3 \in SL(2,R) \ \ \mbox{\rm  with } \ \ {\rm tr} (g_3) = -1,
\]
\[
K_6 = \langle g_4 \rangle \simeq {\mathbb C}_6 \ \ \mbox{ \rm for some } \ \ g_4 \in SL(2,R) \ \ \mbox{\rm  with } \ \ {\rm tr} (g_4) = 1,
\]
\[
K_7 = \langle g_1, g_4 \ \ \vert \ \ g_1^2 = g_4^3 = - I_2, \ \  g_4 g_1 g_4 = g_1 \rangle \simeq {\mathbb Q}_{12}
\]
 for some $g_1, g_4 \in SL(2,R)$  with ${\rm tr}(g_1) = 0$, ${\rm tr} (g_4) =1$,
\[
K_8 = \langle g_1, g_2, g_3 \ \ \vert \ \  g_1 ^2 = g_2 ^2 = - I_2, \ \ g_3 ^3 = I_2, \ \  g_2 g_1 = - g_1 g_2, \ \
\]
\[
g_3 g_1 g_3^{-1} = g_2, \ \ g_3 g_2 g_3^{-1} = g_1g_2 \rangle \simeq SL(2, {\mathbb F}_3)
\]
for some $g_1, g_2, g_3 \in SL(2,R)$, ${\rm tr} (g_1) = {\rm tr}(g_2) = 0$, ${\rm tr} (g_3) = -1$, where ${\mathbb Q}_8$ denotes the quaternion group of order $8$, ${\mathbb Q}_{12}$ stands for the dicyclic group of order $12$ and $SL(2, {\mathbb F}_3)$ is the special linear group over the field ${\mathbb F}_3$ with three elements.
\end{proposition}

\begin{proof}

By Proposition \ref{SylowSubgroupsK}, $K$ is of order $1, 2, 3, 6, 12$ or $24$. The only subgroup $K < SL(2,R)$ of order $1$ is $K = K_1 = \{ I_2 \}$. Since $- I_2$ is the only element of $SL(2,R)$ of order $2$, the group $K = K_2 = \langle - I_2 \rangle \simeq {\mathbb C}_2$ is the only cyclic subgroup of $SL(2,R)$ of order $2$. Any subgroup $K < SL(2,R)$ of order $4$ is cyclic or $K = K_3 = \langle g_1 \rangle$ for some $g_1 \in SL(2,R)$ with ${\rm tr} (g_1) = 0$, because $SL(2,R)$ has a unique element $ - I_2$ of order $2$. Proposition \ref{ListK} has established the existence of elements $g_1 \in SL(2, {\mathbb Z}) \leq SL(2,R)$ of order $4$.

If $K < SL(2,R)$ is a subgroup of order $8$ then it coincides with its Sylow $2$-subgroup
\[
K = H_8 = \langle g_1, g_2 \ \ \vert \ \  g_1^2 = g_2 ^2 = - I_2, \ \  g_2g_1 = - g_1g_2 \rangle = K_4 \simeq {\mathbb Q}_8,
\]
isomorphic to the quaternion group ${\mathbb Q}_8$ of order $8$. Note that there is a realization
\[
{\mathbb Q}_8 \simeq \langle D_1 = \left(  \begin{array}{cc}
i  &  0  \\
0  &  -i
\end{array}   \right), \ \
D_2 = \left( \begin{array}{cc}
0  &  1  \\
-1  &  0
\end{array} \right) \rangle < SL(2, {\mathbb Z} [i])
\]
as a subgroup of $SL(2, {\mathbb Z}[i])$. In general,
\[
D_j = \left( \begin{array}{cc}
a_j  &  b_j  \\
c_j  &  - a_j
\end{array}   \right) \in SL(2,R)
\]
amount to $a_j^2 + b_jc_j = -1$. The anti-commuting  relation $g_2g_1 = - g_1g_2$ is equivalent to $2 a_1a_2 + b_1c_2 + b_2 c_1 = 0$. Therefore $K_4 = \langle g_1, g_2 \rangle < SL(2,R)$ is a realization of ${\mathbb Q}_8$ if and only if $a_j, b_j, c_j \in R$ are subject to
\begin{equation}   \label{Q8System}
\left|  \begin{array}{c}
a_1^2 + b_1 c_1 = -1  \\
a_2^2 + b_2 c_2 = -1 \\
2 a_1 a_2 + b_1 c_2 + b_2 c_1 = 0
\end{array}  \right. .
\end{equation}
The existence of a solution of (\ref{Q8System}) in an arbitrary $R = R_{-d,f} = {\mathbb Z} + f \mathcal{O}_{-d} = {\mathbb Z} + f \omega _{-d} {\mathbb Z}$ is an open problem.

If $|K|=3$ then $K = K_5 = \langle g_3 \rangle \simeq {\mathbb C}_3$ for some $g_3 \in SL(2,R)$ with ${\rm tr} (g_3) = -1$.

From now on, let us assume that $K$ is of order $|K| = 2^a . 3$ for some $1 \leq a \leq 3$ and consider some Sylow subgroups $H_{2в}, H_3 = \langle g_4 \rangle \simeq {\mathbb C}_3 $ of $K$. We claim that the product
\[
H_{2^a} H_3 = \{ g g_4 ^i \ \ \vert \ \  g \in H_{2^a}, \ \  0 \leq i \leq 2 \}
\]
depletes $K$. More precisely, $H_{2^a} \cap H_3 = \{ I_2 \}$, because $2^a$ and $3$ are relatively prime. Therefore
\[
H_{2^a} H_3 / H_{2^a} =  H_{2^a} \cup  H_{2^a}g_4 \cup  H_{2^a} g_4^2
\]
is a right coset decomposition of the subset $H_{2^a} H_3 \subseteq K$ modulo $H_{2^a}$. Due to the disjointness of this decomposition, one has $|H_{2^a} H_3| = 3 |H_{2^a}| = 3.2^a = |K|$. Therefore, the subset $H_{2^a} H_3$ of $K$ coincides with $K$ and $K = H_{2^a} H_3$ is a product of its Sylow subgroups.

If $K = H_2H_3 = \langle - I_2 \rangle \langle g_3 \rangle$ for some $g_3 \in SL(2,R)$ with ${\rm tr} (g_3) = -1$ then $ \pm I_2$ commute with $g_3^j$ for all $0 \leq j \leq 2$ and the group $K$ is abelian. Thus, $K = \langle - g_3 \rangle \simeq {\mathbb C}_6$ is a cyclic group of order $6$, generated by $ - g_3 \in SL(2,R)$ with ${\rm tr} ( - g_3) =1$.

For $K = H_4 H_3 = \langle g_1 \rangle \langle g_3 \rangle$ with $g_1, g_3 \in SL(2,R)$ of ${\rm tr} (g_1) =0$, ${\rm tr} (g_3) = -1$, note that $g_4 = -g_3 \in SL(2,R)$ is of order $6$. Then $g_4 ^3 = -I_2 = g_1^2$, because $-I_2 \in SL(2,R)$ is the only element of order $2$. We claim that $g_1, g_4 \in SL(2,R)$ are subject to $g_4 g_1 g_4 = g_1$. To this end, let $S \in Q(R)  ( \sqrt{-3}))) _{2 \times 2} \subseteq {\mathbb Q} ( \sqrt{-d}, \sqrt{-3})_{2 \times 2}$ be a matrix, whose columns are eigenvectors of $g_1$. Then
\[
D_4 = S^{-1} g_4 S = \left(  \begin{array}{cc}
e^{ \frac{\pi i}{3}}   &   0  \\
0  &  e^{ - \frac{\pi i}{3}}
\end{array}  \right) \ \ \mbox{   and  }
\]
\[
D_1 = S^{-1} g_1 S = \left(  \begin{array}{cc}
a_1  &  b_1  \\
c_1  &  - a_1
\end{array}   \right) \ \ \mbox{ with } \ \  a_1^2 + b_1 c_1 = -1
\]
generate the subgroup $K^o = S^{-1} K S \simeq K$. It suffices to check that $D_4D_1D_4 = D_1$, because then $g_4g_1g_4 =  (SD_4S^{-1})(SD_1S^{-1})(SD_4S^{-1}) = S(D_4D_1D_4)S^{-1} = SD_1S^{-1} = g_1$ and
\[
K = \langle g_1, g_3 \rangle = \langle g_1, g_4 = -g_3 \ \ \vert \ \  g_1^2 = g_4 ^3 = - I_2, \ \  g_4 g_1 g_4 = g_1 \rangle \simeq {\mathbb Q}_{12}
\]
is isomorphic to the dicyclic group ${\mathbb Q}_{12}$ of order $12$. The group $K^o = \langle D_1, D_4 \rangle \simeq K$ of order $12$ has a cyclic subgroup $\langle D_4 \rangle \simeq {\mathbb C}_6$ of order $6$. The index $[ K^o : \langle D_4 \rangle ] = 2$, so that $\langle D_4 \rangle$ is a normal subgroup of $K^o$ and $D_1 D_4 D_1^{-1} \in \langle D_4 \rangle$ is an element of order $6$. More precisely, $D_1 D_4 D_1^{-1} = D_4$ or $ D_1 D_4 D_1 ^{-1} = D_4 ^{-1} = D_4 ^5 = - D_4^2$. If $D_1 D_4 = D_4 D_1$ then $ D_1 D_4 \in K^o$ is of order $12$, as far as $(D_1 D_4) ^{12} = (D_1^{4}) ^3 (D_4^{6}) ^2 = I_2 ^{3} I_2 ^{2} = I_1$, $( D_1 D_4) ^{6} = D_1 ^{2} = - I_2 \neq I_2$, $( D_1 D_4) ^{4} = D_4 ^4 = - D_4 \neq I_2$, whereas $D_1D_4, (D_1D_4) ^{2}, (D_1D_4)^{3} \not \in \{ I_2 \}$. Consequently, $D_1 D_4 = - D_4^2 D_1$, so that $D_4 D_1 D_4 = - D_4^3 D_1 = D_1$ and $K \simeq K^o \simeq {\mathbb Q}_{12}$. For instance, the subgroup
\[
\langle D_1 = \left( \begin{array}{cc}
0  &  1  \\
-1  &  0
\end{array}   \right), \ \
D_4 = \left( \begin{array}{cc}
e^{ \frac{\pi i}{3}}  &  0  \\
0  &  e^{ - \frac{\pi i}{3}}
\end{array}   \right) \ \ \Bigg \vert \ \  D_1 ^2 = D_4 ^3 = - I_2, \ \  D_1 D_4 D_1 ^{-1} = D_4 ^{-1} \rangle
\]
of $SL(2, \mathcal{O}_{-3})$ realizes ${\mathbb Q}_{12}$ as a subgroup of $SL(2, \mathcal{O}_{-3})$. The existence of ${\mathbb Q}_{12} \simeq K < SL(2, R)$ for an arbitrary $R$ is an open problem.

There remains to be shown that any subgroup $K = H_8 H_3 = \langle g_1, g_2, g_3 \rangle \simeq {\mathbb Q}_8 {\mathbb C}_3$ of $SL(2, R)$ of order $24$ is isomorphic to the special linear group $K_8 \simeq SL(2, {\mathbb F}_3)$ over ${\mathbb F}_3$.  In other words, any $K < SL(2,R)$ of order $|K| = 24$ can be generated by such $g_1, g_2, g_3 \in SL(2,R)$ that the subgroup $\langle g_1, g_2 \ \ \vert  \ \ g_1 ^2 = g_2 ^2 = - I_2, \ \  g_2 g_1 = - g_1 g_2 \rangle \simeq {\mathbb Q}_8$ is isomorphic to the quaternion group ${\mathbb Q}_8$ of order $8$, $g_3$ is of order $3$ and $g_3 g_1 g_3^{-1} = g_2$, $g_3 g_2 g_3 ^{-1} = g_1 g_2 $.

First of all,  the Sylow $2$-subgroup $H_8 \simeq {\mathbb Q}_8$ of $K$ is normal. More precisely, by the Third Sylow Theorem, the number $n_2 \in {\mathbb N}$ of the Sylow $2$-subgroups of $K$ (i.e., the number $n_2$ of the subgroups of $K$ of order $8$) divides $|K| = 24$ and $n_2 \equiv 1 ({\rm mod} \ \  2)$. Therefore $n_2 = 1$ or $n_2 = 3$. By Second Sylow Theorem, all Sylow $2$-subgroups are conjugate to each other, so that $n_2 =1$ exactly when $H_8 = \langle g_1, g_2 \rangle \simeq {\mathbb Q}_8$ is a normal subgroup of $K$. Let us assume that $n_2 =3$ and denote by $\nu _s$ the number of the elements $g \in K$ of order $s$. Due to $- I_2 \in H_8 = \langle g_1, g_2 \rangle < K$, one has $\nu _1 =1$, $\nu _2 =1$. Note that $g \in K$ is of order $3$ if and only if $ - g \in K$ is of order $6$, so that $\nu _6 = \nu _3$. By the Third Sylow Theorem, the number $n_3 \in {\mathbb N} $ of the Sylow $3$-subgroups of $K$ divides $|K| =24$  and $n_3 \equiv 1 ({\rm mod} \ \ 3)$. Therefore $n_3 =1$ or $n_3=4$.

If $n_3=1$ and there is a unique normal subgroup $H_3 = \langle g_3 \rangle \simeq {\mathbb C}_3$ of $K$ of order $3$, then $g_j g_3 g_j ^{-1} \in \{ g_3, g_3^2 \} \subset \langle g_3 \rangle$ for $j=1$ and $j=2$. If $g_j g_3 g_j ^{-1} = g_3$ then $g_j g_3 = g_3 g_j$ for $g_j$ of order $4$ and $g_3$ of order $3$, so that $g_j g_3 \in K$ is of order $12$, contrary to the non-existence of an element of $SL(2,R)$ of order $12$. Therefore $g_1 g_3 g_1^{-1} = g_3^2$, $g_2 g_3 g_2^{-1} = g_3 ^2$, whereas
\[
(g_1g_2) g_3 (g_1g_2) ^{-1} = g_1 ( g_2 g_3 g_2 ^{-1}) g_1 ^{-1} = g_1 g_3 ^2 g_1 ^{-1} = ( g_1 g_3 g_1 ^{-1}) ^2 = (g_3 ^2) ^2 = g_3
\]
and $g_1g_2$ of order $4$ commutes with $g_3$ of order $3$. Thus, $(g_1g_2)g_3 \in K$ is of order $12$, which is an absurd. That rejects the assumption $n_3 =1$ and proves that $n_3=4$.

Let $H_{3,j} = \langle g_{3,j} \rangle \simeq {\mathbb C}_3$, $1 \leq j \leq 4$ be the four subgroups of $K$ of order $3$. Then $H_{3,i} \cap H_{3,j} = \{ I_2 \}$ for all $1 \leq i < j \leq 4$, as far as any $g \in H_{3,i} \setminus \{ I_2 \}$ generates $H_{3,i}$. As a result, $\cup _{i=1} ^4 H_{3,i}$ and $K$ contain $8$ different elements $g_{3,i}, g_{3,i} ^2$, $1 \leq i \leq 4$ of order $3$ and $\nu _6 = \nu _3 =8$. Thus,
\[
24 = |K| = \nu _1 + \nu _2 + \nu _3 + \nu _4 + \nu _6 = 18 + \nu _4,
\]
so that $K$ has $\nu _4=6$ elements of order $4$. Since any Sylow $2$-subgroup
\[
H_8 = \langle g_1, g_2 \ \ \vert \ \  g_1 ^2 = g_2 ^2 = - I_2, \ \ g_2 g_1 = - g_1 g_2 \rangle = \{ \pm I_2, \pm g_1, \pm g_2, \pm g_1g_2 \} \simeq {\mathbb Q}_8
\]
of $K$ contains six elements $\pm g_1, \pm g_2, \pm g_1g_2$ of order $4$, there cannot be more than one $H_8$. In other words, $n_2=1$ and $H_8$ is a normal subgroup of $K$.

The above considerations show that
\[
K = H_8 \rtimes H_3 = \langle g_1, g_2 \ \ \vert \ \  g_1^2 = g_2 ^2 = - I_2, \ \  g_2g_1 = - g_1g_2 \rangle \rtimes \langle g_3 \ \ \vert \ \  g_3^3 = I_2 \rangle \simeq {\mathbb Q}_8 \rtimes {\mathbb C}_3
\]
is a semi-direct product of ${\mathbb Q}_8$ and ${\mathbb C}_3$. Up to an isomorphism, $K$ is uniquely de\-ter\-mined by the group homomorphism
\[
\varphi _K : H_3 \longrightarrow Aut (H_8),
\]
\[
\varphi _K ( g_3 ^j ) ( \pm g_1 ^k g_2 ^l ) = g_3 ^j ( \pm g_1 ^k g_2 ^l ) g_3 ^{-j} \ \ \mbox{ for } \ \
\forall \pm g_1^k g_2 ^l \in H_8, \ \ 0 \leq k,l \leq 1.
\]
Since $H_3 = \langle g_3 \rangle \simeq {\mathbb C}_3$ is cyclic, $\varphi _K$ is uniquely determined by $\varphi _K (g_3) \in Aut (H_8)$. On the  other hand, $H_8$ is generated by $g_1,g_2$, so that suffices to specify $\varphi _K(g_3) (g_j) = g_3 g_j g_3^{-1} \in H_8$ for $1 \leq j \leq 2$, in order to determine $\varphi _K$. If the cyclic group $\langle g_1 \rangle \simeq {\mathbb C}_4$ is normalized by $g_3$ then $g_3 g_1 g_3^{-1} \in \{ \pm g_1 \}$, as an element of order $4$. In the case of $g_3g_1 g_3^{-1} = g_1$, the element $g_1 \in K$ of order $4$ commutes with the element $g_3 \in K$ of order $3$ and their product $g_1g_3 \in K$ is of order $12$. The lack of $g \in SL(2,R)$ of order $12$ requires $g_3g_1g_3^{-1} = - g_1$. Now,
\[
g_3^2 g_1 g_3 ^{-2} = g_3 ( g_3 g_1 g_3^{-1}) g_3^{-1} = g_3 ( -g_1) g_3^{-1} = g_1
\]
is equivalent to $g_3^2 g_1 = g_1 g_3^2$ and the product $g_1g_3^2 \in K$ of $g_1 \in K$ of order $4$ with $g_3^2 \in K$ of order $3$ is an element of order $12$. The absurd justifies that neither of the cyclic subgroups $\langle g_1 \rangle \simeq \langle g_2 \rangle \simeq \langle g_1g_2 \rangle \simeq {\mathbb C}_4$ of order $4$ of $H_8$ is normalized by $g_3$. Thus, an arbitrary $g_1 \in H_8 \simeq {\mathbb Q}_8$ of order $4$ is completed by $g_2 := g_3 g_1 g_3 ^{-1} \in H_8 \setminus \langle g_1 \rangle$ of order $4$ to a generating set of $H_8 \simeq {\mathbb Q}_8$. Then
\[
g_3^2 g_1 g_3^{-2} = g_3 (g_3 g_1 g_3^{-1}) g_3^{-1} = g_3 g_2 g_3^{-1} \in H_8 \setminus ( \langle g_1 \rangle \cup \langle g_2 \rangle ) = \{ g_1g_2, g_2g_1 \}
\]
specifies that either $g_3g_2 g_3^{-1} = g_1g_2$ or $g_3g_2g_3^{-1} = g_2g_1$. If $g_3g_2g_3^{-1} = g_2g_1$, we replace the generator $g_3$ of $K$ by $h_3 = g_3^2$ and note that $h_3 g_1 h_3^{-1} = g_2g_1$. Now, $h_1 := g_1$ and $h_2 := g_2g_1$ generate $H_8 = \langle h_1, h_2 \ \ \vert \ \  h_1^2 = h_2 ^2 = - I_2, \ \ h_2h_1 = - h_1 h_2 \rangle$ and satisfy $h_3 h_1 h_3^{-1} = h_2$,
\[
h_3h_2h_3^{-1} = g_3 [ (g_3g_2g_3^{-1})(g_3g_1g_3^{-1}) ] g_3^{-1} = g_3 (g_2g_1g_2) g_3^{-1} = g_3g_1g_3^{-1} =
\]
\[
=  g_2 = - (g_2g_1) g_1 = - h_2h_1 = h_1h_2.
\]
Thus, the group
\[
K'= \langle g_1, g_2, g_3 \, \vert \,
 g_1^2 = g_2 ^2 = - I_2, \, g_2g_1 = - g_1g_2, \, g_3^3 = I_2, \,  g_3g_1g_3^{-1} = g_2, \, g_3g_2g_3^{-1} = g_2g_1 \rangle
\]
is isomorphic ro the group
\[
K = \langle g_1, g_2, g_3 \, \vert \,
  g_1^2 = g_2 ^2 = - I_2, \,  g_2g_1 = - g_1g_2, \, g_3^3 = I_2, \,  g_3g_1g_3^{-1} = g_2, \,  g_3g_2g_3^{-1} = g_1g_2 \rangle.
\]
We shall realize $SL(2, {\mathbb F}_3)$ as a subgroup $K_8 ^o = \langle D_1, D_2, D_3 \rangle$ of $SL(2, {\mathbb Q}( \sqrt{-d}, \sqrt{-3}))$. The existence of subgroups $SL(2, {\mathbb F}_3) \simeq K_8 < SL(2,R)$ is an open problem. Towards the construction of $K_8 ^o$, let us choose
\[
D_j = \left( \begin{array}{cc}
a_j  &  b_j  \\
c_j  &  - a_j
\end{array} \right) \ \ \mbox{ \rm with } \ \ a_j^2 + b_jc_j = -1 \ \ \mbox{ \rm for } \ \ 1 \leq j \leq 2 \ \ \mbox{ \rm and}
\]
\[
D_3 = \left( \begin{array}{cc}
e^{\frac{2 \pi i}{3}}  &  0  \\
0  &  e^{ - \frac{2 \pi i}{3}}
\end{array}  \right)
\]
from $SL(2, {\mathbb Q}( \sqrt{-d}, \sqrt{-3}))$. After computing
\[
D_3D_j D_3^{-1} = \left( \begin{array}{cc}
a_j  &  e^{ - \frac{2 \pi i}{3}}  b_j  \\
\mbox{   }   & \mbox{  } \\
e^{\frac{2 \pi i}{3}} c_j  &  - a_j
\end{array}  \right) \ \ \mbox{\rm for} \ \ 1 \leq j \leq 2,
\]
observe that $D_3D_1D_3^{-1} = D_2$ reduces to
\[
\left|  \begin{array}{c}
a_2 = a_1  \\
b_2 = e^{ - \frac{ 2 \pi i}{3}} b_1  \\
c_2 = e^{\frac{2 \pi i}{3}} c_1
\end{array}  \right. .
\]
The relation $D_2D_1 = - D_1D_2 $ is equivalent to $2 a_1a_2 + b_1c_2 + b_2c_1 = 0$ and implies that $2 a_1^2 = b_1c_1$. Now,
\[
D_3D_2D_3^{-1} = \left( \begin{array}{cc}
a_1  &  e^{\frac{2 \pi i}{3}} b_1 \\
e^{ - \frac{2 \pi i}{3}} c_1  &  - a_1
\end{array}  \right) =
\left( \begin{array}{cc}
\sqrt{-3} a_1^2  &  \sqrt{-3} e^{\frac{2 \pi i}{3}} a_1b_1 \\
\sqrt{-3}  e^{ - \frac{2 \pi i}{3}} a_1c_1  &  - \sqrt{-3}  a_1^2
\end{array}  \right) = D_1D_2
\]
is tantamount to
\[
\left| \begin{array}{c}
a_1 (1 - \sqrt{-3}a_1) = 0 \\
b_1 (1 - \sqrt{-3}a_1) = 0 \\
c_1 (1 - \sqrt{-3}a_1) = 0
\end{array}  \right.
\]
and specifies that $a_1 = \frac{\sqrt{-3}}{3}$. Namely, the assumption $a_1 \neq - \frac{\sqrt{-3}}{3}$ forces $a_1 = b_1 = c_1 =0$, whereas $\det (D_1) = 0$, contrary to the choice of $D_1 \in SL(2, {\mathbb Q}( \sqrt{-d}, \sqrt{-3}))$. As a result, $b_1 \neq 0$, $c_1 = - \frac{2}{3 b_1}$ and
\[
D_1 = \left( \begin{array}{cc}
- \frac{\sqrt{-3}}{3}   &  b_1  \\
\mbox{   }   &  \mbox{   } \\
- \frac{2}{3 b_1}  &  \frac{\sqrt{-3}}{3}
\end{array}  \right), \ \
D_2 = \left( \begin{array}{cc}
- \frac{\sqrt{-3}}{3}  &  e^{ - \frac{ 2 \pi i}{3}} b_1  \\
\mbox{   }   &  \mbox{   } \\
e^{ \frac{2 \pi i}{3}} c_1  &  \frac{\sqrt{-3}}{3}
\end{array}  \right), \ \
D_3 = \left(  \begin{array}{cc}
e^{\frac{2 \pi i}{3}}  &  0  \\
\mbox{   }   &  \mbox{   } \\
0  &  e^{ - \frac{2 \pi i}{3}}
\end{array}  \right)
\]
generate a subgroup $SL(2, {\mathbb F}_3) \simeq K_8 ^o < SL(2, {\mathbb Q} ( \sqrt{-d}, \sqrt{-3}))$.

\end{proof}

\begin{corollary}    \label{SubgroupSL23ButQ12}
If the finite subgroup $K$ of $SL(2,R)$ is not isomorphic to the dicyclic group
\[
K_7 = \langle g_1, g_4 \ \ \vert   \ \  g_1 ^2 = g_4^3 = -I_2,  \ \ g_4 g_1 g_4 = g_1 \rangle =
\]
\[
= \langle g_1, g_3 = - g_4 \ \ \vert \ \  g_1 ^2 = - I_2,  \ \ g_3^3 = I_2, \ \  g_3 g_1 g_3^{-1} = g_3 g_1 \rangle \simeq {\mathbb Q}_{12}
\]
of order $12$ then $K$ is isomorphic to a subgroup of the special linear group
\[
K_8 = \langle g_1, g_2, g_3 \,\vert \, g_1^2 = g_2 ^2 = - I_2, \, g_3 ^3 = I_2,  \, g_2 g_1 = - g_1 g_2, \, g_3 g_1 g_3^{-1} = g_2,
 g_3 g_2 g_3^{-1} = g_1 g_2 \rangle
 \]
 \[
\simeq  SL(2, {\mathbb F}_3)
\]
over the field ${\mathbb F}_3$ with three elements.
\end{corollary}

\begin{proof}

According to Proposition \ref{K}, any finite subgroup $K < SL(2,R)$ is isomorphic to some of the groups $K_1, \ldots , K_8$. Thus, it suffices to establish that any $K_j$, $1 \leq  j \leq 6$ is isomorphic to a subgroup of $K_8$. Note that $K_1 = \{ I_2 \} \subset K_8$ and $K_2 = \langle - I_2 \rangle \subset K_8$ are subgroups of $K_8$. The generator $g_1$ of $K_8$ is of order $4$, so that any subgroup $K_3 \simeq {\mathbb C}_4$ of $SL(2,R)$ is isomorphic to the subgroup $\langle g_1 \rangle$ of $K_8$. In the proof of Proposition \ref{K} we have seen that $K_8$ has a normal Sylow $2$-subgroup
\[
H_8 = \langle g_1, g_2 \ \ \vert \ \ g_1^2 = g_2^2 = -I_2, \ \  g_2g_1 = - g_1 g_2 \rangle \simeq {\mathbb Q}_8,
\]
isomorphic to the quaternion group ${\mathbb Q}_8 \simeq K_4$ of order $8$. The generator $g_3$ of $K_8$ provides a subgroup $\langle g_3 \rangle \simeq {\mathbb C}_3 \simeq K_5$ of $K_8$. The product $(-I_2) g_3$ of the commuting elements $-I_2 \in K_8$ or order $2$ and $g_3 \in K_8$ of order $3$ is an element $-g_3 \in K_8$ of order $6$, so that $K_6 \simeq {\mathbb C}_6$ is isomorphic to the subgroup $\langle - g_3 \rangle$ of $K_8$.

\end{proof}


Towards the classification of the finite subgroups of $GL(2,R)$, we proceed with the following:

\begin{lemma}   \label{HKh}
Let $H$ be a finite subgroup of $GL(2,R)$. Then

(i) $\det(H)$ is a cyclic subgroup of $R^*$;

(ii) $H$ is a product $H = [ H \cap SL(2,R)] \langle h_o \rangle$ of its normal subgroup $H \cap SL(2,R)$ and any ${\mathbb C}_r \simeq \langle h_o \rangle \subseteq H$ with $\det (H) = \langle \det (h_o) \rangle \simeq {\mathbb C}_s$;

(iii) the order $s$ of $\det(H) = \langle \det (h_o) \rangle$ divides the order $r$ of $h_o \in H$ and
\[
[ H \cap SL(2,R)] \cap \langle h_o \rangle = \langle h_o^s \rangle \simeq {\mathbb C} _{\frac{r}{s}};
\]

(iv)  $H$ is of order $s|H \cap SL(2,R)|$;

(v) $s=r$ if and only if $H = [ H \cap SL(2,R)] \leftthreetimes \langle h_o \rangle$ is a semi-direct product.
\end{lemma}

\begin{proof}

(i) The image $\det(H)$ of the group homomorphism $\det : H \rightarrow R^*$ is a subgroup of $R^*$. As far as the units group $R^*$ of the endomorphism ring $R$ of $E$ is cyclic, its subgroup $\det(H)$ is cyclic, as well.

(ii) If $\det(h_o)$ is a generator of the cyclic subgroup $\det(H) < R^*$ then one can represent $H = [ H \cap SL(2,R)] \langle h_o \rangle$. The inclusion $[ H \cap SL(2,R)] \langle h_o \rangle \subseteq H$ is clear by the choice of $h_o \in H$. For the opposite inclusion, note that any $h \in H$ with $\det (h) = \det (h_o)^m$ for some $m \in {\mathbb Z}$ is associated with $h h_o^{-m} \in H \cap SL(2,R)$, so that $h = (h h_o^{-m})h_o^m \in [ H \cap SL(2,R)] \langle h_o \rangle$ and $H \subseteq [ H \cap SL(2,R)] \langle$.

(iii) If $h_o \in H$ is of order $r$ then $h_o^r =I_2$ and $\det(h_o)^r=1$. Therefore the order $s$ of $\det (h_o) \in R^*$ divides $s$. Note that $h_o^s \in [ H \cap SL(2,R)] \cap \langle h_o \rangle$, as far as $\det(h_o^s) = \det(h_o) ^s =1$. Therefore $\langle h_o^s \rangle$ is a subgroup of $[ H \cap SL(2,R)] \cap \langle h_o \rangle$. Conversely, any $h_o^x \in [H \cap SL(2,R)] \cap \langle h_o \rangle$ has $\det(h_o^x) = \det(h_o) ^x =1$, so that $s$ divides $x$ and $h_o^x \in \langle h_o ^s \rangle$. That justifies $[ H \cap SL(2,R)] \cap \langle h_o \rangle \subseteq \langle h_o^s \rangle$ and $[ H \cap SL(2,R)] \cap \langle h_o \rangle = \langle h_o^s \rangle$. The order of $\langle h_o^s \rangle$ and $h_o^s$ is $\frac{r}{s}$, since $s$ divides $r$.

(iv) It suffices to show that
\[
H = \cup _{i=0} ^{s-1} [ H \cap SL(2,R)] h_o^j
\]
is the coset decomposition of $H$ with respect to its normal subgroup $H \cap SL(2,R)$, in order to conclude that the order $|H|$ of $H$ is $s$ times the order  $|H \cap SL(2,R)|$ of $H \cap SL(2,R)$.  The inclusion $H \supseteq \cup _{i=0} ^{s-1} [ H \cap SL(2,R)] h_o^j$ is clear by the choice of $h_o \in H$. According to $H = [ H \cap SL(2,R)] \langle h_o \rangle$, any element of $H$ is of the form $h = gh_o^m$ for some $g \in H \cap SL(2,R)$ and $m \in {\mathbb Z}$. If $m = sq + r_o$ is the division of $m$ by $s$ with residue $0 \leq r_o \leq s-1$ then $h = [ g (h_o^s) ^q] h_o^{r_o} \in [ H \cap SL(2,R)] h_o^{r_o}$, due to $h_o^s \in H \cap SL(2,R)$. Therefore $H \subseteq \cup _{j=0} ^{s-1} [ H \cap SL(2,R)] h_o^j$ and $H = \cup _{j=0} ^{s-1} [ H \cap SL(2,R)] h_o^j$. The cosets $[H \cap SL(2,R)]h_o^{i} $ and $[H \cap SL(2,R)] h_o^j$ are mutually disjoint for any $0 \leq i < j \leq s-1$, because the assumption $g_1 h_{i} = g_2h_o^j$ for $g_1, g_2 \in H \cap SL(2,R)$ implies that $h_o^{j-i} = g_2^{-1} g_1 \in [ H \cap SL(2,R)] \cap \langle h_o \rangle = \langle h_o^s \rangle$. As a result, $s$ divides $0 < j-i < s$, which is an absurd.

(v) According to (iii), the order $s$ of $\det(h_o)$ divides the order $r$ of $h_o$. On the  other hand, $h_o^s = I_2$ exactly when $r$ divides $s$, so that $h_o^s = I_2$ is equivalent to $r=s$. Thus, $r=s$ exactly when
\[
[ H \cap SL(2,R)] \cap \langle h_o \rangle = \{ I_2 \}.
\]
As far as the product of the normal subgroup $H \cap SL(2,R)$ and the subgroup $\langle h_o \rangle$ is the entire $H$, one has a semi-direct product $H = [ H \cap SL(2,R) ] \rtimes \langle h_o \rangle$ if and only if $r=s$.

\end{proof}

\begin{lemma}  \label{StructureH}
Let $H = [ H \cap SL(2,R)] \langle h_o \rangle$ be a finite subgroup of $GL(2,R)$ for $h_o  \in H$ of order $r$ with $\det(H) = \langle \det(h_o) \rangle \simeq {\mathbb C}_s$ and $H \cap SL(2,R)$ be generated by $g_0 = h_o^s, g_1, \ldots , g_t$. Then $H \cap SL(2,R)$, $r$ and
\[
h_o g_i h_o^{-1} \in H \cap SL(2,R) \ \ \mbox{  for all  } \ \ 1 \leq i \leq t
\]
 determine $H$ up to an isomorphism.
\end{lemma}

\begin{proof}

By the proof of Lemma \ref{HKh} (iv), $H$ has a coset decomposition
\[
H = \cup _{j=0} ^{s-1} [ H \cap SL(2,R)] h_o^j
 \]
 with respect to its normal subgroup $H \cap SL(2,R)$. Therefore, the group structures of $H \cap SL(2,R)$ and $\langle h_o \rangle \simeq {\mathbb C}_r$, together with the multiplication rule for $h_1 h_o^{i} , h_2 h_o^j \in H$ with $h_1,h_2 \in H \cap SL(2,R)$ and $ 0 \leq i,j \leq s-1$ determine the group $H$ up to an isomorphism. Let us represent $h_1 = g_{i_1} ^{a_1} g_{i_2} ^{a_2} \ldots g _{i_k} ^{a_k}$ and $h_2 = g_{j_1} ^{b_1} g_{j_2} ^{b_2} \ldots g_{j_l} ^{b_l}$ as words in the alphabet $g_0 = h_o^s, g_1, \ldots , g_t$ with some integral exponents $a_p, b_q \in {\mathbb Z}$. (The group $H$ is finite, so that any $g_i$ is of finite order $r_i$ and  one can reduce the exponent of $g_i$ to a residue  modulo $r_i$.) In order to determine the product $( h_1 h_o^{i}) (h_2 h_o^j)$ as an element of $H = \cup _{j=0} ^{s-1} \langle g_0,g_1, \ldots , g_t \rangle h_o^j$, it suffices to specify $g'_i \in H \cap SL(2,R) = \langle g_0, g_1, \ldots , g_t \rangle$ with $h_o g_i = g'_i h_o$ for all $0 \leq i \leq t$. That allows to move gradually $h_o^{i}$ to the end of $(h_1h_o{i})( h_2 h_o^{j})$, producing $h_1 h'_2 h_o ^{i+j} \in [ H \cap SL(2,R)] h_o ^{(i+j) ({\rm mod} s) }$ for an appropriate $h'_2 \in H \cap SL(2,R)$. In other words, the group structures of $H \cap SL(2,R)$ and $\langle h_o \rangle \simeq {\mathbb C}_r$, together with the conjugates $g'_i = h_o g_i h_o^{-1}$ of $g_i$ determine the group multiplication in $H$. Note that $h_o g_{0} h_o ^{-1} = g_{0}$, since $g_{0} = h_o ^s$ commutes with $h_o$. The conjugates $g'_i = h_o g_i h_o ^{-1}$ with $1 \leq i \leq t$ belong to  the normal subgroup $H \cap SL(2,R) \ni g_i$ of $H$ and have the same orders $r_i$ as $g_i$.

\end{proof}

Any finite subgroup $H = [ H \cap SL(2,R)] \langle h_o \rangle$ of $GL(2,R)$ with determinant \\ $\det(H) = \langle \det(h_o) \rangle \simeq {\mathbb C}_s$ has a conjugate
\[
S^{-1}HS = \{ S^{-1} [ H \cap SL(2,R)] S \} \langle S^{-1} h_o S \rangle = [ S^{-1}HS \cap SL(2, {\mathbb C})] \langle S^{-1}h_o S \rangle
\]
with a diagonal matrix $S^{-1}h_oS$. Mote precisely, if $R$ is a subring of the integers ring $\mathcal{O}_{-d}$ of an imaginary quadratic number field ${\mathbb Q} ( \sqrt{-d})$ and $\lambda _1 = \lambda _1 (h_o)$, $\lambda _2 = \lambda _2 ( h_o)$ are the eigenvalues of $h_o$, then there exists a basis
\[
v_1 = \left( \begin{array}{c}
s_{11}  \\
s_{21}
\end{array}  \right), \ \
v_2 = \left( \begin{array}{c}
s_{12}  \\
s_{22}
\end{array}  \right) \ \ \mbox{ of } \ \ {\mathbb C}^2,
\]
consisting of eigenvectors $v_j$ of $h_o$, associated with the eigenvalues $\lambda _j = \lambda _j (h_o)$.  This is due to the finite order of $h_o$, because the Jordan block
\[
J = \left( \begin{array}{cc}
\lambda _1  &  1  \\
0  & \lambda _1
\end{array}  \right) \ \ \mbox{  with  } \ \ \lambda _1 \in {\mathbb C}^*
\]
is of infinite order in $GL(2,{\mathbb C})$. The matrix $S = (s_{ij}) _{i,j=1}^2$ with columns $v_1,v_2$ is non-singular and its entries belong to the extension ${\mathbb Q} ( \sqrt{-d}, \lambda (h_o)) = {\mathbb Q} ( \sqrt{-d}, \lambda _2(h_o))$ of ${\mathbb Q} ( \sqrt{-d})$ by some of the eigenvalues of $h_o$. Making use of the classification of $h_o \in GL(2,R)$ of finite order $r$ and $\det(h_o) \in R^*$ of order $s$, done in section 2, one determines explicitly the field $F_{-d} ^{(s,r)} = {\mathbb Q} ( \sqrt{-d}, \lambda _1 (h_o))$, obtained from ${\mathbb Q} ( \sqrt{-d})$ by adjoining an eigenvalue $\lambda _1 (h_o)$ of $h_o \in H$. The group
\[
S^{-1}HS = [ S^{-1}HS \cap SL(2, {\mathbb C})] \langle S^{-1} h_o S \rangle
 \]
 has a diagonal generator $D_o = S^{-1} h_o S$ and the conjugates
 \[
 (S^{-1}h_oS)(S^{-1}g_iS)(S^{-1}h_oS)^{-1} = S^{-1} ( h_o g_i h_o ^{-1} )S
 \]
  are easier to be computed.

   The next lemma collects the fields $F_{-d} ^{(s,r)}$.

\begin{lemma}    \label{DiagonalizationField}
Let $H = [ H \cap SL(2,R)] \langle h_o \rangle$ be a finite subgroup of $GL(2,R)$ with $h_o \in H$ of order $r$, $\det(h_o) \in R^*$ of order $s$ and $F_{-d} ^{(s,r)}$ be the number field
\[
F_{-d} ^{(s,r)} = \begin{cases}
{\mathbb Q} ( \sqrt{-d}) &  \text{ for $s=r=2$,}  \\
{\mathbb Q} (i)  & \text{ for $s \in \{ 2, 4 \}$, $r=4$,}  \\
{\mathbb Q} ( \sqrt{-3})  & \text{ for $(s,r) = (2,6)$  or $s \in \{ 3, 6 \}$,}     \\
{\mathbb Q} ( \sqrt{2},i) & \text{ for $s \in \{ 2, 4 \}$, $r=8$,}   \\
{\mathbb Q} ( \sqrt{3},i)  & \text{ for $s=2$, $r=12$.}
\end{cases}.
\]
Then there exists a matrix $S \in GL(2, F_{-d} ^{(s,r)} )$ such that
\[
D_o = S^{-1} h_o S = \left( \begin{array}{cc}
\lambda _1 (h_o)  &  0  \\
0  &  \lambda _2 (h_o)
\end{array}  \right)
\]
is diagonal and
\[
H^{o} = S^{-1}H S = [ S^{-1} H S \cap SL(2, F_{-d} ^{(s,r)} )] \langle D_o \rangle
\]
is a subgroup of $GL(2, F_{-d}^{(s,r)} )$, isomorphic to $H$.
\end{lemma}

Summarizing the results of section 2, one obtains also the following

\begin{corollary}   \label{RatioonGLEigenvalues}
If $h_o \in GL(2,R) \setminus SL(2,R)$ is of order $r$ with $\det(h_o) \in R^*$ of order $s$ and eigenvalues $\lambda _1 (h_o), \lambda _2 (h_o)$, then
\[
\frac{\lambda _1 (h_o)}{\lambda _2(h_o)} \in \left \{ \pm 1, \ \  \pm i, \ \   e^{ \pm \frac{ 2 \pi i}{3}}, \ \  e^{ \pm \frac{ \pi i}{3}} \right \}.
\]
More precisely,

\[
\mbox{ (i)  } \ \  \frac{\lambda _1 (h_o)}{\lambda _2 (h_o)} =1 \ \ \mbox{   exactly when  } \ \
h_o \in \left \{ \pm i I_2, \ \  e^{ \pm \frac{2 \pi i}{3}} I_2, \ \  e^{ \pm \frac{ \pi i}{3}} I_2 \right \}
\]
is a scalar matrix;

\[
\mbox{ (ii)  } \ \ \frac{ \lambda _1 (h_o)}{\lambda _2 (h_o)} = -1 \ \ \mbox{  for}
\]
\[
\mbox{ (a)  } \ \  \lambda _1 (h_o) = 1, \ \ \lambda _2 (h_o) = -1 \ \ \mbox{  and an arbitrary } \ \  R = R_{-d,f};
\]
\[
\mbox{ (b) } \ \  \lambda _1 (h_o) = e^{ \pm \frac{ 3 \pi i}{4}}, \ \ \lambda _2 (h_o) = e^{ \mp \frac{ \pi i}{4}}, \ \ R = {\mathbb Z}[i], \ \ s=4;
\]
\[
\mbox{  (c)  }  \ \ \lambda _1 (h_o) = e^{ \pm \frac{ 5 \pi i}{6}}, \ \  \lambda _2 (h_o) = e^{ \mp \frac{ \pi i}{6}}, \ \ R = \mathcal{O}_{-3}, \ \
s=3 \
\]
\[
\mbox{  (d)  } \ \  \lambda _1 (h_o) = e^{ \pm \frac{ 2 \pi i}{3}}, \ \ \lambda _2 (h_o) = e^{ \mp \frac{ \pi i}{3}}, R = \mathcal{O}_{-3}, \ \ s=6.
\]

\[
\mbox{  (iii) }   \ \ \frac{\lambda _1 (h_o)}{\lambda _2(h_o)} = \pm i \ \  \mbox{ for  }
\]
\[
\mbox{ (a)  } \ \  \lambda _1 (h_o) = e^{ \pm \frac{ 3 \pi i}{4}}, \ \ \lambda _2 (h_o) = e^{ \pm \frac{ \pi i}{4}}, \ \ R = \mathcal{O}_{-2}, \ \  s=2;
\]
\[
\mbox{  (b)  }  \ \ \{ \lambda _1 ( h_o ), \lambda _2 (h_o) \} = \{ \pm i, \pm 1 \} \ \ \mbox{  or } \{ \pm i, \mp 1 \}  \ \ \mbox{ with   } \ \
 R = {\mathbb Z}[i], \ \  s=4.
\]

\[
\mbox{  (iv)    } \ \ \frac{ \lambda _1 (h_o)} {\lambda _2 (h_o)} = e^{ \pm \frac{ 2 \pi i}{3}} \ \ \mbox{  for  }
\]
\[
\mbox{  (a)  } \ \ \lambda _1 (h_o) = e^{ \pm \frac{ 5 \pi i}{6}}, \ \ \lambda _2 (h_o) = e^{ \pm \frac{ \pi i}{6}}, \ \  R = {\mathbb Z}[i], \ \
s=2;
\]
\[
\mbox{  (b)  } \ \ \lambda _1 (h_o) = e^{ \pm \frac{2 \pi i}{3}}, \ \  \lambda _2 (h_o) = 1, \ \ R = \mathcal{O}_{-3}, \ \  s=3;
\]
\[
\mbox{  (c)  } \ \ \lambda _1 (h_o) = e^{ \pm \frac{\pi i}{3}}, \ \ \lambda _2 (h_o) =-1, \ \ R = \mathcal{O}_{-3}, \ \ s=3.
\]

\[
\mbox{  (v)   } \ \ \frac{ \lambda _1 (h_o)}{\lambda _2 (h_o)} = e^{ \pm \frac{ \pi i}{3}} \ \ \mbox { for  }
\]
\[
\mbox{  (a)  } \ \ \lambda _1(h_o) = e^{ \pm \frac{2 \pi i}{3}}, \lambda _2 (h_o) = e^{ \pm \frac{ \pi i}{3}}, R=\mathcal{O}_{-3}, \ \ s=2;
\]
\[
\mbox{  (b) } \ \ \lambda _1 (h_o) = \varepsilon e^{ \eta \frac{ \pi i}{3}}, \ \ \lambda _2 (h_o) = \varepsilon, \ \ R = \mathcal{O}_{-3}, \ \
s =6, \ \ \varepsilon, \eta \in \{ \pm 1 \}.
\]
\end{corollary}

\begin{proposition}   \label{HC1}
Let $H$ be a finite subgroup of $GL(2,R)$,
\[
H \cap SL(2,R) = \{ I_2 \}
\]
and $h_o \in H$ be an element of order $r$ with $\det(H) = \langle \det( h_o) \rangle \simeq {\mathbb C}_s$ and eigenvalues $\lambda _1 (h_o)$,
 $\lambda _2(h_o)$. Then $r=s$ and $H$ is isomorphic to $H_{C1} (j) \simeq {\mathbb C}_{s_j}$ for some $1 \leq j \leq 4$, where
\[
H_{C1} (1) = \langle h_o \rangle \simeq {\mathbb C}_2 \ \ \mbox{   with  } \ \ \lambda _1(h_o) =1, \ \ \lambda _2 (h_o) =-1,
\]
\[
H_{C1} (2) = \langle h_o \rangle \simeq {\mathbb C}_3 \ \ \mbox{  with } \ \  R = \mathcal{O}_{-3}, \ \ h_0 = e^{ - \frac{2 \pi i}{3}} I_2 \ \ \mbox{  or  } \ \ \lambda _1 (h_o) = e^{\frac{ 2 \pi i}{3}}, \ \ \lambda _2 (h_o) =1,
\]
\[
H_{C1} (3) = \langle h_o \rangle \simeq {\mathbb C}_4 \ \ \mbox{ with } \ \ R = {\mathbb Z}[i], \ \  \{ \lambda _1 (h_o), \lambda _2 (h_o) \} = \{ i, 1 \} \ \ \mbox{  or  } \ \ \{ -i, -1 \},
\]
\[
H_{C1} (4) = \langle h_o \rangle \simeq {\mathbb C}_6 \ \ \mbox{   with } \ \ R = \mathcal{O}_{-3},
\]
\[
\{ \lambda _1 (h_o), \lambda _2 (h_o) \} =  \left \{ e^{\frac{ \pi i}{3}}, 1 \right \}, \ \ \left \{ e^{ - \frac{ 2 \pi i}{3}}, -1 \right \} \ \ \mbox{  or  }
\ \ \left \{ e^{\frac{ 2 \pi i}{3}}, e^{ - \frac{ \pi i}{3}} \right \}.
\]
\end{proposition}

\begin{proof}

By Lemma \ref{HKh} (ii), the group $H = \langle h_o \rangle \simeq {\mathbb C}_r$ is cyclic and generated by any $h_o \in H$, whose determinant $\det(h_o)$ generates $\det(H) = \langle \det(h_o) \rangle$. Moreover, Lemma \ref{HKh} (iii) specifies that $ \{ I_2 \} = [ H \cap SL(2,R)] \cap \langle h_o \rangle = \langle h_o^s \rangle$ or the order $r$ of $h_o$ coincides with the order $s$ of $\det(h_o)$.  For $s \in \{ 3,4,6 \}$ one can assume that $\det(h_o) = e^{ \frac{ 2 \pi i}{3}}$, since the  generators of $\det(H) = \langle \det(h_o) \rangle \simeq {\mathbb C}_s$ are $e^{ \frac{ 2 \pi i}{s}}$ and $e^{ - \frac{ 2 \pi i}{s}}$. Making use of the classification of the elements $h_o \in GL(2,R)$ of order $s$ with $\det(h_o ) = e^{ \frac{ 2 \pi i}{s}}$, done in section 2, one concludes that $H \simeq H_{C1} (j)$ for some $1 \leq j \leq 4$.

\end{proof}

\begin{proposition}  \label{HC2}
Let $H$ be a finite subgroup of $GL(2,R)$,
\[
H \cap SL(2,R) = \langle - I_2 \rangle \simeq {\mathbb C}_2
\]
and $h_o \in H$ be an element of order $r$ with $\det(H) = \langle \det(h_o) \rangle \simeq {\mathbb C}_s$ and eigenvalues  $\lambda _1(h_o)$, $\lambda _2 (h_o)$. Then $H$ is isomorphic to $H_{C2}(i)$ for some $1 \leq i \leq 6$, where
\[
H_{C2}(1) = \langle i I_2 \rangle \simeq {\mathbb C}_4 \ \ \mbox{  with } \ \ R = {\mathbb Z}[i],
\]
\[
H_{C_2}(2) = \langle - I_2 \rangle \times \langle h_o \rangle \simeq {\mathbb C}_2 \times {\mathbb C}_2 \ \ \mbox{ with  }
\ \   \lambda _1 (h_o) = 1, \ \ \lambda _2 (h_o) = -1,
\]
\[
H_{C2}(3) = \langle h_o \rangle \simeq {\mathbb C}_6 \ \ \mbox{ with } \ \ R = \mathcal{O}_{-3}, \ \  h_o = e^{ \frac{ \pi i}{3}} I_2 \ \ \mbox{ or } \ \
\lambda _1 (h_o) = e^{ - \frac{ \pi i}{3}}, \ \ \lambda _2 (h_o) =-1,
\]
\[
H_{C2}(4) = \langle h_o \rangle \simeq {\mathbb C}_8 \ \ \mbox{  with } \ \  R = {\mathbb Z}[i], \ \ \lambda _1(h_o) = e^{ \frac{ 3 \pi i}{4}}, \ \
\lambda _2 (h_o) = e^{ - \frac{ \pi i}{4}},
\]
\[
H_{C2}(5) = \langle -I_2 \rangle \times \langle h_o \rangle \simeq {\mathbb C}_2 \times {\mathbb C}_4 \ \ \mbox{  with } \ \
R = {\mathbb Z}[i], \ \  \lambda _1 (h_o) =i, \ \ \lambda _2 (h_o) =1,
\]
\[
H_{C2}(6) = \langle h_o \rangle \simeq {\mathbb C}_8 \ \ \mbox{  with  } \ \ R = {\mathbb Z}[i], \ \ \lambda _1 (h_o) = e^{\frac{ 3 \pi i}{4}}, \ \
\lambda _2 (h_o) = e^{ - \frac{ \pi i}{4}},
\]
\[
H_{C2} (7) = \langle - I_2 \rangle \times \langle h_o \rangle \simeq {\mathbb C}_2 \times {\mathbb C}_6 \ \ \mbox{ with } \ \
R = \mathcal{O}_{-3},
\]
\begin{equation}    \label{EigenvaluesListHC26}
\left \{ \lambda _1 (h_o),  \lambda _2 (h_o) \right \} = \left \{ e^{ \frac{ 2 \pi i}{3}}, e^{ - \frac{ \pi i}{3}} \right \}, \ \
\left \{ e^{ \frac{ \pi i}{3}}, 1 \right \} \ \  \mbox{ or } \ \ \left \{ e^{ - \frac{ 2 \pi i}{3}}, -1 \right \}.
\end{equation}

\end{proposition}

\begin{proof}

By Lemma \ref{HKh} (iii), one has $h_o^s \in H \cap SL(2,R) = \langle - I_2 \rangle$ for some $s \in \{ 2, 3, 4, 6 \}$.

If $h_o^s =I_2$ then $s=r$ and
\[
H = \langle -I_2 \rangle \times \langle h_o \rangle \simeq {\mathbb C}_2 \times {\mathbb C}_s
  \]
 is a direct product, as far as the scalar matrix $ - I_2$ commutes with $h_o$. When  $h_o$ is of odd order $s=3$, its opposite matrix $- h_o \in H$ is of order $6$ and $H = \langle -h_o \rangle \simeq {\mathbb C}_6$. Without loss of generality, $h_1 := - h_o$ has $\det(h_1) = e^{\frac{ 2 \pi i}{3}}$ and Proposition \ref{ListS3+} specifies that either $h_1 = e^{\frac{ \pi i}{3}} I_2$ or $\lambda _1 (h_1) = e^{ - \frac{\pi i}{3}}$, $\lambda _2(h_o) = -1$.
 For $s=2$ the group $H = \langle - I_2 \rangle \times \langle h_o \rangle = H_{C2} (2) \simeq {\mathbb C}_2 \times {\mathbb C}_2$, where $h_o \in H$ has eigenvalues $\lambda _1 (h_o) =1$, $\lambda _2 (h_o) = -1$. The case $s=4$ occurs only for $R = {\mathbb Z}[i]$. Assuming $\det(h_o) =i$, one gets $\lambda _1(h_o) = \varepsilon i$, $\lambda _2 (h_o) = \varepsilon$ for some $\varepsilon \in \{ \pm 1 \}$ by Proposition \ref{ListS4+}. Since $- I_2 \in H$, one can replace $h_o$ by $-h_o$ and reduce to the case of $\varepsilon =1$. If $s=6$, then Proposition \ref{ListS6+} provides (\ref{EigenvaluesListHC26}).

In the case of $h_o^s = - I_2$, the intersection $\langle h_o \rangle SL(2,R) = \langle - I_2 \rangle = H \cap SL(2,R)$ and the group
\[
H = \langle h_o \rangle \simeq {\mathbb C}_{2s}
\]
 is cyclic.  More precisely, for $s=2$ Proposition \ref{ListS2} implies that $h_o = \pm i I_2$ and $H \simeq H_{C2} (1)$. If $s=3$ and $\det(h_o) = e^{\frac{ 2 \pi i}{3}}$ then $H \simeq H_{C2} (3)$ by Proposition \ref{ListS3+}. For $s=4$ and $\det(h_o) =i$ one has $H \simeq H_{C2}(6)$, according to Proposition \ref{ListS4+}. Making use of Proposition \ref{ListS6+}, one observes that there are no $h_o \in GL(2,R)$ of order $12$ with $\det(h_o) = e^{\frac{ \pi i}{3}}$ and concludes the proof of the proposition.

\end{proof}

Towards the description of the finite subgroups $H = [ H \cap SL(2,R)] \langle h_o \rangle$ of $GL(2,R)$ with $H \cap SL(2,R) \simeq {\mathbb C}_t$ for some $t \in \{ 3,4, 6 \}$, one needs the following

\begin{lemma}   \label{CommuteSquares}
If $g \in GL(2, {\mathbb C})$ has different eigenvalues $\lambda _1 \neq \lambda _2$ then any $h \in GL(2, {\mathbb C})$ with $hg \neq gh$ and $h^2g = gh^2$ has vanishing  trace $\tr(h) =0$.
\end{lemma}

\begin{proof}

The trace is invariant under conjugation, so that
\[
g = \left( \begin{array}{cc}
\lambda _1  & 0  \\
0  &  \lambda _2
\end{array}  \right)
\]
can be assumed to be diagonal. If
\[
h = \left( \begin{array}{cc}
a  &  b  \\
c  &  d
\end{array}  \right) \in GL(2, {\mathbb C}),
\]
then $h^2g=gh^2$ is equivalent to
\[
\left| \begin{array}{c}
(\lambda _1 - \lambda _2) b (a+d) = 0 \\
(\lambda _1 - \lambda _2) c (a+d) = 0
\end{array}  \right. .
\]
Due to $\lambda _1 \neq \lambda_2$, there follow $b(a+d)=0$ and $c(a+d)=0$. The assumption $\tr(h) = a + d \neq 0$ leads to $b=c=0$. As a result,
\[
h = \left(  \begin{array}{cc}
a  &  0  \\
0  &  d
\end{array} \right)
\]
is a diagonal matrix and  commutes with $g$. The contradiction justifies that $\tr(h) =0$.

\end{proof}

\begin{lemma}     \label{HK346Necessary}
Let $H = [ H \cap SL(2,R)] \langle h_o \rangle$ be a finite subgroup of $GL(2,R)$ with
\[
H \cap SL(2,R) = \langle g \rangle \simeq {\mathbb C}_t \ \ \mbox{ for some  } \ \   t \in \{ 3, 4, 6 \} \ \ \mbox{ and }
\]
\[
\det(H) = \langle \det(h_o) \rangle = \langle e^{ \frac{ 2 \pi i}{s}} \rangle \simeq {\mathbb C}_s, \ \ s>1
\]
for some $h_o \in H$ of order $r$. Then:

\[
\mbox{ (i)}  \ \ \frac{r}{s} = \begin{cases}
1, 2, 3, 4 \mbox{ or  }   6   & \text{ for $s=2$,}  \\
1, 2 \mbox{ or  }  4  &  \text{ for $s=3$,}  \\
1 \mbox{ or } 2  &  \text{ for $s=4$}   \\
1  & \text{ for $s=6$}
\end{cases}
\]
divides $t$;

(ii) $\frac{r}{s} =t$ if and only if $H = \langle h_o \rangle \simeq {\mathbb C}_r$ is cyclic and $H \cap SL(2,R) = \langle h_o^s \rangle$;

(iii) if $\frac{r}{s} < t$ then $H$ is isomorphic to the non-cyclic abelian group
\[
H'= \langle g, h_o \ \ \vert \ \  g^t = h_o^r = I_2, \ \ h_og = gh_o \rangle
\]
or to the non-abelian group
\[
H'' = \langle g, h_o \ \ \vert \ \  g^t = h_o^r = I_2, \ \ h_og h_o^{-1} = g^{-1} \rangle;
\]

(iv) if $\frac{r}{s} < t$ and $H \simeq H''$ is non-abelian then $h_o$ has eigenvalues $\lambda _1 (h_o) = i e^{ \frac{ \pi i}{s}}$, $\lambda _2 ( h_o) = -i e^{\frac{ \pi i}{s}}$ and
\[
(r,s) \in \{ (2,2), \ \ (6,6) \} \ \ \mbox{  for  } \ \ t=3,
\]
\[
(r,s) \in \{ (2,2), \ \  (8,4),  \ \ (6,6) \} \ \ \mbox{  for  } \ \ t=4,
\]
\[
(r,s) \in \{ (2,2), \ \  (8,4), \ \ (6,6) \} \ \ \mbox{  for  } \ \ t=6.
\]
\end{lemma}

\begin{proof}

(i) Note that if $\det(h_o) \in R^*$ is of order $s$ then $\det( h_o^s) = \det(h_o)^s =1$ and $h_o^s \in H \cap SL(2,R) = \langle g \rangle$ is an element of order $\frac{r}{s}$. Since $\langle g \rangle \simeq {\mathbb C}_t$ is of order $t$, the ratio $\frac{r}{s} \in {\mathbb N}$ divides $t$. Proposition \ref{ListS2} provides the list of $\frac{r}{s} = \frac{r}{2}$ for $s=2$. If $s=3$ then the values of $\frac{r}{s} = \frac{r}{3}$ are taken from Propositions \ref{ListS3+} and \ref{ListS3-}. Propositions \ref{ListS4+} and \ref{ListS4-} supply the range of $\frac{r}{s} = \frac{r}{4}$ for $s=4$, while Propositions \ref{ListS6+} and \ref{ListS6-} give account for $\frac{r}{s} = \frac{r}{6}$ in the case of $s=6$.

(ii)  Note that $h_o^s \in \langle g \rangle$ is of order $\frac{r}{s} = t$ exactly when $\langle g \rangle = \langle h_o ^s \rangle$ and $H = \langle h_o \rangle \simeq {\mathbb C}_r$ is a cyclic group.

(iii) According to Lemma \ref{StructureH}, the group $H = [ H \cap SL(2,R)] \langle h_o \rangle =
\langle g \rangle \langle h_o \rangle$ is completely determined by the order $t$ of $g$, the order $r$ of $h_o$ and the conjugate $x = h_o g h_o^{-1} \in H \cap SL(2,R) = \langle g \rangle$ of $g$ by $h_o$. The order $t$ of $g$ is invariant under conjugation, so that $x=g^m$ for some $m \in {\mathbb Z}_t ^*$. The Euler function $\varphi (t) =2$ for $t \in \{ 3, 4, 6 \}$ and ${\mathbb Z}_t^* = \{ \pm 1({\rm mod} t) \}$. Therefore $x = h_o g h_o^{-1} = g$ or $x = h_o g h_o^{-1} = g^{-1}$.

(iv)  If $H \simeq H''$ is a non-abelian group then
\[
h_o^2 g h_o^{-2} = h_o ( h_o g h_o^{-1}) h_o^{-1} = h_o g^{-1} h_o^{-1} = ( h_o g h_o^{-1} ) ^{-1} = ( g^{-1}) ^{-1} =g,
\]
so that $g$ commutes with $h_o^2$, but does not commute with $h_o$. By Lemma \ref{CommuteSquares} there follows $\tr(h_o) =0$. There exists a matrix  $S \in GL \left( 2, {\mathbb Q} \left( \sqrt{-d}, e^{\frac{ 2 \pi i}{t}} \right) \right)$, such that
\[
D = S^{-1} g S = \left( \begin{array}{cc}
e^{\frac{ 2 \pi i}{t}}  &  0  \\
0  &  e^{ - \frac{2 \pi i}{t}}
\end{array}  \right) \in SL \left( 2, {\mathbb Q} \left(  \sqrt{-d}, e^{ \frac{ 2 \pi i}{t}} \right)  \right)
\]
is diagonal.   Since the trace is invariant under conjugation,
\[
D_o := S^{-1} h_o S = \left( \begin{array}{cc}
a  &  b  \\
c  &  -a
\end{array}  \right) \in GL \left( 2, {\mathbb Q} \left( \sqrt{-d}, e^{\frac{2 \pi i}{t}} \right) \right).
\]
The relation $h_og = g^{-1} h_o$ implies  the vanishing  of $a$. As a result, the characteristic polynomial
\[
\mathcal{X}_{h_o} (\lambda ) = \lambda ^2 + \det(h_o) = \lambda ^2 + e^{ \frac{ 2 \pi i}{s}} =0
\]
has roots $\lambda _1 (h_o) = i e^{ \frac{ \pi i}{s}}$, $\lambda _2 (h_o) = -i e^{ \frac{ \pi i}{s}}$. More precisely, for $s=2$ one has $\lambda _1 (h_o) = -1$, $\lambda _2 (h_o) =1$, so that $h_o$ and $D_o$ are of order $r=2$. The ratio $\frac{r}{s}=1$ divides any $t \in \{ 3, 4, 6 \}$. If $s=3$ then
$\lambda _1 (h_o) = e^{ \frac{ 5 \pi i}{6}}$, $\lambda _2 (h_o) = e^{ - \frac{\pi i}{6}}$, so that $h_o$ and $D_o$ are of order $r=12$. The quotient $\frac{r}{s}=4$ divides only $t=4$. Therefore $\frac{r}{s} =t$ and $H = \langle h_o \rangle \simeq {\mathbb C}_{12}$, according to (ii). In the case of $s=4$, one has $\lambda _1 (h_o) = e^{\frac{ 3 \pi i}{4}}$, $\lambda _2 (h_o) = e^{ - \frac{ \pi i}{4}}$, whereas $h_o$ and $D_o$ are of order $r=8$. The quotient $\frac{r}{s} = 2$ divides only $t \in \{ 4, 6 \}$. Finally, for $s=6$ the automorphism $h_o$ has eigenvalues $\lambda _1 (h_o) = e^{ \frac{ 2 \pi i}{3}}$,
 $\lambda _2 ( h_o) = e^{ - \frac{\pi i}{3}}$. Consequently, $h_o$ and $D_o$ are of order $r=6$ and $\frac{r}{s} =1$ divides all $t \in \{ 3, 4, 6 \}$.

\end{proof}

\begin{lemma}  \label{HK346Attained}
(i) For arbitrary $d \in {\mathbb N}$ and $t \in \{ 3, 4, 6 \}$ there is a dihedral subgroup
\[
\mathcal{D}_t = \langle g, h_o \ \ \vert \ \  g^t = h_o ^2 = I_2, \ \ h_o g h_o ^{-1} = g^{-1} \rangle < GL(2, {\mathbb Q} (\sqrt{-d}))
\]
of order $2t$ with $\mathcal{D}_t \cap SL(2, {\mathbb Q}(\sqrt{-d})) = \langle g \rangle \simeq {\mathbb C}_t$,
$\det( \mathcal{D}_t) = \langle \det( h_o) \rangle = \langle -1 \rangle \simeq {\mathbb C}_2$
and eigenvalues $\lambda _1 (h_o) = -1$, $\lambda _2 (h_o) =1$ of $h_o$.

(ii) For an arbitrary $t \in \{ 3, 4, 6 \}$ there is a subgroup
\[
\mathcal{H}_t = \langle g, h_o \ \ \vert \ \  g^t = h_o ^6 = I_2, \ \ h_o g h_o ^{-1} = g^{-1} \rangle < GL(2, {\mathbb Q} ( \sqrt{-3}))
\]
of order $6t$ with $\mathcal{H}_t \cap SL(2, {\mathbb Q}( \sqrt{-3})) = \langle g \rangle \simeq {\mathbb C}_t$,
$\det( \mathcal{H}_t) = \langle \det( h_o) \rangle = \langle e^{ \frac{ \pi i}{3}} \rangle \simeq {\mathbb C}_6$
and eigenvalues $\lambda _1 (h_o) = e^{\frac{ 2 \pi i}{3}}$, $\lambda _2 (h_o) = e^{ - \frac{ \pi i}{3}}$ of $h_o$.

(iii)  For an arbitrary $t \in \{ 4 , 6 \}$ there is a subgroup
\[
\mathcal{H}'_t = \langle g, h_o \ \ \vert \ \  g^{\frac{t}{2}} = h_o^4 = -I_2, \ \ h_o g h_o ^{-1} = g^{-1} \rangle < GL(2, {\mathbb Q} ( \sqrt{2}, i))
\]
of order $4t$ with
$\mathcal{H}'_t \cap SL(2, {\mathbb Q}( \sqrt{2},i)) = \langle g \rangle \simeq {\mathbb C}_t$,
$\det( \mathcal{H}'_t) = \langle \det(h_o) \rangle = \langle i \rangle \simeq {\mathbb C}_4$
and eigenvalues $\lambda _1(h_o) = e^{\frac{ 3 \pi i}{4}}$, $\lambda _2 (h_o) = e^{ - \frac{ \pi i}{4}}$ of $h_o$.
\end{lemma}

\begin{proof}

(i) Let us choose a diagonalizing matrix $S \in GL(2, {\mathbb Q}( \sqrt{-d}))$ of $h_o$, so that
\[
D_o = S^{-1} h_o S = \left( \begin{array}{rr}
-1  &  0  \\
0  &  1
\end{array}  \right).
\]
Taking into account Proposition \ref{ListK}, one has  to show the existence of
\[
D = S^{-1`}gS = \left( \begin{array}{cc}
a  &  b  \\
c  &  d
\end{array}  \right) \in SL(2, {\mathbb Q}( \sqrt{-d})
\]
with
\[
D_oDD_o^{-1} = \left( \begin{array}{rr}
a  & -b  \\
-c  &  d
\end{array}  \right) =
\left( \begin{array}{rr}
d  &  -b  \\
-c  &  a
\end{array} \right) = D^{-1}
\]
for any trace $\tr(g) = \tr(D) = a + d \in \{ 0, \pm 1 \}$. More precisely, for $a=d=0$, $b\neq 0$ and $c= - b^{-1}$, then the matrix
\[
D = D_4 = \left( \begin{array}{rr}
0  &  b  \\
\mbox{  }  & \mbox{  }  \\
- b^{-1}  &  0
\end{array}  \right)
\]
of order $4$ and the matrix $D_o$ of order $2$ generate a dihedral group $\mathcal{D}_4$ of order $8$. If $a = d = - \frac{1}{2}$, $b \neq 0$ and $c = - \frac{3}{4} b^{-1}$ then
\[
D = D_3 = \left( \begin{array}{rr}
- \frac{1}{2}  &  b  \\
\mbox{  }  & \mbox{  }  \\
- \frac{3}{4} b^{-1}  &  - \frac{1}{2}
\end{array}  \right)
\]
of order $3$ and $D_o$ of order $2$ generate a symmetric group $\mathcal{D}_3 \simeq S(3)$ of degree $3$. In the case of $a = d = \frac{1}{2}$, $b \neq 0$ and $c = - \frac{3}{4} b^{-1}$, the matrix
\[
D = D_6 = \left( \begin{array}{rr}
\frac{1}{2}  &  b  \\
\mbox{  }  & \mbox{  }  \\
- \frac{3}{4} b^{-1}  &  \frac{1}{2}
\end{array}  \right)
\]
of order $6$ and the matrix $D_o$ of order $2$ generate a dihedral group $\mathcal{D}_6$ of order $12$.

(ii) By Proposition \ref{ListS6+}, if $h_o \in GL(2,R)$ has eigenvalues $\lambda _1 (h_o) = e^{\frac{ 2 \pi i}{3}}$, $\lambda _2 (h_o) = e^{ - \frac{ \pi i}{3}}$ then $R = \mathcal{O}_{-3}$. Let us consider
\[
D_o = S^{-1}h_o S = \left( \begin{array}{rr}
e^{\frac{ 2 \pi i}{3}}  &  0  \\
 0  &  e^{ - \frac{\pi i}{3}}
 \end{array}  \right) \in GL(2, {\mathbb Q} ( \sqrt{-3}))
 \]
 for some $S \in GL(2, {\mathbb Q}( \sqrt{-3}))$ and
 \[
 D = S^{-1} g S = \left( \begin{array}{cc}
 a  &  b  \\
 c  &  d
 \end{array}   \right) \in SL(2, {\mathbb Q} ( \sqrt{-3}))
 \]
 with trace $\tr(g) = \tr(D) = a + d \in \{ 0, \pm 1 \}$. Then
 \[
 D_o D D_o^{-1} =
 \left( \begin{array}{rr}
 a  &  -b  \\
 -c  &  d
 \end{array}  \right) =
 \left( \begin{array}{rr}
 d  &  -b  \\
 -c  &  a
 \end{array} \right) = D^{-1}
 \]
 is equivalent to $a=d$. Consequently, $D_3, D_4, D_6$ from the proof of (i) satisfy the required conditions.

 (iii) Note that
 \[
 D_o = S^{-1} h_o S = \left( \begin{array}{cc}
 e^{\frac{ 3 \pi i}{4}}  &  0  \\
 0  &  e^{ - \frac{ \pi i}{4}}
 \end{array}  \right) \in GL(2, {\mathbb Q}( \sqrt{2}, i))
 \]
 for some $S \in GL(2, {\mathbb Q} ( \sqrt{2}, i))$ and
 \[
 D = S^{-1} g S = \left( \begin{array}{cc}
 a  &  b  \\
 c  &  d
 \end{array}  \right) \in SL(2, {\mathbb Q}( \sqrt{2},i))
 \]
 with trace $\tr(g) = \tr(D) = a + d \in \{ 0, 1 \}$ satisfy
 \[
 D_o D D_o^{-1} = \left( \begin{array}{rr}
 a  &  -b  \\
 c  &  d
 \end{array}  \right) =
 \left( \begin{array}{rr}
 d  &  -b  \\
 -c  &  a
 \end{array}  \right) = D^{-1}
 \]
 exactly when $a=d$. In the notations from the proof of (i), one has $\langle D_4, D_o \rangle \simeq \mathcal{H}'_4$ and $\langle D_6, D_o \rangle \simeq \mathcal{H}'_6$.

\end{proof}

\begin{corollary}   \label{HC3}
Let $H$ be a finite subgroup of $GL(2,R)$,
\[
H \cap SL(2,R) = \langle g \rangle \simeq {\mathbb C}_3
\]
and $h_o \in H$ be an element of order $r$ with $\det (H) = \langle \det(h_o) \rangle \simeq {\mathbb C}_s$ and eigenvalues $\lambda _1(h_o)$, $\lambda _2 (h_o)$. Then $H$ is isomorphic to some $H_{C3}(i)$, $1 \leq i \leq 5$, where
\[
H_{C3}(1) = \langle h_o \rangle \simeq {\mathbb C}_6
\]
with $R = R_{-3,f}$, $\lambda _1 (h_o) = e^{\frac{ \pi i}{3}}$, $\lambda _2 (h_o) = e^{\frac{ 2 \pi i}{3}}$,
\[
H_{C3}(2) = \langle g, h_o \ \ \vert \ \  g^3 = h_o^2 = I_2, \ \ h_ogh_o^{-1} = g^{-1} \rangle \simeq S_3
\]
is the symmetric group of degree $3$, $\lambda _1 (h_o) = -1$, $\lambda _2 (h_o) = 1$,
\[
H_{C3}(3) = \langle g \rangle \times \langle e^{\frac{ 2 \pi i}{3}} I_2 \rangle \simeq {\mathbb C}_3 \times {\mathbb C}_3
\]
with $R = \mathcal{O}_{-3}$ and any $g \in SL(2, \mathcal{O}_{-3})$ of trace $\tr(g) = -1$,
\[
H_{C3}(4) = \langle g \rangle \times \langle h_o \rangle \simeq {\mathbb C}_3 \times {\mathbb C}_6
\]
with $R = \mathcal{O}_{-3}$, $\lambda _1 (h_o) = e^{\frac{ \pi i}{3}}$, $\lambda _2 (h_o) = e^{ - \frac{ 2 \pi i}{3}}$,
\[
H_{C3} (5) = \langle g, h_o \ \ \vert \ \  g^3 = h_o ^6 = I_2, \ \  h_ogh_o^{-1} = g^{-1} \rangle
\]
of order $18$ with $R = \mathcal{O}_{-3}$, $\lambda _1 (h_o) = E^{\frac{2 \pi i}{3}}$, $\lambda _2 (h_o) = e^{ - \frac{ \pi i}{3}}$.

There exist subgroups
\[
H_{C3}(1), H_{C3}(3), H_{C3}(4) < GL(2, \mathcal{O}_{-3}),
\]
 as well as subgroups
\[
H^{o}_{C3}(2) < GL(2, {\mathbb Q} ( \sqrt{-d})), \ \
H^{o}_{C3}(5) < GL(2, {\mathbb Q}( \sqrt{-3}))
\]
with $H^{o}_{C3}(j) \simeq H_{C3}(j)$ for $j \in \{ 2, 5 \}$.
\end{corollary}

\begin{proof}

By Lemma  \ref{HK346Necessary} (i), the quotient $\frac{r}{s}$ is a divisor of $t=3$, so that either $r=s$ or $r=3s=6$.

For $s=2$, $r=6$ one has a cyclic group $H = \langle h_o \rangle \simeq {\mathbb C}_6$ with $\det(h_o) = -1$. Up to an inversion $h_o \mapsto h_o^{-1}$ of the generator, Proposition \ref{ListS2} specifies that $\lambda _1 (h_o) = e^{\frac{ \pi i}{3}}$, $\lambda _2 (h_o) = e^{\frac{ 2 \pi i}{3}}$ and justifies the realization of $H_{C3}(1) = \langle h_o \rangle$ over $\mathcal{O}_{-3}$.

Form now on, let $r=s \in \{ 2, ,3 ,4 6 \}$. According to Lemma \ref{HK346Necessary}(iii) and (iv), the group $H = \langle g, h_o \rangle$ is either abelian or isomorphic to some $H_{C3}(j)$  for $j \in \{ 2, 5 \}$.

If $H = \langle g, h_o \ \ \vert \ \  g^3 = h_o^r = I_2, \ \ gh_o = h_og \rangle$ is an abelian group of order $3r$, then $H = \langle g \rangle \times \langle h_o \rangle \simeq {\mathbb C}_3 \times {\mathbb C} _r$ is a direct product by Lemma \ref{HKh} (iv). (Here we use that the semi-direct product $H = [ H \cap SL(2,R)] \rtimes \langle h_o \rangle = \langle g \rangle \rtimes \langle h_o \rangle$ is a direct product if and only if $gh_o = h_og$.)

The order $r=s=2$ of $h_o$ is relatively prime to the order $3$ of $g$, so that $gh_o$ is an element of order $6$ and $\langle g, h_o \rangle = \langle gh_o \rangle \simeq {\mathbb C}_6 \simeq H_{C3}(1)$.

The order $r=s=4$ of $h_o$ is relatively prime to the order $3$ of $g$ and $gh_o$ is of order $12$. By the classification of $x \in GL(2,R)$ of finite order, done in section 2, one has $\det(gh_o) = -1$. Therefore $\det(h_o) = -1$ and $s=2$, contrary to the assumption $s=4$.

For $r=s=3$ one can assume $\det(h_o) = e^{ - \frac{ 2 \pi i}{3}}$, after an eventual inversion $h_o \mapsto h_o^{-1}$. Then by Proposition \ref{ListS3-} one has $h_o = e^{\frac{ 2 \pi i}{3}} I_2$ or $\lambda _1 (h_o) = e^{\frac{ 4 \pi i}{3}}$, $\lambda _2 (h_o) =1$. Assume that $\lambda _1 (h_o) = e^{\frac{4 \pi i}{3}}$, $\lambda _2 (h_o) =1$ and note that the commuting $g$ and $h_o$ can be simultaneously diagonalized by an appropriate $S \in GL(2, {\mathbb C})$. Consequently,
\[
D = S^{-1} g S = \left( \begin{array}{cc}
e^{\frac{ 2 \pi i}{3}}   &  0  \\
0  &  e^{ - \frac{2 \pi i}{3}}
\end{array}  \right) \ \mbox{ and   } \ \
D_o = S^{-1} h_o S = \left( \begin{array}{cc}
e^{ \frac{ 4 \pi i}{3}}   & 0  \\
0  &  1
\end{array}  \right)
\]
are subject to $D^2 D_o = e^{\frac{ 2 \pi i}{3}} I_2$. As a result,
\[
g^2h_o = (SDS^{-1})^{-1} (SD_oS^{-1}) = S (D^2D_o) S^{-1} = e^{\frac{ 2 \pi i}{3}} I_2
\]
and $H = \langle g, h_o \rangle = \langle g, g^2h_o \rangle \simeq H_{C3} (3)$.

Finally, for $r=s=6$, let us assume that $\det(h_o) = e^{ - \frac{ \pi i}{3}}$. Then
\[
\left \{ \lambda _1 (h_o), \lambda _2 (h_o) \right \} = \left \{ e^{\frac{ \pi i}{3}}, e^{ - \frac{ 2 \pi i}{3}} \right \}, \ \
 \left \{ e^{ - \frac{ \pi i}{3}}, 1 \right \} \ \ \mbox{ or  } \ \ \left \{ e^{ \frac{ 2 \pi i}{3}}, -1 \right \}.
 \]
 Similarly to the case of $r=s=3$, the commuting $g$ and $h_o$ admit a simultaneous diagonalization
 \[
 D = S^{-1} g S = \left( \begin{array}{cc}
 e^{ \frac{ 2 \pi i}{3}}  &  0  \\
 0  &  e^{ - \frac{ 2 \pi i}{3}}
 \end{array}  \right), \quad
 D_o = S^{-1} h_o S = \left( \begin{array}{cc}
 \lambda _1 (h_o)   &  0  \\
 0  &  \lambda _2 (h_o)
 \end{array}  \right).
 \]
If $\lambda _1 (h_o) = e^{ - \frac{ \pi i}{3}}$, $\lambda _2 (h_o) =1$ then
\[
D D_o = \left( \begin{array}{cc}
e^{ \frac{ \pi i}{3}}  &  0  \\
0  & e^{ - \frac{ 2 \pi i}{3}}
\end{array}  \right) \ \ \mbox{ and  } \ \ H \simeq \langle D, D_o \rangle = \langle D, DD_o \rangle \simeq H_{C3}(4).
\]
For $\lambda _1 (h_o) = e^{\frac{ 2 \pi i}{3}}$ and $\lambda _2 (h_o) = -1$ note that
\[
DD_o = \left( \begin{array}{cc}
e^{- \frac{ 2 \pi i}{3}}  &  0  \\
0  &  e^{\frac{ \pi i}{3}}
\end{array}  \right), \ \ \mbox{  so that again } \ \
H \simeq \langle D, D_o \rangle = \langle D, DD_o \rangle \simeq H_{C3} (4).
\]
Note that
\[
g = \left( \begin{array}{cc}
e^{\frac{ 2 \pi i}{3}}  &  0  \\
0  &  e^{ - \frac{ 2 \pi i}{3}}
\end{array}  \right), \quad
h_o = \left( \begin{array}{cc}
e^{\frac{ \pi i}{3}}  &  0  \\
0  &  e^{ - \frac{ 2 \pi i}{3}}
\end{array} \right) \in GL(2, \mathcal{O}_{-3})
\]
generate a group, isomorphic to $H_{C3}(4)$.

\end{proof}

\begin{corollary}   \label{HC4}
Let $H$ be a finite subgroup of $GL(2, R)$,
\[
H \cap SL(2, R) = \langle g \rangle \simeq {\mathbb C}_4
\]
and $h_o \in H$ be an element of order $r$ with $\det(H) = \langle \det(h_o) \rangle \simeq {\mathbb C}_s$ and eigenvalues $\lambda _1 (h_o), \lambda _2 (h_o)$. Then $H$ is isomorphic to some $H_{C4} (i)$, $1 \leq i \leq 9$, where
\[
H_{C4} (1) = \langle h_o \rangle \simeq {\mathbb C}_8
\]
with $R = \mathcal{O}_{-2}$, $\lambda _1 (h_o) = e^{\frac{ \pi i}{4}}$, $\lambda _2 (h_o) = e^{\frac{ 3 \pi i}{3}}$,
\[
H_{C4} (2) = \langle g \rangle \times \langle h_o \rangle \simeq {\mathbb C}_4 \times {\mathbb C}_2
\]
with $R = R_{-1,f}$, $\lambda _1(h_o) = -1$, $\lambda _2 (h_o) =1$,
\[
H_{C4}(3) = \langle g, h_o \ \ \vert \ \  g^2 = -I_2, \ \ h_o^2 = I_2, \ \ h_ogh_o^{-1} = g^{-1} \rangle \simeq \mathcal{D}_4
\]
is the dihedral group of order $8$ with $\lambda _1 (h_o) = -1$, $\lambda _2 (h_o) = 1$,
\[
H_{C4}(4) = \langle h_o \rangle \simeq {\mathbb C} _{12}
\]
with $R = \mathcal{O}_{-3}$, $\lambda _1 (h_o) = e^{ \frac{ 5 \pi i}{6}}$, $\lambda _2 (h_o) = e^{ - \frac{ \pi i}{6}}$,
\[
H_{C4}(5) = \langle g \rangle \times \langle e^{\frac{ 2 \pi i}{3}} I_2 \rangle \simeq {\mathbb C}_4 \times {\mathbb C}_3
\]
for  $R = \mathcal{O}_{-3}$ and $\forall g \in SL (2, \mathcal{O}_{-3})$  with $\tr(g) =0$,
\[
H_{C4}(6) = \langle g \rangle \times \langle h_o \rangle \simeq {\mathbb C}_4 \times {\mathbb C}_4
\]
with $R = {\mathbb Z}[i]$, $\lambda _1 (h_o) =i$, $\lambda _2 (h_o) =1$,
\[
H_{C4}(7) = \langle i g \rangle \times \langle h_o \rangle \simeq {\mathbb C}_2 \times {\mathbb C}_8
\]
with $R = {\mathbb Z}[i]$, $\lambda _1 (h_o) = e^{\frac{ 3 \pi i}{4}}$, $\lambda _2 (h_o) = e^{ - \frac{ \pi i}{4}}$,
\[
H_{C4} (8) = \langle g, h_o \ \ \vert \ \  g^2 = h_o^4 = -I_2, \ \ h_ogh_o^{-1} = g^{-1} \rangle
\]
of order $16$ with $R = {\mathbb Z}[i]$, $\lambda _1 (h_o) = e^{\frac{ 3 \pi i}{4}}$, $\lambda _2 (h_o) = e^{ - \frac{ \pi i}{4}}$,
\[
H_{C4} (9) = \langle g, h_o \ \ \vert \ \  g^2 = -I_2, \ \ h_o^6=I_2, \ \ h_ogh_o^{-1} = g^{-1} \rangle
\]
of order $24$ with $R = \mathcal{O}_{-3}$, $\lambda _1 (h_o) = e^{\frac{ 2 \pi i}{3}}$, $\lambda _2 (h_o) = e^{ - \frac{ \pi i}{3}}$.

There exist subgroups
\[
H_{C4}(1) < GL(2, \mathcal{O}_{-2}),  \ \
H_{C4}(4), H_{C4}(5) < GL(2, \mathcal{O}_{-3}),
\]
\[
H_{C4}(2), H_{C4}(6) < GL (2, {\mathbb Z}[i]),
\]
 as well as subgroups
\[
H^{o} _{C4}(7), H^{o} _{C4}(8) < GL(2, {\mathbb Q} ( \sqrt{2}, i)), \ \
H^{o} _{C4}(3) < GL(2, {\mathbb Q} ( \sqrt{-d})),
\]
\[
H^{o} _{C4}(9) < GL(2, {\mathbb Q}(\sqrt{-3})),
\]
with $H^{o} _{C4}(j) \simeq H_{C4} (j)$ for $j \in \{ 3, 7, 8, 9 \}$.
\end{corollary}

\begin{proof}

If $\frac{r}{s} =4$ then either $(s,r) = (2,8)$ and $H \simeq H_{C4} (1)$ or $(s,r) = (3,12)$ and $H \simeq H_{C4}(4)$. By Proposition \ref{ListS2} there exists an element $h_o \in GL(2, \mathcal{O}_{-2})$ of order $8$ with $\det (h_o) = -1$. Proposition \ref{ListS3+} provides an example of $h_o \in GL(2, \mathcal{O}_{-3})$ of order $12$ with $\det( h_o) = e^{\frac{ 2 \pi i}{3}}$. There remain to be considered the cases with $\frac{r}{s} \in \{ 1, 2 \}$. According to Lemma \ref{HK346Necessary}, the non-abelian $H$ under consideration are isomorphic to $H_{C4} (3)$, $H_{C4}(8)$ or $H_{C4}(9)$. By Lemma \ref{HK346Attained} (i) there is a subgroup $H^{o} _{C4} (3) < GL(2, {\mathbb Q} ( \sqrt{-d}))$, conjugate to $H_{C4}(3)$. Lemma \ref{HK346Attained} (iii) provides an example of $S^{-1} H_{C4} (8) S = H^{o} _{C4} (8) < GL(2, {\mathbb Q}( \sqrt{2},i))$, while Lemma \ref{HK346Attained}(ii) justifies the existence of $S^{-1} H_{C4} (9) S = H^{o} _{C4} (9) < GL(2, {\mathbb Q} (\sqrt{-3}))$.

There remain to be classified the non-cyclic abelian groups $H = [ H \cap SL(2,R)] \langle h_o \rangle$ with $H \cap SL(2,R) \simeq {\mathbb C}_4$, $\langle h_o \rangle \simeq {\mathbb C} _r$, $\det(h_o) = e^{\frac{ 2 \pi i}{3}}$ for $s \in \{ 2, 3, 4, 6 \}$, $r \in \{ s, 2s \}$.

If $r=s=2$ then by Proposition \ref{ListS2}, the eigenvalues of $h_o$ are $\lambda _1 (h_o) = -1$ and $\lambda _2 (h_o) =1$. There exists a matrix $S \in GL(2, {\mathbb Q} ( \sqrt{-d})$, such that
\[
D_o = S^{-1} h_o S = \left(  \begin{array}{rr}
-1  &  0  \\
0  &  1
\end{array}  \right).
\]
Proposition \ref{ListK} establishes that $g \in SL(2,R)$ is of order $4$ exactly when $\tr(g) =0$. The trace and the determinant are invariant under conjugation, so that
\[
D = S^{-1} g S = \left( \begin{array}{rr}
a  &  b  \\
c  &  -a
\end{array}  \right) \in SL( 2, {\mathbb Q} (\sqrt{-d})).
\]
The commutation
\[
D D_o = \left( \begin{array}{rr}
-a  &  b  \\
-c  &  -a
\end{array} \right) =
\left( \begin{array}{rr}
-a  &  -b  \\
c  &  -a
\end{array}  \right) = D_oD
\]
holds only when $b=c=0$ and
\[
D = \pm  \left( \begin{array}{rr}
i  &  0  \\
0  &  -i
\end{array} \right).
\]
Bearing in mind that $D \in SL(2, {\mathbb Q} ( \sqrt{-d}))$, one concludes that $i \in {\mathbb Q}(\sqrt{-d})$, whereas $d=1$ and $R = R_{-1,f}$. The matirces
\[
g = \left( \begin{array}{rr}
i  &  0  \\
0  &   -i
\end{array}  \right), \ \
h_o = \left( \begin{array}{rr}
-1  &  0  \\
0  &  1
\end{array} \right) \in GL(2, {\mathbb Z}[i])
\]
generate a subgroup of $GL(2, {\mathbb Z}[i])$, isomorphic to $H_{C4}(2)$.

For $s=2$  and $r=4$ one has $R = {\mathbb Z}[i]$ and $h_o = \pm I_2$. Bearing in mind that $g \in SL(2, R)$ is of order $4$ if and only if $\tr(g) =0$, let
\[
g = \left( \begin{array}{cc}
a  &  b  \\
c  &  d
\end{array}  \right) \in SL(2, {\mathbb Z}[o]).
\]
Then
\[gh_o = \pm \left( \begin{array}{rr}
ai  &  bi  \\
ci  &  -ai
\end{array}  \right) \in {\mathbb Z}[i]_{2 \times 2}
\]
has determinant $\det(gh_o) = \det(g) \det(h_o) = \det(h_o) =-1$ and trace $\tr(gh_o) =0$. By Proposition \ref{ListS2}, $gh_o$ has eigenvalues $\lambda _1 (gh_o) = -1$, $\lambda _2 (gh_o) =1$ and $H \simeq H_{C4}(2)$.

If $s=r=3$ then $R = \mathcal{O}_{-3}$ and either $h_o = e^{  - \frac{ 2 \pi i}{3}}$ or $\lambda _1 (h_o) = e^{\frac{ 2 \pi i}{3}}$,
$\lambda _2 ( h_o) =1$. Replacing $e^{ - \frac{2 \pi i}{3}} I_2$ by its inverse, one observes that $H_{C4}(5) = \langle g, e^{ - \frac{ 2 \pi i}{3}} I_2 \rangle < GL(2, \mathcal{O}_{-3})$. If $\lambda _1 (h_o) = e^{\frac{ 2 \pi i}{3}}$, $\lambda _2 (h_o) =1$, then there exists $S \in GL(2, {\mathbb Q} ( \sqrt{-3}))$, such that
\[
D_o = S^{-1}h_o S = \left( \begin{array}{cc}
e^{\frac{ 2 \pi i}{3}}  &  0  \\
0  &  1
\end{array}  \right).
\]
The determinant and the trace are invariant under conjugation, so that
\[
D = S^{-1} g S = \left( \begin{array}{rr}
a  &  b  \\
c  & -a
\end{array} \right) \in SL(2, {\mathbb Q}( \sqrt{-3})).
\]
Note that
\[
D D_o = \left( \begin{array}{rr}
e^{\frac{ 2 \pi i}{3}} a   &  b  \\
e^{\frac{2 \pi i}{3}} c  &  -a
\end{array} \right) =
\left( \begin{array}{rr}
e^{\frac{ 2 \pi i}{3}} a  &  e^{\frac{ 2 \pi i}{3}} b  \\
c  &  -a
\end{array}  \right) = D_o D
\]
is equivalent to $b=c=0$ and $1 = \det(g) = \det(D) = - a^2$ specifies that
\[
D = \pm \left( \begin{array}{rr}
i  &  0  \\
0  &  -i
\end{array}  \right).
\]That contradicts $F \in SL(2, {\mathbb Q}( \sqrt{-3}))$ and justifies the non-existence of $H$ with $s=r=3$.

Let $s=3$, $r=6$. According to Proposition \ref{ListS3+}, there follows $R = \mathcal{O}_{-3}$ with $h_o = e^{\frac{ \pi i}{3}} I_2$ or
$\lambda _1 (h_o) = e^{ - \frac{ \pi i}{3}}$, $\lambda _2 (h_o) =1$. If $h_o = e^{\frac{ \pi i}{3}}$ then $H = \langle g, h_o \rangle = \langle g, g^2h_o = - h_o = e^{ - \frac{ 2 \pi i}{3}} I_2 \rangle \simeq H_{C4}(5)$. In the case of $\lambda _1 (h_o) = e^{ - \frac{ \pi i}{3}}$, $\lambda _2 (h_o)=1$ let us choose
 $S \in GL(2, {\mathbb Q} ( \sqrt{-3}))$ with
\[
D_o = S^{-1} h_o S = \left( \begin{array}{cc}
e^{ - \frac{ \pi i}{3}}  &  0    \\
0  &  1
\end{array}  \right) \in GL(2, {\mathbb Q} ( \sqrt{-3})) \ \ \mbox{  and  }
\]
\[
D = S^{-1} gS = \left( \begin{array}{rr}
a  &  b  \\
c  & -a
\end{array}  \right) \in SL(2, {\mathbb Q}( \sqrt{-3})).
\]
Then
\[
D D_o = \left( \begin{array}{rr}
e^{ - \frac{ \pi i}{3}} a  &  b  \\
e^{ - \frac{ \pi i}{3}} c  & -a
\end{array}  \right) =
\left( \begin{array}{rr}
e^{ - \frac{ \pi i}{3}} a   &  e^{ - \frac{\pi i}{3}} b  \\
c  & -a
\end{array} \right)  = D_o D
\]
if and only if
\[
D = \pm \left( \begin{array}{rr}
i  &  0  \\
0  &  -i
\end{array}  \right) \in SL(2, {\mathbb Q} ( \sqrt{-3})),
\]
which is an absurd.

Let us suppose that $s=r=4$. The Proposition \ref{ListS4+} specifies that $R = {\mathbb Z}[i]$ and $\lambda _1 (h_o) = \varepsilon i$, $\lambda _2 (h_o) = \varepsilon$ for some $\varepsilon \in \{ \pm 1 \}$. As far as $g^2 = -I_2 \in H$, there is no loss of generality in assuming that $\lambda _1 (h_o) =i$, $\lambda _2 (h_o) =1$ and $H \simeq H_{C4}(6)$. Note that
\[
g = \left( \begin{array}{rr}
i  &  0  \\
0  & -i
\end{array}  \right), \ \ h_o \in \left( \begin{array}{rr}
i  &   0  \\
0  &  1
\end{array}  \right) \in GL(2, {\mathbb Z}[i])
\]
generate a subgroup, isomorphic to $H_{C4}(6)$.

For $s=4$, $r=8$, Proposition \ref{ListS4+} implies that $R = {\mathbb Z}[i]$ and $\lambda _1 (h_o) = e^{ \frac{ 3 \pi i}{4}}$, $\lambda _2 (h_o) = e^{ - \frac{ \pi i}{4}}$. Note that $(ig)^2 = -g^2 = I_2$, so that $ig \in H = \langle g, h_o \rangle$ is of order $2$ and $h_o^6 = i I_2$, according to $\lambda _1 (h_o^6) = \lambda _1 (h_o)^6 = i$, $\lambda _2 (h_o^6) = \lambda _2 (h_o)^6 = i$. Consequently,
\[
H = \langle g, h_o \rangle = \langle h_o^6g = ig, h_o \rangle = \langle ig \rangle \times \langle h_o \rangle \simeq {\mathbb C}_2 \times {\mathbb C}_8,
\]
as far as $\langle ig \rangle \cap \langle h_o \rangle = \{ I_2 \}$. More precisely, if $ig = h_o^m$, then the second eigenvalue
\[
1 = -i^2 = \lambda _2(ig) = \lambda _2 (h_o^m) = e^{ - \frac{ \pi i m}{4}},
\]
whereas $m \in 8 {\mathbb Z}$ and the first eigenvalue
\[
-1 = \lambda _1 (ig) = \lambda _1(h_o^m) = e^{ \frac{ 3 \pi i m}{4}} = 1,
\]
which is an absurd. Thus, $H \simeq H_{C4}(7)$ and there exists a subgroup
\[
H^{o} _{C4}(7) = \langle \left( \begin{array}{rr}
i  &  0  \\
0  &  -i
\end{array}  \right), \ \ \left( \begin{array}{rr}
e^{\frac{ 3 \pi i}{4}}   &  0  \\
0  &  e^{ - \frac{ \pi i}{4}}
\end{array} \right) \rangle < GL(2, {\mathbb Q} ( \sqrt{2},i)),
\]
conjugate to $H_{C4} (7)$.

Let us assume that $s=r=6$. Then Proposition \ref{ListS6+} applies to provide $R = \mathcal{O}_{-3}$ and
\[
\{ \lambda _1 (h_o), \lambda _2 (h_o) \} = \left \{ e^{\frac{ 2 \pi i}{3}}, e^{ - \frac{ \pi i}{3}} \right \}, \ \
\left \{ e^{\frac{ \pi i}{3}}, 1 \right \}, \ \
\left \{ e^{ - \frac{ 2 \pi i}{3}}, -1 \right \}.
\]
Choose a matrix $S \in GL(2, {\mathbb Q}( \sqrt{-3}))$ with
\[
D_o = S^{-1} h_o S = \left( \begin{array}{cc}
\lambda _1 (h_o)  &  0  \\
0  &  \lambda _2 (h_o)
\end{array}  \right) \in GL(2, {\mathbb Q} ( \sqrt{-3})),
\]
\[
D = S^{-1} g S = \left( \begin{array}{rr}
a  &  b  \\
c  &  -a
\end{array}  \right) \in SL(2, {\mathbb Q} ( \sqrt{-3})).
\]If $\lambda _1 (h_o) \neq \lambda _2 (h_o)$ then
\[
D D_o = \left( \begin{array}{rr}
\lambda _1 (h_o) a    &  \lambda _2 (h_o) b  \\
\lambda _1 (h_o) c  &  - \lambda _2 (h_o) a
\end{array}  \right) =
\left( \begin{array}{rr}
\lambda _1 (h_o)  a  &  \lambda _1(h_o) b  \\
\lambda _2 (h_o) c  &  - \lambda _2 (h_o) a
\end{array}  \right) = D_o D
\]
is tantamount to $b=c=0$, $a = \pm i$ and
\[
D = \pm \left( \begin{array}{rr}
i  &  0  \\
0  & -i
\end{array}  \right) \in SL(2, {\mathbb Q}( \sqrt{-3})
\]
is an absurd.

Similarly, in the case of $s=6$, $r=12$, Proposition \ref{ListS6+} derives that $R = \mathcal{O}_{-3}$ and
\[
\{ \lambda _1 (h_o), \lambda _2 (h_o) \} = \left \{ e^{ \frac{ 2 \pi i}{3}}, e^{ - \frac{ \pi i}{3}} \right \}, \ \
\left \{ e^{\frac{ \pi i}{3}}, 1 \right \}, \ \
\left \{ e^{ - \frac{ 2 \pi i}{3}}, -1 \right \}.
\]
Note that $\lambda _1 (h_o) \neq \lambda _2 (h_o)$ for all the possibilities and apply the considerations for $s=r=6$, in order to exclude the case $s=6$, $r=12$.

\end{proof}

\begin{corollary}   \label{HC6}
Let $H$ be a finite subgroup of $GL(2, R)$,
\[
H \cap SL(2,R) = \langle g \rangle \simeq {\mathbb C}_6
\]
and $h_o \in H$ be an element of order $r$ with $\det(H) = \langle \det(h_o) \rangle \simeq {\mathbb C}_s$ and eigenvalues $\lambda _1 (h_o)$, $\lambda _2 (h_o)$.  Then $H$ is isomorphic to some $H_{C6}(i)$, $1 \leq i \leq 7$, where
\[
H_{C6}(1) = \langle h_o \rangle \simeq {\mathbb C}_{12}
\]
with $R = {\mathbb Z}[i]$, $\lambda _1 (h_o) = e^{\frac{ \pi i}{6}}$, $\lambda _2 (h_o) = e^{\frac{ 5 \pi i}{6}}$,
\[
H_{C6}(2) = \langle g \rangle \times \langle h_o \rangle \simeq {\mathbb C}_6 \times {\mathbb C}_{12}
\]
with $R = \mathcal{O}_{-3}$ or $R  = R_{-3,2}$, $\lambda _1 (h_o) = -1$, $\lambda _2 (h_o) = 1$,
\[
H_{C6}(3) = \langle g, h_o \ \ \vert \ \  g^3 = -I_2, \ \ h_o^2 = I_2, \ \ h_ogh_o^{-1} = g^{-1} \rangle \simeq \mathcal{D}_6
\]
is the dihedral group of order $12$, $\lambda _1 (h_o) = -1$, $\lambda _2 ( h_o) =1$,
\[
H_{C6} (4) = \langle g \rangle \times \langle e^{\frac{ 2 \pi i}{3}} I_2 \rangle \simeq {\mathbb C}_6 \times {\mathbb C}_3
\]
with $R = \mathcal{O}_{-3}$ and $\forall g \in SL(2, \mathcal{O}_{-3})$ of $\tr(g) =1$,
\[
H_{C6}(5) = \langle g, h_o \ \ \vert \ \  g^3 = h_o^4 = -I_2, \ \ h_og h_o^{-1} = g^{-1} \rangle
\]
of order $24$ with $R = {\mathbb Z}[i]$, $\lambda _1 (h_o) = e^{\frac{ 3 \pi i}{4}}$, $\lambda _2 (h_o) = e^{ - \frac{ \pi i}{4}}$,
\[
H_{C6}(6) = \langle g, h_o \ \ \vert \ \ g^3 = -I_2, \ \ h_o ^6 = I_2, \ \ h_o g h_o^{-1} = g^{-1} \rangle
\]of order $36$ with $R = \mathcal{O}_{-3}$, $\lambda _1 (h_o) = e^{\frac{ 2 \pi i}{3}}$, $\lambda _2 (h_o) = e^{ - \frac{ \pi i}{3}}$,
\[
H_{C6}(7) = \langle g \rangle \times \langle h_o \rangle \simeq {\mathbb C}_6 \times {\mathbb C}_6
\]
of order $36$ with $R = \mathcal{O}_{-3}$, $\lambda _1 (h_o) = e^{\frac{ 2 \pi i}{3}}$, $\lambda _2 (h_o) = e^{ - \frac{ \pi i}{3}}$.

There exist subgroups
\[
H_{C6}(1) < GL(2, {\mathbb Z}[i]), \ \ H_{C6}(2), H_{C6}(4), H_{C6}(7) < GL(2, \mathcal{O}_{-3}),
\]
as well as subgroups
\[
H^{o} _{C6} (3) < GL(2, {\mathbb Q} ( \sqrt{-d})), \ \ H^{o}_{C6}(5) < GL(2, {\mathbb Q} ( \sqrt{2},i)),
\]
\[
H^{o}_{C6}(6) < GL(2, {\mathbb Q} ( \sqrt{-3}))
\]
with $H^{o}_{C6} (j) \simeq H_{C6}(j)$ for $j \in \{ 3, 5, 6 \}$.
\end{corollary}

\begin{proof}

According to Lemma \ref{HK346Necessary}(i), the ratio $\frac{r}{s} \in \{ 1, 2, 3, 6 \}$ is a divisor of $t=6$. If $r = 6s$ then $s=2$ and $H = \langle h_o \rangle \simeq {\mathbb C}_{12} \simeq H_{C6}(1)$ by  Lemma \ref{HK346Necessary} (i), (ii). According to Proposition \ref{ListS2}, the existence of $h_o \in GL(2,R)$ of order $12$ with $\det(h_o) =-1$ requires   $R = {\mathbb Z}[i]$ and there exist $h_o \in GL(2, {\mathbb Z}[i])$ of order $12$ with $\det(h_o) =-1$.

For $r=3s$ Lemma \ref{HK346Necessary}(i) specifies that $s=2$. Combining with Lemma \ref{HK346Necessary}(iv), one concludes that
\[
H= \langle g, h_o \ \ \vert \ \  g^3 = -I_2, \ \ h_o^6 = I_2, \ \ h_og = gh_o \rangle
\]
is a non-cyclic abelian group of order $st=12$. By Proposition \ref{ListS2}, $R = \mathcal{O}_{-3}$ or $R = R_{-3,2}$ and $h_o$ has eigenvalues $\lambda _1 (h_o) = e^{ \frac{ \varepsilon \pi i}{3}}$, $\lambda _2 (h_o) = e^{ \frac{ \varepsilon 2 \pi i}{3}}$ for some $\varepsilon \in \{ \pm 1 \}$.
Due to $\langle g, h_o \rangle = \langle g, h_o^{-1} = h_o^5 \rangle$ by $h_o = (h_o^5)^5$, one can assume that $\lambda _1 (h_o) = e^{\frac{ \pi i}{3}}$, $\lambda _2 (h_o) = e^{ \frac{ 2 \pi i}{3}}$. The commuting matrices $g$ and $h_o$ admit a simultaneous diagonalization
\[
D = S^{-1} g S = \left( \begin{array}{cc}
e^{\frac{ \pi i}{3}}  &  0  \\
0  &  e^{ - \frac{ \pi i}{3}}
\end{array}  \right), \ \
D_o = S^{-1} h_o S = \left( \begin{array}{cc}
e^{\frac{ \pi i}{3}}  &  0  \\
0  & e^{\frac{ 2 \pi i}{3}}
\end{array}  \right)
\]
by an appropriate $S \in GL(2, {\mathbb Q} ( \sqrt{-3}))$. Then
\[
D^2 D_o = \left( \begin{array}{rr}
-1   &  0  \\
0  &  1
\end{array} \right)
\]
implies that $\lambda _1 (g^2h_o) = -1$, $\lambda _2(g^2h_o) = 1$. As a result, $H = \langle g, h_o \rangle = \langle g, g^2h_o \rangle$ is a subgroup of $GL(2, \mathcal{O}_{-3})$, isomorphic to $H_{C6}(2)$.

Form now on, $\frac{r}{s} \in \{ 1, 2 \}$. In particular, $\frac{r}{s} < t=6$ and the non-abelian
\[
H = \langle g, h_o \ \ \vert \ \  g^6 = h_o^r = I_2, \ \ h_og h_o^{-1} = g^{-1} \rangle
\]
occurs for $(r,s) \in \{ (2,2), (8,4), (6,6) \}$, according to Lemma \ref{HK346Necessary}(iv). Namely, for $r=s=2$ one has a dihedral group $H \simeq \mathcal{D}_6 \simeq H_{C6}(3)$ of order $12$, which is realized as a subgroup of $GL(2, {\mathbb Q} ( \sqrt{-d}))$ by Lemma \ref{HK346Attained}(i). In the case of $s=4$ and $r=8$ the group $H \simeq H_{C6} (5)$ of order $24$ is embedded in $GL(2, {\mathbb Q}( \sqrt{2},i))$ by Lemma \ref{HK346Attained}(iii). In the case of $r=s=6$ one has $H \simeq H_{C6}(6)$ of order $36$, realized as a subgroup of $GL(2, {\mathbb Q} ( \sqrt{-3}))$ by Lemma \ref{HK346Attained}(ii).

There remain to be considered the non-cyclic abelian $H$ with $r=2s$, $s \in \{ 2, 3, 4 \}$ or $r=s \in \{ 2, 3, 4, 6 \}$. If $s=2$, $r=4$ then Proposition \ref{ListS2} requires $R = {\mathbb Z}[i]$ and $h_o = \pm iI_2$. Up to an inversion of $h_o$, one can assume that $h_o = iI_2$. Then $H = \langle g, iI_2 \rangle = \langle -g = (iI_2)^2g, iI_2 \rangle$ is generated by the element $-g$ of order $3$ and the scalar matrix $iI_2 \in H$ of order $4$, so that $-ig = (iI_2) (-g) \in H$ of order $12$ generates $H$, $H \simeq H_{C6}(1) \simeq {\mathbb C}_{12}$. (Note that for $g \in SL(2, {\mathbb Z}[i])$ of order $6$ one has $g^3= -I_2$, whereas $(-g)^3 = -g^3 = I_2$. The assumptions $-g = I_2$ and $(-g)^2 = g^2 = I_2$ lead to an absurd. )

Let us assume that $s=3$ and $r=6$. Then Proposition \ref{ListS3+} implies that $R = \mathcal{O}_{-3}$ with $h_o = E^{\frac{ \pi i}{3}} I_2$ or $\lambda _1 (h_o) = e^{ - \frac{ \pi i}{3}}$, $\lambda _2 (h_o) = -1$. Note that $H = \langle g, e^{\frac{ \pi i}{3}} I_2 \rangle = \langle g, e^{ - \frac{\pi i}{3}} I_2 \rangle$ by
$e^{ - \frac{ \pi i}{3}} = \left( e^{ \frac{ \pi i}{3}} \right) ^5$,
 $e^{\frac{\pi i}{3}} = \left( e^{ - \frac{\pi i}{3}} \right) ^5$.
 Further,
 \[
 g^3 \left( e^{ - \frac{ \pi i}{3}} I_2 \right) = \left( e^{ \pi i} I_2 \right) \left( e^{ - \frac{\pi i}{3}} I_2 \right) = e^{\frac{ 2 \pi i}{3}} I_2
  \]
  implies that
  \[
 H = \langle g, e^{ - \frac{ \pi i}{3}} I_2 \rangle = \langle g, g^3 \left( e^{ - \frac{ \pi i}{3}} I_2 \right)  = e^{\frac{ 2 \pi i}{3}} I_2 \rangle = \langle g \rangle \times \langle e^{ \frac{ 2 \pi i}{3}} \rangle \simeq {\mathbb C}_6 \times {\mathbb C}_3 \simeq H_{C6}(4).
 \]
For any $g \in SL(2, \mathcal{O}_{-3})$ of order $6$, there is a subgroup $H_{C6}(4) = \langle g, e^{\frac{2 \pi i}{3}} I_2 \rangle < GL(2, \mathcal{O}_{-3})$.

For $s=4$, $r=8$ there follow $R = {\mathbb Z}[i]$ and $\lambda _1 (h_o) = e^{\frac{ 3 \pi i}{4}}$, $\lambda _2 (h_o) = e^{ - \frac{ \pi i}{4}}$, according to Proposition \ref{ListS4+}. Suppose that $S \in GL(2, {\mathbb Q}( \sqrt{2},i ))$ diagonalizes $h_o$,
\[
D_o = S^{-1} h_oS = \left(  \begin{array}{cc}
e^{\frac{ 3 \pi i}{4}}   &  0  \\
0  &  e^{ - \frac{ \pi i}{4}}
\end{array}  \right) \in GL(2, {\mathbb Q} ( \sqrt{2},i)).
\]
By Proposition \ref{ListK}, $g \in SL(2, {\mathbb Z}[i])$ is of order $6$ exactly when $\tr(g) =1$. Since the determinant and the trace are invariant under conjugation, one has
\[
D = S^{-1} g S  = \left(  \begin{array}{cc}
a  &  b  \\
c  &  1-a
\end{array}  \right) \in SL(2, {\mathbb Q}( \sqrt{2},i).
\]
However,
\[
D D_o = \left( \begin{array}{cc}
e^{\frac{ 3 \pi i}{4}} a   &  e^{ - \frac{ \pi i}{4}} b  \\
e^{\frac{ 3  \pi i}{4}} c   &  e^{ - \frac{ \pi i}{4}} (1-a)
\end{array}  \right) =
\left( \begin{array}{cc}
e^{\frac{ 3 \pi i}{4}}  a  &  e^{\frac{3 \pi i}{4}} b  \\
e^{ - \frac{ \pi i}{4}} c  &  e^{ - \frac{ \pi i}{4}}  (1-a)
\end{array} \right) = D_oD
\]
if and only if $b=c=0$ and $a = e^{\frac{ \varepsilon \pi i}{3}}$ for some $\varepsilon \in \{ \pm 1 \}$. Now,
\[
D = \left( \begin{array}{cc}
e^{\frac{ \varepsilon \pi i}{3}}  &  0  \\
0  &  1 - e^{ \frac{ \varepsilon \pi i}{3}}
\end{array}  \right) \in SL(2, {\mathbb Q} ( \sqrt{2},i))
\]
is an absurd, justifying the non-existence of $H$ with $s=4$ and $r=8$.

In the case of $r=s=2$ Proposition \ref{ListS2} implies that $\lambda _1 (h_o) = -1$, $\lambda _2 (h_o) =1$, so that $H \simeq H_{C6} (2) \simeq {\mathbb C}_6 \times {\mathbb C}_2$.

For $r=s=3$ Proposition \ref{ListS3+} reveals that $R = \mathcal{O}_{-3}$ with $h_o = e^{ - \frac{ 2 \pi i}{3}} I_2$ or $\lambda _1 (h_o) = e^{\frac{2 \pi i}{3}}$, $\lambda _2 (h_o) =1$. It is clear that
\[
H = \langle g, e^{ - \frac{ 2 \pi i}{3}} I_2 = \left( e^{\frac{ 2 \pi i}{3}} I_2 \right) ^2 \rangle =
\langle g, e^{\frac{ 2 \pi i}{3}} I_2 = \left( e^{ - \frac{ 2 \pi i}{3}} I_2 \right) ^2 \rangle \simeq H_{C6}(4) \simeq {\mathbb C}_3 \times {\mathbb C}_3.
\]
If $\lambda _1 (h_o) = e^{\frac{ 2 \pi i}{3}}$, $\lambda _2 (h_o) =1$ then the commuting matrices $g$ and $h_o$ admit a simultaneous diagonalization
\[
D = S^{-1} g S = \left( \begin{array}{cc}
e^{\frac{ \pi i}{3}} &  0  \\
0  &  e^{ - \frac{ \pi i}{3}}
\end{array} \right), \ \
D_o = S^{-1} h_o S = \left( \begin{array}{cc}
e^{\frac{ 2 \pi i}{3}}   &  0  \\
0  &  1
\end{array}  \right) \in GL(2, {\mathbb Q} ( \sqrt{-3}))
\]
by an appropriate $S \in GL(2, {\mathbb Q} ( \sqrt{-3}))$. Then $D^2 D_o = e^{ - \frac{ 2 \pi i}{3}} I_2$, whereas $g^2h_o = S \left( e^{ - \frac{ 2 \pi i}{3}} I_2 \right) S^{-1} = e^{ - \frac{ 2 \pi i}{3}} I_2 $ and
\[
H = \langle g, h_o \rangle = \langle g, g^2h_o = e^{ - \frac{ 2 \pi i}{3}} I_2 \rangle \simeq H_{C6} (4) \simeq {\mathbb C}_6 \times {\mathbb C}_3.
\]

The assumption $r=s=4$ implies that $R = {\mathbb Z}[i]$ and $\lambda _1 (h_o) = \varepsilon i$, $\lambda _2 (h_o) = \varepsilon$ for some $\varepsilon \in \{ \pm 1 \}$, according to Proposition \ref{ListS4+}. Due to $g^3 = -I_2$, one has $\langle g, h_o \rangle = \langle g, -h_o = g^3h_o \rangle$, so that there is no loss of generality in assuming $\varepsilon =1$. If $S \in GL(2, {\mathbb Q}(i))$ conjugates $h_o$ to its diagonal form
\[
D_o = S^{-1} h_o S = \left( \begin{array}{cc}
i  &  0  \\
0  &  1
\end{array}  \right) \in GL(2, {\mathbb Q}9i)),
\]
then
\[
D = S^{-1} g S = \left( \begin{array}{cc}
a  &  b  \\
c  &  1-a
\end{array}  \right) \in SL(2, {\mathbb Q}(i)).
\]
The relation
\[
D D_o = \left( \begin{array}{cc}
ia  &  b  \\
ic  &  1-a
\end{array}  \right) =
\left( \begin{array}{cc}
ia  &  ib  \\
c  & 1-a
\end{array}  \right) = D_o D
\]
implies that
\[
D = \left( \begin{array}{cc}
e^{\frac{ \varepsilon \pi i}{3}}  &  0  \\
0  &  e^{ - \frac{ \varepsilon \pi i}{3}}
\end{array}  \right) \in SL(2, {\mathbb Q}(i)) \ \ \mbox{ for some  } \ \ \varepsilon \in \{ \pm \}.
\]
The contradiction proves the non-existence of $H$ with $r=s=4$.

Finally, for $r=s=6$ Proposition \ref{ListS6+} specifies that $R = \mathcal{O}_{-3}$ and
\[
\{ \lambda _1 (h_o), \lambda _2 (h_o) \} = \left \{ e^{\frac{ 2 \pi i}{3}}, e^{ - \frac{ \pi i}{3}} \right \}, \ \
\left \{ 1, e^{ \frac{ \pi i}{3}} \right \} \ \ \mbox{  or  } \ \
\left \{ e^{ - \frac{ 2 \pi i}{3}}, -1 \right \}.
\]
The commuting matrices $g$ and $h_o$ admit simultaneous diagonalization
\[
D = S^{-1} g S = \left( \begin{array}{cc}
e^{\frac{ \pi i}{3}}   &  0  \\
0  &  e^{ - \frac{\pi i}{3}}
\end{array}  \right),
\]
\[
D_o = S^{-1} h_o S =  \left(  \begin{array}{cc}
\lambda _1(h_o)  &  0  \\
0  &  \lambda _2 (h_o)
\end{array}  \right) \in GL(2, {\mathbb Q}( \sqrt{-3})
\]
by an appropriate $S \in GL(2,{\mathbb Q} ( \sqrt{-3}))$. Let us denote
\[
D_o := \left( \begin{array}{cc}
e^{\frac{ 2 \pi i}{3}}  &  0  \\
0  &  e^{ - \frac{ \pi i}{3}}
\end{array}  \right), \ \
D'_o := \left( \begin{array}{cc}
1  &  0  \\
0  &  e^{\frac{ \pi i}{3}}
\end{array}  \right), \ \
D''_o := \left( \begin{array}{cc}
e^{ - \frac{2 \pi i}{3}} &  0  \\
0  &  -1
\end{array}  \right) \in GL(2, \mathcal{O}_{-3})
\]
and observe that
\[
D^2 D_o = D''_o, \ \ D62 D''_o = D'_o.
\]
By its very definition,
\[
H = \langle D, D_o \rangle < GL(2, \mathcal{O}_{-3})
\]
is isomorphic to $H_{C6}(7)$. The equalities $\langle D, D'_o = D^2 D''_o \rangle = \langle D, D''_o \rangle$ and
$\langle D, D''_o = D^2 D_o \rangle = \langle D, D_o \rangle$ conclude the proof of the proposition.

\end{proof}


\begin{proposition}      \label{HQ8}
Let $H$ be a finite subgroup of $GL(2,R)$,
\[
H \cap SL(2,R) = \langle g_1, g_2 \ \ \vert \ \  g_1^2 = g_2 ^2 = -I_2,  \ \  g_2g_1 = - g_1g_2 \rangle \simeq {\mathbb Q}_8,
\]
and $h_o \in H$ be an element of order $r$ with $\det(H) = \langle \det (h_o) \rangle \simeq {\mathbb C}_s$ and eigenvalues $\lambda _1 (h_o), \lambda _2(h_o)$. Then $H$ is isomorphic to some $H_{{\mathbb Q}8} (i)$, $1 \leq i \leq 9$, where
\[
H_{{\mathbb Q}8} (1) = \langle g_1, g_2, i I_2 \ \ \vert \ \  g_1^2 = g_2 ^2 = -I_2, \ \  g_2g_1 = -g_1g_2 \rangle
\]
is of order $16$ with  $R = {\mathbb Z}[i]$,
\[
H_{{\mathbb Q}8}(2) = \langle g_1, g_2, h_o \ \ \vert \ \  g_1^2 = g_2^2 = -I_2, \ \  h_o ^2 = I_2, \ \  g_2g_1 = - g_1g_2,
\]
\[
  h_o g_1 h_o^{-1} = -g_1, \ \ h_o g_2 h_o^{-1} = -g_2 \rangle
\]
is of order $16$ with  $R = {\mathbb Z}[i]$, $\lambda _1(h_o) = -1$, $\lambda _2(h_o) =1$,
\[
H_{{\mathbb Q}8}(3) = \langle g_1, g_2, h_o \ \ \vert \ \  g_1^2 = g_2^2 = h_o^4 = -I_2, \ \  g_2g_1 = - g_1g_2,
\]
\[
  h_o g_1 h_o^{-1} = g_2, \ \
h_o g_2 h_o^{-1} = -g_1 \rangle
\]
is of order $16$ with  $R = \mathcal{O}_{-2}$, $\lambda _1(h_o) = e^{\frac{\pi i}{4}}$, $\lambda _2(h_o) = e^{\frac{3 \pi i}{4}}$, $h_o ^2 = \pm g_1g_2$,
\[
H_{{\mathbb Q}8}(4) = \langle g_1, g_2, h_o \ \ \vert \ \  g_1^2 = g_2^2 = -I_2, \ \  h_o^2 = I_2, \ \ g_2g_1 = - g_1g_2,
  \]
  \[
  h_o g_1 h_o ^{-1} = g_2, \ \  h_o g_2 h_o^{-1} = g_1 \rangle
\]
is of order $16$ with  $R = R_{-2,f}$, $\lambda _1(h_o) = -1$, $\lambda _2(h_o) =1$,
\[
H_{{\mathbb Q}8}(5) = \langle g_1, g_2 \rangle \times \langle e^{ \frac{ 2 \pi i}{3}} \rangle \simeq {\mathbb Q}_8 \times {\mathbb C}_3
\]
is of order $24$ with  $R = \mathcal{O}_3$,
\[
H_{{\mathbb Q}8}(6) = \langle g_1, g_2, h_o \ \ \vert \ \  g_1^2 = g_2^2 = - I_2, \ \  h_o^3 =I_2, \ \ g_2g_1 = -g_1g_2,
\]
\[
 h_og_1h_o^{-1} = g_2, \ \ h_o g_2 h_o^{-1} = g_1g_2 \rangle
 \]
 is of order $24$ with  $R = \mathcal{O}_{-3}$, $\lambda _1(h_o) = e^{\frac{2 \pi i}{3}}$, $\lambda _2 (h_o) =1$,
 \[
 H_{{\mathbb Q}8}(7) = \langle g_1, g_2, h_o \ \ \vert \ \ g_1^2 = g_2^2 = h_o ^4 = -I_2, \ \  g_2g_1 = - g_1g_2,
 \]
 \[
  h_o g_1 h_o ^{-1} = - g_1, \ \
 h_o g_2 h_o^{-1} = - g_2 \rangle
 \]
 is of order $32$ with  $R = {\mathbb Z}[i]$, $\lambda _1 (h_o) = e^{\frac{3 \pi i}{4}}$, $\lambda _2(h_o) = e^{ - \frac{\pi i}{4}}$,
\[
H_{{\mathbb Q}8}(8) = \langle g_1, g_2, h_o \ \ \vert \ \  g_1^2 = g_2^2 = h_o^4  = - I_2, \ \  g_2g_1 = - g_1g_2,
\]
\[
  h_o g_1 h_o^{-1} =  g_2, \ \
h_o g_2 h_o^{-1} =  g_1 \rangle
\]
is of order $32$ with  $R = {\mathbb Z} [i]$, $\lambda _1 (h_o) = e^{\frac{3 \pi i}{4}}$, $\lambda _2 (h_o) = e^{ - \frac{pi i}{4}}$,
\[
H_{{\mathbb Q}8}(9) = \langle g_1, g_2, h_o \ \ \vert \ \  g_1^2 = g_2 ^2 = - I_2, \ \  h_o ^4 = I_2, \ \  g_2g_1 = - g_1g_2,
\]
\[
 h_o g_1 h_o ^{-1} = g_2, \ \
h_o g_1 h_0 ^{-1} = g_2 \rangle
\]
is of order $32$ with  $R = {\mathbb Z}[i]$, $\lambda _1 (h_o) =i$, $\lambda _2 (h_o) =1$.

There exist subgroups
\[
H_{{\mathbb Q}8}(1), \ \  H_{{\mathbb Q}8}(2), \ \  H_{{\mathbb Q}8}(9)  < GL(2, {\mathbb Z}[i]), \ \
_{{\mathbb Q}8}(5) < GL(2, \mathcal{O}_{-3}),
 \]
 as well as  subgroups
 \[
 H_{{\mathbb Q}8}^{o}(4) < GL(2, {\mathbb Q}(\sqrt{-2})), \ \ H_{{\mathbb Q}8}^{o}(6) < GL(2, {\mathbb Q}( \sqrt{-3})),
 \]
 \[
 H_{{\mathbb Q}8}^{o}(3), \ \ H_{{\mathbb Q}8}^{o}(7), \ \ H_{{\mathbb Q}8}^{o}(8) < GL(2, {\mathbb Q}( \sqrt{2}, i)),
  \]
  such that  $H_{{\mathbb Q}8}^{o}(j) \simeq H_{{\mathbb Q}8}(j)$ for $j \in \{ 3, 4, 6, 7, 8 \}$.
\end{proposition}

\begin{proof}

According to Lemmas \ref{HKh} and \ref{StructureH}, the group $H = \langle g_1, g_2 \rangle \langle h_o \rangle$ with $\det(H) = \langle \det (h_o) \rangle \simeq {\mathbb C}_s$ is completely determined by the order $r$ of $h_o$ and the elements $x_j = h_o g_j h_o ^{-1} \in \langle g_1, g_2 \rangle$, $1 \leq j \leq 2$ of order $4$. Bearing in mind that $\langle g_1, g_2 \rangle ^{(4)} = \{ \pm g_1, \pm g_2, \pm g_1g_2 \}$, let us split the considerations into Case A with $x_j \in \{ \pm g_j \}$ for $1 \leq j \leq 2$, Case B with $h_o g_1 h_o^{-1} = g_2$, $h_o g_2 h_o ^{-1} = \varepsilon g_1$ for some $\varepsilon = \pm 1$ and Case C with $h_o g_1 h_o^{-1} = g_2$, $h_o g_2 h_o ^{-1} = \varepsilon g_1 g_2$ for some $\varepsilon = \pm 1$.

In the case A, let us distinguish between Case A1 with $x_j = h_o g_j h_o^{-1} = g_j$ for $\forall 1 \leq j \leq 2$ and  Case A2 with $x_k = h_o g_k h_o^{-1} = - g_k$ for some $k \in \{ 1, 2 \}$. Note that if $h_o g_j = g_j h_o$ for $\forall 1 \leq j \leq 2$ then $h_o \in H$ is a scalar matrix. Indeed, if $h_o$ has diagonal form
\[
D_o = S^{-1} h_o S =
\left( \begin{array}{cc}
\lambda _1  &  0  \\
0  &  \lambda _2
\end{array}  \right)
\]
for some $S \in GL(2, {\mathbb Q} ( \sqrt{-d}, \lambda _1))$ and
\[
D_j = S^{-1} g_j S =
 \left( \begin{array}{rr}
a_j  &  b_j  \\
c_j  & -a_j
\end{array}  \right) \in SL(2, {\mathbb Q} ( \sqrt{-d}, \lambda _1 )) \ \ \mbox{ for  } \ \ 1 \leq j \leq 2 \ \ \mbox{ then }
\]
\begin{equation}    \label{ConjugacyFormula}
D_o D_j D_o^{-1} = \left( \begin{array}{rr}
a_j  &  \frac{\lambda _1}{\lambda _2} b_j  \\
\mbox{   }  &  \mbox{  } \\
\frac{\lambda _2}{\lambda _1} c_j  &  -a_j
\end{array}  \right)
\end{equation}
coincides with $D_j$ if and only if
\[
\left| \begin{array}{c}
\left( \frac{\lambda _1}{\lambda _2} - 1 \right) b_j =0  \\
\left( \frac{\lambda _2}{\lambda _1} - 1 \right) c_j =0
\end{array}  \right. .
\]
The assumption $\lambda _1 (h_o) = \lambda _1 \neq \lambda _2 = \lambda _2 (h_o)$ implies $b_j=c_j=0$ for $\forall 1 \leq j \leq 2$, so that
\[
D_1 = \pm i \left(  \begin{array}{rr}
1  &  0  \\
0  &  -1
\end{array}  \right) \ \ \mbox{  and  } \ \ D_2 = \left(  \begin{array}{rr}
1  &  0  \\
0  &  -1
\end{array}  \right)
\]
are diagonal. In particular, $D_1$ commutes with $D_2$, contrary to $D_2D_1 = - D_1D_2$. Thus, in the Case A1 with $h_og_j = g_jh_o$ for $\forall 1 \leq j \leq 2$ the matrix $h_o \in H$ is to be scalar. By Propositions \ref{ListS2}, \ref{ListS4+}, \ref{ListS4-}, \ref{ListS6+}, \ref{ListS6-}, \ref{ListS3+}, \ref{ListS3-}, the scalar matrices $h_o \in GL(2,R) \setminus SL(2,R)$ are
\[
h_o = i I_2 \in GL(2, {\mathbb Z}[i]) \ \ \mbox{  of order } \ \ 4,
\]
\[
h_o = e^{ \pm \frac{2 \pi i}{3}}  I_2 \in GL(2, {\mathbb Z}[i]) \ \ \mbox{  of order } \ \ 3 \ \ \mbox{  and  }
\]
\[
h_o = e^{ \pm \frac{ \pi i}{3}} I_2 \in GL(2, {\mathbb Z}[i]) \ \ \mbox{  of order } \ \ 6.
\]
For any subgroup
\[
{\mathbb Q}_8 \simeq \langle g_1, g_2 \ \ \vert \ \ g_1^2 = g_2^2 = -I_2,  \ \ g_2g_1 = - g_1g_2 \rangle < SL(2, {\mathbb Z}[i])
\]
one obtains a group
\[
H_{Q8}(1) = \langle g_1, g_2, i I_2 \ \ \vert \ \  g_1^2 = g_2^2 = - I_2, \ \ g_2g_1 = - g_1 g_2 \rangle < GL(2, {\mathbb Z}[i])
\]
of order $16$. As far as $-I_2 \in H \cap SL(2, R)$, the group $H$ contains $e^{\frac{2 \pi i}{3}} I_2$ if and only if it contains
 $- e^{\frac{2 \pi i}{3}} I_2 = e^{ - \frac{ \pi i}{3}} I_2$.
 Since $\langle g_1, g_2 \rangle \cap \langle e^{\frac{2 \pi i}{3}} I_2 \rangle = \{ I_2 \}$, any finite group $H$ with $e^{\frac{ 2 \pi i}{3}} I_2 \in H$ is a subgroup of $GL(, \mathcal{O}_{-3})$ of the form
\[
H_{Q8}(5) = \langle g_1,g_2 \rangle \times \langle e^{\frac{ 2 \pi i}{2}} I_2 \rangle \simeq {\mathbb Q}_8 \times  {\mathbb C}_3.
\]
These deplete $H = [ H \cap SL(2, R)] \langle h_o \rangle = \langle g_1, g_2 \rangle \langle h_o \rangle \simeq {\mathbb Q}_8{\mathbb C}_s$ of Case A1.

In the Case A2, one can assume that $h_og_1h_o^{-1} = - g_1$. If $h_og_2h_o=g_2$ then $h_o (g_1g_2)h_o^{-1} = - g_1g_2$, so that there is no loss of generality in supposing $h_o g_2 h_o^{-1} = -g_2$. By Lemma \ref{HK346Necessary}(iv), $h_og_1h_o^{-1} = - g_1$ requires $\lambda _1 (h_o) = i e^{\frac{  \pi i}{s}}$, $\lambda ( h_o) = -i e^{\frac{ \pi i}{s}}$, whereas $\frac{\lambda _1(h_o)}{\lambda _2 (h_o)} +  1 = \frac{\lambda _2(h_o)}{\lambda _1(h_o)} + 1 = 0$.  If
\[
D_o = S^{-1} h_o S = \left(  \begin{array}{rr}
i e^{\frac{ \pi i}{s}}   &  0  \\
0  &  -i e^{\frac{ \pi i}{s}}
\end{array}  \right) \in GL(2, {\mathbb Q} ( \sqrt{-d}, i e^{\frac{\pi i}{s}} ))
\]
is a diagonal form of $h_o \in H$ and
\[
D_j = S^{-1} g_j S = \left(  \begin{array}{rr}
a_j  &  b_j  \\
c_j  & -a_j
\end{array}  \right) \in GL(2, {\mathbb Q}( \sqrt{-d}, i e^{\frac{\pi i}{s}})) \ \ \mbox{  for } \ \ 1 \leq j \leq 2,
\]
then $D_o D_jD_o^{-1} = - D_j$ for $1 \leq j \leq 2$ is equivalent to $a_1=a_2=0$. As a result, $b_j \neq 0$ and $c_j = - \frac{1}{b_j}$. Straightforwardly, $D_2D_1 = - D_1D_2$ amounts to $2a_1a_2 + b_1 c_2 + b_2 c_1 = 0$, whereas $\frac{b_2}{b_1} + \frac{b_1}{b_2} = 0$. Denoting $\beta := \frac{b_2}{b_1} \in {\mathbb Q} ( \sqrt{-d}, i e^{\frac{ \pi i}{s}} )$, one computes that $\beta = \pm i \in {\mathbb Q} ( \sqrt{-d}, i e^{\frac{ \pi i}{s}} )$. Then by Lemma \ref{DiagonalizationField} there follows $s=2$ with $d=1$ or $s=4$. For $s=2$ one has $\lambda _1(h_o) = -1$, $\lambda _2(h_o)=1$, so that $h_o \in H$ is of order $2$ and
\[
H = H_{Q8}(2) = \langle g_1, g_2, h_o \ \ \vert \ \  g_1^2 = g_2^2 = -I_2, \ \ h_o^2 = I_2, \ \
\]
\[
 g_2g_1 = - g_1 g_2, \ \ h_og_1h_o^{-1} = - g_1, \ \
h_og_2h_o^{-1} = - g_2 \rangle
\]
is a subgroup of $GL(2, R_{-1,f})$ of order $16$. Note that
\[
h_o = \left( \begin{array}{rr}
-1  &  0  \\
0  &  1
\end{array}  \right), \ \
g_1 = \left( \begin{array}{cc}
0  &  1  \\
-1  &  0
\end{array}  \right) \ \ \mbox{  and  } \ \
g_2 = \left( \begin{array}{rr}
0  &  i  \\
i  &  0
\end{array}  \right)
\]
generate a subgroup of $GL(2, {\mathbb Z}[i])$, isomorphic to $H_{Q8}(2)$.  In the case of $s=4$, the element $h_o \in H$ with eigenvalues $\lambda _1 (h_o) = e^{\frac{ 3 \pi i}{4}}$, $\lambda _2 (h_o) = e^{ - \frac{ \pi i}{4}}$ is of order $8$ and
\[
H = H_{Q8}(7) = \langle g_1, g_2, h_o \ \ \vert \ \ g_1^2 = g_2^2 = h_o^4 = - I_2, \ \ g_2g_1 = - g_1g_2
\]
\[
h_og_1h_o^{-1} = -g_1, \ \ h_og_2h_o^{-1} = - g_2 \rangle
\]
is a subgroup of $GL(2, {\mathbb Z}[]i)$ of order $32$. The matrices
\[
D_o = \left( \begin{array}{rr}
e^{\frac{ 3 \pi i}{4}}  & 0  \\
0  &  e^{ - \frac{ \pi i}{4}}
\end{array}  \right), \ \
D_1 = \left(  \begin{array}{rr}
0  &  1  \\
-1  &  0
\end{array}  \right), \ \
D_2 = \left( \begin{array}{rr}
0  &  i  \\
i  & 0
\end{array}  \right)
\]
generate a subgroup $H^{o} _{Q8}(7)$ of $GL(2, {\mathbb Q}( \sqrt{2},i))$, isomorphic to $H_{Q8}(7)$. That concludes the Case A.

In the Case B, let us observe that $h_og_1h_o^{-1} =g_2$ and $h_og_2h_o^{-1} = \varepsilon g_1$ imply $h_o^2 g_1 h_o^{-2} = \varepsilon g_1$ and $h_o^2 g_2h_o^{-2} = \varepsilon g_2$. If $h_o^2 \in H \cap SL(2, R)$ then $\det(h_o) = \lambda _1(h_o) \lambda _2 (h_o) = -1$. The matrices
\[
D_o = S^{-1} h_o S = \left( \begin{array}{cc}
\lambda _1 (h_o)   &  0  \\
0  &  \lambda _2 (h_o)
\end{array}  \right) \ \ \mbox{  and  } \ \
D_j = S^{-1} g_j S = \left( \begin{array}{rr}
a_j  &  b_j  \\
c_j  &  - a_j
\end{array}  \right)
\]
with $a_j^2 + b_jc_j = -1$, $2a_1a_2 + b_1c_2 + b_2c_1 =0$ satisfy $D_oD_1D_o^{-1} = D_2$ if and only if
\[
D_2 = \left( \begin{array}{rr}
a_1  &  - \lambda _1^2 (h_o) b_1  \\
- \frac{c_1}{\lambda _1^2(h_o)}  &  - a_1
\end{array}  \right).
\]
Then $D_o D_2 D_o^{-1} = \varepsilon D_1$ is equivalent to
\[
\left| \begin{array}{c}
( \varepsilon -1) a_1 = 0 \\
( \varepsilon - \lambda _1 ^4 (h_o)) b_1 = 0 \\
\left( \varepsilon - \frac{1}{\lambda _1^4(h_o)} \right) c_1 = 0
\end{array}  \right. .
\]
According to $\det(D_1) = 1 \neq 0$, there follows $( \varepsilon -1) ( \varepsilon - \lambda _1^4 (h_o)) =0$. In the case of $-1 = \varepsilon = \lambda _1^4(h_o)$, Proposition \ref{ListS2} implies that $R = \mathcal{O}_{-2}$, $h_o$ is of order $8$ and
\[
D_o = S^{-1} h_o S = \left( \begin{array}{cc}

e^{\frac{ \pi i}{4}}   &  0  \\
0  &  e^{\frac{ 3 \pi i}{4}}
\end{array}  \right) \in GL(2, {\mathbb Q} ( \sqrt{2}, i)).
\]
Moreover,
\[
D_1 = \left( \begin{array}{rr}
0  &  b_1  \\
- \frac{1}{b_1}  &  0
\end{array}  \right), \ \
D_2 = \left( \begin{array}{rr}
0  &  - i b_1  \\
- \frac{i}{b_1}   &  0
\end{array}  \right),
\]
so that the subgroup
\[
H_{Q8} (3) = \langle g_1, g_2, h_o \ \ \vert \ \ g_1^2 = g_2^2 = h_o^4 = - I_2, \ \ g_2g_1 = -g_1 g_2,
\]
\[
h_og_1h_o^{-1} = g_2, \ \ h_og_2h_o^{-1} = - g_1 \rangle < GL(2, \mathcal{O}_{-2})
\]
of order $16$ is conjugate to the subgroup
\[
H^{o}_{Q8}(3) = \langle D_o = \left( \begin{array}{rr}
e^{\frac{ \pi i}{4}}  &  0  \\
0  &  e^{\frac{ 3 \pi i}{4}}
\end{array}  \right), \ \
D_1 = \left( \begin{array}{rr}
0  &  1  \\
-1  &  0
\end{array}  \right),
\]
\[
D_2 = \left(   \begin{array}{rr}
0  &  -i  \\
-i  &  0
\end{array} \right) \rangle < GL(2, {\mathbb Q} ( \sqrt{2},i)).
\]
For $\varepsilon =1$ and $\lambda _1 ^4 (h_o) \neq 1$ there follows
\[
D_2 = D_1 = \pm \left( \begin{array}{rr}
i  &  0  \\
0  &  -i
\end{array}  \right),
\]
which contradicts $D_2D_1 = - D_1D_2$. Therefore $\varepsilon =1$ implies $\lambda _1^4 (h_o) =1$ and
\[
D_o = \left( \begin{array}{rr}
1  &  0  \\
0  &  -1
\end{array}  \right)
\]
is of order $2$, since all $h_o \in H$ of order $4$ with $\det(h_o) =-1$ are scalar matrices and commute with $g_1,g_2$. In such a way, one obtains the group
\[
H_{Q8}(4) = \langle g_1, g_2, h_o \ \ \vert \ \ g_1^2 = g_2^2 = - I_2, \ \ h_o^2 = I_2, \ \ g_2g_1 = - g_1g_2,
\]
\[
h_og_1h_o^{-1} = g_2, \ \ h_og_2h_o^{-1} = g_1  \rangle
\]
of order $16$. The matrices
\[
D_1 = \left( \begin{array}{rr}
a_1  &  b_1  \\
c_1  &  -a_1
\end{array}  \right) \ \ \mbox{  and  } \ \
D_2 = \left(  \begin{array}{rr}
a_1  &  -b_1  \\
-c_1  &  -a_1
\end{array} \right)
\]
generate a subgroup of $GL(2, {\mathbb Q}( \sqrt{-d}))$, isomorphic to ${\mathbb Q}_8$ exactly when $a_1 = \pm \frac{\sqrt{-2}}{2} \in {\mathbb Q}( \sqrt{-d})$ and $c_1 = - \frac{1}{b_1}$ for some $b_1 \in {\mathbb Q}( \sqrt{-d})^*$. Therefore $H_{Q8}(4)$ occurs only as a subgroup of $GL(2, R_{-2,f})$ and
\[
D_o = \left( \begin{array}{rr}
1  &  0  \\
\mbox{  }  &  \mbox{  }  \\
0  &  -1
\end{array}  \right), \ \
D_1 = \left(  \begin{array}{rr}
\frac{\sqrt{-2}}{2}   &   1  \\
\mbox{  }  &  \mbox{  }  \\
- \frac{1}{2}  &  - \frac{\sqrt{-2}}{2}
\end{array}  \right), \ \
D_2 = \left(  \begin{array}{rr}
\frac{\sqrt{-2}}{2}  &  -1  \\
\mbox{  }  &  \mbox{  }  \\
\frac{1}{2}  &  - \frac{\sqrt{-2}}{2}
\end{array}  \right)
\]
generate a subgroup $H^{o}_{Q8}(4)$ of $GL(2, {\mathbb Q}( \sqrt{-2}))$, isomorphic to $H_{Q8}(4)$. That concludes the Case B with $h_o^2 \in H \cap SL(2, R)$.

Let us suppose that $h_og_1h_o^{-1} = g_2$, $h_og_2h_o^{-1} = \varepsilon g_1$ with $\det(h_o) \in R^*$ of order $s>2$. Note that $h_o^s \in H \cap SL(2,R) = \langle g_1, g_2 \rangle$ implies $h_o^s g_j h_o^{-s} \in \{ \pm g_j \}$ for $\forall 1 \leq j \leq 2$, so that $s \in \{ 4, 6 \}$ has to be an even natural number. The group
\[
H'= \langle g_1, g_2, h_o^2 \ \ \vert \ \  g_1^2 = g_2^2 = - I_2, \ \ h_o^r = I_2, g_2g_1 = - g_1 g_2,
\]
\[
h_o^2 g_1h_o^{-2} = \varepsilon g_1, \ \ h_o^2 g_2 h_o^{-2} = \varepsilon g_2 \rangle
\]
with $h_o^2 \in GL(2,R) \setminus SL(2,R)$, $H'\cap SL(2,R) = \langle g_1, g_2 \rangle \simeq {\mathbb Q}_8$ is of order $8 \frac{s}{2} \in \{ 16, 24 \}$ and satisfies the assumptions of Case A. Thus, for $\varepsilon =1$ one has $h_o^2 = i I_2$ or $h_o^2 = e^{\frac{ 2 \pi i}{3}} I_2$. If $h_o^2 = iI_2$ then $h_o \in H$ is of order $8$ with $\det(h_o) = \pm i$. Therefore $R = {\mathbb Z}[i]$ and $h_o$ has eigenvalues $\lambda _1 (h_o) = e^{\frac{ 3 \pi i}{4}}$, $\lambda _2 (h_o) = e^{ - \frac{ \pi i}{4}}$ with $\frac{ \lambda _1 (h_o)}{\lambda _2 (h_o)} = \frac{ \lambda _2 (h_o)}{\lambda _1 (h_o)} = -1$. The relations $D_oD_1D_o^{-1} = D_2$, $D_oD_2D_o^{-1} = D_1$ on  the diagonal form $D_o$ of $h_o$ hold for
\[
D_1 = \left(  \begin{array}{rr}
a_1  &  b_1  \\
c_1  & -a_1
\end{array}  \right), \ \
D_2 = \left( \begin{array}{rr}
a_1  &  -b_1  \\
-c_1  &  -a_1
\end{array}  \right) \in SL(2, {\mathbb Q} ( \sqrt{2},i)).
\]
The group $\langle D_1, D_2 \rangle$ is isomorphic to ${\mathbb Q}_8$ if and only if $a_1 = \pm \frac{\sqrt{-2}}{2}$ and $c_1 = - \frac{1}{b_1}$ for some $b_1 \in {\mathbb Q}( \sqrt{2},i)$. In such a way, one obtains the group
\[
H_{Q8}(8) = \langle g_1, g_2, h_o \ \ \vert g_1^2 = g_2^2 = h_o^4 = - I_2, \ \ g_2g_1 = - g_1g_2,
\]
\[
h_og_1h_o^{-1} = g_2, \ \ h_og_2h_o^{-1} = g_1 \rangle
\]
for $R = {\mathbb Z}[i]$. Note that $H_{Q8}(8)$ is of order $32$ and has a conjugate $H^{o}_{Q8}(8) = \langle D_1, D_2, D_o \rangle < GL(2, {\mathbb Q}(\sqrt{2},i))$. If $h_o^2 = e^{\frac{ 2 \pi i}{3}} I_2$ then $R = \mathcal{O}_{-3}$ and $h_o \in H$ is of order $6$ with $\det(h_o) = e^{ \pm \frac{ 2 \pi i}{3}}$. According to $h_og_1h_o^{-1} = g_2 \neq g_1$, $h_o$ is not a scalar matrix, so that $\lambda _1 (h_o) = e^{ - \frac{ \pi i}{3}}$, $\lambda _2 (h_o) =-1$ for $\det(h_o) = e^{ \frac{2 \pi i}{3}}$. Now, $D_oD_1D_o^{-1} =D_2$ is tantamount to
\[
D_2 = \left(   \begin{array}{cc}
a_1  &  e^{\frac{ 2 \pi i}{3}} b_1  \\
\mbox{  }  &  \mbox{  }  \\
e^{ - \frac{ 2 \pi i}{3}} c_1  & -a_1
\end{array}   \right)
\]
and $D_oD_2D_o^{-1} = D_1$ reduces to
\[
\left| \begin{array}{c}
\left( 1 - e^{ - \frac{ 2 \pi i}{3}} \right) b_1 = 0 \\
\left( 1 - e^{\frac{ 2 \pi i}{3}} \right) c_1 =0
\end{array}  \right. .
\]
As a result, $b_1=c_1$ and
\[
D_1 = D_2 = \pm \left( \begin{array}{rr}
i  &  0  \\
0  &  -i
\end{array}  \right)
\]
commute with each other. Thus, there is no group $H$ of Case B with $h_o^2 = e^{\frac{2 \pi i}{3}} I_2$. If $h_o g_1 h_o^{-1} = g_2$, $ h_o g_2 h_o^{-1} = - g_1$ and $h_o^2 \not \in \langle g_1, g_2 \rangle$ then
\[
H'= \langle g_1, g_2, h_o^2 \ \ \vert \ \ g_1^2 = g_2^2 = -I_2, \ \ h_o^r = I_2, \ \ g_2g_1 = - g_1g_2,
\]
\[
h_o^2 g_1 h_o^{-2} = -g_1, \ \ h_o^2 g_2 h_o^{-2} = - g_2 \rangle
\]
is isomorphic to $H_{Q8}(2)$ or $H_{Q8}(7)$, according to the considerations for Case A. More precisely,  if $H'\simeq H_{Q8}(2)$ then $h_o$ of order $4$ has $\det(h_o) = \pm i$ and $R = {\mathbb Z}[i]$. Due to $-I_2 \in \langle g_1, g_2 \rangle$, one can assume that
\[
D_o = \left( \begin{array}{rr}
i   &  0  \\
0  &  1
\end{array} \right).
\]
Then $D_oD_1D_o^{-1} =D_2$ requires
\[
D_2 = \left( \begin{array}{rr}
a_1  &  ib_1  \\
- ic_1  &  -a_1
\end{array}  \right),
\]
so that $D_oD_2D_o^{-1} = - D_1$ results in $a_1=0$. Bearing in mind that $\det(D_1) = \det(D_2) =1$, one concludes that
\[
D_1 = \left( \begin{array}{rr}
0  &  b_1  \\
- \frac{1}{b_1}  &  0
\end{array}  \right), \ \
D_2 = \left( \begin{array}{rr}
0  &  i b_1  \\
\frac{i}{b_1}  &  0
\end{array}  \right).
\]
For $b_1=1$, one obtains a subgroup $\langle D_1, D_2, D_o \rangle$ of $GL(2, {\mathbb Z}[i])$, isomorphic to
\[
H_{Q8}(9) = \langle g_1, g_2, h_o \ \ \vert \ \ g_1^2 = g_2^2 = - I_2, \ \ h_o^4 = I_2, \ \ g_2g_1 = - g_1g_2,
\]
\[ h_o g_1 h_o^{-1} = g_2, \ \ h_o g_2 h_o^{-1} = - g_1 \rangle < GL(2, {\mathbb Z}[i]).
\]
Since $\det(h_o) =i$ is of order $s=4$, the group $H_{Q8}(9)$ is of order $32$.  If $H'= \langle g_1, g_2, h_o^2 \rangle \simeq H_{Q8}(7)$ then $h_o \in H$ is to be of order $16$, since $h_o^2$ is of order $8$. The lack of $h_o \in GL(2, R)$ of order $16$ reveals that the groups $H_{Q8}(3)$, $H_{Q8}(4)$, $H_{Q8}(8)$, $H_{Q8}(9)$ deplete Case B.

There remains to be considered Case C with $h_og_1 h_o^{-1} =g_2$, $h_o g_2 h_o^{-1} = \varepsilon g_1g_2$, $h_o ( g_1g_2) h_o^{-1} = \varepsilon g_1$ for some $\varepsilon = \pm 1$. Note that $h_o^2 g_1 h_o^{-2} = \varepsilon g_1 g_2$, $h_o^2 g_2 h_o^{-2} = g_1$, $h_o^3 g_1 h_o^{-3} = g_1$, $h_o^3 g_2 h_o^{-3} = g_2$ require the divisibility of $s$ by $3$, as far as $\langle g_j \rangle$ are normal subgroups of $\langle g_1, g_2 \rangle$ and $h_o^s \in \langle g_1, g_2 \rangle$. In other words, $s \in \{ 3, 6 \}$ and $R = \mathcal{O}_{-3}$. The non-scalar matrices $h_o \in GL(2, \mathcal{O}_{-3})$ with $\det(h_o) = e^{\frac{ 2 \pi i}{3}}$ have eigenvalues $\{ \lambda _1 (h_o), \lambda _2 (h_o) \} = \left \{ e^{\frac{ 2 \pi i}{3}}, 1 \right \}$, $\left \{ e^{ - \frac{\pi i}{3}}, -1 \right \}$ or $\left \{ e^{\frac{ 5 \pi i}{6}}, e^{ - \frac{ \pi i}{6}} \right \}$. If $h_o$ is of order $3$ or $6$ then $\frac{\lambda _1(h_o)}{\lambda _2(h_o)} = e^{\frac{ 2 \pi i}{3}}$ and $D_o D_1 D_o^{-1} = D_2$ specifies that
\[
D_2 = \left( \begin{array}{rr}
a_1   &  e^{\frac{ 2 \pi i}{3}}  b_1   \\
e^{- \frac{2 \pi i}{3}}  c_1  &  -a_1
\end{array}  \right).
\]
Now, $2a_1a_2 + b_1c_2 + b_2c_1 =0$ reduces to $2a_1^2 = b_1c_1$ and $a_1^2 + b_1c_1 = -1$ requires $a_1 = \pm \frac{ -3}{3}$, $c_1 = - \frac{2}{3b_1}$ for some $b_1 \in {\mathbb Q} (\sqrt{-3})^*$. Replacing, eventually, $D_j$ by $D_j^3$, one has
\[
D_1 = \left( \begin{array}{rr}
\frac{\sqrt{-3}}{3}   &  b_1  \\
\mbox{  }   &  \mbox{  } \\
- \frac{2}{3b_1}   &  - \frac{\sqrt{-3}}{3}
\end{array}  \right), \ \
D_2 = \left( \begin{array}{rr}
\frac{\sqrt{-3}}{3}   &  e^{\frac{ 2 \pi i}{3}} b_1  \\
\mbox{  }   &  \mbox{  } \\
- \frac{2 e^{ - \frac{ 2 \pi i}{3}}}{3 b_1}  &  - \frac{\sqrt{-3}}{3}
\end{array}  \right).
\]
Now,
\[
D_1D_2 = \left( \begin{array}{rrr}
\frac{\sqrt{-3}}{3} &  \mbox{  }  &  e^{ - \frac{2 \pi i}{3}} b_1  \\
\mbox{  } &  \mbox{  }  &  \mbox{  } \\
- \frac{ 2 e^{\frac{ 2 \pi i}{3}}}{3b_1} &  \mbox{  }  &  - \frac{\sqrt{-3}}{3}
\end{array}  \right)
\]
and $D_o D_2 D_o^{-1} = \varepsilon D_1D_2$ holds for $\varepsilon =1$. Thus,
\[
H^{o}_{Q8}(6) = \langle D_1, D_2, D_o \rangle < GL(2, {\mathbb Q} ( \sqrt{-3}))
\]
 is conjugate to
\[
H_{Q8}(6) = \langle g_1, g_2, h_o \ \ \vert \ \ g_1^2 = g_2^2 = - I_2, \ \ h_o^3 = I_2, \ \ g_2g_1 = - g_1g_2
\]
\[
h_o g_1 h_o^{-1} = g_2, \ \ h_o g_2 h_o^{-1} = g_1g_2 \rangle < GL(2, \mathcal{O}_{-3})
\]
of order $24$ with $\lambda _1(h_o) = e^{\frac{2 \pi i}{3}}$, $\lambda _2 (h_o) =1$ or to
\begin{equation}     \label{Repetition}
H = \langle g_1, g_2, h_o \ \ \vert \ \  g_1^2 = g_2^2 = - I_2, \ \ h_o^3 = - I_2,   g_2g_1 = - g_1 g_2,
\end{equation}
\[
 h_o g_1 h_o^{-1} = g_2, \ \ h_o g_2 h_o^{-1} = g_1 g_2 \rangle < GL(2, \mathcal{O}_{-3})
\]
of order $24$ with $\lambda _1 (h_o) = e^{ - \frac{ \pi i}{3}}$, $\lambda _2 ( h_o) = -1$. Due to $ - I_2 \in \langle g_1, g_2 \rangle$, the presence of $h_o \in H$ of order $6$ with $\det(H) = \langle \det(h_o) \rangle \simeq {\mathbb C}_3$ is equivalent to the existence of $- h_o \in H$ of order $3$ with $\det(H) = \langle \det( - h_o) \rangle \simeq {\mathbb C}_3$ and $H$ from (\ref{Repetition}) is isomorphic to $H_{Q8}(6)$. If $h_o$ has diagonal form
\[
D_o = \left( \begin{array}{cc}
e^{\frac{ 5 \pi i}{6}}  &  0  \\
0  &  e^{ - \frac{ \pi i}{6}}
\end{array}  \right) \in GL(2, {\mathbb Q}( \sqrt{-3}))
\]
of order $12$ with $\det(D_o) = e^{\frac{ 2 \pi i}{3}}$, $\frac{\lambda _1 (h_o)}{\lambda _2 (h_o)} = \frac{ \lambda _2 (h_o)}{\lambda _1(h_o)} = -1$, then $D_o D_1 D_o^{-1} = D_2$ implies that
\[
D_2 = \left( \begin{array}{rr}
a_1  & -b_1    \\
-c_1  &  a_1
\end{array} \right)
\]
with $a_1^2 = b_1c_1 = - \frac{1}{2}$. Therefore, $a_1 = \pm \frac{\sqrt{-2}}{2} \in GL(2, {\mathbb Q} ( \sqrt{-3}))$, which is an absurd. If $h_o g_1 h_o^{-1} = g_2$, $h_o g_2 h_o^{-1} = \varepsilon g_1 g_2 $ and $s=6$ then $h_o \in H$ is of order $6$, according to Proposition \ref{ListS6+}. Now $H'' = \langle g_1, g_2, h_o^3 \rangle < GL(2, R)$ with $h_o^3 \not \in \langle g_1, g_2 \rangle$ is subject to Case A with a scalar matrix $h_o \in H$, according to $h_o^3 g_1 h_o^{-3} = g_1$, $h_o^3 g_2 h_o^{-3} = g_2$. If $h_o^3 = iI_2$ then $h_o$ is of order $r=12$. The assumption $h_o^3 = e^{\frac{ 2 \pi i}{3}} I_2$ holds for $h_o$ of order $r=9$. Both contradict to $r=6$ and establish that any subgroup $H < GL(2, R)$ with $H \cap SL(2, R) \simeq {\mathbb Q}_8$ is isomorphic to $H_{Q8} (i)$ for some $1 \leq i \leq 9$.

\end{proof}

\begin{proposition}    \label{HQ12}
Let $H$ be a finite subgroup of $GL(2,R)$,
\[
H \cap SL(2, R) = K_7 = \langle g_1, g_4, \ \ g_1^2 = g_4^3 = - I_2, \ \ g_1 g_4 g_1 ^{-1} = g_4^{-1} \rangle \simeq {\mathbb Q}_{12}
\]
and $h_o \in H$ be an element of order $r$ with $\det(H) = \langle \det( h_o ) \rangle \simeq {\mathbb C}_s$ and eigenvalues $\lambda _1 (h_o)$, $\lambda _2 (h_o)$. Then $H$ is isomorphic to $H_{Q12}(i)$ for some $1 \leq i \leq 10$, where
\[
H_{Q12} (1) = \langle g_1, g_4, h_o = i I_2 \ \ \vert \ \ g_1^2 = g_4^3 = - I_2, \ \ g_1 g_4 g_1^{-1} = g_4^{-1} \rangle
\]
is of order $24$ with $R - {\mathbb Z}[i]$,
\[
H_{Q12}(2) = \langle g_1, g_4, h_o \ \ \vert \ \  g_1^2 = g_4^3 = - I_2, \ \ h_o^6 = I_2, \ \ g_1g_4g_1^{-1} = g_4^{-1}, \ \
\]
\[
h_o g_1 h_o^{-1}  = g_1 g_4, \ \ h_o g_4 h_o^{-1} = g_4 \rangle
\]
of order $24$, with $R = \mathcal{O}_{-3}$, $\lambda _1 (h_o) = e^{\frac{ \pi i}{3}}$, $\lambda _2 (h_o) = e^{\frac{ 2 \pi i}{3}}$,
\[
H_{Q12}(3) = \langle g_1, g_4, h_o \ \ \vert \ \  g_1^2 = g_4^3 = h_o^6 = - I_2, g_1g_4 g_1^{-1} = g_4^{-1},
\]
\[
h_o g_1 h_o^{-1} = g_1 g_4^2, \ \ h_o g_4 h_o^{-1} = g_4 \rangle
\]
is of order $24$ with $R = \mathcal{O}_{-3}$, $\lambda _1 (h_o) = e^{\frac{ \pi i}{6}}$, $\lambda _2 (h_o) = e^{\frac{ 5 \pi i}{6}}$,
\[
H_{Q12} (4) = \langle g_1, g_4, h_o \ \ \vert \ \ g_1^2 = g_4^3 = - I_2, \ \ h_o ^2 = I_2, \ \ g_1 g_4 g_1^{-1} = g_4^{-1},
\]
\[
h_o g_1 h_o^{-1} = - g_1, \ \ h_o g_4 h_o^{-1} = g_4 \rangle
\]
is of order $24$ with $\lambda _1 (h_o) = -1$, $\lambda _2(h_o) =1$,
\[
H_{Q12}(5) = \langle g_1, g_4, h_o = e^{\frac{ 2 \pi i}{3}} I_2 \ \ \vert \ \  g_1^2 = g_4 ^3 = - I_2, \ \ g_1 g_4 g_1^{-1} = g_4^{-1} \rangle
\]
is of order $36$ with $R = \mathcal{O}_{-3}$,
\[
H_{Q12}(6) = \langle g_1, g_4, h_o \ \ \vert \ \ g_1^2 = g_4^3 = - I_2, \ \ h_o^3 = I_2, \ \ g_1 g_4 g_1 ^{-1} = g_4 ^{-1},
\]
\[
h_o g_1 h_o ^{-1} g_1 g_4^2, \ \ h_o g_4 h_o^{-1} = g_4 \rangle
\]
is of order $36$ with $R = \mathcal{O}_{-3}$, $\lambda _1 (h_o) = e^{\frac{ 2 \pi i}{3}}$, $\lambda _2 (h_o) = 1$,
\[
H_{Q12} (7) = \langle g_1, g_4, h_o \ \ \vert \ \  g_1^2 = g_4^3 = h_o^6 = - I_2, \ \ g_1 g_4 g_1^{-1} = g_4 ^{-1},
\]
\[
h_o g_1 h_o^{-1} = - g_1, \ \ h_o g_4 h_o ^{-1} = g_4 \rangle
\]
is of order $36$ with $R = \mathcal{O}_{-3}$, $\lambda _1 (h_o) = e^{ - \frac{ \pi i}{6}}$, $\lambda _2 (h_o) = e^{\frac{ 5 \pi i}{6}}$,
\[
H_{Q12}(8) = \langle g_1, g_4, h_o \ \ \vert \ \ g_1^2 = g_4^3 = h_o^4 = - I_2, \ \ g_1 g_4 g_1^{-1} = g_4 ^{-1},
\]
\[
h_o g_1 h_o^{-1} = - g_1, \ \ h_o g_4 h_o^{-1} = g_4 \rangle
\]
is of order $48$ with $R = {\mathbb Z}[i]$, $\lambda _1 (h_o) = e^{\frac{ 3 pi i}{4}}$, $\lambda _2 (h_o) = e^{ - \frac{ \pi i}{4}}$,
\[
H_{Q12} (9) = \langle g_1, g_4, h_o \ \ \vert \ \ g_1^2 = g_4 ^3 = - I_2, \ \ h_o^6 = I_2, \ \ g_1 g_4 g_1 ^{-1} = g_4^{-1},
\]
\[
h_o g_1 h_o^{-1} =    g_1g_4, \ \ h_o g_4 h_o^{-1} = g_4 \rangle
\]
is of order $72$ with $R = \mathcal{O}_{-3}$, $\lambda _1 (h_o) = 1$, $\lambda _2 ( h_o) = e^{\frac{ \pi i}{3}}$,
\[
H_{Q12} (10) = \langle g_1, g_4, h_o \ \ \vert \ \  g_1^2 = g_4^3 = - I_2, \ \ h_o^6 = I_2, \ \ g_1 g_4 g_1^{-1} = g_4^{-1},
\]
\[
h_o g_1 h_o^{-1} = - g_1, \ \ h_o g_4 h_o^{-1} = g_4  \rangle
\]
is of order $72$ with $R = \mathcal{O}_{-3}$, $\lambda _1 (h_o) = e^{\frac{ 2 \pi i}{3}}$, $\lambda _2 (h_o) = e^{ - \frac{ \pi i}{3}}$.

There exist subgroups
\[
H_{Q12}(2), H_{Q12}(4), H_{Q12}(5), H_{Q12}(6), H_{Q12}(9), H_{Q12}(10)  < GL(2, \mathcal{O}_{-3}),
\]
as well as subgroups
\[
H^{o} _{Q12}(1), H^{o} _{Q12}(3), H^{o} _{Q12}(7) < GL(2, {\mathbb Q}( \sqrt{3},i)), \ \
H^{o} _{Q12}(8) < GL(2, {\mathbb Q} ( \sqrt{2}, \sqrt{3}, i))
\]
with $H^{o} _{Q12}(j) \simeq H_{Q12}(j)$ for $j \in \{ 1, 3, 7, 8 \}$.
\end{proposition}

\begin{proof}

According to Lemma \ref{StructureH}, the groups $H = K_7 \langle h_o \rangle$ with $\det(H) = \langle \det(h_o) \rangle \simeq {\mathbb C}_s$ are determined up to an isomorphism by the order $r$ of $h_o$, the element $h_o g_1 h_o^{-1} \in K_7$ of order $4$ and the element $h_o g_4 h_o^{-1} \in K_7$ of order $6$. Let us denote by $K_7^{(m)}$ the set of the elements of $K_7$ of order $m$. Straightforwardly,
\[
K_7^{(6)} = \{ g_4, g_4^{-1} \}, \ \ K_7^{(4)} = \{ \pm g_1 g_4 о \ \ \vert \ \  0 \leq i \leq 3 \}.
\]
Inverting $g_1 g_4 g_1 ^{-1} = g_4^{-1}$, one obtains $g_1 g_4^{-1} g_1 ^{-1} = g_4$. If $h_o g_4 h_o^{-1} = g_4^{-1}$ then
\[
(g_1h_o) g_4 (g_1 h_o^{-1} = g_1 ( h_o g_4 h_o^{-1} ) g_1^{-1} = g_1 g_4^{-1} g_1 ^{-1} = g_4.
\]
As far as $K_7 = \langle g_1, g_4, h_o \rangle = \langle  g_1, g_4, g_1 h_o \rangle$, there is no loss of generality in assuming $h_o g_4 h_o^{-1} = g_4$.

We start the study of $H$ by a realization of $K_7$ as a subgroup of the special linear group  $SL(2, {\mathbb Q} ( \sqrt{-d}, \sqrt{-3}))$. Let
\[
D_4 = S^{-1} g_4 S = \left( \begin{array}{cc}
e^{\frac{ \pi i}{3}}  &  0  \\
0  &  e^{ - \frac{ \pi i}{3}}
\end{array}  \right)
\]
be a diagonal form of $g_4$ for some $S \in GL(2, {\mathbb Q} ( \sqrt{-d}, \sqrt{-3}))$ and
\[
D_1 = S^{-1} g_1 S = \left( \begin{array}{rr}
a_1  &  b_1   \\
c_1  & -a_1
\end{array}   \right) \ \ \mbox{ with  } \ \ a_1^2 + b_1c_1 = -1.
\]
Then
\[
D_1 D_4 D_1^{-1} = \left( \begin{array}{rr}
- \sqrt{-3} a_1^2 + e^{ - \frac{ \pi i}{3}}  &  - \sqrt{-3} a_1 b_1  \\
\mbox{  }   &  \mbox{  } \\
- \sqrt{-3} a_1c_1   &  \sqrt{-3} a_1^2 + E^{\frac{ \pi i}{3}}
\end{array}  \right) \in SL(2, {\mathbb Q}( \sqrt{-d}, \sqrt{-3}))
\]
coincides with $D_4^{-1}$ if and only if
\[
D_1 = \left(  \begin{array}{rr}
0  &  b_1  \\
- b_1^{-1}  &  0
\end{array}  \right) \ \ \mbox{  for some  } \ \ b_1 \in {\mathbb Q}( \sqrt{-d}, \sqrt{-3})^*.
\]
That allows to compute explicitly
\[
K_7^{(4)} = \left \{ \pm D_1 = \pm \left( \begin{array}{rr}
0  &  b_1  \\
- b_1^{-1}  &  0
\end{array}  \right), \ \ \pm D_1D_4 = \pm \left( \begin{array}{cc}
0  &  e^{- \frac{ \pi i}{3}} b_1  \\
- \left( e^{ - \frac{ \pi i}{3}} b_1 \right) ^{-1}  &  0
\end{array}  \right),  \right.
\]
\[
\left. \pm D_1 D_4^2 = \pm \left(  \begin{array}{cc}
0  &  e^{ - \frac{ 2 \pi i}{3}} b_1  \\
- \left( e^{ - \frac{ 2 \pi i}{3}}  b_1 \right)^{-1}  &  0
\end{array}  \right) \right \},
\]
 \[
K_7^{(4)} = \left \{ D_1D_4^j  = \left( \begin{array}{rr}
0  &  e^{ - \frac{ j \pi i}{3}} b_1  \\
- \left( e^{ - \frac{ j \pi i}{3}} b_1 \right) ^{-1}   & 0
\end{array}  \right) \ \ \Bigg \vert \ \ 0 \leq j \leq 5 \right \}.
\]
Now, $D_o D_4 D_o^{-1} = D_4$ amounts to
\[
D_o = \left( \begin{array}{cc}
\lambda _1 (h_o)   & 0  \\
0  &  \lambda _2 (h_o)
\end{array}  \right) \ \ \mbox{  and   }
\]
\[
D_o D_1D_o^{-1} = \left( \begin{array}{cc}
0  &  \frac{ \lambda _1 (h_o)}{\lambda _2 (h_o)} b_1  \\
\mbox{  }   &  \mbox{  }  \\
- \left[ \frac{ \lambda _1 (h_o)}{\lambda _2 (h_o)} b_1 \right] ^{-1}   &  0
\end{array}  \right) =
\left( \begin{array}{rr}
0  &  e^{ - \frac{ j \pi i}{3}} b_1  \\
- \left( e^{ - \frac{ j \pi i}{3}} b_1 \right) ^{-1}  & 0
\end{array}  \right) = D_1 D_4^j
\]
if and only if $\frac{\lambda _1 (h_o)}{\lambda _2 (h_o)} = e^{ - \frac{ j \pi i}{3}}$. Note that the ratio $\frac{ \lambda _1 (h_o)}{\lambda _2(h_o)}$ of the eigenvalues of $h_o$ is determined up to an inversion and
\[
\left \{ e^{ - \frac{ j \pi i}{3}} \ \ \vert \ \  0 \leq j \leq 5 \right \} = \left \{ 1 = e^0, \ \ e^{ \mp \frac{ j \pi i}{3}}, \ \ -1 = e^{ \pi i} \ \ \vert  \ \ 1 \leq j \leq 2 \right \}.
\]
For any $h_o \in H$ with $\frac{ \lambda _1 (h_o)}{\lambda _2 (h_o)} = e^{ \mp \frac{ j \pi i}{3}}$, $0 \leq j \leq 3$ the group
\[
H = \langle g_1, g_4, h_o \ \ \vert \ \ g_1 ^2 = g_4^3 = - I_2, \ \ h_o^r = I_2, \ \ g_1 g_4 g_1 ^{-1} = g_4^{-1},
\]
\[
h_o g_1 h_o^{-1} = g_1 g_4^j, \ \ h_o g_4 h_o^{-1} = g_4 \rangle.
\]
Note that  $\frac{\lambda _1 (h_o)}{\lambda _2 (h_o)} =1$ exactly when $h_o \in H \setminus SL(2,R)$ is a scalar matrix. According to Propositions \ref{ListS2}, \ref{ListS4+}, \ref{ListS4-}, \ref{ListS6+}, \ref{ListS6-}, \ref{ListS3+}, \ref{ListS3-}, the only scalar matrices $h_o \in GL(2,R) \setminus SL(2, R)$ are $h_o = \pm i I_2$ for $R = {\mathbb Z} [i]$ and $h_o = e^{ \pm \frac{ 2 \pi i}{3}} I_2$ or $e^{ \pm \frac{ \pi i}{3}} I_2$ with $R = \mathcal{O}_{-3}$. Replacing, eventually, $h_o = -iI_2$ by $h_o^{-1} = i I_2$, one obtains the group $H_{Q12} (1) = \langle g_1, g_4, i I_2 \rangle$ with $R = {\mathbb Z}[i]$. Note that $H^{o} _{Q12}(1) = \langle D_1, D_4, h_o = i I_2 \rangle$ is a realization of $H_{Q12} (1)$ as a subgroup of $GL(2, {\mathbb Q} ( \sqrt{3},i))$. Bearing in mind that $ - I_2 \in K_7$, one observes that $e^{ - \frac{ \pi i}{3}} I_2 \in H$ if and only if $ - e^{ - \frac{ \pi i}{3}} I_2 = e^{\frac{ 2 \pi i}{3}} I_2 \in H$. Replacing, eventually, $e^{\frac{ \pi i}{3}} I_2$ and $e^{ - \frac{ 2 \pi i}{3}} I_2$ by their inverse matrices, one observes that $h_o = e^{\frac{ 2 \pi i}{3}} I_2 \in H$ whenever $H$ contains a scalar matrix of order $3$ or $6$. That provides the group $H_{Q12} (5) = \langle g_1, g_4, e^{\frac{ 2 \pi i}{3}} I_2 \rangle$. Note that
\[
\langle D_1 = \left( \begin{array}{rr}
0  &  1 \\
-1  &  0
\end{array}  \right), \ \
D_4 = \left( \begin{array}{cc}
e^{\frac{ \pi i}{3}} & 0  \\
0  &  E^{ - \frac{\pi i}{3}}
\end{array}  \right), \ \
D_o = e^{\frac{ 2 \pi i}{3}} I_2 \rangle < GL(2, \mathcal{O}_{-3})
\]
is a realization of $H_{Q12}(5)$ as a subgroup of $GL(2, \mathcal{O}_{-3})$.

For $\frac{ \lambda _1 (h_o)}{\lambda _2 (h_o)} = e^{ \mp \frac{ \pi i}{3}}$, Corollary  \ref{RatioonGLEigenvalues} specifies that either $R= \mathcal{O}_{-3}$, $s=2$, $r=6$, $\lambda _1 (h_o) = e^{\frac{ \pi i}{3}}$, $\lambda _2 (h_o) = e^{\frac{ 2 \pi i}{3}}$ and $H \simeq H_{Q12}(2)$ or $R = \mathcal{O}_{-3}$, $s=6$, $r=6$, $\lambda _1 (h_o) = \varepsilon e^{\frac{ \eta \pi i}{3}}$, $\lambda _2 (h_o) = \varepsilon$. In the second case, one can restrict to $\varepsilon =1$, due to $-I_2 \in K_7 \subset H$. The corresponding group $H \simeq H_{Q12}(9)$. Both, $H_{Q12}(2)$ and $H_{Q12}(9)$ can be realized as subgroups of $GL(2, \mathcal{O}_{-3})$, setting
\[
g_1 = \left( \begin{array}{rr}
0  &  1  \\
-1  &  0
\end{array}  \right), \ \
g_4 = \left( \begin{array}{cc}
e^{\frac{ \pi i}{3}}  &  0  \\
0  &  e^{ - \frac{ \pi i}{3}}
\end{array}  \right),
\]
\[
h_o = \left( \begin{array}{cc}
e^{\frac{ \pi i}{3}}   & 0  \\
0  &  e^{\frac{ 2 \pi i}{3}}
\end{array}  \right) \ \ \mbox{  or, respectively, } \ \
h_o = \left( \begin{array}{cc}
e^{ - \frac{ \pi i}{3}}   & 0  \\
0  &  1
\end{array}  \right).
\]

If $\frac{ \lambda _1 (h_o)}{\lambda _2 (h_o)} = e^{ \mp \frac{ 2 \pi i}{3}}$ then, eventually, replacing $h_o$ by $h_o^{-1}$, one has $\lambda _1 (h_o) = e^{\frac{ \pi i}{6}}$, $\lambda _2 (h_o) = e^{\frac{ 5 \pi i}{6}}$, $s=2$, $r=12$, $R = {\mathbb Z}[i]$ and $H \simeq H_{Q12}(3)$ or $\lambda _1 (h_o) = \varepsilon$, $\lambda _2 (h_o) = \varepsilon e^{\frac{ 2 \pi i}{3}}$, $s=3$,  $R = \mathcal{O}_{-3}$, by Corollary \ref{RatioonGLEigenvalues}. Note that $-I_2 \in K_7 \subset H$ reduces the second case to $\lambda _1 (h_o) =1$, $\lambda _2 (h_o) = e^{\frac{ 2 \pi i}{3}}$, $s=3$, $r=3$, $R = \mathcal{O}_{-3}$ and $H \simeq H_{Q12}(6)$. Note that
\[
g_1 = \left( \begin{array}{rr}
0  &  1 \\
-1  &  0
\end{array}  \right), \ \
g_4 = \left( \begin{array}{cc}
e^{\frac{ \pi i}{3}}  &  0  \\
0  &  e^{ - \frac{ \pi i}{3}}
\end{array}  \right), \ \
h_o = \left( \begin{array}{cc}
1  &  0  \\
0  &  e^{\frac{2 \pi i}{3}}
\end{array}  \right)
\]
generate a subgroup of $GL(2, \mathcal{O}_{-3})$, isomorphic to $H_{Q12}(6)$.  In the case of $H \simeq H_{Q12}(3)$ the eigenvalues of $h_o$ are primitive twelfth roots of unity, so that
\[
D_1 = \left( \begin{array}{rr}
0  &   b_1  \\
- b_1^{-1}   &  0
\end{array}  \right), \ \
D_4 = \left( \begin{array}{cc}
e^{\frac{ \pi i}{3}}  &  0  \\
0  &  e^{ - \frac{ \pi i}{3}}
\end{array}  \right), \ \
D_o = \left( \begin{array}{cc}
e^{\frac{ \pi i}{6}}  & 0  \\
0  &  e^{\frac{ 5 \pi i}{6}}
\end{array}  \right)
\]
generate a subgroup $H^{o}_{Q12}(3) < GL(2, {\mathbb Q} (\sqrt{3},i))$, isomorphic to $H_{Q12}(3)$.

For $\frac{ \lambda _1 (h_o)}{\lambda _2 (h_o)} = -1$ there are four non-equivalent possibilities for the eigenvalues $\lambda _1 (h_o)$, $\lambda _2 (h_o)$ of $h_o$. The first one is $\lambda _1 (h_o) =1$, $\lambda _2 (h_o) = -1$ with $s=2$, $r=2$ for any $R = R_{-d,f}$ and $H \simeq H_{Q12}(4)$ of order $24$. Note that
\[
D_1 = \left( \begin{array}{rr}
0  &  1  \\
-1  & 0
\end{array}  \right), \ \
D_4 = \left( \begin{array}{cc}
e^{\frac{ \pi i}{3}}   &  0  \\
0  &  e^{ - \frac{ \pi i}{3}}
\end{array}  \right), \ \
h_o = \left( \begin{array}{rr}
1  &  0  \\
0  &  -1
\end{array} \right)
\]
realizes $H_{Q12}(4)$ as a subgroup of $GL(2, \mathcal{O}_{-3})$. The second one is $\lambda _1 (h_o) = e^{\frac{ 3 \pi i}{4}}$, $\lambda _2 (h_o) = E^{ - \frac{ \pi i}{4}}$ with $s=4$, $r=8$, $R = {\mathbb Z}[i]$ and $H \simeq H_{Q12}(8)$ of order $48$. Observe that
\[
D_1 = \left( \begin{array}{rr}
0   &  1  \\
-1  & 0
\end{array}  \right), \ \
D_4 = \left( \begin{array}{cc}
e^{\frac{ \pi i}{3}}  & 0  \\
0  &  e^{ - \frac{ \pi i}{3}}
\end{array}  \right), \ \
D_o = \left( \begin{array}{cc}
e^{\frac{ 3 \pi i}{4}}  &  0  \\
0  &  e^{ - \frac{ \pi i}{4}}
\end{array}  \right)
\]
generate a subgroup of $GL(2, {\mathbb Q} ( \sqrt{2}, \sqrt{3}, i ))$, isomorphic to $H_{Q12}(8)$. In the third case, $\lambda _1 (h_o) = e^{ - \frac{ \pi i}{6}}$, $\lambda _2 (h_o) = e^{ \frac{ 5 \pi i}{6}}$ with $s=3$, $r=12$, $R = \mathcal{O}_{-3}$ and $H \simeq H_{Q12}(7)$ of order $36$, realized by
\[
D_1 = \left( \begin{array}{rr}
0  &  1  \\
-1  &0
\end{array}  \right), \ \
D_4 = \left( \begin{array}{cc}
e^{\frac{ \pi i}{3}}  &  0  \\
0  &  e^{ - \frac{ \pi i}{3}}
\end{array}  \right), \ \
D_o = \left( \begin{array}{cc}
e^{ - \frac{ \pi i}{6}}   &  0  \\
0  &  e^{\frac{ 5 \pi i}{6}}
\end{array}  \right)
\]
as a subgroup of $GL(2, {\mathbb Q} ( \sqrt{3}, i))$. In the fourth case, $\lambda _1 (h_o) = e^{\frac{ 2 \pi i}{3}}$, $\lambda _2 (h_o) = e^{ - \frac{ \pi i}{3}}$ with $s=6$, $r=6$, $R = \mathcal{O}_{-3}$ and $H \simeq H_{Q12}(10)$ of order $72$. The matrices
\[
g_1 = \left( \begin{array}{rr}
0  &  1  \\
-1  & 0
\end{array}  \right), \ \
g_4 = \left( \begin{array}{cc}
e^{\frac{ \pi i}{3}}  & 0  \\
0  &  e^{ - \frac{\pi i}{3}}
\end{array}  \right), \ \
h_o = \left( \begin{array}{cc}
e^{\frac{ 2 \pi i}{3}}   & 0  \\
0  &  e^{ - \frac{ \pi i}{3}}
\end{array}  \right)
\]generate a subgroup of $GL(2, \mathcal{O}_{-3})$, isomorphic to $H_{Q12}(10)$. The groups $H_{Q12}(4)$, $H_{Q12}(7)$, $H_{Q12}(8)$, $H_{Q12}(10)$ with $\frac{ \lambda _1 (h_o)}{\lambda _2 (h_o)} = -1$ are non-isomorphic, as far as they are of different orders.

\end{proof}

\begin{proposition}  \label{HSL23}
Let $H$ be a finite subgroup of $GL(2,R)$,
\[
H \cap SL(2,R) = K_8 = \langle g_1, g_2, g_3 \ \ \vert \ \ g_1^2 = g_2^2 = -I_2, \ \ g_3^3 = I_2, \ \  g_2g_1 = - g_1g_2,
\]
\[
g_3g_1g_3^{-1} =  g_2, \ \ g_3 g_2 g_3^{-1} = g_1 g_2 \rangle \simeq SL(2, {\mathbb F}_3)
\]
and $h_o \in H$ be an element of order $r$ with $\det(H) = \langle \det(h_o) \rangle \simeq {\mathbb C}_s$ and eigenvalues $\lambda _1 (h_o)$, $\lambda _2 (h_o)$. Then $H$ is isomorphic to $H_{SL(2,3)} (i)$ for some $1 \leq i \leq 9$, where
\[
H_{SL(2,3)} (1) = \langle g_1, g_2, g_3, i I_2 \ \ \vert \ \  g_1^2 = g_2^2 = - I_2, \ \ g_3^3 =I_2, \ \ g_2g_1 = - g_1 g_2,
\]
\[
g_3 g_1 g_3^{-1} = g_2, \ \ g_3 g_2 g_3^{-1} = g_1g_2, \rangle
\]
of order $48$ with $R = {\mathbb Z}[i]$,
\[
H_{SL(2,3)}(2) = \langle g_1, g_2, g_3, h_o \ \ \vert \ \  g_1^2 = g_2^2 = - I_2, \ \ g_3^3 = I_2, \ \ h_o^2 = I_2, \ \  g_2g_1 = - g_1 g_2
\]
\[
g_3 g_1 g_3^{-1} = g_2, \ \ g_3 g_2 g_3^{-1} = g_1 g_2, \ \ h_o g_1 h_o ^{-1} = - g_1, \ \ h_o g_2 h_o^{-1} = - g_2, \ \ h_o g_3 h_o^{-1} = - g_2 g_3 \rangle
\]
of order $48$ with $R = {\mathbb Z}[i]$, $\lambda _1 (h_o) = -1$, $\lambda _2 ( h_o) = 1$,
\[
H_{SL(2,3)}(3) = \langle g_1, g_2, g_3, h_o \ \ \vert \ \ g_1^2 = g_2^2 = h_o^4 = - I_2, \ \ g_3^3 = I_2, \ \ g_2 g_1 = - g_1 g_2,
\]
\[
g_3 g_1 g_3^{-1} = g_2, \ \ g_3 g_2 g_3^{-1} = g_1g_2, \ \ h_o g_1 h_o^{-1} = g_2, \ \ h_o g_2 h_o^{-1} = - g_1, \ \
h_o g_3 h_o^{-1} = g_2 g_3^2 \rangle
\]
of order $48$ with $R = \mathcal{O}_{-2}$, $\lambda _1 (h_o) = e^{\frac{ \pi i}{4}}$, $\lambda _2(h_o) = e^{\frac{ 3 \pi i}{4}}$,
\[
H_{SL(2,3)}(4) = \langle g_1, g_2, g_3, h_o \ \ \vert \ \ g_1^2 = g_2^2 = - I_2, \ \ g_3^3 = I_2, \ \ h_o ^2 = I_2, \ \
g_2 g_1 = - g_1 g_2
\]
\[
g_3 g_1 g_3 ^{-1} = g_2, \ \ g_3 g_2 g_3^{-1} = g_1 g_2, \ \ h_o g_1 h_o^{-1} = g_2, \ \ h_o g_2 h_o^{-1} = g_1, \ \ h_o g_3 h_o^{-1} = g_1 g_3^2 \rangle
\]
of order $48$ with $R = R_{-2,f}$, $\lambda _1(h_o) = -1$, $\lambda _2 (h_o) =1$,
\[
H_{SL(2,3)}(5) = K_8 \times \langle e^{\frac{ 2\pi i}{3}} I_2 \rangle \simeq SL(2, {\mathbb F}_3) \times {\mathbb C}_3
\]
of order $72$ with $R = \mathcal{O}_{-3}$,
\[
H_{SL(2,3)}(6) = \langle g_1, g_2, g_3,  h_o \ \ \vert \ \ g_1^2 = g_2^2 = -I_2, \ \ g_3^3 = I_2, \ \ h_o ^3 = I_2, \ \ g_2 g_1 = - g_1 g_2,
\]
\[
g_3 g_1 g_3^{-1} = g_2, \ \ g_3 g_2 g_3^{-1} = g_1 g_2, \ \ h_o g_1 h_o^{-1} = g_2, \ \ h_o g_2 h_o^{-1} = g_1 g_2, \ \ h_o g_3 h_o^{-1} = g_3 \rangle
\]
of order $72$ with $R = \mathcal{O}_{-3}$, $\lambda_1 (h_o) = e^{\frac{ 2 \pi i}{3}}$, $\lambda _2 (h_o) = 1$,
\[
H_{SL(2,3)}(7) = \langle g_1, g_2, g_3, h_o \ \ \vert \ \ g_1^2 = g_2^2 = h_o^4 = - I_2, \ \ g_3^3 = I_2, \ \ g_2 g_1 = - g_1 g_2
\]
\[
g_3 g_1 g_3^{-1} = g_2, \ \ g_3 g_2 g_3^{-1} = g_1 g_2, \ \ h_o g_1 h_o^{-1} = - g_1, \ \
h_o g_2 h_o^{-1} = - g_2, \ \ h_o g_3 h_o^{-1} = - g_2 g_3 \rangle
\]
of order $96$ with $R = {\mathbb Z}[i]$, $\lambda_1(h_o) = e^{\frac{ 3 \pi i}{4}}$, $\lambda _2 (h_o) = e^{ - \frac{ \pi i}{4}}$,
\[
H_{SL(2,3)}(8) = \langle g_1, g_2, g_3, h_o \ \ \vert \ \ g_1^2 = g_2^2 = h_o^4 = -I_2, \ \ g_3^3 = I_2, \ \ g_2 g_1 = - g_1 g_2
'\]
\[
g_3 g_1 g_3^{-1} = g_2, \ \ g_3 g_2 g_3^{-1} = g_1 g_2, \ \ h_o g_1 h_o^{-1} = g_2, \ \ h_o g_2 h_o^{-1} = g_1, \ \ h_o g_3 h_o^{-1} = g_1 g_3^2 \rangle
\]
of order $96$ with $R = {\mathbb Z}[i]$, $\lambda _1(h_o) = e^{\frac{ 3 \pi i}{4}}$, $\lambda _2 (h_o) = e^{ - \frac{ \pi i}{4}}$,
\[
H_{SL(2,3)}(9) = \langle g_1, g_2, g_3, h_o \ \ \vert \ \ g_1^2 = g_2^2 = - I_2, \ \ g_3^3 = I_2, \ \ h_o^4 = I_2, \ \ g_2 g_1 = - g_1 g_2,
\]
\[
g_3 g_1 g_3^{-1} = g_2, \ \ g_3 g_2 g_3^{-1} = g_1 g_2, \ \ h_o g_1 h_o^{-1} = g_2, \ \ h_o g_2 h_o^{-1} = - g_1, \ \ h_o g_3 h_o^{-1} = g_2 g_3^2 \rangle
\]
of order $96$ with $R = {\mathbb Z}[i]$, $\lambda _1 (h_o) = i$, $\lambda _2 (h_o) = 1$.

There exists a subgroup
\[
H_{SL(2,3)}(5) < GL(2, \mathcal{O}_{-3}),
\]
as well as subgroups
\[
H^{o}_{SL(2,3)}(1),  H^{o}_{SL(2,3)}(2),  H^{o}_{SL(2,3)}(9) < GL(2, {\mathbb Q}( \sqrt{3}, i )),
\]
\[
H^{o}_{SL(2,3)}(4) < GL(2, {\mathbb Q} ( \sqrt{-2}, \sqrt{-3})),
\]
\[
H^{o}_{SL(2,3)}(3), H^{o}_{SL(2,3)}(7), H^{o}_{SL(2,3)}(8) < GL(2, {\mathbb Q} ( \sqrt{2}, \sqrt{3}, i))
\]
with $H^{o}_{SL(2,3)}(j) \simeq H_{SL(2,3)}(j)$ for $1'\leq j \leq 4$ or $6 \leq j \leq 9$.
\end{proposition}

\begin{proof}

According to Lemma \ref{StructureH}, the groups $H$ under consideration are uniquely determined up to an isomorphism by the order $r$ of $h_o$ and by the elements $ h_o g_j h_o^{-1} \in K_8^{(4)}$, $1 \leq j \leq 2$, $x_3 := h_o g_3 h_o^{-1} \in K_8^{(3)}$. (Throughout, $G^{(\nu)}$ denotes the set of the elements of order $\nu$  from a  group $G$.) Recall by Proposition \ref{K} the realization of $K_8 \simeq SL(2, {\mathbb F}_3)$ as a subgroup $\mathcal{K}_8$ of $GL(2, {\mathbb Q} ( \sqrt{-d}, \sqrt{-3}))$, generated by the matrices
\[
D_1  = \left( \begin{array}{rr}
- \frac{\sqrt{-3}}{3}   &  b_1  \\
\mbox{   }   &  \mbox{   } \\
- \frac{2}{3b_1}   &  \frac{\sqrt{-3}}{3}
\end{array}  \right), \ \
D_2 = \left( \begin{array}{rr}
- \frac{\sqrt{-3}}{3}   &  e^{ - \frac{ 2 \pi i}{3}}  b_1  \\
\mbox{   }   &  \mbox{   } \\
- \frac{2 e^{\frac{ 2 \pi i}{3}}}{3b_1}  &  \frac{\sqrt{-3}}{3}
\end{array}  \right), \ \
D_3 = \left( \begin{array}{rr}
e^{\frac{ 2 \pi i}{3}}   &  0  \\
0  &  e^{- \frac{ 2 \pi i}{3}}
\end{array}  \right)
\]
with some $b_1 \in {\mathbb Q} ( \sqrt{-d}, \sqrt{-3})^*$. After computing
\[
D_1 D_2 = \left( \begin{array}{rr}
- \frac{\sqrt{-3}}{3}   &  e^{ - \frac{ 4 \pi i}{3}} b_1  \\
\mbox{   }   &  \mbox{   } \\
- \frac{2 e^{\frac{ 4 \pi i}{3}}}{3b_1}   & \frac{\sqrt{-3}}{3}
\end{array}  \right),
\]
one puts
\[
\delta _j := \left( \begin{array}{rr}
- \frac{\sqrt{-3}}{3}   &  e^{ - \frac{ 2 j \pi i}{3}} b_1  \\
\mbox{   }   &  \mbox{   } \\
- \frac{2 e^{\frac{ 2 j \pi i}{3}}}{3b_1}   & \frac{\sqrt{-3}}{3}
\end{array} \right) \ \ \mbox{  for  } \ \  0 \leq j \leq 2
\]
and observes that $\delta _0 = D_1$, $\delta _1 = D_2$, $\delta _2 = D_1D_2$. The elements of $\mathcal{K}_8$ of order $4$ constitute the subset
\[
\mathcal{K}_8 ^{(4)} = \{ \pm \delta _j \ \ \vert \ \ 0 \leq j \leq 2 \}.
\]
In order to list the elements of $\mathcal{K}_8$ of order $3$, let us note that $D_3 D_1 D_3^{-1} = D_2$ and $D_3 D_2 D_3^{-1} = D_1 D_2$ imply $D_3 (D_1D_2) D_3^{-1} = D_1$. Thus, for any even permutation $j, l, m$ of $0,1,2$, one has
\begin{equation}   \label{ConjugationFormulae}
\left|  \begin{array}{c}
D_3 \delta _j D_3^{-1} = \delta _l \\
D_3 \delta _l D_3^{-1} = \delta _m \\
D_3 \delta _m D_3^{-1} = \delta _j
\end{array}  \right. \ \ \mbox{  or, equivalently, } \ \
\left| \begin{array}{c}
D_3 \delta _j = \delta _l D_3 \\
D_3 \delta _l = \delta _m D_3 \\
D_3 \delta _m = \delta _j D_3
\end{array} \right. .
\end{equation}
Making use of (\ref{ConjugationFormulae}, one computes that
\[
( - \delta _j D_3) ^2 = \delta _m D_3^2, \ \ ( - \delta _j D_3) ^3 = ( - \delta _j D_3) ( - \delta _j D_3)^2 = I_2 \ \ \mbox{  for all } \ \  0 \leq j \leq 2,
\]
so that $- \delta _j D_3 \in \mathcal{K}_8^{(3)}$. As a result, $\delta _j D_3^2 = ( - \delta _l D_m)^2 \in \mathcal{K}_8 ^{(3)}$ for all $0 \leq j \leq 2$ and
\[
\mathcal{K}_8 ^{(3)} = \{ D_3, \ \ - \delta _j D_3, \ \ D_3^2, \ \ \delta _j D_3^2 \ \ \vert \ \ 0 \leq j \leq 2 \}.
\]
Proposition \ref{K} has established that $\mathcal{K}_8$ has a unique Sylow $2$-subgroup
\[
\mathcal{H}_8 = \langle  \delta _0, \delta _1 \ \ \vert \ \ \delta _0 ^2 = \delta _1 ^2 = - I_2, \ \ \delta _1 \delta _0 = - \delta _0 \delta _1 \rangle = \{ \pm I_2, \pm \delta _j \ \ \vert \ \ 0 \leq j \leq 2 \},
\]
so that the set $\mathcal{K}_8^{(4)} = \mathcal{H}_8 ^{(4)}$ of the elements of $\mathcal{K}_8$ of order $4$ are contained in $\mathcal{H}_8 \simeq {\mathbb Q}_8$. In other words, $x_j := h_o \delta _j h_o^{-1} \in \mathcal{H}_8$ and $H'= \langle g_1, g_2, h_o \rangle \simeq \mathcal{H}'= \langle \delta _0, \delta _1, D_o \rangle$ is a subgroup of $H$ with $H \cap SL(2,R) \simeq {\mathbb Q}_8$. Proposition \ref{HQ8}  establishes that any such $H'$ is isomorphic to $H_{Q8}(i)$ for some $1 \leq i \leq 9$.

We claim that for any $1 \leq i \leq 9$ there is (at most) a unique finite subgroup $H = \langle g_1, g_2, g_3, h_o \rangle$ of $GL(2,R)$ with $\langle g_1, g_2, h_o \rangle \simeq H_{Q8} (i)$, $H \cap SL(2,R) = \langle g_1, g_2, g_3 \rangle \simeq SL(2, {\mathbb F}_3 )$ and $\det(H) = \langle \det(h_o) \rangle$. To this end, let us consider the adjoint representation
\[
{\rm Ad} : \mathcal{K}_8 \longrightarrow S ( \mathcal{K}_8^{(4)} ) \simeq S_6
\]
\[
{\rm Ad} _x (y) = x y x^{-1} \ \ \mbox{  for  } \ \ \forall x \in \mathcal{K}_8, \ \ \forall y \in \mathcal{K}_8^{(4)}
\]
and its restriction
\[
{\rm Ad} : \mathcal{K}_8 ^{(3)} \longrightarrow  S( \mathcal{K}_8 ^{(4)}) \simeq S_6
\]
to the elements of $\mathcal{K}_8$ of order $3$.  Note that
\[
\langle x_0, x_1 \rangle = h_o \langle \delta _0, \delta _1 \rangle h_o^{-1} = h_o \mathcal{H}_8 h_o^{-1} = \mathcal{H}_8,
\]
as far as $\mathcal{H}_8 \simeq {\mathbb Q}_8$ is  normal subgroup of $\mathcal{H}'= \mathcal{H}_8 \langle h_o \rangle$. The adjoint action
\[
{\rm Ad} _{h_o}  : \mathcal{K}_8 \longrightarrow \mathcal{K}_8
\]
\[
{\rm Ad} _{h_o} (x) = h_o x h_o^{-1} \ \ \mbox{  for  } \ \ \forall x \in \mathcal{K}_8
\]
of $h_o$ is a group homomorphism and transforms the relations $D_3 \delta _s D_3^{-1} = \delta _{s+1}$ for $0 \leq s \leq 1$ into the relations $x_3 x_s x_3^{-1} = x_{s+1}$ for $ 0 \leq s \leq 1$. For any $1 \leq i \leq 9$ the subgroup $\mathcal{H}' \simeq H_{Q8}(i)$ of $\mathcal{H}$ determines uniquely $x_0, x_1 \in \mathcal{H}_8$. We claim that for any such $x_0,x_1$ there is a unique $x_3 \in \mathcal{K}_8^{(3)}$ with
\begin{equation}  \label{DerivedRelations}
{\rm Ad} _{x_3} ( x_0) = x_1, \quad {\rm Ad} _{x_3} (x_1) = x_0x_1.
\end{equation}
Indeed, Proposition \ref{HQ8} specifies the following five possibilities:
\[
\mbox{   Case 1} \ \ x_0 = \delta _0, \ \ x_1 = \delta _1;
\]
\[
\mbox{   Case 2} \ \ x_0 =  - \delta _0, \ \ x_1 =  - \delta _1;
\]
\[
\mbox{   Case 3} \ \ x_0 = \delta _1, \ \ x_1 =  - \delta _0;
\]
\[
\mbox{   Case 4} \ \ x_0 = \delta _1, \ \ x_1 = \delta _0;
\]
\[
\mbox{   Case 5} \ \ x_0 = \delta _1, \ \ x_1 = \delta _2.
\]
 For any $ 0 \leq s \neq t \leq 2$ and $\varepsilon, \eta \in \{ \pm 1 \}$ note that
\[
{\rm Ad} _{\varepsilon \delta _s} ( \eta \delta _s) = \eta \delta _s, \ \
{\rm Ad}_{ \varepsilon \delta _s} ( \eta \delta _t) = - \eta \delta _t.
\]
Combining with (\ref{ConjugacyFormula}), one concludes that
\[
{\rm Ad} _{D_3} ( \langle \delta _j \rangle ) = {\rm Ad} _{ (- \delta _s D_3 )} ( \langle \delta _j \rangle ) = \langle \delta _l \rangle,
\]
\[
{\rm Ad} _{D_3} ( \langle \delta _l \rangle ) = {\rm Ad} _{( - \delta _s D_3)} ( \langle \delta _l \rangle ) = \langle \delta _m \rangle,
\]
\[
{\rm Ad} _{D_3} ( \langle \delta _m \rangle ) = {\rm Ad} _{( - \delta _s D_3)} ( \langle \delta _m \rangle ) = \langle \delta _j \rangle
\]
for any $0 \leq s \leq 2$ and any even permutation $j, l, m$ of $0, 1, 2$. Similarly,
\[
{\rm Ad} _{D_3^2}( \langle \delta _j \rangle ) = {\rm Ad} _{\delta _s D_3^2} ( \langle \delta _j \rangle ) = \langle \delta _m \rangle,
\]
\[
{\rm Ad} _{D_3^2}( \langle \delta _l \rangle ) = {\rm Ad} _{\delta _s D_3^2} ( \langle \delta _l \rangle ) = \langle \delta _j \rangle,
\]
\[
{\rm Ad} _{D_3^2}( \langle \delta _m \rangle ) = {\rm Ad} _{\delta _s D_3^2} ( \langle \delta _m \rangle ) = \langle \delta _l \rangle
\]
for any $0 \leq s \leq 2$ and any even permutation $j,l,m$ of $0,1,2$. In the case 1, (\ref{DerivedRelations}) read as ${\rm Ad}_{x_3} (\delta _0) = \delta _1$, ${\rm Ad} _{x_3} ( \delta _1 ) = \delta _2$ and imply that $x_3=D_3$, according to (\ref{ConjugationFormulae}) and ${\rm Ad} _{( - \delta _s)} \not \equiv Id_{\mathcal{K}_8}$ for all $ 0 \leq s \leq 2$. In the Case 2, ${\rm Ad} _{x_3} ( \delta _0) = \delta _1$ and ${\rm Ad} _{x_3} ( \delta _1) = - \delta _2$ specify that $x_3 = - \delta _1 D_3 = - D_2D_3$. In the next Case 3, the relations ${\rm Ad} _{x_3} (\delta _1 ) = - \delta _0$, ${\rm Ad} _{x_3}( \delta _0 = \delta _2$ hold if and only if $x_3 = \delta _1 D_3^2 = D_2 D_3^2$. Further, ${\rm Ad} _{x_3} ( \delta _1) = \delta _0$, ${\rm Ad} _{x_3} ( \delta _0) = - \delta _2$ in Case 4 are satisfied by $x_3 = \delta _0 D_3^2 = D_1D_3^2$ and ${\rm Ad}_{x_3} ( \delta _1) = \delta _2$, ${\rm Ad} _{x_3} ( \delta _2) = \delta _0$ in Case 5 are valid for $x_3=D_3$.  Given a presentation of $H'\simeq H_{Q8}(i)$ with generators $g_1, g_2, h_o$, one adjoins a  generator $g_3 \in SL(2,R)$ of order $3$ and the relation $h_o g_3 h_o^{-1} = x_3$, in order to obtain a presentation of $H \simeq H_{SL(2,3)} (i)$, $1 \leq i \leq 9$.

\end{proof}


\section{Explicit Galois groups for $A/H$ of fixed Kodaira-Enriques type}

In order to classify the finite subgroups $H$ of $Aut(A)$, for which $A/H$ is of a fixed Kodaira-Enriques classification type, one needs to describe the finite subgroups $H$ of $Aut(A)$ for $A = E \times E$. Making use of the classification of the finite subgroups $\mathcal{L}(H)$ of $GL(2,R)$, done in section 3, let $\det \mathcal{L}(H) = \langle \det \mathcal{L}(h_o) = e^{\frac{ 2 \pi i}{s}} \rangle \simeq {\mathbb C}_s$ for some $s \in \{ 1, 2, 3, 4, 6 \}$, $h_o \in H$. (In the case of $s=1$, we choose $h_o = Id_A$.) By Proposition \ref{K} one has $\mathcal{L} (H) \cap SL(2,R) = \langle \mathcal{L}(h_1), \ldots , \mathcal{L}(h_t) \rangle$ for some $0 \leq t \leq 3$. (Assume $\mathcal{L}(H) \cap SL(2,R) = \{ I_2 \}$ for $t=0$.) The linear part
\[
\mathcal{L}(H) = [ \mathcal{L}(h) \cap SL(2,R)] \langle \mathcal{L}(h_o) \rangle =
\langle \mathcal{L}(h_1), \ldots , \mathcal{L}(h_t) \rangle \langle \mathcal{L}(h_o) \rangle
\]
of $H$ is a product of its normal subgroup $\langle \mathcal{L}(h_1), \ldots , \mathcal{L}(h_t) \rangle$ and the cyclic group $\langle \mathcal{L}(h_o) \rangle$. The translation part $\mathcal{T}(H) = \ker( \mathcal{L} \vert _H )$ of $H$ is a finite subgroup of $( \mathcal{T}_A, +) \simeq (A,+)$. The lifting $\left( \widetilde{\mathcal{T}_A}, + \right) < \left( \widetilde{A} = {\mathbb C}^2, + \right)$ of $\mathcal{T}(H)$ is a free ${\mathbb Z}$-module of rank $4$. Therefore $\left( \widetilde{ \mathcal{T}(H)} , + \right)$ has at most four generators and
\[
\mathcal{T}(H) = \langle \tau _{(P_i,Q_i)} \ \ \vert \ \  1 \leq i \leq m \rangle \ \ \mbox{   for some   } \ \  0 \leq m \leq 4.
\]
(In the case of $m=0$ one has $\mathcal{T}(H) = \{ Id_A \}$.) We claim that
\[
H = \mathcal{T}(H) \langle h_1, \ldots , h_t, h_o \rangle =
\langle \tau _{(P_i,Q_i)}, h_j, h_o \ \ \vert \ \ 1 \leq i \leq m, \ \ 1 \leq j \leq t \rangle
\]
for some $0 \leq m \leq 4$, $0 \leq t \leq 3$. The choice of $\tau _{(P_i,Q_i)}, h_j, h_o \in H$ justifies the inclusion $\langle \tau _{(P_i,Q_i)}, h_j, h_o \ \ \vert \ \ 1 \leq i \leq m, 1 \leq j \leq t \rangle \subseteq H$. For the opposite inclusion, an arbitrary element $h \in H$ with  $\mathcal{L}(h) = \mathcal{L}(h_1) ^{k_1} \ldots \mathcal{L}(h_t) ^{k_t} \mathcal{L}(h_o)^{k_o}$ for some $k_j \in {\mathbb Z}$ produces a translation $\tau _{(U,V)} := h h_o ^{ - k_o}  h_t ^{ - k_t} \ldots h_1 ^{ - k_1} \in \ker \left( \mathcal{L} \vert _H \right) = \mathcal{T}(H) =
\langle \tau _{(P_i,Q_i)} \ \ \vert \ \  1 \leq i \leq m \rangle$, so that  $ h = \tau _{(U,V)} h_1^{k_1} \ldots h_t ^{k_t} h_o^{k_o} \in \langle \tau _{(P_i,Q_i)}, h_j, h_o \ \ \vert \ \  1 \leq i \leq m, \ \ 1 \leq j \leq t \rangle$ and $H \subseteq \langle \tau _{(P_i,Q_i)} , h_j, h_o \ \ \vert \ \  1 \leq i \leq m, \ \ 1 \leq j \leq t \rangle$.  In such a way, we have derived the following

\begin{lemma}     \label{GeneratingFiniteSubgroupsAutA}
If $H$ is a finite subgroup of $Aut(A)$, $A = E \times E$ with
 \[
 \det \mathcal{L}(H) = \langle \det \mathcal{L}( h_o)  = e^{\frac{ 2 \pi i}{s}} \rangle  \simeq {\mathbb C}_s \ \ \mbox{   and  }
 \]
 \[
 \mathcal{L}(H) \cap SL(2,R) = \langle \mathcal{L}(h_1), \ldots , \mathcal{L}(h_t) \rangle \ \ \mbox{  for some  } \ \ 0 \leq t \leq 3 \ \ \mbox{  then}
 \]
 \[
 H = \langle \tau _{(P_i,Q_i)}, h_j h_o \ \ \vert \ \  1 \leq i \leq m, \ \ 1 \leq j \leq t \rangle
 \]
 is generated by $0 \leq m \leq 3$ translations and at most four non-translation elements.
\end{lemma}

Bearing in mind that $A/H$ is birational to a K3 surface exactly when $\mathcal{L}(H)$ is a subgroup of $SL(2,R)$, one obtains the following

\begin{corollary}     \label{K3AH}
The quotient $A/H$ by a finite subgroup $H$ of $Aut(A)$ has a smooth K3 model if and  only if $H$ is isomorphic to some $H^{K3}(j,m)$ with $1 \leq j \leq 8$, $0 \leq m \leq 3$, where
\[
H^{K3}(1.m) = \langle \tau _{(P_i,Q_i)}, \ \ \tau _{(U_1,V_1)} (-I_2) \ \ \vert \ \  1 \leq i \leq m \rangle
\]
\[
H^{K3}(2,m) = \langle \tau _{(P_i,Q_i)},  \ \ h_1 \ \ \vert \ \  1 \leq i \leq m \rangle
\]
for $\mathcal{L}(h_1) \in SL(2,R)$, $\tr\mathcal{L}(h_1) = 0$,
\[
H^{K3}(3,m) = \langle \tau _{(P_i,Q_i)}, \ \  h_1, \ \ h_2 \ \ \vert \ \  1 \leq i \leq m \rangle
\]
for $\mathcal{L}(h_1), \mathcal{L}(h_2) \in SL(2,R)$, $\tr \mathcal{L}(h_1) = \tr \mathcal{L}(h_2) = 0$, $\mathcal{L}(h_2) \mathcal{L}(h_1) = - \mathcal{L}(h_1) \mathcal{L}(h_2)$,
\[
H^{K3}(4,m) = \langle \tau _{(P_i,Q_i)}, \ \ h_3 \ \ \vert \ \  1 \leq i \leq m \rangle
\]
for $\mathcal{L}(h_3) \in SL(2,R)$, $\tr \mathcal{L}(h_3) = -1$,
\[
H^{K3} (5,m) = \langle \tau _{(P_i,Q_i)}, \ \ h_4 \ \ \vert \ \  1 \leq i \leq m \rangle
\]
for $\mathcal{L}(h_4) \in SL(2,R)$, $\tr \mathcal{L}(h_4) =1$,
\[
H^{K3}(6,m) = \langle \tau _{(P_i,Q_i)}, \ \  h_1, \ \ h_4 \ \ \vert \ \  1 \leq i \leq m \rangle
\]
for $\mathcal{L}(h_1), \mathcal{L}(h_4) \in SL(2,R)$, $\tr \mathcal{L}(h_1) =0$, $\tr \mathcal{L}(h_4) =1$,
$\mathcal{L}(h_1) \mathcal{L}(h_4) [\mathcal{L}(h_1)]^{-1}  = [\mathcal{L} (h_4)]^{-1}$,
\[
H^{K3}(7,m) = \langle \tau _{(P_i,Q_i)}, \ \ h_1, h_2, h_3 \ \ \vert \ \ 1 \leq i \leq m \rangle
\]
for $\mathcal{L}(h_1), \mathcal{L}(h_2), \mathcal{L}(h_3) \in SL(2,R)$, $\tr \mathcal{L}(h_1) = \tr  \mathcal{L}(h_2) = 0$,
$\tr \mathcal{L}(h_3) = -1$, \\
\[
 \mathcal{L}(h_2) \mathcal{L}(h_1) = - \mathcal{L}(h_1) \mathcal{L}(h_2),
 \]
 \[
\mathcal{L}(h_3) \mathcal{L}(h_1) [ \mathcal{L}(h_3)] ^{-1} = \mathcal{L}(h_2) \ \
 \mathcal{L}(h_3)  \mathcal{L}(h_2) [ \mathcal{L}(h_3)] ^{-1} = \mathcal{L}(h_1) \mathcal{L}(h_2),
 \]
\[
H^{K3}(8,m) = \langle \tau _{(P_i,Q_i)}, \ \  h_1, h_2, h_3 \ \ \vert \ \ 1 \leq i \leq m \rangle
\]
for $\mathcal{L}(h_1), \mathcal{L}(h_2), \mathcal{L}(h_3) \in SL(2,R)$, $\tr \mathcal{L}(h_1) = \tr \mathcal{L}(h_2) = 0$, $\tr \mathcal{L}(h_3) = -1$, \\
\[
 \mathcal{L}(h_2) \mathcal{L}(h_1) = - \mathcal{L}(h_1) \mathcal{L}(h_2),
 \]
 \[
 \mathcal{L}(h_3) \mathcal{L}(h_1) [ \mathcal{L}(h_3)] ^{-1} = \mathcal{L}(h_2), \ \
\mathcal{L}(h_3) \mathcal{L}(h_2) [ \mathcal{L}(h_3)]^{-1} = \mathcal{L}(h_1) \mathcal{L}(h_2).
\]
\end{corollary}

We are going to show that for an arbitrary finite subgroup $H < Aut(A)$ with an abelian linear part $\mathcal{L}(H) < GL(2,R)$, there exist an isomorphic model $F_1 \times  F_2$ of $A$ and a normal subgroup $N_1$ of $H$, embedded in $Aut(F_1)$, such that the quotient group $H / N_1$ is an automorphism group of $F_2$. This result can be viewed as a generalization of Bombieri-Mumford's  classification \cite{BM1} of the hyper-elliptic surfaces. More precisely, if $H = \mathcal{T}(H) \langle h_o \rangle$ for some $h_o \in H$ with eigenvalues $\lambda  _1 \mathcal{L}(h_o) = 1$, $\lambda _2 \mathcal{L}(h_o) = \det \mathcal{L}(h_o) = e^{\frac{ 2 \pi i}{s}}$, $s \in \{ 2, 3, 4, 6 \}$, then there is a translation subgroup $N_1$ of $Aut(F_1)$, such that $G \simeq H / N_1$ is a non-translation group, acting on the split abelian surface $F'_1 \times F_2 = ( F_1 / N_1) \times F_2$. According to  Proposition  \ref{AbelianCoverL1}, the quotient $A/H$ is hyper-elliptic  (respectively, ruled with elliptic base) exactly when the finite Galois covering $A \rightarrow A/H$ is unramified (respectively, ramified). Since $F_1 \rightarrow F_1 / N_1 = F'_1$ is unramified for a translation subgroup $N_1 \mathcal{T}_{F_1} < Aut (F_1)$, the covering $A \rightarrow A/H$ is unramified is and only if the covering $F'_1 \times F_2 \rightarrow ( F'_1 \times F_2) / G$ is unramified for $G = H / N_1$. In particular, the first canonical projection ${\rm pr} _1 : G \rightarrow Aut( F'_1)$ is a group monomorphism and $G$ is an abelian group with at most two generators, according to the classification of the finite translation groups of $F'_1$. Thus, Bombieri-Mumford's classification of the hyper-elliptic surfaces $(F'_1 \times F_2) / G$ reduces to the classification of the split, fixed point free abelian subgroups $G < Aut ( F'_1 \times F_2)$ with at most two generators, for which the canonical projections ${\rm pr} _1 : G \rightarrow Aut(F'_1)$ and ${\rm pr}_2 : G \rightarrow Aut(F_2)$ are injective group homomorphisms.

Towards the classification of the finite subgroups of $Aut(E)$, let us recall that the semi-direct products $\langle a \rangle \rtimes \langle b \rangle \simeq {\mathbb C} _m \rtimes {\mathbb C}_s$ of cyclic groups are completely determined by the adjoint action of $b$ on $a$. Namely, ${\rm Ad} _b (a) = b a b^{-1} = a^j$  for some residue $j \in {\mathbb Z}_m^*$ modulo $m$, relatively prime to $m$. Now ${\rm Ad} _{b^s} (a) = a^{j^s} =a$ requires $j^s \equiv 1 ({\rm mod} m)$. In other words, $j \in {\mathbb Z}_m^*$ is of order $r$, dividing $s$ and $\langle a \rangle  \leftthreetimes \langle b \rangle $ is isomorphic to
\begin{equation}    \label{NotationAutEllipticCurve}
G_s ^{(j)} (m) := {\mathbb C}_m \rtimes _j {\mathbb C} _s = \langle a, \ \  b \ \ \vert \ \  a^m =1, \ \ b^s =1, \ \  b a b^{-1} = a^j \rangle
\end{equation}
for some $j \in {\mathbb Z}_m^*$ of order   $r$, dividing $s$. Form now on, we use the notation (\ref{NotationAutEllipticCurve}) without further reference.
Note that the only $j \in {\mathbb Z}_m^*$ of order $1$ is $j \equiv 1({\rm mod} m)$ and $G_s ^{(1)}(m) = \langle a \rangle \times \langle b \rangle \simeq {\mathbb C}_m \times {\mathbb C}_s$ is the direct product of $\langle a \rangle = {\mathbb C}_m$ and $\langle b \rangle = {\mathbb C}_s$.

\begin{lemma}    \label{FiniteSubgroupsAutE}
Let $G$ be a  finite subgroup of the automorphism group  $Aut(E)$  of an elliptic curve $E$ with endomorphism ring $End(E) =R$. Then $G$ is isomorphic to some of the groups $G_1(m,n)$, $G_2 ^{(-1,-1)} (m,n)$, $G_s^{(j)} (m)$, $s \in \{ 3, 4, 6\}$, where
\[
G_1 (m,n) = \langle \tau _{P_1}, \ \ \tau _{P_2} \rangle \simeq {\mathbb C}_m \times {\mathbb C}_n, \ \ m,n \in {\mathbb N}
\]
is a translation group with at most two generators,
\[
G_2^{(-1,-1)}(m,n)  = \langle \tau _{P_1}, \ \ \tau _{P_2} \rangle \rtimes \langle -1 \rangle \simeq
 ( {\mathbb C} _m \times {\mathbb C}_n) \rtimes _{(-1,-1)} {\mathbb C}_2  =
 \left( \langle a \rangle \times \langle b \rangle \right) \rtimes _{(-1,-1)}  \langle c \rangle =
 \]
 \[
= \langle a, \ \ b, \ \ c \ \ \vert \ \  a^m=1, \ \ b^n=1, \ \ c^2 =1, \ \ cac^{-1} =a, \ \ cbc^{-1} = b^{-1} \rangle
\]
\[
G_3^{(j)} (m) = \langle \tau _{P_1} \rangle \rtimes _j \langle e^{\frac{ 2 \pi i}{3}} \rangle \simeq
{\mathbb C}_m \rtimes _j {\mathbb C}_3 =
\langle a \rangle \rtimes _j \langle c \rangle =
\]
\[
= \langle a, \ \  c \ \ \vert \ \  a^m=1, \ \ c^3=1, \ \ cac^{-1} = a^j \rangle
\]
for some $j \in {\mathbb Z}_m^*$ of order $1$ or $3$, $R = \mathcal{O}_{-3}$,
\[
G_4^{(j)} (m) = \langle \tau _{P_1} \rangle \rtimes _j \langle i \rangle \simeq
{\mathbb C}_m \rtimes _j {\mathbb C}_4 =
\langle a \rangle \rtimes _j \langle c \rangle =
\]
\[
= \langle a, \ \ c \ \ \vert \ \  a^m=1, \ \ c^4=1, \ \ cac^{-1} = a^j \rangle
\]
for some $j \in {\mathbb Z}_{m}^*$ of order $1,2$ or $4$, $R = {\mathbb Z}[i]$,
\[
G_6^{(j)} (m) = \langle \tau _{P_1} \rangle \rtimes _j \langle e^{\frac{ \pi i}{3}} \rangle \simeq
{\mathbb C}_m \rtimes _j {\mathbb C}_6 =
\langle a \rangle \rtimes _j \langle c \rangle =
\]
\[
= \langle a, \ \ c \ \ \vert \ \  a^m=1, \ \ c^6=1, \ \ cac^{-1} = a^j \rangle
\]
for some $j \in {\mathbb Z}_m ^*$ of order $1,2,3$ or $6$.
\end{lemma}

\begin{proof}

Any finite translation group $G < ( \mathcal{L}_E, +)$ lifts to a lattice $\widetilde{G} < ( \widetilde{E} = {\mathbb C}, +)$ of rank $2$, containing $\pi _1(E)$. By the Structure Theorem for finitely generated modules over the principal ideal domain ${\mathbb Z}$, there exists a ${\mathbb Z}$-basis $\lambda _1 , \lambda _2$ of $\widetilde{G}$ and natural numbers $m,n \in {\mathbb N}$, such that
\[
\widetilde{G} = \lambda _1 {\mathbb Z} + \lambda _2 {\mathbb Z}, \ \ \pi _1(E) = m \lambda _1 {\mathbb Z} + m n \lambda _2 {\mathbb Z}.
\]
As a result, $P_1 = \lambda _1 + \pi _1(E) \in (E,+)$ of order $m$ and $P_2 = \lambda _2 + \pi _1(E) \in (E,+)$ of order $mn$ generate the finite translation group $G = \widetilde{G} / \pi _1(E) \simeq {\mathbb C}_m \times {\mathbb C}_{mn}$.

If $G$ is a finite non-translation subgroup of $Aut(E)$ then the linear part $\mathcal{L}(G)$ of $G$ is a non-trivial subgroup of the units group $R^*$. Bearing in mind that
\[
R^* = \begin{cases}
\langle -1 \rangle \simeq {\mathbb C}_2  &  \text{ for $R \neq {\mathbb Z}[i], \mathcal{O}_{-3}$,}  \\
\langle i \rangle \simeq {\mathbb C}_4  &  \text{ for $R = {\mathbb Z}[i]$,}   \\
\langle e^{\frac{ \pi i}{3}} \rangle & \text{ for $R = \mathcal{O}_{-3}$,}
\end{cases}
\]
one concludes that $\mathcal{G} = \langle e^{\frac{ 2 \pi i}{s}} \rangle \simeq {\mathbb C}_s$ for some $s \in \{ 2, 3, 4, 6 \}$.
Any lifting $g_0 = \tau _U e^{\frac{ 2 \pi i}{s}} \in  G$ of $\mathcal{L}(g_0) = e^{\frac{ 2 \pi i}{s}}$ has a fixed point $P_0 \in E$. After moving the origin of $E$ at $P_0$, one can assume that $g_0 = e^{\frac{ 2 \pi i}{s}}$. Bearing in mind that the translation part $\mathcal{T}(G) = \ker ( \mathcal \vert _G )$,
one observes that $G = \mathcal{T}(G) \langle e^{\frac{ 2 \pi i}{s}} \rangle$. The inclusion $\mathcal{T}(G) \langle e^{\frac{ 2 \pi i}{s}} \rangle \subseteq G$ is clear. For any $g \in G$ with $\mathcal{L}(g) = e^{\frac{ 2 \pi i j}{s}}$ for some $ 0 \leq j \leq s-1$, one has $g \left( e^{\frac{ 2 \pi i}{s}} \right) ^{-j} \in \ker ( \mathcal{L} \vert _G ) = \mathcal{T}(G)$, so that $G \subseteq \mathcal{T}(G) \langle e^{\frac{ 2 \pi i}{s}} \rangle$ and $G = \mathcal{T}(G) \langle e^{\frac{ 2 \pi i}{s}} \rangle$. Note that $\mathcal{T}(G)$ is a normal subgroup of $G$ with $\mathcal{T}(G) \cap \langle e^{\frac{ 2 \pi i}{s}} \rangle = \{ Id _E \}$, so that
\[
G = \mathcal{T}(G) \rtimes \langle e^{\frac{ 2 \pi i}{s}} \rangle
\]
is a semi-direct product. As a result, there is an adjoint action
\[
{\rm Ad} : \langle e^{\frac{ 2 \pi i}{s}} \rangle \longrightarrow Aut ( \mathcal{T}(G)),
\]
\[
{\rm Ad} _{e^{\frac{ 2 \pi i j}{s}}} ( \tau _{P_1}) = e^{\frac{ 2 \pi i j}{s}} \tau _{P_1} e^{- \frac{ 2 \pi i j}{s}} = \tau _{ s^{\frac{ 2 \pi i j}{s}} P_1}
\]
of $\langle e^{\frac{ 2 \pi i}{s}} \rangle$ on $\mathcal{T}(G)$, which is equivalent to the invariance of $\mathcal{T}(G)$ under a multiplication by $e^{\frac{ 2 \pi i}{s}} \in R^*$.   The translation group $'\mathcal{T}(G) = \langle \tau _{P_1}, \ \ \tau _{P_2} \rangle$ has at most two generators, so that
\[
G = \langle \tau _{P_1}, \ \ \tau _{P_2} \rangle \rtimes \langle e^{\frac{ 2 \pi i}{s}} \rangle
\]
for some $s \in \{ 2, 3, 4, 6\}$. If $s=2$ and $\langle \tau _{P_1}, \tau _{P_2} \rangle \simeq {\mathbb C} _m \times {\mathbb C}_n = \langle \tau _{Q_1} \rangle \times \langle \tau _{Q_2} \rangle$, then ${\rm Ad} _{-1} ( \tau _{Q_1}) = \tau _{-Q_k} $ for $1 \leq k \leq 2$. The residue classes $-1({\rm mod} m) \in {\mathbb Z}_m^*$ and $-1 ({\rm mod} n) \in {\mathbb Z}_n^*$ are order $1$ or $2$.

We claim that $G = \langle \tau _{P_1} \rangle \rtimes \langle e^{\frac{ 2 \pi i}{s}} \rangle$ has at most two generators for $s \in \{ 3, 4, 6\}$. Indeed, $\tau _{P_1} \in \mathcal{T}(G)$ implies that ${\rm Ad} _{e^{\frac{ 2 \pi i}{s}}} ( \tau _{P_1}) = \tau _{ e^{\frac{ 2 \pi i}{s}} P_1} \in \mathcal{T}(G)$. For $s \in \{ 3, 4, 6 \}$ the points $P_1$, $e^{\frac{ 2 \pi i}{s}} P_1$ have ${\mathbb Z}$-linearly independent liftings from $\widetilde{\mathcal{T}(G)}$, so that $\mathcal{T}(G) = \langle \tau _{P_1}, \ \ \tau _{P_2} \rangle = \langle \tau _{P_1}, \ \ \tau _{ e^{\frac{2 \pi i}{s}} P_1} \rangle$. As a result,
\[
G == \langle \tau _{P_1}, \ \ e^{\frac{ 2 \pi i}{s}} \tau _{P_1} e^{ - \frac{ 2 \pi i}{s}} \rangle \rtimes \langle e^{\frac{ 2 \pi i}{s}} \rangle =
\langle \tau _{P_1} \rangle \rtimes \langle e^{\frac{ 2 \pi i}{s}} \rangle \simeq
{\mathbb c}_m \rtimes _j {\mathbb C}_s = \langle a \rangle \rtimes _j \langle c \rangle =
\]
\[
\langle a, \ \  c \ \ \vert \ \  a^m=1, \ \ c^s=1, \ \ cac^{-1} = a^j \rangle
\]
for some $j \in {\mathbb Z}_m^*$ of order $r$, dividing $s \in \{3, 4, 6 \}$.

\end{proof}

Let us put $G_1^{(1,1)} (m,n) := G_1(m,n)$, in order to list the finite subgroups of $Aut(E)$ as $G_s ^{(j_1,j_2)} (m,n)$ with $s \in \{ 1, 2\}$ and $G_s ^{(j)} (m)$ with $s \in \{ 3, 4, 6 \}$.

\begin{lemma}   \label{SimultaneousDiagonalization}
Let $H$ be a finite subgroup of $Aut(A)$ with abelian linear part $\mathcal{L}(H)$. Then:

(i) there exists $S \in GL(2, {\mathbb C})$, such that all the elements of
\[
S^{-1} H S = \{ S^{-1} h S = ( \tau _{U_1} \lambda _1 \mathcal{L}(h), \tau _{U_2} \lambda _2 \mathcal{L} ) \ \ \vert \ \ h \in H \} < Aut (S^{-1}A)
\]
have diagonal linear parts;

(ii) if $F_1 = S^{-1} (E \times \check{o}_E)$, $F_2 = S^{-1} ( \check{o}_E \times E)$ then $S^{-1}A = F_1 \times F_2$ and the canonical projections
\[
{\rm pr} _k : S^{-1} H S \longrightarrow Aut(F_k),
\]
\[
{\rm pr} _k ( \tau _{U_1} \lambda _1 \mathcal{L}(h), \tau _{U_2} \lambda _2 \mathcal{L}(h)) = \tau _{U_k} \lambda _k \mathcal{L}(h),
\]
are group homomorphisms with ${\rm pr} _k (S^{-1}HS) \simeq G_s^{(j_1,j_2)} (m,n)$, $s \in \{ 1,2\}$ or $G_s ^{(j)}$, $s \in \{ 3, 4, 6\}$;

(iii)  $S^{-1}HS = \ker( {\rm pr} _2 ) \langle h_1, \ldots , h_t \rangle$ for any liftings $h_j = (\alpha _j, \beta _j) \in S^{-1}HS$ of the generators $\beta _1, \ldots , \beta _t$ of ${\rm pr} _2 (S^{-1}HS)$, $1 \leq t \leq 3$;

(iv) $S^{-1}A / \ker( {\rm pr} _2) = C_1 \times F_2$, where $C_1$ is an elliptic curve for a translation subgroup $\ker ( {\rm pr}_2 ) < ( \mathcal{T}_{F_1},+) < Aut(F_1)$ or a rational curve for a non-translation subgroup $\ker ( {\rm pr} _2) < Aut (F_1)$, $\ker( {\rm pr} _2 ) \setminus (\mathcal{T}_{F_1}, +) \neq \emptyset$;

(v)  $A/H \simeq (C_1 \times F_2) / G$ for
'\[
G := \langle h_1, \ldots , h_t \rangle / ( \langle h_1, \ldots , h_t \rangle \cap \ker ( {\rm pr}_2 ))
\]
with isomorphic second projection
\[
\overline{{\rm pr}_2}  : G \longrightarrow {\rm pr} _2 (S^{-1}HS)
\]
and first projection
\[
\overline{{\rm pr}_1} : G \rightarrow \overline{{\rm pr}_1} (G) < Aut (C_1)
\]
with kernel $\ker( \overline{{\rm pr}_1} \vert _G ) \simeq \ker( {\rm pr} _1 \vert _{S^{-1}HS} )$.
\end{lemma}

\begin{proof}

(i) It is well known that for any finite set $\{ \mathcal{L}(h) \ \ \vert \ \ h \in H \}$ of commuting matrices, there exists $S \in GL(2, {\mathbb C})$, such that
\[
S^{-1} \mathcal{L}(h) S = \mathcal{L} (S^{-1}hS) = \left( \begin{array}{cc}
\lambda _1 \mathcal{L}(h)   &  0  \\
0  &  \lambda _2 \mathcal{L}(h)
\end{array} \right)
\]
are diagonal for all $h \in H$. Namely, if there is $h_o \in H$, whose linear part $\mathcal{L}(h_o)$ has two different eigenvalues $\lambda _1 \mathcal{L}(h_o) \neq \lambda _2 \mathcal{L}(h_o)$, then one takes the $j$-th column of $S \in {\mathbb Q} ( \sqrt{-1}) _{2 \times 2}$ to be an eigenvector, associated with $\lambda _j \mathcal{L}(h_o)$, $1 \leq j \leq 2$. The conjugate $S^{-1} \mathcal{L}(h_o)S$  is a diagonal matrix. It suffices to show that $v_j$ are eigenvectors of all $\mathcal{L}(h)$, in order to conclude that $S^{-1} \mathcal{L}(h) S$ are diagonal, as the matrices of $\mathcal{L}(h)$ with respect to the basis $v_1,v_2$ of ${\mathbb C}^2$. Indeed, for any $h \in H$ the relation  $\mathcal{L}(h) \mathcal{L}(h_o) = \mathcal{L}(h_o) \mathcal{L}(h)$ implies that
\[
\lambda _j \mathcal{L} (h_o) [ \mathcal{L}(h) v_j] = \mathcal{L}(h) \mathcal{L}(h_o) v_j = \mathcal{L}(h_o) [ \mathcal{L}(h) v_j.
\]
Therefore $\mathcal{L}(h) v_j$ is an eigenvector of $\mathcal{L}(h_o)$ with associated eigenvalue $\lambda _j
\mathcal{L}(h_o$, so that $\mathcal{L}(h) v_j$ is proportional to $v_j$, i.e., $\mathcal{L}(h) v_j =  c_h v_j$ for some $c_h \in {\mathbb C}$, which turns to be an eigenvalue $c_h = \lambda _j \mathcal{L}(h)$ of $\mathcal{L}(h)$. If $\lambda _1 \mathcal{L}(h) = \lambda _2 \mathcal{L}(h)$ for $\forall h \in H$ then all $\mathcal{L}(h)$ are scalar matrices. In particular, $\mathcal{L}(h)$ are diagonal.

(ii) Note that the direct product $A = E \times E$ of elliptic curves coincides with their direct sum. If
\[
S^{-1} A := S^{-1} \widetilde{A} / S^{-1} \pi _1(A) = {\mathbb C}^2 / S^{-1} \pi _1(A),
\]
then $S^{-1} A \rightarrow S^{-1}A$ is an isomorphism of abelian surfaces and
\[
S^{-1}(A) = S^{-1} (E \times E) = S^{-1} [ (E \times \check{o}_E ) \times ( \check{o}_E \times E) ] =
\]
\[
= S^{-1} (E \times \check{o}_E) \times S^{-1} (\check{o}_E \times E) = F_1 \times F_2.
\]
The canonical projections ${\rm pr} _k : S^{-1}HS \rightarrow Aut(F_k)$ are group homomorphisms, according to
\[
{\rm pr} _k ( ( \tau _{V_1} \lambda _1 \mathcal{L}(g), \tau _{V_2} \lambda _2 \mathcal{L}(g))
( \tau _{U_1} \lambda _1 \mathcal{L}(h), \tau _{U_2} \lambda _2 \mathcal{L}(h)) =
\]
\[
= {\rm pr} _k ( \tau _{V_1 + \lambda _1 \mathcal{L} (g) U_1} ( \lambda \mathcal{L}(g). \lambda _1 \mathcal{L}(h)),
\tau _{V_2 + \lambda _2 \mathcal{L}(g) U_2} ( \lambda _2 \mathcal{L}(g). \lambda _2 \mathcal{L}(h)) ) =
\]
\[
= \tau _{V_k \lambda _k \mathcal{L}(g) U_k} ( \lambda _k \mathcal{L}(g) . \lambda _k \mathcal{L}(h)) =
( \tau _{V_k} \lambda _k \mathcal{L}(g)) ( \tau _{U_k} \lambda _j \mathcal{L}(h)) =
\]
\[
= {\rm pr} _k ( \tau _{V_1} \lambda _1 \mathcal{L}(g), \tau _{V_2} \lambda _2 \mathcal{L}(h)) .
( {\rm pr} _k ( \tau _{U_1} \lambda _1 \mathcal{L}(h), \tau _{U_2} \lambda _2 \mathcal{L}(h))
\]
for $\forall g,h \in H$ with $S^{-1}gS = \tau _{(V_1,V_2)} \mathcal{L}(S^{-1}gS)$, $S^{-1}hS = \tau _{(U_1,U_2)} \mathcal{L}(S^{-1}hS)$. The image ${\rm pr} _k (S^{-1}HS)$ of $S^{-1}HS$ is a finite subgroup of $Aut ( F_k)$ for $1 \leq k \leq 2$.

(iii) If $h_j = ( \alpha _j, \beta _j ) \in S^{-1}HS$ are liftings of the generators $\beta _j$ of ${\rm pr} _2 (S^{-1}HS)$, then $\ker( {\rm pr}_2)\langle h_1, \ldots , h_t \rangle$ is a subgroup of $S^{-1}HS$, as far as $\ker( {\rm pr}_2)$ is a normal subgroup of $S^{-1}HS$. For any ${\rm pr}_2 (S^{-1}hS) = \beta _1 ^{m_1} \ldots \beta _t ^{m_t}$ for some $m_i \in {\mathbb Z}$, one has $(S^{-1}HS)(h_1^{m_1} \ldots h_t ^{m_t}) \in \ker( {\rm pr}_2)$, so that $S^{-1}hS \in \ker( {\rm pr}_2) \langle h_1, \ldots , h_t \rangle $and $S^{-1}HS = \ker( {\rm pr}_2 ) \langle h_1, \ldots , h_t \rangle$.

(iv)  The subgroup $\ker( {\rm pr} _2)$ of $S^{-1}HS$ acts identically on $F_2$ and can be thought of as a subgroup of $Aut(F_1)$, ${\rm pr} _1( \ker( {\rm pr}_2)) \simeq \ker( {\rm pr}_2)$. Thus,
\[
S^{-1} A / \ker( {\rm pr}_2) \simeq \left[ F_1 / {\rm pr}_1 ( \ker( {\rm pr}_2 ) \right] \times F_2 = C_1 \times F_2
\]with an elliptic curve $C_1$ exactly when ${\rm pr}_1 ( \ker( {\rm pr}_2))$ is a translation subgroup of $Aut(F_1)$ or a rational curve $C_1$ for a non-translation subgroup ${\rm pr}_1 ( \ker( {\rm pr}_2))$ of  the automorphism group $Aut (F_1)$ of $F_1$.

(v) Since $\ker( {\rm pr}_2)$ is a normal subgroup of $S^{-1}HS$ with quotient
\[
S^{-1}HS / \ker( {\rm pr}_2) = [ \ker( {\rm pr}_2 ) \langle h_1, \ldots , h_t \rangle ] / \ker( {\rm pr}_2) =
\]
\[
= \langle h_1, \ldots , h_t \rangle / ( \langle h_1, \ldots , h_t \rangle \cap \ker( {\rm pr} _2 )) = G,
\]
one has
\[
A / H \simeq (S^{-1}A) / (S^{-1}HS) \simeq [ S^{-1}A / \ker( {\rm pr}_2) ] / [ S^{-1}HS / \ker( {\rm pr}_2) ] =
(C_1 \times F_2) / G.
\]
By the First Isomorphism Theorem, the epimorphism ${\rm pr}_2 : S^{-1}HS \rightarrow {\rm pr}_2 (S^{-1}HS)$ gives rise to an isomorphism
\[
\overline{ {\rm pr}_2} : S^{-1}HS / \ker( {\rm pr}_2)  = G  \longrightarrow {\rm pr}_2 (S^{-1}HS).
\]
The homomorphism ${\rm pr}_1 : S^{-1} HS \rightarrow Aut(F_1)$ induces a homomorphism
\[
\overline{ {\rm pr}_1} : S^{-1}HS / \ker( {\rm pr}_2) = G  \longrightarrow Aut(F_1) / {\rm pr}_1 ( \ker( {\rm pr}_2) ) \simeq  Aut (C_1).
\]
in the automorphism group of $C_1 = F_1 / {\rm pr} _1 ( \ker( {\rm pr}_2) )$.  It suffices to show that the kernel
\[
\ker( \overline{{\rm pr}_1}) = \{ S^{-1} h S \ker( {\rm pr }_2 ) \ \   \vert \ \  {\rm pr} _1 (S^{-1}hS) \in {\rm pr}_1 \ker( {\rm pr}_2 ) \} =
\]
\[
[ \ker( {\rm pr}_2 ) \ker( {\rm pr}_1) ] / \ker( {\rm pr}_2 ),
\]
since
\[
[ \ker( {\rm pr}_2 ) \ker( {\rm pr}_1) ] / \ker( {\rm pr}_2 ) \simeq \ker( {\rm pr}_1) / [ \ker( {\rm pi} _2) \cap \ker( {\rm pr}_1) ] = \ker( {\rm pr}_1).
\]
Indeed, if there exists $S^{-1} h_1 S ({\rm pr}_1 (S^{-1}hS), Id_{F_2} ) \in \ker( {\rm pr}_2)$ then
\[
S^{-1} (h_1^{-1} h) S = ( Id_{F_1}, {\rm pr}_2 ( S^{-1}hS)) \in S^{-1}HS \cap \ker( {\rm pr}_1),
\]
so that $S^{-1}h S \in S^{-1} h_1 S \ker( {\rm pr} _1) \subset \ker( {\rm pr}_2) \ker( {\rm pr}_1)$ for $\forall S^{-1} h S \ker( {\rm pr}_2) \in \ker( \overline{ {\rm pr}_1})$. Conversely, any element of  $[ \ker( {\rm pr}_2) \ker( {\rm pr}_1) ] / \ker( {\rm pr}_2)$ is of the form
\[
( g_1, Id_{F_2}) (Id _{F_1}, g_2) \ker( {\rm pr}_2) = (g_1, g_2) \ker( {\rm pr}_2)
 \]
 for some $(g_1, Id_{F_2}), (Id_{F_1}, g_2) \in S^{-1} HS \cap [ Aut(F_1) \times Aut(F_2)]$,  so that
\[
{\rm pr}_1 (g_1,g_2) = g_1 = {\rm pr}_1 ((g_1, Id _{F_2})) \in {\rm pr}_1 \ker( {\rm pr}_2)
\]
reveals that $(g_1,g_2) \ker( {\rm pr}_2 ) \in \ker( \overline{{\rm pr}_1})$.

\end{proof}

According to Lemma \ref{FiniteSubgroupsAutE}, the finite automorphism groups of elliptic curves have at most three generators. Combining with Lemma  \ref{SimultaneousDiagonalization}(iii), one concludes that the finite subgroups $H$ of $Aut(E \times  E)$ with abelian linear part $\mathcal{L}(H)$
have at most six generators. Their linear parts $\mathcal{L}(H)$ have at most two generators.

\begin{lemma}   \label{ FixedPointsUnderLifting}
Let $h = \tau _{(U,V)} \mathcal{L}(h)$ be an automorphism of $A = E \times E$ and $w=(u,v) \in {\mathbb C}^2 = \widetilde{A}$ be a lifting of $(u,v) + \pi _1 (A) = (U,V) \in A$. Then $h$ has no fixed points on $A$ if and only if for any $\mu = ( \mu _1, \mu _2) \in \pi _1(A)$ the affine-linear transformation
\[
\widetilde{h} (w,\mu) = \tau _{w + \mu} \mathcal{L}(h) \in Aff({\mathbb C}^2, R) := ({\mathbb C}^2,+) \leftthreetimes GL(2,R)
\]
has no fixed points on ${\mathbb C} ^2$.
\end{lemma}

\begin{proof}

The statement of the lemma is equivalent to the fact that $Fix _A(h) \ne q\emptyset$ exactly when $Fix _{{\mathbb C}^2} ( \widetilde{h} (w,\mu)) \neq \emptyset$ for some $\mu \in \pi _1(A)$. Indeed, if $(p,q) \in Fix _{{\mathbb C}^2} ( \widetilde{h}(w,\mu))$ then $(P,Q) = (p + \pi _1(E), q + \pi _1 (E)) \in A$ is a fixed point of $h$, according to
\[
h(P,Q) = \mathcal{L}(h)
\left( \begin{array}{c}
P  \\
Q
\end{array}  \right) +
\left( \begin{array}{c}
U  \\
V
\end{array}  \right) =
 \mathcal{L}(h)
\left( \begin{array}{c}
p   \\
q
\end{array} \right) +
\left( \begin{array}{c}
u  \\
v
\end{array} \right) +
\left( \begin{array}{c}
\mu _1  \\
\mu _2
\end{array}  \right)  +
\left( \begin{array}{c}
\pi _1 (E)  \\
\pi _1(E)
\end{array} \right)   =
\]
\[
= \left( \begin{array}{c}
p   \\
q
\end{array} \right) +
\left( \begin{array}{c}
\pi _1 (E)  \\
\pi _1(E)
\end{array} \right) =
\left( \begin{array}{c}
P  \\
Q
\end{array}  \right).
\]
Conversely, if
\[
\mathcal{L}(h)
\left( \begin{array}{c}
P  \\
Q
\end{array}  \right) +
\left( \begin{array}{c}
U  \\
V
\end{array}  \right) =
\left( \begin{array}{c}
P  \\
Q
\end{array}  \right),
\]
then for any lifting $(p,q) \in {\mathbb C}^2$ of $(P,Q) = (p + \pi _1(E), q + \pi _1 (E))$, one has
\[
\mathcal{L}(h)
\left( \begin{array}{c}
p   \\
q
\end{array} \right) +
\left( \begin{array}{c}
U   \\
V
\end{array} \right) +
\left( \begin{array}{c}
\pi _1 (E)  \\
\pi _1(E)
\end{array} \right) =
\left( \begin{array}{c}
p   \\
q
\end{array} \right) +
\left( \begin{array}{c}
\pi _1 (E)  \\
\pi _1(E)
\end{array} \right).
\]
In other words,
\[
\mu = \left( \begin{array}{c}
\mu _1  \\
\mu _2
\end{array}  \right) :=
\mathcal{L}(h)
\left( \begin{array}{c}
p   \\
q
\end{array} \right) +
\left( \begin{array}{c}
u   \\
v
\end{array} \right) -
\left( \begin{array}{c}
p   \\
q
\end{array} \right) \in
\left( \begin{array}{c}
\pi _1 (E)  \\
\pi _1(E)
\end{array} \right)
\]
and $(p,q) \in Fix _{{\mathbb C}^2} ( \widetilde{h} (w, - \mu ))$.

\end{proof}

Now we are ready to characterize the automorphisms $h \in Aut(A)$ without fixed points

\begin{lemma}   \label{FixedPointFree}
An automorphism $h = \tau _{(U,V)} \mathcal{L}(h) \in Aut (A) \setminus ( \mathcal{T}_A, +)$ acts without fixed points on $A = E \times E$ if and only if its linear part $\mathcal{L}(h)$ has eigenvalues $\lambda _1 \mathcal{L}(h) =1$, $\lambda _2 \mathcal{L}(h) \neq 1$ and
\[
\mathcal{L}(h)
\left( \begin{array}{c}
u  \\
v
\end{array}  \right) \neq \lambda _2
\left( \begin{array}{c}
u  \\
v
\end{array}  \right)
\]
 for any lifting  $(u,v) \in {\mathbb C}^2$ of  $(u + \pi _1(E), v + \pi _1(E)) = (U,V)$.
\end{lemma}

\begin{proof}

The fixed points $(P,Q) \in A$ of $h = \tau _{(U,V)} \mathcal{L}(h)$ are  described  by the equality
\begin{equation}  \label{FixedPointEquation}
( \mathcal{L}(h) - I_2)
\left( \begin{array}{c}
P  \\
Q
\end{array} \right) =
\left( \begin{array}{c}
-U  \\
-V
\end{array} \right).
\end{equation}

If $\det ( \mathcal{L}(h) -I_2) \neq 0$ or $1 \in {\mathbb C}$ is not an eigenvalues of $\mathcal{L}(h)$, then consider the adjoint matrix
\[
( \mathcal{L}(h) - I_2)^* = \left( \begin{array}{rr}
d  &  -b  \\
-c  &  a
\end{array}  \right) \in R_{2 \times 2} \ \ \mbox{  of   }
\]
\[
\mathcal{L}(h) - I_2 = \left( \begin{array}{rr}
a  &  b  \\
c  &  d
\end{array}  \right) \in R_{2 \times 2}.
\]
According to $( \mathcal{L}(h) - I_2 )^* ( \mathcal{L}(h) -I_2) = \det( \mathcal{L}(h) -I_2) I_2 = ( \mathcal{L}(h) - I_2) ( \mathcal{L}(h) - I_2 )^*$,
one obtains
\begin{equation}   \label{DiagonalizedFixedPointEquation}
\det( \mathcal{L}(h) - I_2)
\left( \begin{array}{c}
P   \\
Q
\end{array}  \right) =
( \mathcal{L}(h) - I_2) ^* ( \mathcal{L}(h) - I_2 )
\left( \begin{array}{c}
u  \\
v
\end{array}  \right) =
- ( \mathcal{L}(h) - I_2 )^*
\left( \begin{array}{c}
U  \\
V
\end{array}  \right).
\end{equation}
Then for an arbitrary lifting $(u_1,v_1) \in {\mathbb C}^2$ of
\[
\left( \begin{array}{c}
u_1 + \pi _1(E)   \\
v_1 + \pi _1 (E)
\end{array}  \right) =
\left( \begin{array}{c}
U_1 \\
V_1
\end{array}  \right) := - ( \mathcal{L}(h) - I_2) ^*
\left( \begin{array}{c}
U  \\
V
\end{array}  \right),
\]
the point
\[
(p,q) = \left( \frac{u_1}{\det( \mathcal{L}(h) - I_2)}, \frac{v_1}{\det ( \mathcal{L}(h) - I_2)} \right) \in {\mathbb c}^2
\]
descends to $(P,Q) = ( p + \pi _1 (E), q + \pi _1 (E))$, subject to (\ref{DiagonalizedFixedPointEquation}). As a result,
\[
( \mathcal{L}(h) - I_2 )
\left( \begin{array}{c}
P  \\
Q
\end{array}  \right) = \frac{1}{\det( \mathcal{L}(h) -I_2)} ( \mathcal{L}(h) -I_2)
\left( \begin{array}{c}
u_1  \\
v_1
\end{array}  \right) +
\left( \begin{array}{c}
 \pi _1 (E)  \\
 \pi _1 (E)
 \end{array}  \right) =
 \]
 \[
=  \left( \begin{array}{c}
 u  \\
 v
 \end{array}  \right) +
\left( \begin{array}{c}
 \pi _1 (E)  \\
 \pi _1 (E)
 \end{array}  \right)
 \]
 and $(P,Q) \in Fix _A(h)$.

 From now on, let us suppose that the linear part $\mathcal{L}(h) \in GL(2,R)$ of $h \in Aut (A) \setminus ( \mathcal{T}_A, +)$ has eigenvalues $\lambda _1 \mathcal{L}(h) =1$ and $\lambda _2 \mathcal{L}(h) = \det \mathcal{L}(h) \in R^* \setminus \{ 1\}$. We claim that a lifting $(u,v) \in {\mathbb c}^2$ of $( u + \pi _1(E), v + \pi _1 (E)) = (U,V) \in A$ satisfies
 \[
 \mathcal{L}(h)
 \left( \begin{array}{c}
 u  \\
 v
 \end{array}  \right) = \lambda _2 \mathcal{L}(h)
  \left( \begin{array}{c}
 u  \\
 v
 \end{array}  \right)
 \]
 if and only if there exists $(p,q) \in {\mathbb C}^2$ with
 \[
 ( \mathcal{L}(h) - I_2)
  \left( \begin{array}{c}
 p  \\
 q
 \end{array}  \right) =
  \left( \begin{array}{c}
 - u  \\
 - v
 \end{array}  \right),
 \]
 which amounts to $(p,q) \in Fix _{{\mathbb C}^2} ( \tau _{(u,v)} \mathcal{L}(h))$. To this end, let us view $\mathcal{L}(h)  : {\mathbb C}^2 \rightarrow {\mathbb C}^2 $ as a linear  operator in ${\mathbb C} ^2$ and reduce the claim to the equivalence of $( -u,-v) \in \ker( \mathcal{L}(h) - \lambda _2 \mathcal{L}(h) I_2)$ with $(-u,-v) \in Im ( \mathcal{L}(h) - I_2)$. In other word, the statement of the lemma reads as $\ker ( \mathcal{L}(h) - \lambda _2 \mathcal{L}(h) I_2) = Im ( \mathcal{L}(h) - I_2)$ for the linear operators $\mathcal{L}(h) - \lambda _2 \mathcal{L}(h) I_2$ and $\mathcal{L}(h) - I_2$ in ${\mathbb C}^2$. By Hamilton -Cayley  Theorem, $\mathcal{L}(h) \in {\mathbb C} _{2 \times 2}$ is a root of its characteristic polynomial
 \[
 \mathcal{X} _{\mathcal{L}(h)} ( \lambda ) = ( \lambda - \lambda _1 \mathcal{L}(h)) ( \lambda -1).
 \]
 Thus,
 \[
 ( \mathcal{L}(h) - \lambda _2 \mathcal{L}(h) I_2) Im ( \mathcal{L}(h) - I_2) = \{ (0,0) \}
 \]
 is the zero subspace of ${\mathbb C}^2$  and $Im ( \mathcal{L}(h) - I_2 ) \subseteq \ker( \mathcal{L}(h) - \lambda _2 \mathcal{L}(h) I_2)$. However, $\dim Im ( \mathcal{L}(h) - I_2) = {\rm rk} ( \mathcal{L}(h) - I_2 ) =1$ and
 \[
 \dim \ker( \mathcal{L}(h) - \lambda _2 \mathcal{L}(h)) = 2 - {\rm rk} ( \mathcal{L}(h) - \lambda _2 \mathcal{L}(h) I_2) = 2-1=1,
  \]
  so that $Im ( \mathcal{L}(h) - I_2) = \ker( \mathcal{L}(h) - \lambda _2 \mathcal{L}(h) I_2)$.

\end{proof}

\begin{corollary}  \label{HyperellipticAH}
Let $H = \mathcal{T}(h) \langle h_o \rangle$ be a finite subgroup of $Aut(A)$ for some $h_o \in H$ with
\[
\lambda _1 \mathcal{L}(h_o) =1,  \ \  \lambda _2 \mathcal{L}(h_o) = e^{\frac{ 2 \pi i}{s}},  \ \ s \in \{ 2, 3 ,4,  6 \},
 \]
 $S \in GL(2, {\mathbb Q} ( \sqrt{-d}))$ be a diagonalizing matrix for $h_o$ and
 \[
 S^{-1} h_o S = \left( \tau _W, e^{\frac{ 2 \pi i}{s}} \right)
 \]
 after appropriate choice of an origin of $S^{-1}A = F_1 \times F_2$, $F_1 = S^{-1}( E \times \check{o}_E)$, $F_2 = S^{-1} ( \check{o}_E \times E)$.
 Then $A/H$ is a hyper-elliptic surface if and only if the kernel $\ker( {\rm pr}_1)$ of the first canonical projection ${\rm pr}_1 : S^{-1} HS \rightarrow Aut(F_1)$  is a translation subgroup of $Aut(F_2)$. If so, then
 \[
 S^{-1}A / [ \ker( {\rm pr}_2) \ker( {\rm pr}_1) ] \simeq C_1 \times C_2
 \]
 for some elliptic curves $C_1, C_2$ and
 \[
 A/H \simeq ( C_1 \times C_2 ) / G,
 \]
 where the group $G$ is isomorphic to some of the groups
 \[
 G_2^{HE}  = \langle ( \tau _{U_1}, -1 ) \rangle \simeq {\mathbb C}_2
 \]
 with $U_1 \in C_1^{2-{\rm tor}} \setminus  \{ \check{o}_{C_1} \}$,
 \[
 G_{2,2}^{HE} = \langle \tau _{(P_1, Q_1)} \rangle \times \langle ( \tau _{U_1}, -1) \rangle \simeq {\mathbb C}_2 \times {\mathbb C}_2
 \]
 with $P_1, U_1 \in C_1 ^{2-{\rm tor}} \setminus \{ \check{o}_{C_1} \}$, $Q_1 \in C_2 ^{2-{\rm tor}}$,
 \[
 G_3 ^{HE} = \langle ( \tau _{U_1}, e^{\frac{ 2 \pi i}{3}} ) \rangle \simeq {\mathbb C}_3
 \]
 with $R = \mathcal{O}_{-3}$, $ U_1 \in C_1 ^{3-{\rm tor}} \setminus C_1 ^{2-{\rm tor}}$,
 \[
 G_{3,3}^{HE} = \langle \tau _{(P_1,Q_1)}  \rangle \times \langle \left(  \tau _{U_1}, e^{\frac{ 2 \pi i}{3}} \right) \rangle \simeq
 {\mathbb C}_3 \times {\mathbb C}_3
 \]
 with $R = \mathcal{O}_{-3}$, $P_1, U_1 \in C_1 ^{3-{\rm tor}} \setminus C_1 ^{2-{\rm tor}}$, $Q \in C_2 ^{3-{\rm tor}} \setminus \{ \check{o}_{C_2} \}$,
 \[
 G_4^{HE} = \langle ( \tau _{U_1}, i ) \rangle \simeq {\mathbb C}_4
 \]
 with $R = {\mathbb Z}[i]$, $U_1 \in C_1 ^{4-{\rm tor}} \setminus ( C_1 ^{2-{\rm tor}} \cup C_1^{3-{\rm tor}} )$,
 \[
 G_{4,4} ^{HE} = \langle \tau _{(P_1,Q_1)} \rangle \times \langle ( \tau _{U_1}, i) \rangle \simeq {\mathbb C}_2  \times {\mathbb C}_4
 \]
 with $R = {\mathbb Z}[i]$, $P_1 \in C_1 ^{2-{\rm tor}} \setminus \{ \check{o}_{C_1} \}$, $Q_1 \in C_2 ^{(1_i)-{\rm tor}} \setminus \{ \check{o}_{C_2} \}$,
 $U_1 \in C_1 ^{4-{\rm tor}} \setminus ( C_1 ^{2-{\rm tor}} \cup C_1 ^{3-{\rm tor}} )$,
 \[
 G_6^{HE} = \langle \left(  \tau _{U_1}, e^{\frac{ \pi i}{3}} \right) \rangle \simeq {\mathbb C}_6
 \]
 with $R = \mathcal{O}_{-3}$, $U_1 \in C_1 ^{6-{\rm tor}} \setminus ( C_1 ^{3-{\rm tor}} \cup C_1^{4-{\rm tor}} \cup C_1 ^{5-{\rm tor}} )$.

 In the notations from Proposition \ref{HC1}, $A/H$ is a hyper-elliptic surface exactly when $H \simeq S^{-1}HS$ is isomorphic to some of the groups:
 \[
 H_2 ^{HE} (m,n) = \langle ( \tau _{M_j}, Id _{F_2}), (Id_{F_1}, \tau _{N_k} ), ( \tau _W, -1) \ \ \vert \ \  1 \leq j \leq m, \ \ 1 \leq k \leq n \rangle
 \]
 with $W \not \in \ker( {\rm pr}_2)$, $2W \in \ker( {\rm pr}_2)$, $\mathcal{L}(H_2^{HE} (m,n) ) \simeq H_{C1}(1) \simeq {\mathbb C}_2$,
 \[
 H_{2,2}^{HE} (m,n) = \langle ( \tau _{M_j},  Id _{F_2}),  \ \ (Id_{F_1},  \tau _{N_k}),  \ \ \tau _{(X,Y)}, \ \  (\tau _W, -1)
 \ \ \vert \ \  1 \leq j \leq m, \ \ 1 \leq k \leq n  \rangle
 \]
 with $2X. 2W \in \ker( {\rm pr}_2)$, $X,W \not \in \ker( {\rm pr}_2)$, $2Y \in \ker( {\rm pr}_1)$, $Y \not \in \ker( {\rm pr}_1)$,  \\
 $\mathcal{L} (  H_{2,2}^{HE} (m,n) ) \simeq H_{C1} (1) \simeq {\mathbb C}_2$
 \[
 H_3^{HE} (m,n) = \langle ( \tau _{M_j}, If_{F_2}), \ \ (Id_{F_1}, \tau _{N_k}), \ \ \left( \tau _W, e^{\frac{2 \pi i}{3}} \right)
 \ \ \vert \ \  1 \leq j \leq m, \ \ 1 \leq k \leq n  \rangle
 \]
 with $R = \mathcal{O}_{-3}$, $3 W \in \ker( {\rm pr}_2)$, $2 W \not \in \ker({\rm pr} _2)$,
 $\mathcal{L} (H_3^{HE} (m,n) ) \simeq H_{C1}(2) \simeq {\mathbb C}_3$,
 \[
 H_{3,3}^{HE} (m,n) = \langle ( \tau _{M_j}, Id _{F_2}), \, (Id_{F_1}, \tau _{N_k}), \,  \tau _{(X,Y)}, \,  \left( \tau _W, e^{\frac{ 2 \pi i}{3}} \right)
 \, \vert \,   1 \leq j \leq m, \,  1 \leq k \leq n  \rangle
 \]
 with $R = \mathcal{O}_{-3}$, $3X, 3W \in \ker( {\rm pr}_2)$, $2X, 2W \not \in  \ker( {\rm pr}_2)$, $3 Y \in \ker( {\rm pr} _1)$,
 $Y \not \in  \ker( {\rm pr}_1)$, $\mathcal{L}( H_{3,3}^{HE} (m,n)) \simeq H_{C1}(2) \simeq {\mathbb C}_3$,
 \[
 H_4^{HE}(m,n) = \langle ( \tau _{M_j}, Id_{F_2}), \ \ (Id_{F_1}, \tau _{N_k}), \ \  (\tau _W, i)
 \ \ \vert \ \  1 \leq j \leq m, \ \ 1 \leq k \leq n  \rangle
 \]
 with $R = {\mathbb Z}[i]$, $4W \in \ker( {\rm pr} _2)$, $2W, 3W \not \in \ker( {\rm pr}_2)$,
 $\mathcal{L}( H_4^{HE}(m,n)) \simeq H_{C1} (e) \simeq {\mathbb C}_4$,
 \[
 H_{4,4}^{HE}(m,n) = \langle ( \tau _{M_j}, Id_{F_2}), \ \ ( Id _{F_1}, \tau _{N_k}), \ \ \tau _{(X,Y)}, \ \ ( \tau _W, i)
 \ \ \vert \ \  1 \leq j \leq m, \ \ 1 \leq k \leq n  \rangle
 \]
 with $R = {\mathbb Z}[i]$, $2X \in \ker( {\rm pr}_2)$, $X \not \in \ker( {\rm pr}_2)$, $(1_i) Y \in \ker( {\rm pr}_1)$, $Y \not \in \ker( {\rm pr} _1)$,
 $4W \in \ker( {\rm pr}_2)$, $2W, 3W \not \in \ker( {\rm pr}_2)$,
 $\mathcal{L}( H_{4,4}^{HE} (m,n) \simeq H_{C1} (3) \simeq {\mathbb C}_4$,
 \[
 H _6 ^{HE} (m,n) = \langle ( \tau _{M_j}, Id _{F_2} ), \ \ ( Id _{F_1}, \tau _{N_k} ), \left( \tau _W, e^{\frac{\pi i}{3}} \right)
 \ \ \vert \ \  1 \leq j \leq m, \ \ 1 \leq k \leq n  \rangle
 \]
 with $R = \mathcal{O}_{-3}$, $6W \in \ker( {\rm pr} _2)$, $3W, 4W, 5W \not \in \ker( {\rm pr}_2)$, where $m, n \in \{ 0, 1, 2 \}$.
  \end{corollary}

\begin{proof}

In the notations from Lemma \ref{SimultaneousDiagonalization}, the kernel $\ker( {\rm pr}_2)$ of the second canonical projection ${\rm pr}_2 : S^{-1}HS \rightarrow Aut(F_2)$ is a translation group, so that
\[
S^{-1}A \rightarrow S^{-1}A / \ker( {\rm pr}_2) = C_1 \times F_2
 \]
 is unramified and $C_1$ is an elliptic curve. Thus, the covering $A \rightarrow A/H$ is unramified if and only if $C_1 \times F_2 \rightarrow (C_1 \times F_2) / G \simeq A/H$ is unramified. In other words, $A/H$ is a hyper-elliptic surface   exactly when the group $G$ has no fixed point on $C_1 \times F_2$. For any $g \in G$ with $\mathcal{L}(g) \neq I_2$ the second component $\overline{ {\rm pr}_2} (g) = \tau _{V_2} e^{ \frac{ 2 \pi i j}{s}}$ for some $1 \leq j \leq s-1$, $V_2 \in F_2$ has a fixed point on $F_2$. Towards $Fix _{C_1 \times F_2} (g) = \emptyset$ one has to have $\overline{ {\rm pr}_1} (g) \neq Id _{C_1}$, so that $\ker ( \overline{ {\rm pr}_1}) \subseteq \mathcal{T}(G) = G \cap \ker( \mathcal{L})$ and $\ker( {\rm pr}_1) \subseteq \mathcal{H} = H \cap \ker( \mathcal{L})$ are translation groups. The covering $C_1 \times F_2 \rightarrow (C_1 \times F_2 ) / \ker( \overline{ {\rm pr}_1} ) = C_1 \times C_2$ is unramified, $C_2$ is an elliptic curve and $A/H$ is a hyper-elliptic surface exactly when $G_o = G / \ker( \overline{ {\rm pr}_1} )$ has no fixed points on $(C_1 \times F_2) / \ker( \overline{ {\rm pr}_1})$. The canonical projections
 \[
 \overline{ {\rm pr}_1} : G_o \longrightarrow Aut(C_1) \ \ \mbox{  and   } \ \
 \overline{ {\rm pr}_2} : G_o \longrightarrow Aut(C_2)
 \]
 are injective. Since $\overline{ {\rm pr}_1} (G_o)$ is a translation subgroup of $Aut (C_1)$, the group $G_o \simeq \overline{ {\rm pr}_1}$ is abelian and has at most two generators.  As a result, $\overline{ {\rm pr}_2} (G_o) \simeq G_o$ is an abelian subgroup of $Aut (C_2)$ with at most two generators and non-trivial linear part $\mathcal{L} \left( \overline{ {\rm pr}_2} (G_o) \right) = \langle e^{\frac{ 2 \pi i}{s}} \rangle \simeq {\mathbb C}_s$ for some $s \in \{ 2, 3, 4, 6 \}$. According to Lemma \ref{FiniteSubgroupsAutE},
 \[
 \overline{ {\rm pr}_2} (G_o) \simeq \langle \tau _{Q_1} \rangle \times \langle e^{\frac{ 2 \pi i}{s}} \rangle \simeq {\mathbb C}_m \times {\mathbb C}_s
 \]
 for some $Q_1 \in C_2$ with $\tau _{Q_1} = {\rm Ad} _{e^{\frac{ 2 \pi i}{s}}} ( \tau _{Q_1} ) = \tau _{e^{\frac{2 \pi i}{s}} Q_1}$. In other words,  the point $Q_1 \in C_2^{\left( e^{\frac{ 2 \pi i}{s}} -1 \right)-{\rm tor}} \setminus \{ \check{o}_{C_2} \}$. If $s=2$ then any $Q_1 \in C_2^{2-{\rm tor}}$ works out and the order of $Q_1 \in (C_2,+)$ is $m=2$.

 For $s=3$ note that the endomorphism ring of $C_2$ is $End(C_2) = \mathcal{O}_{-3}$. Therefore the fundamental group $\pi _1(C_2) = c( {\mathbb Z} + \tau {\mathbb Z})$ for some $\tau \in {\mathbb Q}( \sqrt{-3})$ and $c \in {\mathbb C}^*$. By $c \in \pi _1(C_2)$ and $e^{\frac{ \pi i}{3}} \in End(C_2)$ one has $e^{\frac{ \pi i}{3}} c \in \pi _1 (C_2)$. Due to the linear independence of $c$ and $e^{\frac{ \pi i}{3}}$ over ${\mathbb Z}$, one has $\pi _1 (C_2) = c {\mathbb Z} + e^{\frac{ \pi i}{3}} c {\mathbb Z} = c \mathcal{O}_{-3}$. For $\alpha = e^{\frac{ 2 \pi i}{3}} -1 = - \frac{3}{2} + \frac{\sqrt{3}}{2} i$ the equation
 \[
 \alpha \left( x + e^{\frac{ \pi i}{3}} y \right) = \left(  a + e^{\frac{ \pi i}{3}} b \right)  c \ \ \mbox{  for some } \ \ a,b \in {\mathbb Z}
 \]
 has a solution $x = \frac{ -a+b}{3}$, $y = \frac{ -a-2b}{3}$. Note that $x ({\rm mod} {\mathbb Z}) \equiv y ({\rm mod} {\mathbb Z})$ and
 \[
 \left(  x + e^{\frac{ \pi i}{3}} y \right)  c \left( {\rm mod} {\mathbb Z} + e^{\frac{ \pi i}{3}} {\mathbb Z} \right) =
 \left( x + e^{\frac{ \pi i}{3}} \right) \left( {\rm mod} \pi _1(C_2) \right) \in
 \]
 \[
 \left \{ \check{o}_{C_2}, \ \ \pm \left( 1 + e^{\frac{ \pi i}{3}} \right) ({\rm mod} \pi _1 (C_2)  )\right \} = C_2 ^{3-{\rm tor}},
 \]
whereas $C_2^{\alpha-{\rm tor}} = C_2 ^{3-{\rm tor}}$ and $m=3$. Thus, $Q_1 \in C_2^{3-{\rm tor}} \setminus \{ \check{o}_{C_2} \}$ in the case of $s=3$.

If $s=4$ then $End(C_2) = {\mathbb Z}[i]$ and $\pi _1(C_2) = c {\mathbb Z}[i]$ for some $c \in {\mathbb C}^*$. The equation $(i-1) (x + iy) c = (a+bi) c$ for some $a,b \in {\mathbb Z}$ has a solution $x = \frac{ -a+b}{2}$, $y= \frac{-a-b}{2}$ with
\[
(x+iy) c ({\rm mod} {\mathbb Z}[i]) = x + iy ( {\rm mod} \pi _1 (C_2)) \in
\]
\[
\left \{ \check{o}_{C_2}, \left( \frac{1+i}{2} \right) c ({\rm mod} \pi _1 (C_2) ) \right \} = C_2 ^{(i+1)-{\rm tor}},
\]
so that $m=4$ and $Q_1 \in C_2^{(i+1)-{\rm tor}} \setminus \{ \check{o}_{C_2} \}$.

For $s=6$ one has $e^{\frac{ \pi i}{3}} -1 = e^{\frac{ 2 \pi i}{3}}$ and $C_2 ^{e^{\frac{ 2 \pi i}{3}}-{\rm tor}} = \{ \check{o}_{C_2} \}$, Therefore
$\overline{ {\rm pr}_2} (G_o) = \langle e^{\frac{ \pi i}{3}} \rangle \simeq {\mathbb C}_6$ in this case.

The restrictions on $P_1, U_1 \in C_1$ arise  from the isomorphism $G_o \simeq  \overline{ {\rm pr}_1} (G_o) \simeq \overline{ {\rm pr}_2} (G_o)$. Namely, $\left( \tau _{U_1}, e^{\frac{ 2 \pi i}{s}} \right) \in G_o$ with $\overline{ {\rm pr}_2} \left( \tau _{U_1}, e^{\frac{ 2 \pi i}{s}} \right) = E^{\frac{ 2 \pi i}{s}}$ of order $s \in \{ 2, ,3 4, 6 \}$ has to have $\tau _{U_1} = \overline{ {\rm pr}_1} \left( \tau _{U_1}, e^{\frac{ 2 \pi i}{s}} \right) \in (C_1,+)$ of order $s$. That amounts to $U_1 \in C_1 ^{s-{\rm tor}}$ and $U_1 \not \in C_1^{t-{\rm tor}}$ for all $1 \leq t < s$. If $\overline{ {\rm pr}_2} (G_o) = \langle \tau _{Q_1} \rangle \times \langle e^{\frac{ 2 \pi i}{s}} \rangle$ with $Q_1 \neq \check{o}_{C_2}$ then the order $m$ of $Q_1 \in C_2$ has to coincide with the order of $P_1 \in C_1$.

In order to relate the classification $G_s^{HE}$, $G_{m,s} ^{HE}$ of $G_o$ with the classification of the groups $H_s^{HE} (m,n)$, $H_{s,s} ^{HE} (m,n)$ of $H \simeq S^{-1}HS$, note that $P_1, U_1 \in C_1 ^{p-{\rm tor}} \setminus C_1 ^{q-{\rm tor}}$ for some natural numbers $p >q$ exactly when the corresponding liftings $X,W \in F_1$ are subject to $pX, pQ \in \ker( {\rm pr}_2)$, $qX, qW \not \in \ker( {\rm pr}_2)$. Similarly, $Q_1 \in C_2^{p-{\rm tor}} \setminus C_2 ^{q-{\rm tor}}$ for $p,q \in  {\mathbb N}$, $P >q$ if and only if an arbitrary lifting $Y \in F_2$ satisfies $pY \in \ker( {\rm pr} _1)$, $qY \not \in \ker( {\rm pr}_1)$.

\end{proof}

Bearing in mind that $A/H$ with $H = \mathcal{T}(H) \langle h_o \rangle$, $\lambda _1 \mathcal{L}(h_o) =1$, $\lambda _2 \mathcal{L}(h_o) \in R^* \setminus \{ 1 \}$ is either hyper-elliptic or a ruled surface with an elliptic base, one obtains the following

\begin{corollary}   \label{RuledAHEllipticBase}
Let $H = \mathcal{T}(H) \langle h_o \rangle$ be a finite subgroup of $Aut(A)$ for some $h_o \in H$ with $\lambda _1 \mathcal{L}(h_o) =1$, $\lambda _2 \mathcal{L}(h_o) = e^{\frac{ 2 \pi i}{s}}$, $s \in \{ 2, 3, 4, 6 \}$, $S \in GL(2, {\mathbb Q}( \sqrt{-d}))$ be a diagonalizing matrix for $h_o$ and
\[
S^{-1} h_o S = \left( \tau _{U_1}, e^{\frac{ 2 \pi i}{s}} \right)
\]
after an appropriate choice of an origin of $S^{-1}(A) = F_1 \times F_2$, $F_1 = S^{-1} ( E \times \check{o}_E )$, $F_2 = S^{-1} ( \check{o}_E \times E)$.
Then $A/H$ is a ruled surface with an elliptic base if and only if the kernel $\ker( {\rm pr} _1)$  of the first canonical projection ${\rm pr}_1 : S^{-1}HS \rightarrow Aut(F_1)$ contains a non-translation element $S^{-1} h S = \left( Id_{F_1}, \tau _{V_2}  e^{\frac{ 2 \pi ik}{s}} \right)$ for some $1 \leq k \leq s-1$, $V_2 \in F_2$.

In the notations from Lemma \ref{SimultaneousDiagonalization}, the quotient  $A/H \simeq (C_1 \times F_2) / G$ of  the split abelian surface $C_1 \times F_2 = S^{-1}A / \ker( {\rm pr}_2 )$ by  its finite automorphism group $G = S^{-1}HS / \ker( {\rm pr}_2 )$  is a ruled surface with an elliptic base exactly when $G$ is isomorphic to some of the groups
\[
G_2^{RE}(m,n) = \langle \tau _{(P_1,Q_1)}, \ \ \tau _{(P_2, Q_2)}, \ \ \rangle \rtimes \langle ( \tau _{U_1}, -1 ) \rangle \simeq
( {\mathbb C}_m \times {\mathbb C}_n) \rtimes _{(-1,-1)}  {\mathbb C}_2 =
\]
\[
= ( \langle a \rangle \times \langle b \rangle ) \rtimes _{(-1,-1)} \langle c \rangle =
\langle a, \ \ b, \ \ c \ \ \vert \ \  a^m=1, \ \ b^n=1, \ \ cac^{-1} = a^{-1}, \ \ cbc^{-1} = b^{-1} \rangle
\]
with $\tau _{U_1} \in \left( \langle \tau _{P_1}, \tau _{P_2} \rangle , + \right) \simeq {\mathbb C}_m \times {\mathbb C}_n$ for some $m,n \in {\mathbb N}$,
\[
G_3^{RE}(m,j) = \langle \tau _{(P_1,Q_1)} \rangle \rtimes \langle \left( \tau _{U_1}, e^{\frac{2 \pi i}{3}} \right) \rangle \simeq
{\mathbb C}_m \rtimes _j {\mathbb C}_3 =
\]
\[
= \langle a \rangle \rtimes _j \langle c \rangle =
\langle a, \ \  c \ \ \vert \ \ a^m=1, \ \ c^3=1, \ \ cac^{-1} = a^j \rangle
\]
with $R = \mathcal{O}_{-3}$, $2 U_1 \in \left( \langle \tau _{P_1} \rangle , + \right) \simeq {\mathbb C}_m$ for some $j \in {\mathbb Z}_m ^*$ of order $1$ or $3$,
\[
G_4 ^{RE} (m,j) = \langle \tau _{(P_1,Q_1)} \rangle \rtimes \langle ( \tau _{U_1}, i) \rangle
\simeq {\mathbb C}_m \rtimes _j {\mathbb C}_4 =
\]
\[
= \langle a \rangle \rtimes _j \langle c \rangle =
\langle a, \ \ c \ \ \vert \ \  a^m=1, \ \ c^4=1,  \ \ cac^{-1} = a^j \rangle
\]
with $R = {\mathbb Z}[i]$ for some $j \in {\mathbb Z}_m^*$ or order $1, 2$ or $4$,
\[
G_6^{RE} (m,j) = \langle \tau _{(P_1,Q_1)} \rangle \rtimes \langle \left( \tau _{U_1}, e^{\frac{ \pi i}{3}} \right) \rangle
\simeq {\mathbb C}_m \rtimes _j {\mathbb C}_6 =
\]
\[
= \langle a \rangle \rtimes _j \langle c \rangle =
\langle a, \ \ c \ \ \vert \ \  a^m=1, \ \ c^6=1, \ \ cac^{-1} = a^j \rangle
\]
with $R = \mathcal{O}_{-3}$ and at least one of $3U_1, 4U_1$ or $5U_1$ from $( \langle \tau _{P_1} \rangle , +)$ for some $j \in {\mathbb Z} _m^*$ of order $1, 2, 3$ or $6$.
\end{corollary}

The classification of $G$ is an immediate application of the group isomorphism $\overline{ {\rm pr}_2} : G \rightarrow {\rm pr}_2 ( S^{-1}HS)$ from Lemma \ref{SimultaneousDiagonalization} (v) and the classification of $Aut(F_2)$, given in Lemma \ref{FiniteSubgroupsAutE}.

\begin{lemma}   \label{PreEnriquesH}
Let $G$ be a finite subgroup of $GL(2,R)$ with $G \cap SL(2,R) \neq \{ I_2 \}$, such that any $g \in G \setminus SL(2,R) \neq \emptyset$ has an eigenvalue $\lambda _1 (g) =1$. Then:

 (i) $G = G_s = \langle g_s, g_o \rangle$ is generated by $g_s \in SL(2,R)$ of order $s \in \{ 2, 3, 4, 6 \}$ and $g_o \in GL(2,R)$ with $\det (g_o) = -1$, $\tr(g_o) = 0$, subject to $g_o g_s g_o^{-1} = g_s^{-1}$;

 (ii) and $g \in G \setminus SL(2,R)$ has eigenvalues $\lambda _1 (g) =1$ and $\lambda _2 (g) = -1$;

 (iii) the group
 \[
 G_s = \langle g_s, \ \ g_o \ \ \vert \ \  g_s^s = I_2, \ \ g_o^2 = I_2, \ \ g_o g_s g_o^{-1} = g_s^{-1} \rangle \simeq \mathcal{D}_s
 \]
 is dihedral of order $2s$  for $s \in \{ 3, 4, 6 \}$ or the Klein group $G_2 \simeq {\mathbb C}_2 \times {\mathbb C}_2$ for $s=2$.
\end{lemma}

\begin{proof}

Note that $g \in G \setminus SL(2,R)$ has an eigenvalue $1$ exactly when the characteristic polynomial $\mathcal{X}_g ( \lambda ) = \lambda ^2 - \tr(g) \lambda + \det(g) \in R[ \lambda]$ of $g$ vanishes at $\lambda =1$. This is equivalent to
\[
\tr(g) = \det(g) + 1.
\]
If $- I_2 \not \in G$, then Proposition \ref{K} specifies that $G \cap SL(2,R) = \langle g_3 \rangle \simeq {\mathbb C}_3$. In the notations from Proposition \ref{HC3}, all the finite subgroups $H_{C3}(i) = [ H_{C3}(i) \cap SL(2,R)] \langle g_o \rangle$ of $GL(2,R)$ with $H_{C3}(i) \cap SL(2,R) \simeq {\mathbb C}_3$, such that $g_o$ has an eigenvalue $\lambda _1 (g_o) =1$ are isomorphic to
\[
H_{C3} (4) = \langle g, \ \  g_o \ \  g^3 = g_o^3 = I_2, \ \ g_o g g_o^{-1} = g^{-1} \rangle \simeq S_3 \simeq \mathcal{D}_3
\]
for some $g \in SL(2,R)$ with $\tr(g) = -1$ and $\lambda _1 (g_o) = 1$, $\lambda _2(g_o) = -1$. Since $g_o$ is of order $2$, the complement
\[
H_{C3}(4) \setminus SL(2,R) = \langle g \rangle g_o = \{ g^j g_o \ \ \vert \ \ 0 \leq j \leq 2 \}
\]
consists of matrices $g^jg_o$ of determinant $\det(g^jg_o) = \det( g_o) = -1$ and $g \in H_{C3} (4) \setminus SL(2,R)$ has as eigenvalue $1$ exactly when $\tr(g^jg_o) = 0$. Bearing in mind the invariance of the trace under conjugation, one can consider
\[
g = \left( \begin{array}{cc}
e^{\frac{ 2 \pi i}{3}}   &   0  \\
0  &  e^{- \frac{ 2 \pi i}{3}}
\end{array}  \right) \ \ \mbox{  and  } \ \
g_o = \left( \begin{array}{rr}
a_o   &  b_o  \\
c_o   & -a_o
\end{array}  \right)
\]
with $a_o^2 + b_oc_o=1$. Then
\[
g_ogg_o^{-1} = g_ogg_o = \left( \begin{array}{cc}
e^{- \frac{ 2 \pi i}{3}} + \sqrt{-3} a_o^2   &  \sqrt{-3} a_ob_o   \\
\mbox{  }  &  \mbox{  }  \\
\sqrt{-3} a_oc_o   &  e^{\frac{ 2 \pi i}{3}} + \sqrt{-3} a_o^2
\end{array} \right) =
\left(  \begin{array}{cc}
e^{- \frac{ 2 \pi i}{3}}   &  0  \\
\mbox{  }  &  \mbox{  }  \\
0  &  e^{\frac{ 2 \pi i}{3}}
\end{array}  \right) = g^{-1}
\]
is equivalent to $a_o=0$ and
\[
g^jg_o = \left( \begin{array}{cc}
e^{\frac{ 2 \pi i j}{3}}   &  0  \\
\mbox{  }  &  \mbox{  }  \\
0  &  e^{ - \frac{ 2 \pi i j}{3}}
\end{array}  \right)
\left(  \begin{array}{cc}
0  &  b_o  \\
\mbox{  }  &  \mbox{  }  \\
\frac{1}{b_o}   & 0
\end{array}  \right) =
\left( \begin{array}{cc}
0  &  e^{\frac{ 2 \pi i j}{3}} b_o  \\
\frac{ e^{ - \frac{ 2 \pi i j}{3}}}{b_o}  &  0
\end{array} \right)
\]
have $\tr( g^jg_o) =0$ for all $0 \leq j \leq 2$. Thus, any $g \in H_{C3}(4) \setminus SL(2,R)$ has an eigenvalue $\lambda _1 (g)=1$.

If $- I_2 \in G$, then for any $g \in G \setminus SL(2,R)$ with $\lambda _1 (g) =1$, $\lambda _2 (g) = \det( g) \in R^* \setminus \{ 1 \}$, one has $ - g \in G \setminus SL(2,R)$ with $\lambda _1 (-g) = -1$, $\lambda _2 (-g) = - \det(g)$. Thus, $-g$ has an eigenvalue $1$ exactly when $\lambda _2 (-g) = - \det(g) =1$ or $\lambda _2 (g) = \det(g) = -1$. In particular,
\[
G = [ G \cap SL(2,R)] \langle g_o \rangle
\]
for some $g_o \in G$ with $\det(g_o)=-1$, $\tr(g_o) =0$ and $G \setminus SL(2,R) = [ G \cap SL(2,R)] g_o$. Thus, for any $g \in G \setminus SL(2,R)$ has $\det(g) = -1$ and $g$ has an eigenvalue $\lambda _1 (g) =1$ exactly when $\tr(g) = 0$.

We claim that $\tr(g_1g_o) = 0$ for all $g_1 \in G \cap SL(2,R)$ and some $g_o \in G$ with $\det(g_o) = -1$, $\tr(g_o) = -1$ requires $G \cap SL(2,R)$ to be a cyclic group. Assume the opposite. Then by Proposition \ref{K}, either $G \cap SL(2,R)$ contains a subgroup
\[
K_4 = \langle g_1, g_2 \ \ \vert \ \ g_1^2 = g_2^2 = -I_2, \ \ g_1 g_2 g_1^{-1} = g_2^{-1} \rangle \simeq {\mathbb Q}_8
\]
isomorphic to the quaternion group ${\mathbb Q}_8$ of order $8$, or
\[
G \cap SL(2,R) = K_7 = \langle g_1, \ \ g_4 \ \ \vert \ \ g_1^2 = g_4^3 = - I_2, \ \ g_1 g_4 g_1 ^{-1} = g_4^{-1} \rangle \simeq {\mathbb Q}_{12}
\]
is isomorphic to the dicyclic group ${\mathbb Q}_{12}$ of order $12$. In either case, one has $h_1, h_2 \in SL(2,R)$ with $\tr(h_1) =0$ and $h_2$ of order $s \in \{ 4, 6 \}$, such that $h_1 h_2 h_1^{-1} = h_2^{-1}$. Let us consider
\[
D_1 = S^{-1}h_1S = \left( \begin{array}{rr}
a_1  &  b_1  \\
c_1  &  -a_1
\end{array}  \right)  \in SL \left( 2, {\mathbb Q} \left( \sqrt{-d}, E^{\frac{ 2 \pi i}{s}} \right)  \right),
\]
\[
D_2 = S^{-1}h_2S = \left( \begin{array}{cc}
e^{\frac{ 2 \pi i}{s}}   &   0  \\
0  &  e^{-\frac{ 2 \pi i}{s}}
\end{array} \right) \ \ \mbox{  and  }
\]
\[
D_o = S^{-1} g_o S = \left( \begin{array}{rr}
a_o  &  b_o   \\
c_o  &  -a_o
\end{array}  \right) \in GL \left( 2,  {\mathbb Q} \left( \sqrt{-d}, e^{\frac{ 2 \pi i}{s}} \right) \right)
\]
with $a_o^2 + b_oc_o =1$. The relation
\[
D_1D_2D_1^{-1} = - D_1 D_2 D_1 = \left(  \begin{array}{ccc}
e^{- \frac{ 2 \pi i}{s}} - 2i Im \left( e^{\frac{ 2 \pi i}{s}} \right) a_1^2    & \mbox{  }  & - 2i Im \left( e^{\frac{ 2 \pi i}{s}} \right) a_1b_1  \\
\mbox{  }   & \mbox{  }  & \mbox{  } \\
- 2i Im \left( e^{\frac{ 2 \pi i}{s}} \right) a_1c_1   & \mbox{  }  &  e^{\frac{2 \pi i}{s}} + 2 i Im \left(  e^{\frac{ 2 \pi i}{s}} \right) a_1^2
\end{array}  \right) =
\]
\[
= \left( \begin{array}{cc}
e^{ - \frac{ 2 \pi i}{s}}   &   0  \\
\mbox{  }   & \mbox{  }  \\
0  &  e^{\frac{ 2 \pi i}{s}}
\end{array} \right) = D_2^{-1}
\]
requires $a_1=0$ and
\[
D_1 = \left(  \begin{array}{cc}
0  &  b_1  \\
- \frac{1}{b_1}  & 0
\end{array}  \right) \ \ \mbox{  for some  } \ \ b_1 \in {\mathbb Q} \left( \sqrt{-d}, e^{\frac{ 2 \pi i}{s}} \right).
\]
Now,
\[
\tr(D_2D_o) = \tr \left( \begin{array}{cc}
e^{\frac{ 2 \pi i}{s}} a_o  &  e^{\frac{ 2 \pi i}{s}} b_o  \\
 \mbox{  }  & \mbox{  } \\
 e^{- \frac{ 2 \pi i}{s}} c_o  &  - e^{ - \frac{ 2 \pi i}{s}} a_o
 \end{array}  \right) = 2 i Im \left( e^{\frac{ 2 \pi i}{s}} \right) a_o = 0
 \]
 specifies the vanishing of $a_o$, whereas
 \[
 D_o = \left( \begin{array}{cc}
 0  &  b_o  \\
 \frac{1}{b_o}   &  0
 \end{array}  \right) \ \ \mbox{  for some  } \ \  b_o \in {\mathbb Q} \left( \sqrt{-d} e^{\frac{ 2 \pi i}{s}} \right).
 \]
 The condition
 \[
 \tr(D_1D_o) = \tr \left( \begin{array}{cc}
 \frac{b_1}{b_o}   &  0  \\
 0  &   - \frac{b_o}{b_1}
 \end{array}  \right) = \frac{ b_1}{b_o} - \frac{b_o}{b_1} = 0
 \]
requires $b_1 = \varepsilon b_o$ for some $\varepsilon \in \{ \pm \}$ and
\[
\tr ( D_1 D_2 D_o) = \tr \left(  \begin{array}{rr}
\varepsilon e^{ - \frac{ 2 \pi i}{s}}  &  0  \\
mbox{  }   &  \mbox{  }  \\
0  &  - \varepsilon e^{\frac{ 2 \pi i}{s}}
\end{array}  \right) = - \varepsilon \left( e^{\frac{ 2 \pi i}{s}} - e^{ - \frac{ 2 \pi i}{s}} \right) =
- 2i Im \left(  e^{\frac{ 2 \pi i}{s}} \right) \varepsilon \neq0
\]
contradicts the assumption. Therefore $G \cap SL(2,R) = \langle g \rangle \simeq  {\mathbb C}_s $ is cyclic group of order $s \in \{ 2, 4, 6 \}$.
If $G = [ G \cap SL(2,R)] \langle g_o \rangle$ has a normal subgroup $G \cap SL(2,R) = \langle g \rangle \simeq {\mathbb C}_2$ then $g = -I_2$ and $g_o (-I_2) = (-I_2) g_o$, as far as $-I_2$ is a scalar matrix. As a result, $G = \langle g \rangle  \times \langle g_o \rangle \simeq {\mathbb C}_2 \times {\mathbb C}_2$. For $G = [ G \cap SL(2,R)] \langle g_o \rangle$ with a normal subgroup $G \cap SL(2,R) = \langle g \rangle \simeq {\mathbb C}_s$ of order $\{ 4, 6 \}$ note that the element $g_ogg_o^{-1}$ of $\langle g \rangle$  is of order $s$, so that either $g_o g g_o^{-1} =g$ or $g_o g g_o^{-1} = g^{-1}$, according to ${\mathbb Z}_4^* = \{ \pm 1({\rm mod}4) \}$, ${\mathbb Z}_6 ^* = \{ \pm 1({\rm mod}6) \}$. If $g_o g = g g_o$ then there exists a matrix $S \in GL \left( 2,  {\mathbb Q} \left( \sqrt{-d}, e^{\frac{ 2 \pi i}{s}} \right) \right)$, such that
\[
D = S^{-1} g S = \left( \begin{array}{cc}
e^{\frac{ 2 \pi i}{s}}   &  0  \\
0  &  e^{ - \frac{ 2 \pi i}{s}}
\end{array}  \right) \ \ \mbox{  and  } \ \
D_o = S^{-1} g_oS = \left( \begin{array}{rr}
1   &  0  \\
0  &  -1
\end{array} \right)
\]
are diagonal. Then $\tr( g g_o) = \tr( D D_o) = e^{\frac{ 2 \pi i}{s}} - e^{ - \frac{ 2 \pi i}{s}} = 2 i Im \left( e^{\frac{ 2 \pi i}{s}} \right) \neq 0$ and $1$ is not an eigenvalue of $g g_o$. Therefore $g_o g g_o^{-1} = g^{-1}$. If
\[
D = S^{-1} g S = \left( \begin{array}{cc}
e^{\frac{ 2 \pi i}{s}}   &  0  \\
0   & e^{ - \frac{ 2 \pi i}{s}}
\end{array}  \right)  \ \ \mbox{   and   }
\]
\[
D_o = S^{-1} g_o S = \left( \begin{array}{rr}
a_o   &  b_o  \\
C_o   & -a_o
\end{array}  \right) \in GL \left( 2, {\mathbb Q} \left( \sqrt{-d} , e^{\frac{ 2 \pi i}{s}} \right) \right) \ \ \mbox{  with  } \ \ a_o^2 + b_oc_o =1,
\]
then the relation
\[ D_o D D_o^{-1} = D_o D D_o = \left( \begin{array}{ccc}
e^{ - \frac{ 2 \pi i}{s}} + 2 i Im \left( e^{\frac{ 2 \pi i}{s}} \right) a_o^2   & \mbox{  }   & 2 i Im \left( e^{\frac{ 2 \pi i}{s}} \right) a_o b_o \\
\mbox{  }  & \mbox{  }  & \mbox{  } \\
2 i Im \left( e^{\frac{ 2 \pi i}{s}} \right) a_o c_o   & \mbox{  }  &  e^{\frac{ 2 \pi i}{s}} - 2i Im \left( e^{\frac{ 2 \pi i}{s}} \right) a_o^2
\end{array} \right) =
\]
\[
= \left( \begin{array}{cc}
e^{ - \frac{2 \pi i}{s}}   & 0  \\
0  &  e^{\frac{ 2 \pi i}{s}}
\end{array}  \right) = D^{-1}
\]
specifies that $a_o=0$ and
\[
D_o = \left( \begin{array}{cc}
0  &  b_o  \\
\frac{1}{b_o}   &  0  \\
\end{array}   \right) \ \ \mbox{  for some } \ \ b_o \in {\mathbb Q} \left(  \sqrt{-d}, e^{\frac{2 \pi i}{s}} \right).
\]
The non-trivial coset
\[
S^{-1}GS \setminus SL \left( 2, {\mathbb Q} \left( \sqrt{-d}, e^{\frac{2 \pi i}{s}} \right) \right) =
\langle D \rangle D_o = \{ D^j D_o \ \ \vert \ \ 0 \leq j \leq s-1 \}
\]
consists of elements of trace
\[
\tr ( D^jD_o) = \tr \left( \begin{array}{cc}
0   &  e^{\frac{ 2 \pi i j}{s}} b_o  \\
\frac{e^{- \frac{ 2 \pi i j}{s}}}{b_o}  & 0
\end{array} \right) = 0,
\]
so that any $\Delta \in S^{-1}GS \setminus SL \left( 2, {\mathbb Q} \left( \sqrt{-d}, e^{\frac{ 2 \pi i}{s}} \right) \right)$ has an eigenvalue $1$ and any $g = S \Delta S^{-1} \in G \setminus SL(2,R)$ has an eigenvalue $1$.

\end{proof}

\begin{proposition}    \label{EnriquesAH}
The quotient $A/H$ of $A = E \times E$ is an Enriques surface if and only if $H$ is generated by $h \in H$  of order $s \in \{ 2, 3, 4, 6 \}$ with $\mathcal{L}(h) \in SL(2,R)$ and $h_o \in H$ with $\lambda _1 \mathcal{L}(h_o) =1$, $\lambda _2 \mathcal{L}(h_o) = -1$, $\tau (h_o) = h_o \mathcal{L}(h_o) ^{-1} = \tau _{(U_o,V_o)}$, subject to $h_o h h_o^{-1} = h_o h h_o = h^{-1}$ and
\begin{equation}   \label{Assumption}
\mathcal{L} ( h_o)
\left(  \begin{array}{cc}
U_o  \\
V_o
\end{array}  \right) \neq
- \left(  \begin{array}{cc}
U_o  \\
V_o
\end{array}  \right).
\end{equation}
In particular, for $s=2$ the group
\[
H \simeq \mathcal{L}(H) \simeq {\mathbb C}_2 \times {\mathbb C}_2
\]
is isomorphic to the Klein group of order $4$, while for $s \in \{ 3, 4, 6 \}$ one has a dihedral group
\[
H \simeq \mathcal{L}(H) \simeq \mathcal{D}_s = \langle a, \ \ b \ \ \vert \ \  a^s=1, \ \ b^2=1, \ \ bab^{-1} = a^{-1} \rangle
\]
of order $2s$.
\end{proposition}

\begin{proof}

According to Lemmas \ref{GeneratingFiniteSubgroupsAutA} and \ref{PreEnriquesH}, the finite subgroups $H$ of $Aut (E \times E)$ with Enriques quotient $A/H$ are of the form
\[
H = \langle \tau _{(P_i, Q_i)}, \ \ h, \ \ h_o \ \ \vert \ \  1 \leq i \leq m \rangle
\]
with $0 \leq m \leq 3$ and
\[
\mathcal{L}(H) = \langle \mathcal{L}(h), \ \ \mathcal{L}(h_o) \ \ \mathcal{L}(h) ^s = I_2, \ \ \mathcal{L}(h_o) ^2 = I_2, \ \ \mathcal{L}(h_o) \mathcal{L}(h) \mathcal{L}(h_o)^{-1} = \mathcal{L}(h_)^{-1} \simeq \mathcal{D}_s
\]
for some $\mathcal{L}(h) \in SL(2,R)$, $\mathcal{L}(h_o) \in GL(2,R)$, $\lambda _1 \mathcal{L}(h_o) =1$, $\lambda _2 \mathcal{L}(h_o) = -1$. Note that
\[
K := \mathcal{L}^{-1} ( \mathcal{L}(H) \cap SL(2,R)) = \langle \tau _{(P_i,Q_i)} \ \ \vert \ \  1 \leq i \leq m \rangle \langle h \rangle
\]
is a normal subgroup of $H$ with a single non-trivial coset
\[
H \setminus K = K h_o =
 \left \{ \tau _{  h(z,j) =\sum\limits _{i=1}^m z_i (P_i,Q_i)}  h^j h_o \ \ \vert \ \ z_i \in {\mathbb Z}, \ \ 0 \leq j \leq s-1 \right \}.
\]
The automorphism $h$, whose linear part $\mathcal{L}(h)$ has eigenvalues $\lambda _1 \mathcal{L}(h) = e^{\frac{2 \pi i}{s}}$, $\lambda _2 \mathcal{L}(h) = e^{ - \frac{ 2 \pi i}{s}}$, different from $1$ has always a fixed point on $A$. Without loss of generality, one can assume that $h = \mathcal{L}(h) \in GL(2,R)$, after moving the origin of $A$ at a fixed point of $h$. If $h_o = \tau _{(U_o,V_o)} \mathcal{L}(h_o)$ for some $(U_o,V_o) \in A$ then the translation parts
\[
\tau ( h(z,j)) = h(z,j) \mathcal{L}( h(z,j))^{-1} = \tau _{ \sum\limits _{i=1} ^m z_i (P_i,Q_i) + h^j (U_o,V_o)} \ \ \mbox{  for } \ \
\forall z = (z_1, \ldots , z_m) \in {\mathbb Z}^m
 \]
and $ 0 \leq j \leq s-1$. The linear parts $\mathcal{L} ( h(z,j)) = \mathcal{L}(h^jh_o) = h^j \mathcal{L}(h_o)$ have eigenvalues $\lambda _1 ( h^j \mathcal{L}(h_o)) =1$, $\lambda _2 ( h^j \mathcal{L}(h_o)) = -1$ for all $0 \leq j \leq s-1$. Applying Lemma \ref{FixedPointFree}, one concludes that $Fix _A(h(z,j)) = \emptyset$ if and only if no one lifting $(x(z,j), y(z,j)) \in {\mathbb C}^2$ of $\tau (h(z,j))$ is in the kernel of the linear operator $\psi _j = h^j \mathcal{L} ( h_o) + I_2 : {\mathbb C}^2 \rightarrow {\mathbb C}^2$. For any fixed $0 \leq j \leq s-1$, note that $(x(z,j), y(z,j)) \not \in \ker( \phi _j)$ for all $z=(z_1, \ldots , z_m) \in {\mathbb Z}^m$ implies that the lifting of the ${\mathbb R}$-span of $\langle \tau _{(P_i,Q_i)} \ \ \vert \ \  1 \leq i \leq m \rangle$ to ${\mathbb C}^2$ is parallel to $\ker( \psi _j)$. It suffices to establish that $\ker( \psi _0) \cap \ker( \psi _1) = \{ (0,0) \}$, in order to conclude that $m=0$ and
$
H = \langle h, \ \ h_o \rangle = \langle h_o, \ \ h \rangle .
$
 Since the claim $\ker( \psi _0) \cap \ker( \psi _1) = \{ (0,0) \}$ is independent on the choice of a coor\-di\-nate system on ${\mathbb C}^2$, one can use Lemma \ref{PreEnriquesH} to assume that
 \[
 \mathcal{L}(h_o) = D_o = \left( \begin{array}{cc}
 0   &  b_o  \\
 \frac{1}{b_o}   & 0
 \end{array}  \right) \ \ \mbox{  and  } \ \
 h = \mathcal{L}(h) = \left( \begin{array}{cc}
 e^{\frac{ 2 \pi i}{s}}   &  0  \\
 0  &  e^{ - \frac{2 \pi i}{s}}
 \end{array}  \right)
 \]
 for some $s \in \{ 2, 3, 4, 6 \}$. Then $\psi _0 = \mathcal{L}(h_o) + I_2$ has kernel $\ker( \psi _0) =  {\rm Span} _{\mathbb C} (b_o, -1)$, while
 \[
 \psi _1 = h \mathcal{L}(h_o) + I_2 = \left( \begin{array}{cc}
 1   &  e^{\frac{ 2 \pi i}{s}} b_o  \\
 e^{ - \frac{ 2 \pi i}{s}} b_o^{-1}   &  1
 \end{array}  \right)
 \]
has kernel $\ker( \psi _1) = {\rm Span} _{\mathbb C} \left( e^{\frac{ 2 \pi i}{s}} b_o, -1 \right)$. For $s \in \{ 2, ,3 4, 6 \}$ the vectors $(b_o,-1)$ and $\left( e^{\frac{ 2 \pi i}{s}} b_o, -1 \right)$ are linearly independent over ${\mathbb C}$, so that $\ker( \psi _0) \cap \ker( \psi _1 ) = \{ (0,0) \}$.
Now, $\mathcal{L}( h^j h_o) = h^j \mathcal{L}(h_o) \neq I_2$ for any $0 \leq j \leq s-1$, as far as $\mathcal{L}(h_o) \not \in \langle h \rangle < SL(2,R)$.  On the other hand, the subgroup $\langle h = \mathcal{L}(h) \rangle$ of $H$ is contained in $SL(2,R)$, so that the translation part $\mathcal{T}(H) = \ker ( \mathcal{L} \vert _H ) = Id _A$ is trivial. As a result, $\mathcal{L} : H \rightarrow \mathcal{L}(H)$ is a group isomorphism and the relation $\mathcal{L}(h_o) h \mathcal{L}(h_o)^{-1} = h^{-1}$ implies that
\[
h_o h h_o^{-1} = \left( \tau _{(U_o,V_o)} \mathcal{L}(h_o) \right) h \left( \tau _{ - \mathcal{L} (h_o)^{-1} (U_o,V_o)} \mathcal{L}(h_o) ^{-1} \right) =
\]
\[
= \tau _{(U_o,V_o) - \mathcal{L}(h_o) h \mathcal{L}(h_o) ^{-1} (U_o,V_o)} [ \mathcal{L}(H_o) h \mathcal{L}(h_o) ^{-1}] =
\tau _{(U_o,V_o) - h^{-1} (U_o,V_o)} h^{-1} = h^{-1}.
\]
After acting by $h$ on $(U_o,V_o) = h^{-1} (U_o,V_o)$, one obtains that $h(U_o,V_o) = (U_o,V_o)$, or $(U_o,V_o) \in A$ is a fixed point of $h$. Bearing in mind that $K = \langle h \rangle \simeq \langle \mathcal{L}(h) \rangle = \mathcal{L}(H) \cap SL(2,R)$ is a normal subgroup of $H \simeq \mathcal{L}(H) = [ \mathcal{L}(H) \cap SL(2,R) ] \langle \mathcal{L}(h_o) \rangle$, let us represent the complement $H \setminus K$  as the set of the entries of the left coset
\[
H \setminus K = h_o K = \{ h_o h^j \ \ \vert \ \  0 \leq j \leq s-1 \}.
\]
Then $h_o h^j = \tau _{(U_o,V_o)} ( \mathcal{L}(h_o) h^j)$ have translation parts
\[
\tau ( h_o h^j) = h_o h^j \mathcal{L} (h_o h^j) ^{-1} = h_o \mathcal{L}(h_o) ^{-1} = \tau ( h_o) = \tau _{(U_o, V_o)}
\]
and linear parts $\mathcal{L}(h_o) h^j$ with eigenvalues $\lambda _1 ( \mathcal{L}(h_o) h^j) =1$, $\lambda _2 ( \mathcal{L}(h_o) h^j ) = -1$. According to Lemma \ref{FixedPointFree}, the automorphism $h_o h^j \in Aut(A)$ has no  fixed point on $A$ if and only if no one lifting $(u_o,v_o) \in {\mathbb C}^2$ of $(u_o + \pi _1(E), v_o + \pi _1 (E)) = (U_o,V_o)$ is in the kernel of $\varphi _j = \mathcal{L} (h_o) h^j + I_2$. We claim that if
\[
h \left( \begin{array}{c}
u_o  \\
v_o
\end{array}  \right) =
\left( \begin{array}{c}
u_o  \\
v_o
\end{array}  \right) +
\left( \begin{array}{c}
\mu _1  \\
\mu _2
\end{array}   \right)  \ \ \mbox{ for some  } \ \ (\mu _1, \mu _2) \in \pi _1 (A),
\]
then $\varphi _j (u_o,v_o) - \varphi _0 (u_o,v_o) \in \pi _1(A)$. Indeed, by an induction on $j$, one has
\[
h^j \left( \begin{array}{c}
u_o  \\
v_o
\end{array}  \right) -
\left( \begin{array}{c}
u_o  \\
v_o
\end{array}  \right) \in \pi _1 (A),
\]
whereas
\[
\varphi _j (u_o,v_o) - \varphi _0 (u_o,v_o) = \mathcal{L} (h_o) h^j
\left( \begin{array}{c}
u_o  \\
v_o
\end{array}  \right) - \mathcal{L}(h_o)
\left( \begin{array}{c}
u_o  \\
v_o
\end{array}  \right) \in \pi _1 (A).
\]
Thus, the assumption $(u_o,v_o) \in \ker( \varphi _j)$ implies that
\[
\varphi _0 (u_o,v_o) = \mathcal{L}(h_o) (u_o,v_o) + (u_o,v_o) = ( \mu '_1, \mu '_2) \in \pi _1 (A),
\]
whereas
\[
\mathcal{L}(h_o)
\left( \begin{array}{c}
U_o  \\
V_o
\end{array}  \right) =
 - \left( \begin{array}{c}
U_o  \\
V_o
\end{array}  \right),
\]
contrary to the assumption (\ref{Assumption}). Note that (\ref{Assumption}) is equivalent to $\varphi _0 (u_o,v_o) \not \in \pi _1 (A)$ for all liftings $(u_o,v_o) \in {\mathbb C}^2$ of $( u_o + \pi _1 (E), v_o + \pi _1 (E)) = (U_o,V_o)$ and is slightly stronger than $Fix _A (h_o) = \emptyset$, which amounts to
$\varphi _0 (u_o,v_o) \neq 0$ for $\forall (u_o,v_o) \in {\mathbb C}^2$ with $(u_o + \pi _1 (E), v_o + \pi _1 (E)) = (U_o,V_o)$.

\end{proof}

\newpage



\end{document}